\documentclass[11pt]{amsart}

\usepackage{palatino, mathpazo}
\usepackage[mathscr]{eucal}
\usepackage{amssymb}
\usepackage[all]{xy}
\usepackage{graphicx}
\usepackage{texdraw}
\usepackage{color}

\usepackage{etoolbox}
\patchcmd{\section}{\scshape}{\bfseries}{}{}
\makeatletter
\renewcommand{\@secnumfont}{\bfseries}
\makeatother

\numberwithin{equation}{subsection}

\newtheorem{theorem}[subsection]{Theorem}
\newtheorem{lemma}[subsection]{Lemma}
\newtheorem{proposition}[subsection]{Proposition}
\newtheorem{corollary}[subsection]{Corollary}

\newtheorem{theoremss}[subsubsection]{Theoremss}
\newtheorem{propositionss}[subsubsection]{Proposition}
\newtheorem{lemmass}[subsubsection]{Lemma}
\newtheorem{corollaryss}[subsubsection]{Corollary}
\newtheorem{conjecturess}[subsubsection]{Conjecture}

\theoremstyle{definition}

\newtheorem{remark}[subsection]{Remark}
\newtheorem{example}[subsection]{Example}
\newtheorem{definitionss}[subsubsection]{Definition}
\newtheorem{remarkss}[subsubsection]{Remark}
\newtheorem{exampless}[subsubsection]{Example}

\newtheorem*{remark*}{Remark}

\newcommand{\p}{\partial}

\newcommand{\f}{{\bf f}}

\def\NN{\mathbb{N}}
\def\QQ{\mathbb{Q}}
\def\ZZ{\mathbb{Z}}
\def\RR{\mathbb{R}}
\def\CC{\mathbb{C}}
\def\PP{\mathbb{P}}
\def\HH{\mathbb{H}}
\def\ringO{\mathscr{O}}

\newcommand\myatop[2]{\genfrac{}{}{0pt}{}{#1}{#2}}

\begin{document}

\marginpar{\tiny version 12/25/2014}

\title[Mean field equations, hyperelliptic curves
and modular forms: I]{Mean field equations, hyperelliptic\\ curves
and modular forms: I}
\author{Ching-Li Chai}\thanks{\tiny C.-L.\,C.\ would like to acknowledge
the support of NSF grants
DMS-0901163 and DMS-1200271.}
\address{Department of Mathematics, University of Pennsylvania, Philadelphia}
\email{chai@math.upenn.edu}
\author{Chang-Shou Lin}
\address{Department of Mathematics and Center for Advanced Studies 
in Theoretic Sciences (CASTS), National Taiwan University, Taipei}
\email{cslin@math.ntu.edu.tw}
\author{Chin-Lung Wang}\thanks{\tiny C.-L.\,W. is supported by 
MOST grant NSC 100-2115-M-002-001-MY3}
\address{Department of Mathematics and Taida Institute of
Mathematical Sciences (TIMS), National Taiwan University, Taipei}
\email{dragon@math.ntu.edu.tw}

\begin{abstract}
We develop a theory connecting the following three areas:
(a) the mean field equation (MFE)
\begin{equation*}
\triangle u + e^u = \rho\, \delta_0, \qquad\rho \in \mathbb R_{>0}
\end{equation*}
on flat tori \(E_\tau  = \mathbb C/(\mathbb Z + \mathbb Z\tau)\), 
(b) the classical Lam\'e equations and (c) modular forms. 
A major theme in part I is a classification of 
developing maps $f$ attached to solutions $u$ of 
the mean field equation according to the
type of transformation laws (or monodromy) 
with respect to \(\Lambda\)
satisfied by \(f\).

We are especially interested in the case when the
parameter \(\rho\) in the mean field equation is an integer
multiple of \(4\pi\).
In the case when $\rho = 4\pi(2n + 1)$ 
for a non-negative integer 
\(n\), we prove that the number of solutions is
$n + 1$ 
except for a finite number 
of conformal isomorphism classes of flat tori, 
and we give a family of polynomials 
which characterizes the developing maps for
solutions of mean field equations
through the configuration of their zeros and poles.
Modular forms appear naturally already in the simplest situation 
when
$\,\rho=4\pi$

In the case when $\rho = 8\pi n$ for a positive integer \(n\),
the solvability of the MFE depends on the  
\emph{moduli} of the flat tori
$E_\tau$ and leads naturally to a hyperelliptic curve  
$\bar X_n=\bar X_{n}(\tau)$
arising from the Hermite-Halphen ansatz solutions
of Lam\'e's differential equation
\[\frac{d^2 w}{dz^2}-(n(n+1)\wp(z;\Lambda_{\tau}) + B) w=0.\]
We analyse the curve $\bar X_n$ from both the analytic and the
algebraic perspective, including its local coordinate 
near the point at infinity, which turns out to be 
a smooth point of \(\bar{X}_n\).
We also specify the role of the branch points of the
hyperelliptic projection
$\bar X_n \to \mathbb P^1$ when the parameter $\rho$ varies 
in a neighborhood of $\rho = 8\pi n$. 
In part II, we study a ``pre-modular form'' 
$Z_n(\sigma; \tau)$, a real-analytic function in two
variables associated to $\bar X_n(\tau)$, 
which has many symmetries and also 
the property that the \(\tau\)-coordinates of
zeros of $Z_n(\sigma; \tau)$
correspond exactly to those flat tori where the MFE with parameter 
\(\rho=8\pi n\) has a solution. 
\end{abstract}

\maketitle
\small
\tableofcontents
\normalsize
\setcounter{section}{-1}

\section{Introduction} \label{intro}
\setcounter{equation}{0}

\subsection{}
How to study the \emph{geometry} of a flat torus 
$E = E_\Lambda := \mathbb C/\Lambda$? 
There are at least two
seemingly different 
approaches
to this problem. In the first approach one studies
the \emph{Green's function} $G=G_E$ of \(E\),
characterized by 
\begin{equation}\label{eqn-green}
\begin{cases}
\displaystyle -\triangle G = \delta_0 - \frac{1}{|E|} & 
\textup{on } \,E, \\
\int_{E} G = 0,
\end{cases}
\end{equation}
where the Laplacian \(\triangle = \frac{\partial^2}{\partial x^2}
+\frac{\partial^2}{\partial y^2}
= 4\,\partial_z\partial_{\bar z}\) on \(E\)
is induced by the Laplacian of the covering space \(\CC\) of \(E\),
\(\,|E|=\int_E\,dx\, dy 
=\frac{\sqrt{-1}}{2} \int_{\CC/\Lambda}dz\wedge d\bar{z}\,\)
is the area of \(E\), 
$\delta_0$ is the Dirac delta measure at the zero point 
$[0]= 0\,\textup{mod}\,\Lambda \in E$,
and we have identified functions on \(E\) with measures on \(E\)
using the Haar measure \(\,dx\,dy = \frac{1}{2}\,\vert dz d\bar{z}\vert\)
on \(E\).\footnote{So the first equation in (\ref{eqn-green})
means that 
\[
-\int_E\, G\cdot \triangle(f)\,dx\,dy
= f(0) - \vert E\vert^{-1}\cdot \int_E\, f\,dx\,dy
\]
for all smooth functions \(f\in C^{\infty}(E)\).
}
In the second approach one studies the classical \emph{Lam\'e equation}
\begin{equation} \label{Lame-eq-n}
L_{\eta, B}\,w := w'' - (\eta(\eta + 1) \wp(z) + B) w = 0
\end{equation}
with parameters\footnote{We imposed the
condition \(\,\eta>0\,\)
on the real parameter \(\eta\) of the Lam\'e equation
\(L_{\eta, B}\) so that these Lam\'e equations are related to
the mean field equations.  
The parameter \(\eta\) are positive integers
in classical literature such as 
\cite{Halphen, Whittaker}.} 
\(\eta \in \mathbb R_{> 0}\) and \(B\in \CC\),
where $\wp(z)=\wp(z;\Lambda)$ 
is the Weierstrass elliptic function 
$$
\wp(z;\Lambda) = \frac{1}{z^2} + \sum_{w \in \Lambda \smallsetminus \{0\}} 
\Big( \frac{1}{(z - w)^2} - \frac{1}{w^2}\Big), \quad z \in\CC.
$$

Throughout this paper, we denote by \(\omega_1, \omega_2\) a \(\ZZ\)-basis of the lattice \(\Lambda\), $\tau = \omega_2/\omega_1$ with \(\textup{Im}(\tau) > 0\), and \(\omega_3= -\omega_1-\omega_2\).

\subsubsection{}\label{subsec:thmA}
The Green's function is closely related to the so-called 
\emph{concentration phenomenon} of some non-linear elliptic 
partial differential equations in two dimensions. 
For example, consider the following \emph{singular Liouville equation}
with parameter \(\rho\in \RR_{>0}\)\footnote{Equation
(\ref{Liouville-eq}) with parameter \(\rho<0\) is 
not very interesting---it becomes too easy.}
\begin{equation} \label{Liouville-eq}
\triangle u + e^u = \rho\cdot \delta_0 \quad \textup{on } \,E.
\end{equation}
It is proved in \cite{CL0} that for a sequence of 
blow-up solutions $u_k$ of (\ref{Liouville-eq}) 
corresponding to $\rho = \rho_k$ 
with $\rho_k \to 8\pi n$, $n \in \mathbb Z_{>0}$, the set $\{p_i, \ldots, p_n\}$ of blow-up points satisfies the following equations:
\begin{equation} \label{Green-eq}
n \frac{\partial G}{\partial z} (p_i) = 
\sum_{1\leq j \leq n, \,j\ne i} \frac{\partial G}{\partial z} (p_i - p_j) 
\qquad \forall i=1,\ldots, n.
\end{equation}
For $n = 1$, the blow-up set consists of only one point $p$ which 
by (\ref{Green-eq}) is a critical point of $G$:
\begin{equation} \label{Green-critical}
\frac{\partial G}{\partial z}(p) = 0.
\end{equation}
Such a connection of (\ref{Green-eq}) with the Green's function also
appears in many gauge field theories in physics. The well-known
examples are the Chern--Simons--Higgs equation for the abelian case,
and the ${\rm SU}(m)$ Toda system for the non-abelian case. See
\cite{LinYan, NT, NT1} and
references therein.

This connection leads to the following question: How many solutions 
might the system (\ref{Green-eq}) have? Or an even more basic
question: 
\begin{quote}
\emph{How many critical points might the Green function $G$ have?} 
\end{quote}
Surprisingly, this problem has never been answered until \cite{LW}, 
where the second and the third authors proved the following result.

\bigbreak
\noindent
{\bf Theorem A.}\enspace \emph{For any flat torus $E$, 
the Green function $G$ has either three or five critical points.}
\smallskip

The statement of Theorem A looks deceptively simple
at first sight. However its proof uses the non-linear PDE  
(\ref{Liouville-eq}) and is not elementary.

\subsubsection{}
In view of Theorem A 
it is natural to study the system (\ref{Green-eq}) 
and the degeneracy question related
to each solution of it. 
We will see in a moment that such an investigation leads naturally to
a fundamental hyperelliptic curve $\bar X_n$ (which varies with
\(n\) and \(E\)), and requires us to
study its geometry---especially the branch points for the 
hyperelliptic structural map
$\bar X_n \to \mathbb P^1(\mathbb{C})$. 
The precise definition of $\bar X_n$ will later be given in (\ref{X-n}).

In the literature there are at least two situations in which one 
encounters this hyperelliptic curve $\bar X_n$. Both of them are related 
to the Lam\'e equation (\ref{Lame-eq-n}) with $\eta \in \mathbb N_{>0}$. 
One of them is well-known, namely the spectral curve of the Lam\'e 
equation in the KdV theory. In \S \ref{hyper-ell-str} we will prove that
this spectral curve is identical to the hyperelliptic curve $\bar X_n$;
see Remark \ref{KdV}. 

The second situation has a more algebraic flavor and is perhaps less 
known to the analysis community. Let 
\[p(x)=4 x^3 - g_2(\Lambda) x - g_3(\Lambda).\]
The differential equation
\begin{equation}\label{Lame-eq-alg-form}
p(x)\frac{d^2y}{dx^2}+\frac{1}{2} p'(x)\frac{dy}{dx} -
\big(\eta(\eta+1)x+B\big)y=0 
\end{equation}
on \(\PP^{1}(\CC)\) is
related to (\ref{Lame-eq-n}) by the change of variable
\(x=\wp(z;\Lambda)\), where 
\(p'(x)=\frac{d}{dx}p(x)\).
People have been interested in 
describing those parameters $B$ so that 
the Lam\'e equation (\ref{Lame-eq-alg-form}) 
has algebraic solutions only, 
or equivalently the global monodromy group 
of (\ref{Lame-eq-alg-form}) is \emph{finite}. 
This question seems not to have been
fully solved, though significant progress had been achieved
and there are algorithms to generate all cases; 
see \cite{BW, Waall, Dahmen} and references therein.

\subsubsection{}
A related and even more
classical question 
is: 
\begin{quote}\emph{When is the global monodromy group 
of the Lam\'e equation} (\ref{Lame-eq-alg-form}) \emph{reducible?} 
\end{quote}
That is, there is a one 
dimensional subspace
$\mathbb C\cdot y_1(x)$ of the space of local solutions which is 
stable under the action of the fundamental group 
\[\pi_1\big(\PP^1(\CC)\smallsetminus
\{\wp(z_1;\Lambda)\mid z_1\in \tfrac{1}{2}\Lambda/\Lambda\}\big)\]
of the complement of the \(4\) singular points of
(\ref{Lame-eq-alg-form}) on \(\PP^1(\CC)\).
This question has a fairly complete answer:
\begin{itemize}
\item[(i)] It is known that if the global monodromy group is reducible
then the parameter \(\eta\) of the Lam\'e
equation is an integer and the monodromy group is
infinite.
\item[(ii)] If $\CC\cdot y_1(x)$ is a one-dimensional space
of solutions of (\ref{Lame-eq-alg-form}), then
\(y_1(x)\) is  an algebraic function 
{in \(x\)}.
Moreover solutions which are not multiples of
\(y_1(x)\) are not algebraic in \(x\), 
and the action of the monodromy group of (\ref{Lame-eq-alg-form}) 
is not completely reducible (as a two-dimensional
linear representation of the fundamental group).
\end{itemize} 
See \cite[\S4.4]{Waall} and \cite[\S3]{BW}. A solution of
(\ref{Lame-eq-n}) of the form \(\,y_1\big(\!\wp(z;\Lambda)\!\big)\,\)
with \(\eta=n\) is traditionally known as a 
\emph{Lam\'e function}.\footnote{When
\(n\) is even, $1+n/2$ of the \(2n+1\) Lam\'e functions are
polynomials of degree \(n/2\) in \(\wp(z)\), and \(3n/2\) of the form 
\(\sqrt{(\wp(z)-e_i)(\wp(z)-e_j)}
\cdot Q(\wp(z))\)
for some polynomial \(Q(x)\) of degree \((n/2)-1\) and 
some \(i\neq j\) with \(i,j=1, 2\) or \(3\), where
\(e_i:=\wp(\omega_i/2)\) for \(i=1, 2\) or \(3\). 

When \(n\) is odd, \(3(n+1)/2\) of the \(2n+1\) Lam\'e functions are of the form
\(\sqrt{\wp(z)-e_i}
\cdot Q(\wp(z))\) for some polynomial \(Q(x)\) of
degree \((n-1)/2\) and
\(i =1, 2\) or \(3\). The rest \((n-1)/2\) 
Lam\'e functions are of the form
\(\wp'(z)\cdot 
Q(\wp(z))
=\sqrt{(\wp(z)-e_1)(\wp(z)-e_2)(\wp(z)-e_3)}\cdot Q(\wp(z))
\)
for some polynomial \(Q(x)\) of degree \((n-3)/2\).
} 

\subsubsection{}
For each fixed torus $E$ and each $n \in \mathbb Z_{>0}$,
those parameters $B$ such that the monodromy group of
the Lam\'e equation \(L_{n,B}\) is reducible are
characterized by the classical Theorem B below. 

\medbreak\noindent
{\bf Theorem B.} \cite{Halphen, Whittaker, Poole} \emph{Suppose that
$\eta = n \in \mathbb N$. Then there is a polynomial $\ell_n(B)$ of degree
$2n + 1$ in $B$ such that equation} (\ref{Lame-eq-n}) 
\emph{has a Lam\'e
function as its solution if and only if $\ell_n(B) = 0$.}
\bigskip 

It turns out that the 
algebraic curve 
\[\{(B, C) \mid C^2 = \ell_n(B)\}\] is identical to 
``the affine part'' $Y_n$
of a complete hyperelliptic curve $\bar X_n$, to be
defined later in (\ref{Yn1}). 
Such an identification was only implicitly stated 
in Halphen's classic \cite{Halphen}.
In this paper we will give a detailed and rigorous proof of
the statement; see Theorem \ref{hyp-ell-thm}.
See also \cite[Ch.\,12]{Halphen}, \cite[Ch.\,23]{Whittaker} 
and \cite[Ch.\,9]{Poole} for traditional treatments of the Lam\'e 
equation. 

\subsubsection{}
The main theme of this paper is to explore the 
connection between
the Liouville equation (\ref{Liouville-eq}) and the Lam\'e equation
(\ref{Lame-eq-n}) when the parameter \(\rho\) in (\ref{Liouville-eq})
and the parameter \(\eta\) in (\ref{Lame-eq-n}) satisfies the
linear relation \(\eta =\rho/8\pi\).\footnote{In this article
\(2\eta=\rho/4\pi\in\ZZ_{n>0}\) most of the time.}
Equation (\ref{Liouville-eq}) has its origin in the
prescribed curvature problem in conformal geometry. In general, for any
compact Riemann surface $(M, g)$ we may consider the following equation
\begin{equation} \label{pkeq}
\triangle u + e^u - 2K = 4\pi \sum_{j = 1}^n \alpha_j \,\delta_{Q_j} 
\quad \mbox{on $M$},
\end{equation} 
where $K=K(x)$ is the Gaussian curvature of the given metric $g$ 
at $x \in M$, $Q_j \in M$ are distinct points, and $\alpha_j > -1$ are
constants.\footnote{The points \(Q_j\) and the
constants \(\alpha_j\) are regarded as parameters of (\ref{pkeq})).} 
For any solution $u(x)$ to (\ref{pkeq}), the Gaussian
curvature of the new metric 
\[\tilde g := \frac{1}{2} e^u\cdot g\] 
has constant Gaussian curvature $\tilde K = 1$ outside those $Q_j$'s.
Since (\ref{pkeq}) has singular sources at the $Q_j$'s, the 
metric $\,\frac{1}{2}e^u g\,$ 
{may}
degenerate at $Q_j$ for each \(j\) and is
called a metric on $M$ with \emph{conic singularities} 
at the points $Q_j$'s. 
\smallbreak

\noindent
{\small{\bf Digression.}\enspace 
There is also an application of 
(\ref{Liouville-eq}) 
to the complex Monge-Amp\`{e}re equation:}
\begingroup\makeatletter\def\f@size{9}\check@mathfonts
\begin{equation} \label{MA}
\det\Big(\frac{\p^2 w}{\p z_i \p\bar z_j}\Big)_{1\leq i, j \leq d} =
e^{-w} 
\quad \mbox{on $(E \!\smallsetminus \!\{0\})^d$},
\end{equation}
\endgroup
{\small 
where \((E \!\smallsetminus \!\{0\})^d\) is
the $d$-th Cartesian product of $E \!\smallsetminus\! \{0\}$. 
Obviously, for any solution $u$ to (\ref{Liouville-eq}), the function
}
\begingroup\makeatletter\def\f@size{12}\check@mathfonts
\[
\scriptstyle w(z_1, \ldots, z_d) = -\sum_{i = 1}^d u(z_i) 
\,+\, d\cdot \log 4\] 
\endgroup
{\small
satisfies (\ref{MA}) 
with a logarithmic singularity along the normal crossing divisor
\begingroup\makeatletter\def\f@size{9}\check@mathfonts
$D = E^d \!\smallsetminus\! (E \!\smallsetminus \!\{0\})^d.\,$
\endgroup
In particular, bubbling solutions to (\ref{Liouville-eq}) will 
give examples of bubbling solutions to 
the complex Monge-Amp\`{e}re equation (\ref{MA}), 
whose bubbling behavior could be understood from our theory 
developed in this paper.  Those examples might be useful for
studying the geometry related 
to the degenerate complex Monge-Amp\`{e}re equations.
}

\subsubsection{}
Equation (\ref{pkeq}) is a special case\footnote{namely the case when
\(h\) is the constant function \(1\), \(n=1\) and
\(\alpha_1={\rho}/{4\pi}\).} of a general class of equations, 
called \emph{mean field equations}:
\begin{equation} \label{mfeq}
\triangle u + \rho \left( \frac{h e^u}{\int h e^u} 
- \frac{1}{|M|}\right) = 
4 \pi \sum_{j = 1}^n \alpha_j \left(\delta_{Q_j} - \frac{1}{|M|}\right) 
\quad \mbox{on $M$},
\end{equation}
where $h(x)$ is a \emph{positive} $C^1$-function on $M$ and $\rho$ 
is a \emph{positive} real number. Equation (\ref{mfeq}) arises not only 
from geometry, but also from many applications in physics. 
For example it appears in statistical physics as the equation for
the \emph{mean field limit} of the Euler flow in Onsager's vortex
model, hence its name. 
Recently the equation (\ref{mfeq}) was shown to be related to 
the self-dual condensation of 
the Chern--Simons--Higgs model. We refer the readers to 
\cite{CY, Choe, LinYan, LinYan2, NT, NT1} and 
references therein for recent developments on this subject. 

Equation (\ref{mfeq}) has been studied extensively for 
over three decades. 
It can be proved that 
outside a countable set of 
critical parameters $\rho$, solutions $u$ of (\ref{Liouville-eq})
have 
\emph{uniform} 
\emph{a priori} bounds in 
$C^2_{{loc}}(M \smallsetminus \{Q_1, \ldots, Q_n\})$:
\begin{quote}
{For any closed interval \(I\) not containing 
any of the critical parameters and any compact subset 
\(\Phi\subset M\smallsetminus \{Q_1, \ldots, Q_n\}\),
there exists a constant \(C_{I,\Phi}\) such that
\(|u(z)|\leq C_{I,\Phi}\) for all \(z\in \Phi\) and 
every solution 
\(u(z)\) of (\ref{Liouville-eq}) with parameter \(\rho\in I\)};
\end{quote}
see  
\cite{BT, CL0, CL1, LS}.

{The existence of uniform {a priori} bounds 
for solutions of (\ref{Liouville-eq}) implies that} 
the topological Leray-Schauder degree $d_\rho$ is well-defined
when $\rho$ is a  \emph{non-critical} parameter. 
Recently, an explicit degree counting formula 
has been proved in \cite{CL, CL2}, 
{which has the following consequence:} 
\begin{quotation}
Suppose that 
{$\rho \in (\RR_{>0}\!\smallsetminus\! 8\pi \mathbb N)$},
$\alpha_j \in \mathbb N$ for all \(j\) and the genus
$g(M)$ of \(M\) is at least \(1\).
Then $d_\rho > 0$, hence the mean field equation (\ref{mfeq}) 
has a solution. \footnote{For any natural number \(m\in\NN_{\geq 0}\),
the Leray-Schauder degree \(d_\rho\) is constant
in the open interval \((8\pi, 8\pi(m+1))\)
by homotopy invariance of topological degree, 
and 
\(\,d_{\rho}=m+1\) in this open interval
according to \cite[Thm.\,1.3]{CL2}.}
\end{quotation}

{However} when $\rho \in 8\pi \mathbb N_{>0}$, 
\emph{a priori} bounds for solutions of (\ref{mfeq}) might not exist,
{and} 
the existence of solutions 
{becomes an intricate question}. 
The singular Liouville equation (\ref{Liouville-eq}) 
{on flat tori} with 
$\rho \in 8\pi \NN_{>0}$ is the simplest class of mean field equations 
where the parameter $\rho$ is \emph{critical}, and 
the existence problem for equation (\ref{Liouville-eq})
{is already} a delicate one 
in the 
case when $\rho = 8\pi$. 
In \cite{LW} the second and third authors proved that equation
(\ref{Liouville-eq}) has a solution if and only if the Green's 
function on the torus has \emph{five} critical points;
c.f.\ Theorem A in \ref{subsec:thmA}.

\subsubsection{}
In this paper we will consider the case
when the parameter \(\rho\) is of the form 
$\rho = 4\pi l$ for some positive integer 
$l \in \mathbb Z_{>0}$.
We note that if $l$ is odd, $\rho = 4\pi l$ is \emph{not} a 
critical parameter. In this case the degree counting formula 
in \cite{CL2, ChenLW} gives the following result:
\bigbreak
\noindent
{\bf Theorem C.}\enspace {\it Suppose that $l = 2n + 1$ is a positive odd
integer. Then the Leray-Schauder degree $d_{4\pi l}$ of} of
equation (\ref{Liouville-eq}) 
\emph{is $\tfrac{1}{2}(l + 1) = n + 1$.}
\bigbreak

Theorem C will be sharpened in corollaries \ref{exact-sol} and
\ref{cor:odd-lame-nber-soln} to: 
\begin{quotation}
\emph{Let \(\rho=4\pi (2 n+1)\) for some 
\(n\in \ZZ_{\geq 0}\). Except for a finite number of tori up to isomorphism,
 equation} (\ref{Liouville-eq}) 
\emph{has exactly $n + 1$ solutions}. 
\end{quotation}

\subsection{}
The above sharpening of Theorem C will be established via 
the connection between the Liouville equation (\ref{Liouville-eq}) and
the Lam\'e equation (\ref{Lame-eq-n}). 
Indeed the equation (\ref{Liouville-eq}) is 
\emph{locally completely integrable} 
according to the following theorem of Liouville: 
\begin{quotation} \emph{Any solution $u$ to 
\textup{(\ref{Liouville-eq})} 
can be expressed locally on \(\,E\smallsetminus \Lambda\,\) 
as}
\begin{equation} \label{dev-map}
u(z) = \log \frac{8|f'(z)|^2}{(1 + |f(z)|^2)^2}
\qquad \forall z \in E \smallsetminus \{0\},
\end{equation}
\emph{where $f(z)$ is a multi-valued meromorphic function on 
$\mathbb C \smallsetminus \Lambda$ (i.e.\ a meromorphic function on an
unramified covering of $\mathbb C \smallsetminus \Lambda$)
such that the right hand side of the above displayed expression
is a well defined function on $\mathbb C \smallsetminus \Lambda$.} 
\end{quotation}
Such a function $f$ is called a \emph{developing map} of 
the solution $u$ of (\ref{Liouville-eq}). 

\subsubsection{}
When $\rho = 4\pi l$, $l \in \mathbb Z_{>0}$, 
{it is a fact} 
that every developing map 
$f(z)$ of a solution to (\ref{Liouville-eq})
extends to a \emph{single valued} meromorphic function 
on $\mathbb C$. Such a developing map $f$ is not doubly periodic in
general; rather it is \(\textup{SU}(2)\)-automorphic for 
the period lattice \(\Lambda\):
\begin{quote}
\emph{For every \(\omega\in \Lambda\) there exists
an element \(T = \begin{pmatrix}a&b\\c&d
\end{pmatrix}\in\textup{SU}(2)\) such that
\(\,f(z+\omega) = Tf(z) = \frac{af(z)+b}{cf(z)+d}\,\) for all \(\,z\in\CC\).}
\end{quote}
Moreover 
\begin{emph}{\(f(z)\) has multiplicity \(l+1\) at points of
the lattice \(\Lambda\subset \CC\), and no critical point
elsewhere on \(\CC\).}
\end{emph}
Conversely every meromorphic function \(f(z)\) on \(\CC\)
satisfying the above two properties is the developing map
of a solution to (\ref{Liouville-eq}); see Lemma \ref{order}.

\subsubsection{}
After 
replacing $f$ by \(Tf\) for 
a suitable element \(T\in\textup{SU}(2)\) we get a \emph{normalized} developing map \(f\) satisfying one of the following conditions:  

\begin{itemize}
\item[(i)] Type I 
{(the monodromy of \(f\) is a Klein four)}: 
\begin{equation} \label{feq-I}
\begin{split}
f(z + \omega_1) &= -f(z) \\
f(z + \omega_2) &= \frac{1}{f(z)}
\qquad \forall\,z\in\CC.
\end{split}
\end{equation}

\item[(ii)] Type II 
{(the monodromy of \(f\) is contained in a 
maximal torus)}: There exist real numbers \(\theta_1, \theta_2\)
such that 
\begin{equation} \label{feq-II}
f(z + \omega_i) = e^{2\theta_i} f(z)\qquad \forall\,z\in\CC,\ 
\forall\,i=1,2.
\end{equation}
\end{itemize}
Here \(\omega_1, \omega_2\) is a \(\ZZ\)-basis of
\(\Lambda\) with ${\rm Im}(\omega_2/\omega_1) > 0$. See \S \ref{Liouville-thm} for more details.

\subsubsection{}
We have seen that when the parameter \(\rho\) of the Liouville
equation (\ref{Liouville-eq}) is \(4\pi l\) with \(l\in\NN_{>0}\),
solving (\ref{Liouville-eq}) is equivalent to finding 
normalized developing maps, i.e.\ meromorphic functions 
on \(\CC\) with multiplicity \(l+1\) at points of the lattice 
\(\Lambda\) and
no critical points on \(\CC\smallsetminus\Lambda\),
whose monodromy with respect to \(\Lambda\) is specified as one
of the two types above. 
It turns out that Liouville equation (\ref{Liouville-eq})
with \(\rho/4\pi\in\NN_{>0}\) is \emph{integrable}
in the sense that
the configuration of the zeros and poles of such a 
normalized developing map can be described by either
system of polynomial equations, or as the zero locus
of an explicitly defined \(\CC\)-valued real analytic function 
on an algebraic variety.
In other words the Liouville equation (\ref{Liouville-eq})
with \(\rho/4\pi\in\NN_{>0}\) is \emph{integrable}
in the sense that the problem of solving this 
partial differential equation is reduced to finding the
zero locus of some explicit system of equations on
a finite dimensional space;
see Theorem \ref{thm-type I} and
Theorem \ref{thm-type II} below.

\subsubsection{}\label{subsec_dev-lame}
Let \(f(z)\) be a developing map of a solution \(u(z)\) of 
the Liouville equation (\ref{Liouville-eq}).
Of course the formula (\ref{dev-map}) expresses \(u\) in
terms of \(f\).
There is a simple way to ``recover'' the developing map $f$ from $u$ 
for a general parameter $\rho\in \RR_{>0}$. 
Notice first that \(\frac{\partial}{\partial z}\)
of (\ref{Liouville-eq}) gives
\[
\frac{\partial}{\partial {z}} 
\left(\frac{\partial^2 u}{\partial z\partial \bar{z}}
-\frac{1}{2} \left(\frac{\partial u}{\partial z}\right)^2
\right)=\frac{\rho}{4} \, \frac{\partial}{\partial {z}}\delta_0,
\]
which implies that $\,u_{zz} - \tfrac{1}{2} u_z^2\,$ is 
a meromorphic function on \(E\). 
A simple computation using (\ref{dev-map}) 
gives the following formula of this meromorphic function in terms
of the developing map \(f\) of \(u\):
\begin{equation} \label{S-Der}
u_{zz} - \tfrac{1}{2} u_z^2 = \frac{f'''}{f'} - \frac{3}{2} 
\Big(\frac{f''}{f'}\Big)^2.
\end{equation} 
The right hand side of (\ref{S-Der}) is the 
\emph{Schwarzian derivative} $S(f)$ of the meromorphic function $f$, 
while the left hand side is
\(\Lambda\)-periodic, with only one singularity at $0$ which is a pole of
order 
{at most} \(2\), hence must be equal to
a \(\CC\)-linear combination of the Weierstrass function
\(\,\wp(z;\Lambda)\,\) and the constant function \(1\).
It is not difficult to determine the coefficient of 
\(\,\wp(z;\Lambda)\,\) in this linear combination:
We know from equation (\ref{Liouville-eq}) that 
\[\,u(z)= 2 l \cdot \log |z| + (\textup{a }C^{\infty}\textup{-function})
\]
for all \(\,z\,\) in a neighborhood of \(0\in E\).
A straightforward calculation shows that either $f(z)$ has a pole
of order \(l+1\) at \(z=0\), or \(f(z)\) is holomorphic at \(z=0\)
and its derivative \(f'(z)\) has a zero of order \(l\) at \(z=0\).
In either case the Schwarzian derivative \(S(f)\) has a double
pole at \(0\) and 
\[\,\lim_{z\to 0} z^2 S(f)= (l+2)(l+3)-\frac{3}{2}(l+2)^2
= -(l^2 + 2 l)/2.\]
In other words there exists a constant \(B\in\CC\) such that
\begin{equation} \label{S-der-f}
S(f) = -2(\eta (\eta + 1) \wp(z) + B), 
\end{equation}
where \(\eta := \rho/8\pi = l/2\).

On the other hand it is well-known that the ``potential'' of
a second order linear ODE can be recovered from the Schwarzian 
derivative of the ratio of two linear independent (local) solutions
of the ODE.  In the case of the Lam\'e equation (\ref{Lame-eq-n})
this general fact specializes as follows.
\begin{quote}
\emph{If
$w_1, w_2$ are two linearly independent local solutions 
to the Lam\'e equation $L_{\eta, B}\,w = 0$ in (\ref{Lame-eq-n})
with a general parameters \(\eta\) and \(B\)
and \(h(z):=w_1(z)/w_2(z)\), then
$$
S(h) = -2(\eta (\eta + 1) \wp(z) + B).
$$}
\end{quote}

Combining the above discussions, we conclude:
\begin{quote}
\emph{If
\(\,\eta = \rho/8\pi\,\) 
then any developing map $f(z)$ of a solution $u(z)$ to} (\ref{Liouville-eq})
\emph{can be expressed as a ratio of two \(\CC\)-linearly independent
solutions of the Lam\'e equation} (\ref{Lame-eq-n})
\emph{for some \(B\in \CC\).}\footnote{The ``constant'' \(B\in\CC\) 
depends on both the (isomorphism class of the) flat torus \(E\)
and the solution \(u(z)\) of (\ref{Liouville-eq}).}
\end{quote}
We say that the Lam\'e equation $L_{\eta, B}\,w = 0$ on \(\CC/\Lambda\)
\emph{corresponds to} a solution $u$ of (\ref{Liouville-eq}) on \(\CC/\Lambda\)
with $\,\rho = 8\pi \eta \in 4\pi\cdot \ZZ_{>0}\,$ 
if there exist two linearly independent meromorphic solutions
\(w_1, w_2\) of $L_{\eta, B}\,w = 0$
on \(\CC\) such that \(w_1/w_2\) is a developing
map of \(u\).  This is a property of the parameter \(B\) of the
Lam\'e equation.


\subsubsection{}
We make a simple observation about a \emph{normalized} developing map 
of a solution
to (\ref{Liouville-eq}) with \(\rho/4\pi\in \ZZ_{>0}\). 
\begin{quotation}
\emph{If $f(z)$ is normalized of type II, then $e^\lambda f(z)$ 
also satisfies the type II condition for all 
$\lambda \in \mathbb R$. Thus}
(\ref{dev-map}) 
\emph{gives rise to a} scaling \emph{family of 
solutions $u_\lambda(z)$, where}
\begin{equation} \label{u-lambda}
u_\lambda(z) = \log \frac{8 e^{2\lambda} |f'(z)|^2}{(1 + e^{2\lambda} |f(z)|^2)^2}.
\end{equation}
\emph{Consequently if} 
(\ref{Liouville-eq}) \emph{has a type II solution, 
then the same equation has infinitely many solutions.}
\end{quotation}
From (\ref{u-lambda}), it is easy to see that $u_\lambda(z)$ 
blows up as $\lambda \to \pm \infty$. As we have discussed earlier, 
if $\rho = 4\pi l$ with $l = 2n + 1$ 
an odd positive integer, then solutions of (\ref{Liouville-eq}) have 
\emph{a priori} bound in $C^2_{{loc}}(E\smallsetminus\{0\})$. 
Thus we conclude that when $\rho = 4\pi l$ with $l$ 
a positive odd integer, 
(\ref{Liouville-eq}) has a solution because the topological degree 
is positive. 
Moreover such a solution must be of type I 
by the 
{existence of uniform a priori bound
on compact subsets of \(E\smallsetminus \{0\}\).}

Our first main theorem in this paper says that the converse 
to the statement in the previous paragraph
also holds. At the same time we provide a self-contained proof of 
the above implication without using the 
{uniform a priori bound}:

\begin{theorem} [c.f.~Proposition \ref{non-exist} 
and Theorem \ref{even-II}]\label{Main-corr}
Let $\rho = 4\pi l$ with $l \in \mathbb Z_{>0}$. 
Then equation \textup{(\ref{Liouville-eq})} 
admits a type I solution if and only if $l$ is odd.
\end{theorem}
\bigbreak

We will ``classify'' type I solutions 
for an odd positive integer $l = 2n + 1$ in the 
next theorem \ref{thm-type I}.
Let \(f(z)\) be a developing map of a solution of (\ref{Liouville-eq})
satisfying the normalized tranformation formula 
of type I in \eqref{feq-I}.
Consider the logarithmic derivative $g = (\log f)' = f'/f$, which is an elliptic function on the double cover 
\[E' := \mathbb C/\Lambda' \to E,\]
of \(E\), where
\[\Lambda' := \mathbb Z \omega_1' + \mathbb Z \omega_2',\quad
\omega_1' := \omega_1 \ \ \textup{and} \ \
\omega_2' := 2\omega_2.
\] 
Our next goal is to find all possible type I developing maps $f$ 
for some solution \(u\) of the Liouville equation 
\textup{(\ref{Liouville-eq})} whose parameter \(\rho\) is an odd
integral multiple of \(4\pi\). 
To do so, we have to locate the position of poles of $g$, 
or equivalently the position of zeros and poles of $f$. 

\begin{theorem}[Type I evenness and algebraic integrability] 
\label{thm-type I} 
Let $u$ be a solution to \textup{(\ref{Liouville-eq})} with 
$\rho = 4\pi (2n + 1)$,
$n \in \mathbb Z_{\ge 0}$. 
Let $f$ be a normalized type I developing map of $u$, \(\Lambda'=\ZZ\omega_1+ \ZZ 2\omega_2
=\ZZ\omega_1'+\ZZ\omega_2'\), and
\(e_i := \wp(\tfrac{1}{2}\omega_i'; \Lambda')\) for \(i=1,2\).

\begin{itemize}
\item [(1)] The solution $u(z)$ is even and 
the developing map $f(z)$ of \(u(z)\) is also even;
 i.e.~$u(-z) = u(z)$ for all \(z\in E\) and
\(f(z)=f(-z)\) for all \(z\in \CC\). 

\item[(2)] There exist $p_1, \cdots,p_n \in \mathbb C$ 
satisfying the following properties.
\begin{itemize}
\item $2p_i \not\in\Lambda'$ for $i=1,\ldots, n$, 

\item $p_i\pm p_j \not\in \Lambda'$ for all $i\neq j$,
\(1\leq i, j\leq n\), 

\item $f$ has simple zeros at $\tfrac{1}{2}\omega_1$ and $\pm p_i$
for $i = 1, \ldots, n$, 

\item \(f\) has simple poles at $\tfrac{1}{2} \omega_1 + \omega_2$ 
and $\pm p_i + \omega_2$ for $i = 1, \ldots, n$.

\item 
{Every zero or pole of \(f\) is 
congruent modulo \(\Lambda'\)
to one of the zeros or poles listed above.
}

\end{itemize}
Note that the \textup{unordered} set 
\(\{p_i\ \textup{mod}\,\Lambda'\}\subset E'\) is uniquely
determined by the normalized developing map \(f\).

\item [(3)] Let $q_i := p_i + \omega_2$, $i = 1, \ldots, n$ 
for \(i=1,\ldots, n\) and
let 
\[z_i := \wp(p_i;\Lambda') - e_2, \qquad
\tilde z_i = \wp(q_i;\Lambda') - e_2.\] 
There exist constants $\mu$ and $C_1,\ldots,C_n$ which depend only 
on the modular constants
$e_1$, $e_3$, $g_2(\Lambda')$ and $g_3(\Lambda')$ 
such that the following polynomial equations hold.
\begin{equation} \label{system-I}
\sum_{i = 1}^n z_i^j - \sum_{i = 1}^n \tilde
z_i^j = C_j \quad\textup{and}\quad
z_j \tilde z_j = \mu \quad \forall\,j = 1, \ldots, n.
\end{equation}

\item[(4)] Conversely let \(\mu, C_1,\ldots,C_n\) be the
constants in \textup{(3)}, and suppose that the \(2n\)-tuple
\((z_1,\ldots, z_n; \tilde{z}_1,\ldots,\tilde{z}_n)\in \CC^{2n}\)
is a solution of the system of polynomial equations 
\textup{(\ref{system-I})}.
There is an even type I developing map $f$ 
and \(p_1,\ldots, p_n\in \CC\)
with the following properties:
\begin{itemize}
\item $f$ has simple zeros at $\tfrac{1}{2} \omega_1$ and 
$\pm p_i$ for \(i=1,\ldots,n\).

\item \(f\) has simple poles at 
$\tfrac{1}{2}\omega_1 + \omega_2$ and $\pm p_i + \omega_2$
for \(i=1,\ldots,n\).

\item \(z_i = \wp(p_i;\Lambda') - e_2\) and
\(\tilde z_i = \wp(p_i+\omega_2;\Lambda') - e_2\)
for \(i=1,\ldots,n\).

\end{itemize}
\end{itemize}
\end{theorem}

We will prove Theorem \ref{thm-type I} in \S\ref{I-int}. 
In view of this result, it is interesting to know how many solutions 
the system (\ref{system-I}) has. Since the topological degree 
for (\ref{Liouville-eq}) with $\rho = 4\pi(2n + 1)$ is known 
to be $n + 1$ (by Theorem C), it is reasonable to conjecture 
that (\ref{Liouville-eq}) has $n + 1$ solutions, and then 
(\ref{system-I}) has $(n + 1)!$ solutions due to the 
permutation symmetry on $\{1, 2, \ldots, n\}$ 
(c.f.~{\cite[Conjecture 6.1]{LW2}} where a related version 
of this counting conjecture was first formulated). 
This conjecture had been verified previously 
up to $n \le 5$; see Remark \ref{KL-Lin}. 
However for higher $n$ it seems to be a non-trivial task 
to work on the affine polynomial system (\ref{system-I}) directly.

We will 
{affirm} this conjecture using 
the connection between the Liouville equation 
(\ref{Liouville-eq}) and the Lam\'e equation (\ref{Lame-eq-n}) 
discussed earlier. 

\begin{theoremss} \label{thm-K4}
For any 
$n \in \mathbb N_{\geq 0}$, 
the 
{projective} 
monodromy group of
a Lam\'e equation $L_{n + (1/2),\, B}\, w = 0$ 
on \(E\)
is isomorphic to the Klein-four group
$(\mathbb Z/2\mathbb Z)^2$ if and only if 
it \emph{corresponds to} a type I solution of 
\textup{(\ref{Liouville-eq})}, in the sense that 
there exist two meromorphic functions \(w_1, w_2\) on \(\CC\)
such that \(L_{n+(1/2),\,B}\,w_1 =0= L_{n+(1/2),\,B}\,w_2\)
and the quotient \(w_1/w_2\) is a developing map of a
solution of a Liouville equation \textup{(\ref{Liouville-eq})} 
with $\rho = 4\pi(2n + 1)$ and
the monodromy group for \(w_1/w_2\) is
isomorphic to $(\mathbb Z/2\mathbb Z)^2$.
Moreover, each parameter $B$ with the above property
corresponds to exactly one type I solution
of \textup{(\ref{Liouville-eq})}.
\end{theoremss}

Theorem \ref{thm-K4} will be proved in 
\S\,\ref{monodromy-I}, Theorem \ref{K4-thm}. 
Its proof shows that the number of solutions to 
(\ref{Liouville-eq}) with $\rho = 4\pi(2n + 1)$ is equal to 
the number of $B$'s in \(\CC\) such that all solutions to 
$L_{n + (1/2),\, B}\, w = 0$ are without logarithmic singularity. 
A classical theorem of Brioschi, Halphen and Crawford says that 
there exists a polynomial \(p_n(B)\) of degree $n + 1$ in \(B\)
whose roots are exactly the parameter values having the
above property.
Hence we have the following corollary which sharpens Theorem C:

\begin{corollaryss} \label{exact-sol}
Let $\rho = 4\pi(2n + 1)$, $n \in \mathbb Z_{\geq 0}$. 
There exists a finite set \(\mathcal S_n\) of tori such that for
every torus not isomorphic to anyone 
in the exceptional set \(\mathcal S_n\) the 
Liouville equation \textup{(\ref{Liouville-eq})} 
possesses exactly $n + 1$ distinct solutions.
\end{corollaryss} 

\subsubsection{}
In \S \ref{monodromy-I} we will also give a new proof of the 
Brioschi-Halphen-Crawford theorem by exploring the fact that 
the ratio $f=w_1/w_2$ of two linearly independent
solutions of the Lam\'e equation \(\,L_{n+({1}/{2}), B}\, w=0\,\)
satisfies the
equation (\ref{S-der-f}) for the Schwarzian derivative
with \(\eta=n\!+\!\tfrac{1}{2}\);
see Theorem \ref{BH-poly}. The new proof has 
{several} advantages. 
It provides 
a 
{convenient} 
way to compute the polynomial $p_n(B)$ 
for each $n$. Moreover it is local in nature. 
Thus it can be used to treat the mean field equation
with multiple singular sources of the form 
\begin{equation} \label{mfe-many-4pi}
\triangle u + e^u = 4\pi \sum_{j = 1}^l \alpha_j\,\delta_{Q_j} 
\quad \mbox{on $E$},
\end{equation}
where $Q_1,\ldots, Q_l$ are distinct points in $E$ 
and \(\alpha_1,\ldots,\alpha_l\) are positive integers. 
In a forthcoming paper \cite{CLW2} 
we will prove that for generic $Q_1, \ldots, Q_n \in E$, 
equation (\ref{mfe-many-4pi}) has exactly 
$$
\tfrac{1}{2} \prod\nolimits_{j = 1}^l (\alpha_j + 1)
$$ 
distinct solutions provided that $\sum_{j = 1}^l \alpha_j$ 
is an odd positive integer.

An immediate consequence of Corollary \ref{exact-sol} 
is a solution of the counting conjecture 
stated in the paragraph after \ref{thm-type I}:
\begin{quotation}
\emph{There exists a finite set \(\mathcal S_n\) of tori such that
for every torus not isomorphic to anyone in the exceptional set
\(\mathcal S_n\)
the polynomial system} (\ref{system-I})
\emph{has $(n + 1)!$ solutions.} 
\end{quotation}
This might be helpful when we come to study the 
\emph{excess intersection} 
at $\infty$ 
for the projectivized 
{version} of the 
system of 
{equations}
(\ref{system-I}) for general $n$.

\subsubsection{}
Another important consequence of Theorem \ref{thm-type I} is the 
\emph{holomorphic dependency} 
of \(\,f(z;\tau)\,\) on
the moduli variable $\tau = \omega_2/\omega_1$ in the upper half
plane $\mathbb H$ 
for normalized developing maps \(\,f(z;\tau)\,\) of 
solution to \textup{(\ref{Liouville-eq})} with 
$\rho = 4\pi (2n + 1)$ 
as in (\ref{thm-type I}); 
%
we have not been able to prove this statement directly
from Liouville's equation (\ref{Liouville-eq}).
The modular dependency of the constants $\mu$, $C_j$'s 
in Theorem \ref{thm-type I} indicates that 
the \emph{normalized} developing map might be invariant 
under modular transformations 
of $\tau$ for some congruence subgroup of 
\(\textup{SL}_2(\ZZ)\). 
To illustrate this connection between (\ref{Liouville-eq}) 
and modular forms, 
we will consider in \S\ref{mod-form} 
the simplest case $\rho = 4\pi$, 
where (\ref{Liouville-eq}) has exactly one solution for any torus. 
In this situation 
{we can \emph{specify} a unique} developing map 
$f(z; \tau)$
on $E_\tau = \mathbb C/\Lambda_\tau$, 
where $\Lambda_\tau = \mathbb Z + \mathbb Z \tau$ and 
$\tau \in \mathbb H$;
see Proposition \ref{prop:extra-normalized} for the
definition of this function \(f(z;\tau)\)
and explicit formulas for it.
When $f(z;\tau)$ is written as a power series
$$
f(z; \tau) = a_0(\tau) + a_2(\tau) z^2 + a_4(\tau) z^4 + \cdots,
$$
for each \(k\) the coefficient 
$a_k(\tau)$ of \(\,z^k\,\) is a \emph{modular form} of weight $k$ 
for the principal congruence subgroup $\Gamma(4)$ 
which is holomorphic on the upper-half plane \(\HH\)
but may have poles at the cusps of the modular curve \(X(4)\);
see Corollary \ref{modular-form-1stcase}.
In addition we will show 
that the constant term \(a_0(\tau)\) of \(f(z;\tau)\) is a 
\(\,\QQ(\sqrt{-1})\)-rational \emph{Hauptmodul} which is also a
\emph{modular unit}; i.e.\ (a) \(a_0(\tau)\) is holomorphic and 
everywhere non-zero on \(\HH\), 
(b) \(a_0(\tau)\) defines a meromorphic
function on \(X(4)\) with \(\QQ(\sqrt{-1})\)-rational \(q\)-expansion
at all cusps of \(X(4)\),  
and (c) every meromorphic function on \(X(4)\) is a rational
function of \(a_0(\tau)\).
See 
Corollary \ref{cor:hauptm} and Remark \ref{rem:modunit}\,(b).

Underlying the above statements is the fact that the function
\(f(z;\tau)\) satisfies a transformation law for the 
full modular group \(\textup{SL}_2(\ZZ)\); 
see Proposition \ref{prop:transf_4pi} for the precise statement. 
Modulo a question \ref{question:irred}\,(a) on the irreducibility
of certain branched covering of the upper-half plane \(\HH\),
this transformation law generalizes to the case when 
\(\rho=4\pi(2n+1)\) for any natural number \(n\in\NN_{\geq 0}\);
see Corollary \ref{cor:gen-psi}.

In a forthcoming paper we will consider equation 
(\ref{mfe-many-4pi}) with multiple singular sources and 
show that for each \(k\) the space of modular forms of weight \(k\)
arising from (\ref{mfe-many-4pi}) is invariant under 
(suitably defined) Hecke operators.

\subsection{}
Next we want to classify solutions of the Liouville equation
(\ref{Liouville-eq}) 
with $\rho = 8\pi n$ for some positive integer $n$. 
By Theorem \ref{Main-corr} any solution of (\ref{Liouville-eq}), 
if exists, 
must be of type II. Hence any solution of (\ref{Liouville-eq}) 
begets infinitely many solutions. 
We remark that not every torus admits a solution to
(\ref{Liouville-eq}). For instance 
when $\rho = 8\pi$, there are no solutions 
to (\ref{Liouville-eq}) for rectangular tori, 
while there do exist solutions for $\tau$ 
close to $e^{\pi i/3}$; see \cite[Example 2.5, 2.6]{LW}). 
Indeed 
\cite[Theorem 1.1]{LW} asserts that 
in the case when $\rho = 8\pi$, the Liouville equation
(\ref{Liouville-eq}) has a solution
if and only the Green's function $G_E$ for the torus
has a critical point which is not a \(2\)-torsion.
Hence by Theorem A the Liouville equation
(\ref{Liouville-eq}) with $\rho = 8\pi$ has a solution for 
if and only if the Green function has five critical points. 
Let $\mathcal{M}_1 = {\rm SL}_2(\mathbb Z)\backslash \mathbb H$ 
and let
$$
\Omega_5 := \{\,\bar\tau \in \mathcal{M}_1 \mid \mbox{$G(z; \bar\tau)$ 
has five critical points}\,\}.
$$ 
By the uniqueness theorem in \cite[Theorem 4.1]{LW}, 
we know that $\Omega_5$ is open; while it is easy to see 
(using the holomorphic \((\ZZ/3\ZZ)\)-action on the torus)
that the image of \( e^{\pi i/3}\) in
\(\mathcal{M}_1\) lies in \(\Omega_5\).
It is important to further investigate the geometry of this 
moduli subset $\Omega_5$. In Part II of this series of papers, 
we shall use methods for non-linear PDE's to the 
Liouville equation (\ref{Liouville-eq}) 
and the theory of modular forms to prove that 
$\Omega_5$ is a simply connected domain 
and the boundary \(\p \Omega_5\) of \(\Omega\)
is 
real-analytically isomorphic to a circle,
{thereby} settling the conjecture on the shape of 
$\Omega_5$ raised in \cite[\S1\,p.\,915]{LW}.\footnote{This phenomenon 
was observed in computer simulations.} 

\subsubsection{}\label{subsec:ansatz}
In this paper (Part I of the series), we classify 
all type II solutions for general $n \in \mathbb Z_{>0}$ and study 
their connection with the geometry of a family of hyperelliptic
curves. This will form the foundation of an investigation on 
certain modular forms to be developed in 
Part II of this series \cite{LW-II}.

We will also consider the 
logarithmic derivative $g = {f'}/{f}$ of a \emph{normalized}
type II developing map
$f$. 
The type II condition \eqref{feq-II} implies that $g$ is 
an elliptic function on $E = E_\Lambda$. 

\subsubsection{}
As was explained in \ref{subsec_dev-lame}, 
the Liouville equation (\ref{Liouville-eq}) is related
to the Lam\'e equation (\ref{Lame-eq-n}) whose
parameter \(\eta\) is equal to \(\rho/8\pi\).
In the case when \(\eta\) is a positive integer \(n\), 
there are explicit formulas for solutions,
of the Lam\'e equation (\ref{Lame-eq-n}), 
called the \emph{Hermite-Halphen ansatz}; c.f.\ 
\cite[I--VII]{Hermite} and
\cite[pp.\,495-498]{Halphen}: 

\begin{quotation}
\emph{For any $a_1, \ldots, a_n \in \mathbb C \smallsetminus \Lambda$ 
such that the images $[a_i] \in E$
of \(a_i\) under the projection 
$\mathbb C \to E = \mathbb C/\Lambda$, $i = 1, \ldots, n$, 
represent} \(n\) mutually distinct \emph{points in 
\(E\smallsetminus\{[0]\}\), the function} 
\begin{equation}\label{fnc-w_a}
w_a(z) = e^{z\cdot\sum_{i = 1}^n \zeta(a_i;\Lambda)} 
\prod_{i = 1}^n \frac{\sigma(z - a_i;\Lambda)}{\sigma(z;\Lambda)}
\end{equation}
\emph{is a solution to} (\ref{Lame-eq-n}) 
\emph{for some $B \in \mathbb C$ 
if and only if $\{[a_i]\} \in Y_n 
\subset \textup{Sym}^n(E\smallsetminus\{0\})$, where}
\begin{equation}\label{Yn1}
Y_n :=\left\{\{[a_1],\ldots,[a_n]\} \left| 
\begin{array}{lll}
[a_i]\in E\!\smallsetminus\!\{0\}\ \forall i,\ \
[a_i] \ne [a_j]\,   \ \textup{for all } i \ne j,\\
\sum_{1\leq j\leq n,\,j\neq i} 
\left(\zeta(a_i \!-\! a_j) + \zeta(a_j) - \zeta(a_i)\right) = 0\\
\textup{for}\, \ i = 1, \ldots, n.\ \ 
\end{array}\right.\right\}.
\end{equation}
{
\emph{Moreover if \(\{[a_i]\}_{i=1}^n\) is a point of
\(Y_n\) and \(w_a(z)\) is a solution of 
a Lam\'e equation} (\ref{Lame-eq-n}) 
\emph{with \(\eta=n\),
then \(B=(2n - 1) \sum_{i = 1}^n \wp(a_i)\).
}}
\end{quotation}

\noindent
Note that \(\,w_{b}(z)\in \CC^{\times}\cdot w_a(z)\,\) if 
\(b=(b_1,\ldots, b_n)\) and \(b_i\equiv a_i\pmod{\Lambda}\) 
\(i=1,\ldots,n\).

\subsubsection{}\label{subsub:properties_ansatz}
The following properties are known from classical literature.
\begin{itemize}
\item[(i)] 
Each ansatz solution \(\,w_a(z)\,\) of 
\(\,L_{n,B}\,w=0\,\) satisfies
\begin{equation*}
w_a(z+\omega)=e^{\sum_{i=1}^n\zeta(a_i;\Lambda)\omega 
-\sum_{i=1}^n a_i\eta(\omega;\Lambda)}\cdot w_a(z)
\quad \forall\,\omega\in \Lambda.
\end{equation*}
In other words \(\,w_a(z)\,\) is a common eigenvector for the global
monodromy representation of \(\Lambda =\pi_1(E)\) on the 
\(2\)-dimensional space of local solutions of the 
Lam\'e equation \(\,L_{n,B}\,w=0\).

\item[(ii)]  Every one-dimensional eigenspace of the monodromy
representation of a Lam\'e equation \(\,L_{n,B}\,w=0\,\) is of
the form \(\CC\cdot w_a(z)\) for some \(a\) such that
\(\,B=(2n - 1) \sum_{i = 1}^n \wp(a_i)\).
In other words the map \(\,\pi: Y_n\to \CC\,\) given by
\(\,\{[a_i]\}_{i=1}^n\,\mapsto\,(2n - 1) \sum_{i = 1}^n \wp(a_i)\,\)
is surjective. Note that \(\pi\circ\iota=\pi\) for the
involution \(\,\iota:\{[a_i]\}_{i=1}^n\mapsto \{[-a_i]\}_{i=1}^n\) 
on \(Y_n\).

\item[(iii)] For every \(B\in \CC\), the set
\(\pi^{-1}(B)\) is an orbit of the involution \(\iota\),
and \(\pi^{-1}(B)\) is a singleton if and only if
\(L_{n,B}\,w=0\) has a Lam\'e function as a solution.

\end{itemize}
The above properties tell us that \(\,Y_n\,\) can be regarded
as the parameter space of all one-dimensional eigenspaces
of the monodromy representations on the solutions of the
Lam\'e equation \(\,L_{n,B}\,w=0\,\) on \(E\) when the parameter
\(B\) varies over \(\CC\).
This and the fact that \(\pi:Y_n\to \CC\) is a double cover
drives home the compelling picture that \(Y_n\) can be regarded
as a ``spectral curve'' for the global 
monodromy representation.\footnote{This is more than an analogy:
\(Y_n\) is indeed a spectral curve in KdV theory. It parametrizes
one-dimensional common eigenspaces for the commutator subring
of the differential operator \(\,\frac{d^2}{dz^2}-n(n+1)\wp(z)\,\)
in the ring of linear differential operators in one variable.}
The algebraic structure on \(Y_n\) is explained in
\ref{subsub:Yn} below.

\subsubsection{}\label{subsub:Yn}
The analytic set of solutions of the system of equations
\begin{equation}\label{Yn}
\sum_{1\leq j\leq n,\,j\neq i} 
\left(\zeta(a_i - a_j;\Lambda) + \zeta(a_j;\Lambda) - \zeta(a_i;\Lambda)\right) = 0\quad
\forall\, i = 1, \ldots, n
\end{equation}
in variables 
\(\,\left(a_1,\ldots, a_n\right)\,\) under the constraint that
\[
a_i\not\in \Lambda\ \ \ \forall\,i=1,\ldots,n\ \ 
\textup{and}\ \ 
a_i-a_j\not\in \Lambda\ \ \forall\, i\neq j
\]
descends to a locally closed algebraic subvariety of
\[{\rm Sym}^n (E\smallsetminus\{0\}) = (E\smallsetminus\{0\})^n/S_n\] 
because \(Y_n\) is stable under the symmetric group \(S_n\),
and
the classical addition formula (c.f.~\cite[20\(\cdot\)53 Example 2]{Whittaker}) 
\[\frac{1}{2}\,\frac{\wp'(z)+\wp'(w)}{\wp(z)-\wp(w)}
=\zeta(z-w)-\zeta(z)+\zeta(w)
\]
for elliptic functions allows us to express the definition of
\(Y_n\) algebraically:
Let \(\tilde{\Delta}\) be the divisor of \(E^n\) consisting
of all points of \(E^n\) where at least two components are equal,
and let \(\Delta\) be the image of \(\tilde{\Delta}\) in
\(\textup{Sym}^n E\). 
Denote by \(\tilde{U}\) the algebraic variety
\((E\smallsetminus\{[0]\})^n\smallsetminus \Delta\).
\begin{quote}
\emph{\(Y_n\) is the closed subvariety of 
\(\textup{Sym}^n(E\smallsetminus\{[0]\})\smallsetminus \Delta\)
whose inverse image \(\tilde{Y_n}\) in the affine algebraic variety
\(\tilde{U}\)
is defined by the system of equations}
\begin{equation} \label{pm-sys}
\sum_{1\leq j\leq n,\,j \ne i} \frac{y_i + y_j}{x_i - x_j} = 0, 
\qquad \forall\,i = 1, \ldots, n.
\end{equation}
\end{quote}
Here \(x_i, y_i\) are the pull-back via the \(i\)-th projection
of coordinates of the Weierstrass form 
\(\,y^2=4x^3-g_2(\Lambda)x-g_3(\Lambda)\,\) of
\(E=\CC/\Lambda\).
For each pair \((i,j)\) with \(i\neq j\), 
the regular function \(\,\frac{y_i + y_j}{x_i - x_j}\,\)
on the affine open subset \(\tilde{U}[1/(x_i-x_j)]\) of \(\tilde{U}\)
where \(\,x_i\neq x_j\,\) extends to a regular function on
\(\tilde{U}\), therefore the above description defines
an affine closed subvariety \(\tilde{Y}_n\) of \(\tilde{U}\).



In view of the algebraic structure of \(Y_n\), 
the classically known facts recalled in \ref{subsub:properties_ansatz}
means that \(Y_n\) is ``the affine part'' of a hyperelliptic curve
and the \(\pi:Y_n\to \CC\) is the restriction to \(Y_n\)
of the hyperelliptic projection.
On the other hand, solutions to the 
Liouville equation (\ref{Liouville-eq}) 
with $\rho = 8\pi n$ admit the following description.

\begin{theorem} [Type II evenness and Green/algebraic system] 
\label{thm-type II} 
Let $u$ be a solution to \textup{(\ref{Liouville-eq})} 
with $\rho = 8\pi n$ on \(E=\CC/\Lambda\)
and let $f$ be a \emph{normalized} developing map of \(u\)
of type II. 

\begin{itemize}
\item[(1)] The developing map \(f\) is a local unit
at points of \(\Lambda\).

\item[(2)] There are \(2n\) elements 
\(p_1,\ldots,p_n, q_1,\ldots, q_n\in \CC\)
with the following properties.
\begin{itemize}
\item 
\([p_1],\ldots, [p_n], [q_1],\ldots, [q_n]\) 
are \(2n\) 
\emph{distinct} points
in \(E\), where \([p_i]:=p_i\, \textup{mod}\,\Lambda\)
for all \(i\) and similarly for the \([q_i]\)'s.
\item $f$ has simple zeros 
{at points above} 
$p_1, \ldots, p_n$ 
and simple poles 
{at points above} 
$q_1, \ldots, q_n$.
\item 
\(f\) is holomorphic and non-zero at every point of \(\CC\)
which is not congruent modulo \(\Lambda\) to
one of \(\{p_1,\ldots, p_n, q_1,\ldots, q_n\}\).

\end{itemize}

\item[(3)] The zeros and poles of the developing map \(f\)
are related by
\[\{[q_1], \ldots, [q_n]\} = \{[-p_1], \ldots, [-p_n]\}.\]

\item[(4)] There is a unique even solution in the one-parameter 
{scaling} family of solutions
\(\,u_{\lambda}(z)=\log\frac{8\,e^{2\lambda}\,\vert f'(z)\vert^2}{(1+
e^{2\lambda}\vert f(z)\vert^2)^2}\,\) 
of \textup{(\ref{Liouville-eq})}
with parameter \(\lambda\in\RR\).

\item[(5)] The 
{``zero points'' $p_1, \ldots, p_n\in E$
of \(f\)}  satisfy the following 
$n$ equations: 
\begin{equation} \label{system-II}
\sum_{i = 1}^n\wp'(p_i;\Lambda) \cdot \wp^r(p_i;\Lambda) = 0, \quad r
= 0,\ldots, n - 2, 
\end{equation}
\begin{equation} \label{G-eq}
\sum_{i = 1}^n \frac{\partial G}{\partial z}(p_i) = 0,
\end{equation}
where \(G(z)\) is the Green's function of \(E\).

\item[(6)] The meromorphic function \(\,g:=\frac{d}{dz}\log f
  =f'/f\,\) on \(E\) is 
\emph{even}
and is determined by the points  
\([p_1],\ldots,[p_n]\in E\), while
the normalized developing map \(f\) is determined up to \(\CC^{\times}\)
by \([p_1],\ldots, [p_n]\), via the following formulas:
\begin{equation} \label{g-expr}
g(z) = \sum_{i = 1}^n \frac{\wp'(p_i;\Lambda)}{\wp(z;\Lambda) -
  \wp(p_i;\Lambda)}, 
\qquad
f(z) = f(0)\cdot \exp \int^z_{{0}} g(\xi)\, d\xi.
\end{equation}
\end{itemize}
Conversely, if $\{[p_1], \ldots, [p_n]\}$ is a set of \(n\) distinct points
of \(E\smallsetminus\{0\}\) which satisfies equations
\textup{(\ref{system-II})} and \textup{(\ref{G-eq})}, and
\begin{equation} \label{pn-p}
\{[p_1], \ldots, [p_n]\} \cap \{[-p_1], \ldots, [-p_n]\} = \emptyset,
\end{equation}
then the function $f$ defined by \textup{(\ref{g-expr})}
is a type II normalized developing map
{of a solution of \textup{the Liouville 
equation (\ref{Liouville-eq})}
with \(\rho=8\pi n\)}.
\end{theorem}

Theorem \ref{thm-type II}
will be proved in \S \ref{II-even}. While the type I system 
is purely algebraic, the type II system is 
{somewhat} 
transcendental as it involves the Green function of \(E\). 
It is natural to isolate the 
Green equation $\sum \nabla G(p_i) = 0$ and consider 
the remaining $n - 1$ algebraic equations (\ref{system-II}) first. 

\subsubsection{} 
Let $a = \{a_1, \ldots, a_n\}$ be an
unordered set of complex numbers with \emph{distinct} images 
in \(E\smallsetminus \{0\}=(\CC\smallsetminus \Lambda)/\Lambda\)
such that the
equation (\ref{system-II}) is satisfied
with \(p_i=a_i\) for \(i=1,\ldots, n\).
Let $x = \wp(z;\Lambda)$, $y = \wp'(z;\Lambda)$, 
so that the torus $E$ is given by 
the Weierstrass equation 
\[y^2 = 4x^3 - g_2(\Lambda) x - g_3(\Lambda).\]
Let $(x_i, y_i) = (\wp(a_i;\Lambda), \wp'(a_i;\Lambda))$ 
for $i = 1, \ldots, n$. 
Then the system of $n - 1$ equations (\ref{system-II}) 
takes the algebraic form
\begin{equation} \label{power-sys}
\sum_{i = 1}^n x_i^r\cdot y_i = 0, \qquad i = 0, 1, \ldots, n - 2.
\end{equation}
Recalled that the algebraic variety \(Y_n\) is 
defined by the system of equations (\ref{pm-sys}).

\begin{theoremss} [$=$ Theorem \ref{eq-2-sys}] \label{eq2sys}
For any set of \emph{distinct} complex numbers
$x_1, \ldots, x_n$, the two 
systems of \emph{linear} equations in variables $y_1,\ldots, y_n$
\textup{(\ref{power-sys})} 
and \textup{(\ref{pm-sys})} are equivalent.
\end{theoremss}

Define $X_n \subset {\rm Sym}^n E$ by 
\begin{equation} \label{X-n}
X_n = \left\{\{(x_i,y_i)\}_{i = 1}^n
\left\vert 
\begin{array}{ll}
(x_i,y_i)\in E\smallsetminus 
{{E[2]}}
\ \ \forall\,i,\ 
x_i\neq x_j \ \ \forall\,i\neq j
\\
\sum_{i = 1}^n x_i^r\cdot y_i
 = 0\ \ \textup{for} \ r=0,1,\ldots,n-2
\end{array}
\right.
\right\},
\end{equation}
where \(E[2]=\frac{1}{2}\Lambda/\Lambda\) is the subset of
\(2\)-torsion points of \(E\).
This variety $X_n$ is an affine algebraic curve, 
which will be called 
the ($n$-th) \emph{Liouville curve}. 
Theorem {\ref{eq2sys}} implies that
$$
X_n = \big\{\{[a_i]\}_{i = 1}^n \in Y_n \,\big\vert\, \mbox{$\wp(a_i) 
\ne \wp(a_j)\,$ whenever $\,i \ne j$},\ 
{\wp'(a_i)\neq 0\ \forall\,i} \,\big\}. 
$$

The following theorem says that the Liouville curve $X_n$ is the unramified locus
of the Lam\'e curve 
{\(\bar X_n\)} 
for the hyperelliptic projection.

\begin{theorem}[Hyperelliptic structure on 
$X_n \subset Y_n \subset \bar X_n$] \label{hyp-ell-thm} {\ }
\begin{itemize}
\item [(1)]
Let $a = \{[a_i]\}_{i=1}^n 
$ 
be a point of \(Y_n\).
The corresponding $B$ in the Lam\'e equation,
in the sense of the Hermite--Halphen ansatz 
recalled in \textup{\ref{subsec:ansatz}},
is given by
\begin{equation*}
B_a = (2n - 1) \sum_{i = 1}^n \wp(a_i).
\end{equation*}

\item [(2)] The map $\pi: Y_n \to \mathbb C$ defined by 
$a \mapsto B_a$ is a 
{proper} 
surjective branched double cover. 


\item [(3)] The map $\pi:Y_n\to \CC$ has a natural extension to a
proper morphism 
\[\bar \pi: \bar X_n \to \mathbb P^1(\mathbb{C}) =\CC\cup\{\infty\},\] 
where \(\bar X_n\)  
is the closure of $X_n$ in ${\rm Sym}^n E$, for both the
Zariski and the complex topologies.

\item[(4)] The restriction 
\[\pi\vert_{X_n}:X_n\to \pi(X_n)=:U_n\] 
of \(\pi\) to the Zariski open subset \(X_n\subset Y_n\) 
is a finite \'etale double cover of 
the Zariski open subset \(U_n\subset\CC\).
Points of the finite set $\bar X_n \smallsetminus X_n$ are precisely 
the ramification points 
{of \(\bar\pi\)}

\item[(5)] The curve \(\bar X_n\) is a 
(possibly singular) hyperelliptic curve 
of arithmetic genus $n$ and
\(\bar\pi\) is the hyperelliptic structural morphism.
Moreover \(\bar X_n=Y_n\cup\{[0]^n\}\) is the union of \(Y_n\) 
with a \emph{single point}
\([0]^n :=\{[0], \ldots, [0]\}\in \textup{Sym}^n E \) 

\item[(6)] 
{The curve \(\bar X_n\) is stable under the involution
\(\tilde\iota\)
of \(\,\textup{Sym}^n E,\) defined by 
\[
\iota\colon \{P_1,\ldots,P_n\} \mapsto
\{-P_1,\ldots, -P_n\}\qquad
\forall\,P_1,\ldots, P_n\in E.
\]
The restriction \(\bar\iota\) of \(\tilde\iota\) to \(\bar X_n\) is
the \emph{hyperelliptic involution} on \(\bar X_n\).}
The set of ramification points \(\bar X_n\!\smallsetminus\! X_n\) of 
\(\bar\pi\) coincides with the fixed point set of the 
hyperelliptic involution \(\bar\iota\) on \(\bar X_n\).

\item[(7)] The map \(\pi\) induces a bijection from
the finite set
\(Y_n\!\smallsetminus \! X_n\) 
to the finite set \(\CC\!\smallsetminus \! U_n\).
A point \(\{[a_i]\}_{i=1}^n\) of \(Y_n\) lies in 
\(Y_n\smallsetminus X_n\) if and only if
the function $w_a$ as defined in \textup{\ref{fnc-w_a}} is a Lam\'e function. 
Hence $\,\#(Y_n \!\smallsetminus\! X_n) = 2n + 1\,$ 
{when \(Y_n\!\smallsetminus\! X_n\) is counted} with 
multiplicities inherited from \(\CC\!\smallsetminus\! U_n\) 
when \(\CC\!\smallsetminus\! U_n\) is identified with the set of roots of the
polynomial $\ell_n(B) = 0$ of degree \(2n+1\) in Theorem B.

\item[(8)] 
The inverse image of the point 
\(\infty = \mathbb P^1(\mathbb{C}) \!\smallsetminus\! \CC\) under \(\bar\pi\)
consists of the single point \(0^n\).
This point \(0^n\) ``at infinity'' is a smooth point
of \(\bar X_n\) for every torus \(E\).
 \end{itemize}
\end{theorem}

\subsubsection{}
A complete proof of Theorem \ref{hyp-ell-thm}
is given in \S \ref{hyper-ell-str}, 
after some preparation in \S \ref{X_n} on characterizations of $Y_n$ 
and $X_n$ related to the Lam\'e equations;
see Theorems \ref{hyper}, \ref{alg-hyper}, Corollary
\ref{cor:smooth_at_infty}
and Proposition \ref{prop:use-purity}.
In particular, the affine hyperelliptic curve $Y_n$ 
is defined by the explicit equation 
$$
C^2 = \ell_n(B; g_2, g_3) \quad \text{with}\quad \deg \ell_n(B) = 2n + 1.
$$ 
This curve, called \emph{the $n$-th Lam\'e curve}, 
is smooth for generic tori. It is an irreducible algebraic curve 
since the degree of $\ell_n$ is odd.

Due to its fundamental importance, we 
offer several 
proofs, from both the analytic and the algebraic perspectives,
for (part of) the theorem.
We must mention that the polynomial $\ell_n(B)$ 
in Theorem B has been treated 
in the literature in several different contexts,
{including} the investigation of 
Lam\'e equations with algebraic solutions \cite{BW, Dahmen} and 
the mathematical physics related to Lam\'e equations. 
Thus a substantial portion of Theorem \ref{hyp-ell-thm} overlaps with existing
literature. 
However there are a number of issues for which we were unable 
to locate satisfactory treatments in the literature. 
For instance, why the 
closure $\bar X_n$ of \(Y_n\) in ${\rm Sym}^n E$ 
coincides with the \emph{projective hyperelliptic model}\footnote{See 
\ref{subsubsec:compare-setup}.e for the
definition of the projective hyperelliptic model
defined by \(\ell_n(B)\).}
of the affine curve \(\,C^2=\ell_n(B)\,\)
at the infinity point,
instead of the closure in \(\PP^2\) of the latter curve.

In this paper we attempt to provide a self-contained account 
of the hyperelliptic structure of $\bar{X}_n$, from both the analytic 
and the algebraic point of view, for the convenience of the readers. 
The readers will find in our treatment
the precise behavior of the local structure near every 
$a \in \bar X_n \!\smallsetminus\! X_n$, 
\emph{including the infinity point}, 
and also 
``the meaning'' of the coordinate $\,C\,$ of the Lam\'e curve 
in various contexts (c.f.~Theorem \ref{alg-hyper} and 
Remark \ref{MeaningC}, as well as formulas (\ref{poly=const}), 
(\ref{WronskianC}), (\ref{C2}) and (\ref{(0)})).

\subsubsection{}
Theorem \ref{thm-type II} tells us that algebraic geometric structure
\[\bar{\pi}_n: (\bar{X}_n, Y_n, X_n)\to (\PP^1(\CC), \CC, {U}_n)\]
provides a scaffold for analyzing the mean field equation
\(\,\triangle u+ e^u=8\pi n\delta_0\,\) on a flat torus:
a necessary and sufficient condition 
{for a point 
to be attached to a type II}
solution of the mean field equation 
(\ref{Liouville-eq}) with parameter \(\rho = 8\pi n\)
is that \(\{p_1,\ldots, p_n\}\)
satisfies the Green equation (\ref{G-eq}).  
Of course one wish to pursue the above thread 
to bring about a complete analysis of the set of all
solutions of (\ref{Liouville-eq}).
The case \(n=1\), where \(\bar{X}_1=E\), 
has been successfully treated in \cite{LW} with a 
combination of two techniques.
Naturally one would like to
extend these methods to higher values of \(n\).

\subsubsection{}
The first technique is to
use the double cover 
$E \to \mathbb P^1(\CC) \cong S^2$ and
the evenness of $u$ to transform 
the equation to another one on $S^2$ (with more singular sources). 
To extend this step to a general positive integer $n$, we believe that 
the hyperelliptic structure $\bar{\pi}:\bar X_n\to \PP^1({\CC})$ 
{is}
the right replacement of $E$. 

It will be shown in Part II of this series of articles that the map 
\[\sigma: \bar X_n \to E,\qquad 
\{p_1,\ldots,p_n\} \mapsto \sigma(\{p_1,\ldots,p_n\}) = \sum p_i
\]
is a branched covering of degree $\tfrac{1}{2} n(n + 1)$,
and the rational function 
$$
{\bf z}([a]) := \zeta(\sum a_i) - \sum \zeta(a_i)
\qquad\textup [a]=\{[a_1],\ldots,[a_n]\}\in 
\bar{X}_n
$$
on $\bar X_n$ 
is a primitive generator of the extension field 
$K(\bar X_n)$ over $K(E)$. Using this, 
a ``pre-modular form'' $Z_n(\sigma; \tau)$ for $\tau \in \mathbb H$ 
and $\sigma \in E_\tau$ will be constructed, which has the property 
that non-trivial solutions to the 
Green equation (\ref{G-eq}) on $\bar X_n$ 
correspond exactly to the zeros of 
the single function $Z_n(\sigma; \tau)$.

\subsubsection{}
The second technique employed for the case \(\rho = 8\pi\)
is to use the \emph{method of continuity} 
to connect the equation for $\rho = 8\pi$ to the known 
{case when}  $\rho = 4\pi$ 
by establishing the non-degeneracy of the
linearized equations of (\ref{Liouville-eq}). 
For general $\rho$, such a non-degeneracy statement is out of reach 
at this moment. However, since equation (\ref{Liouville-eq}) 
has a solutions $u_\eta$ for every 
$\rho = 8\pi \eta \not\in 8\pi\mathbb N$, 
it is natural to study the limiting behavior of $u_\eta$ 
as $\eta \to n$. If the limit does not blow up, 
it will converge to a solution $u$ for $\rho = 8\pi n$. 
For the blow-up case, we will establish a connection between 
the location of the blow-up set and the hyperelliptic geometry 
of $Y_n \to \mathbb P^1(\CC)$:

\begin{theoremss} \label{blow-up-set}
Let $S = \{p_1, \ldots, p_n\}$ be an element 
{of \(\textup{Sym}^n E\) such that 
\(p_i\neq p_j\) whenever \(i\neq j\).}
Suppose that \(S\) is the blow-up set 
of a sequence of solutions $u_k$ of the 
Louiville equation \textup{(\ref{Liouville-eq})} 
{with paratmeter \(\rho_k\) such that} 
$\rho_k \to 8\pi n$ as $k \to \infty$. Then $S\in Y_n$. Moreover, 
\begin{itemize}
\item[(1)] If $\rho_k \ne 8\pi n$ for every $k$ then $S$ 
is a branch (or ramification) point of $Y_n$.

\item[(2)] If $\rho_k = 8\pi n$ for all $k$ then $S$ is 
\emph{not} a branch point of $Y_n$.
\end{itemize}
\end{theoremss}

The proof 
of Theorem \ref{blow-up-set} will be given in \S\ref{deformation}. 
Theorems \ref{blow-up-set} and \ref{hyp-ell-thm} provide rather precise 
information on the blow-up set of sequences of
solutions of (\ref{Liouville-eq}), 
which we believe will play a
fundamental role in future research on the mean field equations.
\bigbreak

It is a pleasure to thank the referee for detailed comments
on the literature on integrable systems and for suggestions
on directions for future research.
C.-L.\ C.\ would like to thank Frans Oort for discussions 
on equation (\ref{power-sys}) during lunch in November 2011. 
He would also like to thank the Institute of Mathematics of Academia Sinica
for support during the academic year 2012--2013, and
also the Taida Institute for Mathematical Sciences (TIMS) 
and the Department of Mathematics 
of National Taiwan University for hospitality.

\section{Liouville equations with singular source} \label{Liouville-thm}
\setcounter{equation}{0}

\subsection{A Theorem of Liouville}

\subsubsection{}
We begin with a quick review of a
classical theorem of Liouville.
\begin{propositionss}\label{prop:liouville}
Every \(\RR\)-valued \(C^2\) solution $u$ of the differential equation 
\begin{equation} {\label{Liouville}}
\triangle u + e^u = 0
\end{equation}
in a 
simply connected domain $D \subset \mathbb{C}$ 
can be expressed in the form
\begin{equation}{\label{3-2}}
u = \log \frac{8|f'|^2}{(1 + |f|^2)^2}
\end{equation}
where $f$ is a holomorphic
function on $D$ whose derivative \(f'\) 
does not vanish on \(D\).
Conversely for every meromorphic function \(f\) on an open
subset \(V\subset \CC\) with at most simple poles whose
derivative does not vanish on \(V\), the function
\(\,\log \frac{8|f'|^2}{(1 + |f|^2)^2}\,\) is 
a smooth function which satisfies equation
\textup{(\ref{3-2})}.
\end{propositionss}
A proof of Liouville's theorem \ref{prop:liouville} 
is given in \ref{pf_liouville}
for the convenience of the readers. 

\begin{definitionss}\label{defnss_develop}
Let \(u\) be a real-valued \(C^2\)-function
on a domain \(D\subset \CC\) 
and satisfies
equation (\ref{Liouville}) on \(D\)
A meromorphic function \(f\) on a 
(not necessarily connected) covering space
\(\pi:\tilde D\to D\) of 
\(D\) is a
\emph{developing map} of $u$ if 
\begin{equation}{\label{def_devel}}
u(z) = \log \frac{8|f'(\tilde z)|^2}{(1 + |f(\tilde z)|^2)^2}
\quad
\textup{for every } z\in D\ 
\textup{and every } \tilde z\in \tilde D
\ \textup{above } z.
\end{equation}
For any pole \(\tilde{z}_0\in \tilde D\) of \(f\), 
the equality (\ref{def_devel}) for \(z=z_0\) means that
the right hand side of (\ref{def_devel}) has a finite
limit as \(z\to \tilde{z}_0\), and this limit is equal to 
\(u(\pi(\tilde{z}_0))\).
\end{definitionss}
\begin{remarkss}\label{rmk_etale}
(a) It is easy to see that every developing map
\(f:\tilde D\to \mathbb P^1(\CC)\) 
of a \(C^2\) solution \(u\) of 
(\ref{Liouville}) on \(D\) has no critical point on \(\tilde D\).
In other words the holomorphic map 
\(f:\tilde D\to \mathbb P^1(\CC)\) is \'etale.
More explicitly this means that the derivative 
\(f'\) of the meromorphic function \(f\) does not vanish 
at every point where \(f\) is holomorphic, and
\(f\) has at most simple poles.
\smallbreak

\noindent
(b) The proof of Liouville's theorem \ref{prop:liouville} 
in \ref{pf_liouville} 
provides another interpretation of developing maps:
a developing map \(f\) for a solution
\(u\) of (\ref{Liouville}) is an orientation-preserving
local isometry, from a covering space of 
\(D\) with the Riemannian metric \(\frac{1}{2}\,e^u\,(dx^2+dy^2)\), 
to \(\mathbb{P}^1(\CC)\) with the  
Fubini-Study metric
(or equivalently the unit sphere \(S^2\) with the standard metric) 
which has constant Gaussian curvature \(1\).
\end{remarkss}

Developing maps are not unique. In Lemma \ref{lemma:psu2}
below we show that different developing maps of a solution
\(u\) are related by special \emph{unitary} M\"obius transformations. 


\begin{lemmass}\label{lemma:psu2}
Let \(u\) be a \(C^2\) solution of equation 
\textup{(\ref{Liouville})} 
on a domain \(D\subset \CC\).
Let \(f\) be a developing map for \(u\) on a covering space \(\tilde D\)
of \(D\) as in as in Definition
\textup{\ref{defnss_develop}}.
\begin{itemize}
\item[(1)] The solution \(u\) of \textup{(\ref{Liouville})} 
and its developing map \(f\) are related by\footnote{The right hand
  side 
of equation (\ref{3-3}) is the 
\emph{Schwarzian derivative}
$S(f)$ of $f$; the equality here means that the Schwarzian
derivative of \(f\) descends to the function
\(u_{zz} - \frac{1}{2} u_z^2\) on \(D\).}
\begin{equation}{\label{3-3}}
u_{zz} - \frac{1}{2} u_z^2 = \frac{f'''}{f'} - \frac{3}{2}
\left(\frac{f''}{f'}\right)^2.
\end{equation}

\item[(2)] Let \(U\) be an element of 
\(\textup{PSU}(2)\) represented by a \(2\times 2\)
special unitary matrix \(\begin{pmatrix}a&b\\c&d
\end{pmatrix}\) with
\(a,b,c,d\in\CC\).
The function \(Uf:=\frac{af+b}{cf+d}\) 
is also a developing map of
\(u\) on \(\tilde D\).

\item[(3)] Assume that the covering space
\(\tilde D\) of \(D\) is connected.
Suppose that \(\tilde f\) is another developing of
\(u\) on \(\tilde D\).
There exists an element \(T\in \textup{PSU}(2)\)
such that \(\tilde{f}= Tf\).

\end{itemize}
\end{lemmass}

\begin{proof}
The statements (1) and (2) are easily verified by direct
calculations. It remains to prove (3). Notice first that
the Schwarzian derivatives
of \(f\) and \(\tilde f\) are equal by (1).
Hence there exists a M\"obious transformation \(T\), 
say represented by an element
\(\begin{pmatrix}a&b\\c&d\end{pmatrix}\in \textup{SL}_2(\CC)\), such that
\(\tilde f = Tf = \frac{af+b}{cf+d}\).\footnote{Here we have used
the assumption that \(\tilde D\) is connected and 
a basic property of Schwarzian derivatives: if \(S(g_1)=S(g_2)\) 
for two locally non-constant meromorphic functions \(g_1\) and
\(g_2\), then \(g_1\) and \(g_2\) differ by a M\"obius transformation.
This is consequence of the special case that \(S(g)=0\)
if and only if \(g\) is a linear fractional transformation and the
cocycle property of Schwarzian derivatives: 
\(S(g\circ h)(z)= S(g)(h(z))\cdot h'(z)^2 + S(h)(z)\).}
From 
\[\log\frac{8\vert f'\vert^2}{(1+\vert f\vert^2)^2}
=\log \frac{8\vert (Tf)'\vert^2}{(1+\vert Tf\vert^2)^2}\,,\qquad
(Tf)'=\frac{f'}{(cf+d)^2}\,,\]
we deduce that 
\(\vert af+b\vert^2+ \vert cf+d\vert^2 = 1+ \vert f\vert^2\) on
\(\tilde D\).
Hence the quality
\[
\vert az+b \vert^2 + \vert cz+d\vert^2 = 1+ \vert z\vert^2
\]
holds on \(\CC\) because meromorphic maps are open.
Applying \(\,\partial\bar{\partial} \log\,\) 
to both sides of the last displayed
equality, we see that the M\"obius transformation \(T\)
preserves the Fubini-Study metric on 
\(\mathbb P^1(\CC)\), or equivalently 
the spherical metric on the \(2\)-sphere \(S^2\). 
So \(T\) is an element of \(\textup{PSU}(2)\),
because \(\textup{PSU}(2)\) 
is the group of all orientation preserving
isometries of \(\mathbb P^1(\CC)\).
\end{proof}

%

\begin{remarkss}
In the notation of 
Lemma \ref{lemma:psu2}, 
let 
$V$ be an element of  ${\rm SU(2)}$ 
such that 
$VUV^{-1} = \begin{pmatrix} e^{i\theta} & 0\\
0 & e^{-i\theta}\end{pmatrix}$ for some $\theta\in \mathbb R$.
Then  
the two developing maps \(V \tilde f\) and \(Vf\) of
\(u\) are related by
\begin{equation*}
V\tilde f = e^{2 i \theta} Vf.
\end{equation*}
\end{remarkss}

\subsubsection{}\label{reformulation1}
Lemmas \ref{lemma:psu2} 
can be reformulated as follows. See also \ref{pf_liouville}.
\medbreak

\emph{Let \(u\) be a  \(C^2\) solution of equation (\ref{Liouville}) 
on a domain \(D\subset \CC\).
There exists a (not necessarily connected)
covering space \(\pi_u:\mathscr{D}_{u,\textup{univ}}\to D\),
a left action of \(\textup{PSU}(2)\) on \(\mathscr{D}_u\),
and a meromorphic function \(f_{u,\textup{univ}}:
\mathscr{D}_{u,\textup{univ}}\to \mathbb{P}^1(\CC)\) on 
\(\mathscr{D}_{u,\textup{univ}}\)
satisfying the following properties.}
\begin{itemize}
\item[(a)] \emph{\(\pi:\mathscr{D}_{u,\textup{univ}}\to D\) 
is a left principle homogeneous
space for \(\textup{PSU}(2)\).}
\item[(b)] 
\emph{\(f_{u,\textup{univ}}\) is a developing map for \(u\).}

\item[(c)] \emph{For any open subset \(U\subset D\) and any
developing map \(f\) for \(u\vert_U\) on a 
covering space \(\pi:\tilde{U}\to U\) of \(U\), 
there exists a unique holomorphic map
\(g:\tilde{U}\to \mathscr{D}_{u,\textup{univ}}\) such that
\(f(\tilde{z})=f_{{u,\textup{univ}}}(g(\tilde{z}))\)
for all \(\tilde{z}\in \tilde{U}\)}.

\end{itemize}


\subsubsection{}\label{pf_liouville}
{\scshape A proof of Proposition \ref{prop:liouville}.}
\enspace
From the perspective of differential geometry, 
equation (\ref{Liouville}) is simply
the 
{prescribed} 
Gaussian curvature equation for the metric 
\[
g=\frac{1}{2}\,e^{u}\,(dx^2+dy^2)
\]
on the domain \(D\) to have Gaussian curvature \(K_g=1\).
Given that \(u\) is a solution of the equation (\ref{Liouville}),
a simple calculation shows that
\begin{equation} \label{Gauss}
K_g = -e^{-u}\triangle u =1.
\end{equation}
So at any given point \(z_1\in D\), there exists a 
meromorphic function \(f_1(z)\) on a neighborhood of \(V(z_1)\) of
\(z_1\) such that
\(u(z)=\log\frac{8 |f_1'(z)|^2}{(1+|f_1(z)|^2)^2}\)
for all \(z\in V(z_1)\),
because the Fubini-Study metric on 
\(\CC=\mathbb{P}^1(\CC)\smallsetminus \{\infty\}\) is
\[\frac{4\,dz\,d\bar{z}}{(1+z\bar{z})^2}.\]
The collection of germs of all such local developing maps \(f_1\)
form a locally constant sheaf \(\mathscr{D}ev(u)\) over \(D\) 
with monodromy group \(\textup{PSU}(2)\) according to
Lemma \ref{lemma:psu2}. 
Note that \(\mathscr{D}ev(u)\) is the locally constant sheaf
attached to the covering space \(\mathscr{D}_{u,\textup{univ}}\)
in \ref{reformulation1}.
The locally constant sheaf \(\mathscr{D}ev(u)\) has a global section
\(g\) because the domain \(D\) is simply connected.
This global section \(g\) is a developing map of \(u\) on \(D\).
The holomorphic map \(\tilde g\)
from \(D\) to \(\mathbb P^1(\CC)\) defined by 
the meromorphic function \(g\) is \'etale as we have
remarked in \ref{rmk_etale}\,(b), hence the image \(\tilde g\) 
misses some point \(Q\) of \(\mathbb P^1(\CC)\).
Pick an element \(U\in \textup{PSU}(2)\) which sends \(Q\) to
\(\infty\). Then \(f:=Ug\) is a holomorphic function on \(U\)
which is also a developing map of \(u\).
\qed


\subsection{Liouville theory on tori with isolated singular data} 
\label{0-constraint}
It is a challenge to extend the Liouville theory recalled in
the previous section to
oriented Riemann surfaces and with singular sources.
In this paper we will consider the genus one case, so the 
Riemann surface will be a flat torus \(E\).  
Moreover we put just one singular source on \(E\), and we will make
this singular source the
additive unity \(0\)
for a holomorphic group law on \(E\).

\subsubsection{}
We choose and fix a non-zero global holomorphic one-form \(\beta\)
on \(E\), so that integrating \(\beta\) along paths starting from
\(0\) gives an isomorphism
\(\int_?\,\beta: E\xrightarrow{\sim} \CC/\Lambda\)
for a lattice \(\Lambda\subset \CC\), so that 
\(\beta\) is the pull-back of the one-form 
\(\,dz=dx+\sqrt{-1}dy\,\) descended to 
\(\CC/\Lambda\). 
The flat torus \(E\) will be identified with \(\CC/\Lambda\) in
the rest of this paper.
Let \(\omega_1, \omega_2\) be a 
\(\ZZ\)-basis of \(\Lambda\) such that 
\(\tau:=\omega_2/\omega_1\) satisfies
\(\textup{Im}(\tau)>0\).
Let \(\omega_3:= -\omega_1-\omega_2\),
so that \(\omega_1+\omega_2+\omega_3=0\).
\medbreak

We will consider the 
\emph{mean field equation} 
\begin{equation}{\label{3-1}}
\triangle u + e^u = \rho\cdot\delta_0,\quad \rho\in \mathbb{R}_{>0}
\end{equation}
on $E=\CC/\Lambda$, where 
\(\,\triangle = \frac{\partial^2}{\partial x^2}+
\frac{\partial^2}{\partial y^2}\), 
$\delta_0$ is the Dirac measure at
$0$ and we have identified
\(L^1\)-functions with (signed) measures using the Lebesgue
measure \(\,dx\,dy\,\) on \(\CC/\Lambda\), so 
the equation (\ref{3-1}) means that
\[
\int_E (u\cdot \triangle h + e^u\cdot h)\,dx\,dy =
\rho\cdot h(0)
\]
for every smooth function \(h\) on \(E\).
The corresponding geometric problem is the equation
\begin{equation*}
K_g = -e^{-u} \triangle u = 1 - \rho e^{-u} \delta_0
\end{equation*}
for the Gaussian curvature \(K_g\) of the metric
\(\,g=\frac{1}{2} e^u(dx^2+dy^2)\,\) on \(E\),
which has a highly non-classical character.
\medbreak

We will be mostly interested in the case when the parameter 
\(\rho\) of the equation (\ref{3-1}) is an
integer multiple of \(\,4\pi\).  This integrality condition
on \(\rho\) implies that every developing map 
of a solution of (\ref{3-1}) is meromorphic
locally on \(E\) (and not just on \(E\smallsetminus\{0\}\)).

\begin{lemmass}\label{lemma:mero}
Let \(u\) be a solution of \textup{(\ref{3-1})} 
on \(E\), where the parameter
\(\rho=4\pi l\) for a positive integer \(l\).
Let \(f_1\) be a developing map of the restriction to 
\(E\smallsetminus \{0\}\) of \(u\), so that 
\(f_1\) is a holomorphic function on a universal covering 
\(U\)
of \(E\smallsetminus \{0\}\) whose derivative
does not vanish on \(U\).
Then \(f_1\) extends to a meromorphic function on
a covering space of \(E\) in the following sense:
There exists a covering space \(\gamma:\tilde E\to E\) of \(E\) 
such that the following
statements hold.
\begin{itemize}
\item[(a)] The holomorphic function
\(f_1\) on \(U\)
descends to a function \(f_2\) on the covering space
\(\gamma^{-1}(E\smallsetminus\{0\})\) of \(E\smallsetminus\{0\}\)

\item[(b)] The holomorphic function \(f_2\)
on the open subset 
\(\gamma^{-1}(E\smallsetminus\{0\})\) of \(\tilde E\) 
extends to a meromorphic function on \(\tilde E\).
\end{itemize}
Equivalently, \(f_2\) defines a holomorphic map from \(\tilde E\) to 
\(\mathbb P^1(\CC)\).
\end{lemmass}

\begin{proof}
This statement is local at \(0\in E\).  A proof can be found in
\cite{CW, LW2, PT}, based on the following inequality:
For a punctured disk \(\Delta^{\times}_{\epsilon}\) with a small
radius \(\epsilon\)
we have 
\begin{equation*}
\infty > \int_{\Delta^\times_{\epsilon}} e^u \,dA =
\int_{\Delta^\times_{\epsilon}} \frac{8|f'|^2}{(1 +
|f|^2)^2}\,dA,
\end{equation*}
where the right hand side is the spherical area under the
inverse stereographic projections covered by
$f(\Delta^\times_{\epsilon})$. 

Alternatively, from the well-known formula
\[
\triangle \log \sqrt{x^2+y^2} = 2\pi\cdot \delta_{(0,0)}
\]
on \(\RR^2\), one sees that every 
holomorphic map from a neighborhood \(V\) of \(0\in E\) to 
\(\mathbb P^1(\CC)\) with multiplicity \(l+1\) at \(0\)
is a developing map of a solution of 
(\ref{3-1}) in \(V\). Since any two local developing maps
of any local solution of (\ref{3-1}) differ by an element
of \(\textup{PSU}(2)\), we conclude that every
developing map of every solution of (\ref{3-1}) in a neighborhood
of \(0\in E\) ``is'' a meromorphic function in a neighborhood
of \(0\in E\).
\end{proof}

\begin{remarkss}
As an immediate consequence of the fact that
\(\triangle \log \vert z\vert = 2\pi\cdot \delta_{0}\), one
sees that if a meromorphic function \(f\) on an open neighborhood \(U\)
of \(0\in \CC\) such that
the locally \(L^1\) function
\(u = \log \frac{8|f'|^2}{(1 + |f|^2)^2}\) 
satisfies
\(\triangle u + e^u = \rho\cdot\delta_0\) on \(U\)
for some real number \(\rho\), then
\(\rho=4\pi\cdot l\), where \(l+1\in \NN_{>0} \) is the multiplicity
of \(f\) at \(0\). 
So the parameter \(\rho\) in the equation (\ref{3-1}) must be
in \(4\pi\cdot \NN_{\geq 0}\) if a developing map of a solution \(u\) 
is a meromorphic function on \(\CC\). 
Note also that the equation (\ref{3-1}) has no solution when 
\(\rho=0\), for otherwise the elliptic curve has a metric
with constant Gaussian curvature \(1\), contradicting the
Gauss--Bonnet theorem.
\end{remarkss}
\medbreak



\begin{lemmass}\label{order}
Let \(u\) be a solution of the
of equation \textup{(\ref{3-1})} on \(E\).
Assume that the parameter \(\rho\) is of the
form \(\rho=4\pi l\) where \(l\) is a positive integer.
\begin{itemize}
\item[(1)] There exists a meromorphic function \(f\)
on the universal covering \(\CC\) of 
\(E\) which is a developing map of \(u\). 
Let \(\tilde{f}:\CC\to \mathbb{P}^1(\CC)\) be the holomorphic
map corresponding to the meromorphic function \(f\) on \(\CC\).

\item[(2)] For every \(T\in \textup{PSU}(2)\), 
the meromorphic function \(Tf\) is also a developing map 
of \(u\).  Moreover every developing map of \(u\) is equal
to \(Tf\) for some element \(T\in \textup{PSU}(2)\).

\item[(3)] Suppose that \(z_0\) is an element of the lattice \(\Lambda\).
The holomorphic map \(\tilde f:\CC\to \mathbb P^1(\CC)\) has
multiplicity \(l+1\) at \(z_0\). In other words either 
\(f\) is holomorphic at \(z_0\) and \(f'\) has a zero of order
\(l\) at \(z_0\), or 
\(f\) has a pole of order \(l+1\) at \(z_0\).

\item[(4)] The holomorphic map \(\tilde f:\CC\to \mathbb P^1(\CC)\)
has no critical point outside \(\Lambda\). 
In other words if \(z_1\in \CC\smallsetminus \Lambda\), then
either \(f\) is holomorphic at \(z_1\) and \(f'(z_1)\neq 0\),
or \(f\) has a simple pole at \(z_1\).
\end{itemize}
\end{lemmass}

\proof
The statement (1) is a corollary of Lemma \ref{lemma:mero}.
The statement (2) is a consequence of the interpretation of
developing maps as local isometries from 
the conformal metric \(\,\frac{1}{2}\,e^u\, dz\,d\bar{z}\,\)
to the Fubini-Study metric 
\(\frac{4\,dz\,d\bar{z}}{(1+z\bar{z})^2}\)
and the fact that \(\textup{PSU}(2)\) is the group of all
orientation preserving isometries of \(\mathbb{P}^2(\CC)\)
with the Fubini-Study metric; c.f.\ \ref{pf_liouville}.
The statement (3) is a consequence of the last paragraph of 
the proof of Lemma \ref{lemma:mero}.
The statement (4) follows from \ref{rmk_etale}\,(a). \qed

\begin{remarkss} 
We discuss how to relate solutions of (\ref{3-1}) on 
\(\CC/\Lambda\) and \(\,\CC/(t\!\cdot\!\Lambda)\) for
\(t\in \CC^{\times}\).
Suppose that 
\(u(z;\Lambda)\) is a solution of the singular Liouville equation
(\ref{3-1}) on \(\CC/\Lambda\) and
\(\,f(z;\Lambda)\) is a developing map on \(\CC\) for
\(u(z;\Lambda)\).
It is easy to check that \(\,u(w;t\Lambda):=
u(t^{-1}w;\Lambda)-\log(t\bar{t})\,\) is a 
solution of (\ref{3-1}) on the elliptic curve \(\,\CC/t\Lambda\)
whose universal covering is the complex plane \(\CC_w\) with coordinate
\(\,w=tz\).
Moreover 
\(\,f(t^{-1} w;\Lambda)\,\) is a developing map for the
solution \(\,u(w;t\Lambda)\) on \(\,\CC_w\).

Of course the above ``gauge transformation rules'' reflects the
fact that the three terms of equation (\ref{3-1}) scale differently
when the coordinate of \(\CC\) changes from \(z\) to \(w=tz\) for
a non-zero constant \(t\): the equation (\ref{3-1})
is better written as
\[
2\sqrt{-1}\,\partial\bar{\partial}u 
+ \tfrac{\sqrt{-1}}{2} e^u\, dz\!\wedge\! d\bar{z}
=\rho\cdot \delta_0,
\]
where the last term \(\delta_0\) is the \(\delta\)-\emph{measure}
at \([0]\in\CC/\Lambda\). The second term 
\(\,\tfrac{\sqrt{-1}}{2} e^u\, dz\!\wedge\! d\bar{z}\,\) in the above 
equation depends on the choice of a global holomorphic \(1\)-form
on the elliptic curve, while the other two terms do not.

\end{remarkss}

\begin{lemmass} \label{lemma:n-s-cond-develop}
Let \(f\) be a meromorphic function on the universal covering
\(\CC\) of \(E=\CC/\Lambda\).
This function \(f\) is a developing map of a solution 
\(u\) of \textup{(\ref{3-1})} with parameter 
\(\rho=4\pi l\in 4\pi \NN_{> 0}\)
if and only if the following conditions hold.
\begin{itemize}

\item[(1)] The holomorphic map \(\tilde f:\CC\to \mathbb{P}^1(\CC)\)
corresponding to \(f\) has multiplicity \(l+1\)
at every point above \(\Lambda\), and it has no critical
point outside \(\Lambda\).

\item[(2)] For every \(\omega\in \Lambda\), there exists a
unique element \(T\in\textup{PSU}(2)\) such that
\[f(z+\omega)=(Tf)(z)\quad \forall z\in \CC.\]

\end{itemize}

\end{lemmass}

\proof
The condition (1) means that the equality
\(\triangle u + e^u 
=4\pi\,l\cdot \sum_{\omega\in \Lambda}\delta_{\omega}
\) holds 
for the function
\(u:=\frac{8\vert f'\vert}{(1+\vert f\vert^2)2}\) on \(\CC\).
The condition (2) means that 
\(u\) descends to a function on \(\CC/\Lambda\).
\qed

\begin{definitionss}
Let \(u\) be a  solution \(u\) of the equation
\textup{(\ref{3-1})} on \(E\), where the parmeter
\(\rho\in 4\pi\cdot \NN_{\geq 0}\).
Let \(f\) be a meromorphic function on \(\CC\)
which is a developing map of \(u\).
\begin{itemize}

\item[(1)] The \emph{monodromy representation} \(\rho_f\) 
of the fundamental group \(\Lambda\) of \(E\)
attached to the developing map \(f\) 
is the
group homomorphism
\(\rho_f: \Lambda\to \textup{PSU}(2)\) such that
\[f(z+\omega)=(\rho(\omega)f)(z)\quad
\forall\omega\in \Lambda,\ \forall z\in \CC .\]

\item[(2)] The \emph{monodromy} of the solution 
\(u\) of equation \textup{(\ref{3-1})}
is the
\(\textup{PSU}(2)\)-conjugacy class
of the homomorphism \(\rho_f:\Lambda\to \textup{PSU}(2)\),
which depends only on \(u\) and not on 
the choice of developing map \(f\).

\end{itemize}
\end{definitionss}
 
\subsection{Monodromy constraints}
Next we review the monodromy constraints on a developing map
\(f\) of a solution of (\ref{3-1}) on \(E\), resulting from
the fact that the fundamental group of \(E\) is a free abelian
group of rank two.
By Lemma \ref{lemma:n-s-cond-develop}, there exist
$T_1=\rho_f(\omega_1),T_2=\rho_f(\omega_2)\in {\rm PSU}(2)\) 
with the following 
properties:
\begin{equation}{\label{can}}
\begin{split}
f(z + \omega_1) &= T_1 f,\\
f(z + \omega_2) &= T_2 f.
\end{split}
\end{equation}
In addition \(T_1 T_2=T_2 T_1\) in \(\textup{PSU}(2)\) because 
the source of
the monodromy representation 
\(\rho_{f}: \pi_1(E)\to \textup{PSU}(2)\) is
commutative.


\begin{lemmass}\label{lemma:mono_grpthy}
Let \(\Gamma\) be a commutative subgroup of \(\textup{PSU}(2)\).
\begin{itemize}
\item[(1)] Suppose that \(\Gamma\) is isomorphic to the Klein-four group
\((\ZZ/2\ZZ)^2\). 

\begin{itemize}
\item[(1a)] \(\Gamma\) is conjugate to \(\Gamma_0\), where
\(\Gamma_0\) is the image
in \(\textup{PSU}(2)\) of
\[\left\{\begin{pmatrix}1&0\\0&1\end{pmatrix},
\begin{pmatrix}\sqrt{-1}&0\\0&-\sqrt{-1}\end{pmatrix},
\begin{pmatrix}0&\sqrt{-1}\\\sqrt{-1}&0\end{pmatrix},
\begin{pmatrix}0&1\\-1&0\end{pmatrix}\right\}
\]

\item[(1b)] The centralizer subgroup of \(\Gamma\) in \(\textup{PSU}(2)\)
is equal to \(\Gamma\).
 
\item[(1c)] The centralizer subgroup of \(\Gamma\) in 
\(\textup{PSL}_2(\CC)\) is also equal to \(\Gamma\).

\item[(1d)] The normalizer subgroup
\(\textup{N}_{\textup{PSU}(2)}(\Gamma)\) is
isomorphic to the symmetric group \(S_4\). 
In other words
\(\textup{N}_{\textup{PSU}(2)}(\Gamma)/\Gamma\)
is a semi-direct product of 
\(\Gamma\), and the conjugation action of 
\(\textup{N}_{\textup{PSU}(2)}(\Gamma)\) on \(\Gamma\)
induces an isomorphism 
\[
\textup{N}_{\textup{PSU}(2)}(\Gamma)/\Gamma
\xrightarrow{\sim}
\textup{Aut}_{\textup{grp}}(\Gamma)\cong
\textup{Perm}(\Gamma\smallsetminus\{\textup{Id}\}),
\]
where \(\textup{Perm}(\Gamma\smallsetminus\{0\})\) is the 
set of all permutations of the set
\(\Gamma\smallsetminus\{\textup{Id}\}\).

\item[(1e)] The subset \(\,\{x\in \textup{PSL}(2,\CC)\,\mid\, 
x\cdot \Gamma\cdot x^{-1}\subset \textup{PSU}(2)\,\}\)
of \(\,\textup{PSL}_2(\CC)\,\) is equal to \(\textup{PSU}(2)\).
\end{itemize}

\item[(2)] If \(\Gamma\) is not isomorphic to \((\ZZ/2\ZZ)^2\),
then \(\Gamma\) is contained in a maximal torus of 
\(\textup{PSU}(2)\); i.e.\ there exists an element
\(T_0\in \textup{PSU}(2)\) such that 
\(T_0\cdot \Gamma\cdot T_0^{-1}\) is contained in the image 
in \(\textup{PSU}(2)\)
the diagonal maximal torus 
\[\left\{\begin{pmatrix}e^{\sqrt{-1}\theta}&0\\
0&e^{-\sqrt{-1}\theta}\end{pmatrix}\colon \theta\in \RR/\pi\ZZ\right\}
\subset\ \textup{SU}(2)\]

\end{itemize}
\end{lemmass}

\begin{proof}
The spectral theorem tells us that every element of 
\(\textup{U}(2)\) is conjugate in \(\textup{U}(2)\) to a
diagonal matrix.  Using this it is easy to verify the
following assertion, whose proof is omitted here.
\begin{quotation}
\emph{Suppose that \(u\) is a non-trivial element of
\(\textup{PSU}(2)\).}
\begin{itemize}

\item \emph{If \(u^2\neq 1\)
in \(\textup{PSU}(2)\), then the centralizer 
\(\textup{Z}_{\textup{PSU}(2)}(u)\) of \(u\) is a maximal torus
in \(\textup{PSU}(2)\), i.e.\ a conjugate of the image of the
diagonal maximal torus}
\[\left\{\begin{pmatrix}e^{\sqrt{-1}\theta}&0\\
0&e^{-\sqrt{-1}\theta}\end{pmatrix}\colon \theta\in
\RR/\pi\ZZ\right\}\]

\item \emph{If \(u\) is an element of order two in 
\(\textup{PSU}(2)\), then the centralizer subgroup
\(\textup{Z}_{\textup{PSU}(2)}(u)\) of \(u\) 
in \(\textup{PSU}(2)\) is a semi-direct product
of a maximal torus of \(\textup{PSU}(2)\) with a group
of order two, equal to the normalizer of a maximal torus.
Moreover \(\textup{Z}_{\textup{PSU}(2)}(u)\) contains a unique
subgroup which is isomorphic to 
\((\ZZ/2\ZZ)^2\).}
\end{itemize}
\end{quotation}
The statement (2) follows, so do (1a), (1b) and (1c). 
\smallbreak

To prove (1d),
by (1a) and (1b) it suffices to show that 
the normalizer subgroup
\(\textup{N}_{\textup{PSU}(2)}(\Gamma_0)\) on 
\(\Gamma_0\) contains a subgroup \(S\) of order \(6\)
which intersect \(\Gamma\) trivially.
Let \(\delta\) be the image of 
\(\begin{pmatrix}0&-e^{-\pi\sqrt{-1}/4}\\
e^{\pi\sqrt{-1}/4}&0\end{pmatrix}\) in \(\textup{PSU}(2)\)
and let \(\gamma\) be the image of 
\(\frac{1}{\sqrt{2}}\!\begin{pmatrix}-1&1\\ \sqrt{-1}&\sqrt{-1}\end{pmatrix}
\) in \(\textup{PSU}(2)\).
It is straightforward to check that 
\(\delta\) has order \(2\) and induces a transposition
on \(\Gamma\smallsetminus\{\textup{Id}\}\),
\(\gamma\) has order \(3\) and
\(\delta\cdot \gamma\cdot \delta^{-1}=\gamma^{-1}\).
It follows that 
\(\textup{N}_{\textup{PSU}(2)}(\Gamma_0)\) is a semi-direct
product of \(\Gamma_0\rtimes S_3\), so
\(\textup{N}_{\textup{PSU}(2)}(\Gamma_0)\) is isomorphic to \(S_4\).
We have proved (1d).  
Alternatively, it is well known that \(\textup{PSU}_2\) contains 
a finite subgroup isomorphic to \(S_4\). The statement (1d) follows
from this fact, (1a) and (1b).

\smallbreak
Finally the statement (1e) follows from (1a), (1d) and (1c):
Suppose that \(x\in \textup{PSL}_2(\CC)\) and 
\(\,\textup{Ad}(x)(\Gamma_0)
=x\cdot \Gamma_0\cdot x^{-1}\subset \textup{PSU}(2)\).
By (1a) and (1d), there exists \(y\in \textup{PSU}(2)\) such that 
\(\,y\cdot x\,\) commutes with every element of \(\Gamma_0\).
By (1c) \(y\cdot x\in \Gamma_0\), hence 
\(x\in y^{-1}\cdot \Gamma_0\subset \textup{PSU}(2)\).
\end{proof}

\begin{remark*} The 
group \(\textup{N}_{\textup{PSU}(2)}(\Gamma)\) is also isomorphic to
\(\textup{SL}_2(\ZZ/4\ZZ)/\{\pm \textup{I}_2\}\), the quotient
of \(\textup{SL}_2(\ZZ)\) by the subgroup generated by the
principal congruence subgroup of level \(4\) and 
\(\{\pm\textup{I}_2\}\).
\end{remark*}

\begin{corollaryss} \label{cor:comm_homo}
Let \(\rho:\Lambda\to \textup{PSU}(2)\) be a group
homomorphism. 
\begin{itemize}
\item[(i)] If the image of \(\rho\) is isomorphic to 
\((\ZZ/2\ZZ)^2\), then \(\rho\) is conjugate to the homomorphism
which sends \(\omega_1\) to the image of 
\(\begin{pmatrix}\sqrt{-1}&0\\0&-\sqrt{-1}
\end{pmatrix}\)
and \(\omega_2\) to the image of 
\(\begin{pmatrix}0&\sqrt{-1}\\\sqrt{-1}&0
\end{pmatrix}\).

\item[(ii)] If the image of \(\rho\) is not isomorphic to 
\((\ZZ/2\ZZ)^2\), then there exists real numbers \(\theta_1,\theta_2\)
such that \(\rho\) is conjugate to the homomorphism which
sends \(\omega_i\) to
\(\begin{pmatrix}e^{\sqrt{-1}\theta_1}&0\\0&e^{\sqrt{-1}\theta_2}
\end{pmatrix}\) for \(i=1,2\).

\end{itemize}
\end{corollaryss}

\noindent
Lemma \ref{cor:types} below follows from 
Corollary \ref{cor:comm_homo} and Lemma \ref{order}\,(1),\,(2).

\begin{lemmass}\label{cor:types}
Let \(u\) be solution of equation 
\textup{(\ref{3-1})}
where the parameter \(\rho>0\) is an integer multiple of 
\(4\pi\). 
\begin{itemize}
\item[Type I.] If the image in \(\textup{PSU}(2)\) of 
the monodromy of \(u\) is isomorphic to \((\ZZ/2\ZZ)^2\),
then there exists a developing map \(f\) of \(u\) such that 
\begin{equation}{\label{case-type-I}}
\begin{split}
f(z + \omega_1) &= -f(z)\quad \forall z,\\
f(z + \omega_2) &= \frac{1}{f(z)}\quad \forall z.
\end{split}
\end{equation}
Moreover the set \(\{f, -f, f^{-1}, -f^{-1}\}\) is uniquely
determined by the solution \(u\).

\item[Type II.] Suppose that the image in \(\textup{PSU}(2)\) of 
the monodromy of \(u\) is not isomorphic to \((\ZZ/2\ZZ)^2\).
There exists a developing map \(f\) of \(u\)
and two real numbers \(\theta_1, \theta_2\) such that 
\begin{equation}{\label{case-type-II}}
\begin{split}
f(z + \omega_1) &= e^{2i\theta_1}f(z)\quad \forall z,
\\f(z + \omega_2) &= e^{2i\theta_2}f(z)\quad \forall z.
\end{split}
\end{equation}
If moreover \(\{\theta_1,\theta_2\}\not\subseteq
\frac{1}{2}\ZZ\), then the set
\(\CC^{\times}_1\cdot f \cup \CC^{\times}_1\cdot f^{-1}\)
is uniquely determined by \(u\),
where \(\CC^{\times}_1:=\{w\in\CC:\vert w\vert =1\}\).

\end{itemize}

\end{lemmass}

\begin{definitionss} (a) Let \(f\) be a solution of equation 
\textup{(\ref{3-1})} where the parameter \(\rho>0\) is an integer multiple of 
\(4\pi\).  If the image
of the monodromy representation \(\rho_f\) of \(f\) is 
isomorphic to \((\ZZ/2\ZZ)^2\), then
say that \(f\) is \emph{of type I};
otherwise we say that \(f\) is \emph{of type II}.
\smallbreak

\noindent
(b) A developing map \(f\) which satisfies equation \textup{(\ref{case-type-I})}
(respectively \textup{(\ref{case-type-II})}) will be said to be
\emph{normalized} of type I (respectively type II).
\end{definitionss}

\begin{lemmass}\label{lemma:non-triv-fam}
Let \(f\) be a developing map of a solution of a solution 
\(u\) of equation \textup{(\ref{3-1})}
where the parameter \(\rho>0\) is an integer multiple of \(4\pi\). 
\begin{itemize}
\item[(1)] If \(f\) is of type I and \(T\in \textup{PGL}_2(\CC)\)
is a linear fractional transformation such that
\(T\cdot f\) is again a developing map of a solution of 
equation \textup{(\ref{3-1})} with the same parameter \(\rho\),
then \(T\in \textup{PSU}(2)\) and \(Tf\) is a developing map
of the same solution \(u\) of (\ref{3-1}).

\item[(2)] Suppose that \(f\) is of type II.  
There exists a closed subgroup \(A\) of \(\textup{PGL}_2(\CC)\),
conjugate to
the image in \(\textup{PGL}_2(\CC)\) of the diagonal non-compact real torus
\(\,A_0 :=\left\{\begin{pmatrix}a&0\\0&1\end{pmatrix}\,
:\,a\in \RR^{\times}_{>0}\right\}\), such that the following 
statements hold.
\begin{itemize}
\item[(2a)] \(T\cdot f\) is a developing map of a solution
\(u_{Tf}\) of equation \textup{(\ref{3-1})} with the same parameter 
\(\rho\).

\item[(2b)] \(u_{_{T_1 f}}\neq u_{_{T_2 f}}\) for any
two distinct elements \(T_1, T_2\) in \(A\).

\end{itemize}

\end{itemize}

\end{lemmass}

\begin{proof}
The statement (1) follows from Lemma \ref{lemma:n-s-cond-develop}
and Lemma \ref{lemma:mono_grpthy}\,(1e).
To show (2), we may assume that \(f\) is normalized of type II
and take \(A\) to be the image in \(\textup{PGL}(2,\CC)\) of
\(A_0\).  Then the statement (2a) follows from
Lemma \ref{lemma:n-s-cond-develop}.
The statement (2b) follows from Lemma \ref{lemma:psu2}\,(3)
because the only element of
\(A\) which is conjugate in \(\textup{PGL}_2(\CC)\) to 
an element in \(\textup{PSU}(2)\) is the unity element of \(A\).
\end{proof}

\subsubsection{Logarithmic derivatives of normalized developing maps}
\enspace
In this article we approach the 
mean field equations (\ref{3-1}) 
with $\rho = 4\pi l$, $l \in
\mathbb{N}_{>0}$ 
through the 
logarithmic derivative 
\begin{equation*}
g := (\log f)' = \frac{f'}{f}.
\end{equation*}
of a normalized developing map \(f\) of a solution of (\ref{3-1}).
Recall that such developing maps are meromorphic functions \(f\) on \(\CC\)
satisfying \ref{lemma:n-s-cond-develop}\,(1) and 
either of equations (\ref{case-type-I}), (\ref{case-type-II}).

\begin{lemmass}\label{lemma_zero-pole_g}
Suppose that \(f\) is a normalized developing map of
a solution of \textup{(\ref{3-1})}, and 
\(l:=\rho/4\pi\) is a positive integer.
Let \(g:= f'/f\).
\begin{itemize}
\item[(1)] The developing map
\(f\) on \(\CC\) is holomorphic and non-zero 
at every point of
\(\Lambda\); i.e.\ 
\(f(\Lambda)\subset \mathbb P^1(\CC)\smallsetminus\{0,\infty\}\).

\item[(2)] The meromorphic function \(g\) on \(\CC\)
has a zero of order \(l\) at every point of \(\Lambda\),
no zeros and at most simple poles 
on \(\CC\smallsetminus \Lambda\).

\item[(3)] If \(f\) is of type II, then 
\(g\) descends to a meromorphic function on \(E=\CC/\Lambda\).

\item[(4)] If \(f\) is of type I, then
\(g\) descends to a meromorphic function on
the double cover \(E'=\CC/\Lambda'\), where
\(\Lambda'=\ZZ\cdot \omega_1+\ZZ\cdot 2\omega_2\).

\end{itemize}

\end{lemmass}
\proof
The statements (3) and (4) are immediate from the equations
(\ref{case-type-I}) and (\ref{case-type-II}) for
normalized developing maps.

Clearly \(g\) has at most simple poles on \(\CC\).
Lemma \ref{order}\,(4) implies 
that \(g\) has no zeros on \(\CC\setminus \Lambda\). 
For any point \(z_0\in \Lambda\), if \(f\) has either 
a zero or a pole at \(z_0\), then \(g\) will have a simple
pole at every point of \(\Lambda\), and the meromorphic function
\(\bar g\) on \(E'=\CC/\Lambda'\) defined by \(g\) will have 
no zero but at least one pole, a contradiction.
Therefore \(f\) has values in \(\CC^{\times}\) in a neighborhood
of \(z_0\); we have proved the statement (1).
Lemma \ref{order}\,(3) then implies that
\(g\) has a zero of order \(l\) at every point of \(\Lambda\).
We have proved statement (2).
\qed

%

\subsection{Type I solutions}
In this subsection we will 
show that the existence of
solution of \textup{(\ref{3-1})} 
such that the image of the 
monodromy representation is \((\ZZ/2\ZZ)^2\)
implies that the parameter \(l=\rho/4\pi\) is
an \emph{odd} positive integer.

\subsubsection{Notation for type I}
\begin{itemize}
\item Let \(\omega_1'=\omega_1\), \(\omega_2'=2\omega_2\)
and let 
\(\Lambda':= \ZZ\cdot \omega_1'+ \ZZ\cdot \omega_2'\).

\item Let \(\wp(z)=\wp(z;\Lambda')\) 
be the Weierstrass \(\wp\)-function for the lattice \(\Lambda'\subset \CC\).

\item Let $\zeta(z)=\zeta(z;\Lambda')
= -\int^z \wp(u)du = z^{-1} + \cdots$ be the
Weierstrass \(\zeta\)-function 
and 
let $\sigma (z) = \sigma(z;\Lambda')= \exp \int^z \zeta(u)du
= z + \cdots$ be the 
Weierstrass \(\sigma\)-functions for 
\(\Lambda'\subset \CC\).

\item Let \(g\) be the logarithmic
derivative of the normalized developing map \(f\) of
a type I solution \(u\) of \textup{(\ref{3-1})}.
Let \(\bar{g}\) be the function on \(E'\) defined by \(g\).

\end{itemize}
The standard references for elliptic functions are
\cite[Ch.\,20]{Whittaker}, \cite[Ch.\,7]{Ahlfors} and
\cite[Ch.\,18\,\S1]{Lang-Ell};
we have followed the notation in 
\cite[Ch.\,7]{Ahlfors}:\footnote{The notation in \cite{Lang-Ell} is the same as in
\cite{Ahlfors} except that \(\textup{Im}(\omega_1/\omega_2)>0\). 
The notation in \cite{Whittaker}
is: \(2\omega_1', 2\omega_2'\) form a \(\ZZ\)-basis of 
\(\Lambda'\) with \(\textup{Im}(\omega_2'/\omega_1')>0\), and 
\(\zeta(z+2\omega_i';\Lambda')= \zeta(z)+2\eta_i(\Lambda')\)
for \(i=1,2\).} 
\begin{itemize}
\item \(\omega_1', \omega_2'\) form a \(\ZZ\)-basis of the lattice
\(\Lambda'\) with \(\textup{Im}(\omega_2'/\omega_1')>0\).
Note that the latter condition means that 
\((\omega_1', \omega_2)\) is an oriented basis for the standard orientation of the complex plane.

\item \(\eta_i=\eta(\omega_i';\Lambda')\) for \(i=1, 2\), where
\(\omega\mapsto \eta(\omega;\Lambda')\) is the 
\(\ZZ\)-linear function from \(\Lambda'\) to \(\CC\) such
that 
\[\zeta(z+\omega;\Lambda')= \zeta(z)+\eta(\omega;\Lambda')
\qquad\forall\,z\in\CC,\ \forall\,\omega\in\Lambda'.\]

\item The classical Legendre relation 
\[
\eta_1\cdot \omega_2'- \eta_2\cdot \omega_1' =2\pi\sqrt{-1}
\]
means that 
\[\eta(\alpha)\beta - \eta(\beta)\alpha = 2\pi \sqrt{-1}\,\psi(\alpha,\beta)
\quad \forall\, \alpha, \beta\in \Lambda',
\]
where \(\psi:\Lambda'\times\Lambda'\to \ZZ\) is the alternating
pairing on \(\Lambda'\) which sends an oriented \(\ZZ\)-basis
\((\omega_1',\omega_2')\) of \(\Lambda'\) to \(1\).
\end{itemize}

\subsubsection{}
Recall that the type I condition implies that
\begin{equation}{\label{g:I}}
\begin{split}
g(z + \omega_1) &= g(z)\quad \forall z,\\
g(z + \omega_2) &= -g(z)\quad \forall z.
\end{split}
\end{equation}
According to Lemma \ref{lemma_zero-pole_g}, 
the meromorphic function \(\bar{g}\) on \(E'\)
has zeros of order \(l\) at the two points of \(\Lambda/\Lambda'\),
no zeros and at most simple poles elsewhere on \(E'\).

From \ref{lemma_zero-pole_g}\,(1), the principal divisor
\((\bar g)\) of the meromorphic function \(\bar g\) on \(E'\)
has the form
\[
(\bar g)= \ell\cdot 0_{E'} + \ell \cdot [\omega_2]_{E'} 
- \sum_{P\in \bar{g}^{-1}(\infty)} P,
\]
where \([\omega_2]_{E'}= \omega_2\,\textup{mod}\,\Lambda'\)
is the image of \(\omega_2\) in \(E'\),
and \(\sum_{P\in \bar{g}^{-1}(\infty)}P\) is 
the polar divisor \((\bar g)_{\infty}\) of 
\(\bar{g}\), an effective divisor of degree \(2 l\) which is 
a sum of \(2\ell\) distinct points of \(E'\).
Clearly the sum of the polar divisor under the group law of \(E'\)
is equal to \(\ell\) times the \(2\)-torsion point
\([\omega_2]_{E'}\).
We know from the condition \(g(z+\omega_2)=-g(z)\) that
the polar divisor \((\bar g)\) of \(\bar{g}\) is stable under the
translation by the \(2\)-torsion point \([\omega_2]_{E'}\).
Let \(P_1,\ldots, P_{l}\) be a set of representatives of the quotient of
\(\bar{g}^{-1}(\infty)\) under the translation action by 
\([\omega_2]_{E'}\).
The sum \(\mu_{E'}(P_1,\ldots, P_l)=P_1+_{E'}\cdots+_{E'} P_l\) 
of this set of representatives under the group law of \(E'\)
is a \(2\)-torsion point because the sum of the polar divisor 
\((\bar g)_{\infty}\) is \([2]_{E'}([\omega_2]_{E'})\).
Moreover it is clear that the image of
\(\mu_{E'}(P_1,\ldots,P_l)\) in the quotient group
\(E'[2]/\{0_{E'},[\omega_2]_{E'}\}\) 
is independent of the choice of representatives \(P_1,\ldots, P_l\).
The following lemma says that this image is equal to the
\emph{non-trivial} element of \(E'[2]/\{0_{E'},[\omega_2]_{E'}\}\).


\begin{lemmass}\label{lemmass:nontriv}
Notation as above.
The sum \(P_1+_{E'}\cdots+_{E'} P_l\) in \(E'\) of any set
of representatives of the quotient 
\(\,\bar{g}^{-1}(\infty)/\{0_{E'},[\omega_2]_{E'}\}\,\) is congruent to
the non-zero \(2\)-torsion point 
\(\frac{\omega_1}{2}\,\textup{mod}\,\Lambda'\)
modulo the subgroup 
\(\{0_{E'},[\omega_2]_{E'}\}\) of the group \(E'[2]\)
of all \(2\)-torsion points of \(E'\).
\end{lemmass}

\subsubsection{}\label{subsubsec:properties_sigma}
Lemma \ref{lemmass:nontriv} is a consequence of a more precise
statment Lemma \ref{half-cong}; the latter uses
Weierstrass \(\sigma\)-function. 
The \(\sigma\)-function is essentially the 
odd theta function \(\theta_{11}(z)\) with half-integer characteristics, 
up to rescaling of the \(z\)-variable, a harmless factor
\(-\pi\cdot e^{-\eta_2 z^2/2}\) and the product 
\[\theta_{00}(0)\cdot \theta_{01}(0)\cdot \theta_{10}(0)
\]
of three even theta constants.
We recall some of the basic properties of the \(\sigma\)-function
below.

\begin{itemize}
\item[(i)] The function
\[\sigma(z)=\sigma(z;\Lambda')
=z\cdot \prod_{\omega\in \Lambda'\smallsetminus\{0\}}
\left[\left(1-\frac{z}{\omega}\right)\cdot \exp\left(
\frac{z}{\omega}+\frac{z^2}{2\omega^2}\right)
\right]
\]
is an entire odd function on
\(\CC\), with simples zeros on points of \(\Lambda'\)
and non-zero elsewhere.
In addition \(\zeta(z)\) satisfies the following transformation law
for translation by elements of \(\Lambda'\).
\begin{equation} \label{sigma-law}
\sigma(z+\alpha)= \epsilon(\alpha)\cdot e^{\eta(\alpha)(z+\frac{\alpha}{2})}
\cdot \sigma(z)
\qquad \forall\,z\in\CC,\ \forall\,\alpha\in \Lambda',
\end{equation}
where \(\epsilon: \Lambda'\to \{\pm1\}\) is the quadratic character 
on \(\Lambda'\) given by 
\[\epsilon(\alpha)=\left\{
\begin{array}{rl}
1&\textup{if}\ \ \alpha\in 2\Lambda'\\
-1& \textup{if}\ \ \alpha\not\in 2\Lambda'
\end{array}
\right.\]


\item[(ii)] 
Suppose that \(m\) is a positive integer 
and \(a_1,\ldots, a_m; b_1,\ldots, b_m\) are elements of \(\CC\). 
The meromorphic function
\[
h(z; a_1,\ldots,a_m;b_1,\ldots,b_m) :=
\frac{\prod_{i = 1}^m\sigma(z - a_i)}{\prod_{i = 1}^m \sigma(z - b_i)}
\]
on \(\CC\) is \(\Lambda'\)-periodic if and only if 
\(\sum_{i=1}^m a_i = \sum_{i=1}^m b_i\).
Moreover if \(\sum_{i=1}^m a_i = \sum_{i=1}^m b_i\), then
the principle divisor of the meromorphic function on
\(E'=\CC/\Lambda'\) defined by 
\({\prod_{i = 1}^m\sigma(z - a_i)}/{\prod_{i = 1}^m \sigma(z - b_i)}\)
is \[\sum_{i=1}^m [a_i]_{E'} - \sum_{i=1}^m [b_i]_{E'},\]
where \([a_i]_{E'}\) (respectively \([b_i]_{E'}\))
is the image of \(a_i\) (respectively \(b_i\)) in 
\(E'\) for \(i=1,\ldots, m\).

\item[(iii)] 
Suppose that \(\,(a_1,\ldots,a_m;b_1,\ldots,b_m)\,\) and
\(\,(a_1',\ldots, a_m'; b_1',\ldots, b_m')\,\) are two 
\(2m\)-tuples of complex numbers such that 
\(\sum_{i=1}^m a_i =\sum_{i=1}^m b_k\),
\(\sum_{i=1}^m a_i' = \sum_{i=1}^m b_i'\),
\(a_i'\equiv a_i\ (\textup{mod}\,\Lambda')\) and
\(b_i'\equiv b_i\ (\textup{mod}\,\Lambda')\) for 
\(i=1,\ldots, m\).
Then 
\[
\frac{\prod_{i = 1}^m\sigma(z - a_i)}{\prod_{i = 1}^m \sigma(z - b_i)}
=\frac{\prod_{i = 1}^m\sigma(z - a_i')}{\prod_{i = 1}^m \sigma(z - b_i')}
\]

\item[(iv)] Let \(h\) be a non-constant meromorphic function on 
\(E'\).  Let 
\[a_1,\ldots, a_m, b_1,\ldots,b_m\] be elements of
\(\CC\) such that 
the principle divisor of \(h\) is equal to 
\(\sum_{i=1}^m [a_i]_{E'} - \sum_{i=1}^m [b_i]_{E'}\).
Then there exists a constant \(A\in \CC^{\times}\)
such that 
\[
h([z]_{E'})= A\cdot
\frac{\prod_{i = 1}^m\sigma(z - a_i)}{\prod_{i = 1}^m \sigma(z - b_i)}
\qquad \forall\,z\in\CC.
\]
\end{itemize}
Note that (ii) and (iii) are consequences of the transformation law
(\ref{sigma-law}) in (i), and (iv) follows from (ii).





\begin{lemmass} \cite{LW2} \label{half-cong}
Let \(l\) be a positive integer.
Let \(p_1, \ldots, p_l; q_1,\ldots, q_l\) be
elements in \(\,\CC\smallsetminus \Lambda\) satisfying
\begin{equation}  \label{sum-constraint}
\sum p_i + \sum q_i = l \omega_2
\end{equation}
and \(p_i+ \omega_2\equiv q_i\pmod{\Lambda'}\) for
\(i=1,\ldots,l\).
Let 
\begin{equation} \label{unshift}
h(z) 
= \frac{\sigma^l(z) \sigma^l(z - \omega_2)}{\prod_{i =
1}^{l} \sigma(z - p_i) \prod_{i = 1}^l \sigma(z - q_i)}
\end{equation}
be the meromorphic function on \(\CC\) attached to
the \(4l\)-tuple 
\[(0,\ldots, 0, \omega_2,\ldots,\omega_2; p_1,\ldots,p_l,
q_1,\ldots,q_l)\]
as in \textup{\ref{subsubsec:properties_sigma}\,(ii)}.
Note that \(h(z)\) descends to
a meromorphic function on \(E'\) whose principle 
divisor is 
\[l\cdot 0_{E'}+l\cdot [\omega_2]_E'-\sum_{i=1}^l [p_i]_{E'}
-\sum_{i=1}^l [q_i]_{E'}
\]
according to \textup{\ref{subsubsec:properties_sigma}\,(ii)}.
\begin{itemize}
\item[(a)] The function \(h(z)\) satisfies
\(h(z+\omega_2)= - h(z)\) for all \(z\in \CC\)
if and only if 
\[\sum_{i=1}^l p_i \equiv
\tfrac{1}{2}\omega_1 \pmod{\Lambda}.\]
Note that the above displayed formula means that
\(\sum_{i=1}^l p_i\) is congruent to either
\(\frac{1}{2}\omega_1\) or \(\frac{1}{2}\omega_1 + \omega_2\) 
modulo \(\Lambda'\).

\item[(b)] Suppose that
\(\,p_1+\cdots+p_l\equiv \frac{1}{2}\omega_1\pmod{\Lambda'}\).
There exist elements
\(p_1',\ldots,p_l';\,q_1',\ldots,q_l'\) in \(\CC\)
satisfying the following conditions.
\begin{itemize}

\item[(b1)] \(p_i'\equiv p_i \pmod{\Lambda'}\) and 
\(q_i'\equiv q_i\pmod{\Lambda'}\) for \(i=1,\ldots, l\),

\item[(b2)] \(\sum_{i=1}^l p_i' =\frac{1}{2}\omega_1\),

\item[(b3)] \(q_i' = p_i' + \omega_2\) for \(i=1,\ldots, l-1\)
and \(q_l' = p_l' + \omega_2 - \omega_1\),

\item[(b4)] \(\displaystyle{h(z) 
= \frac{\sigma^l(z) \sigma^l(z - \omega_2)}{\prod_{i =
1}^{l} \sigma(z - p_i') \prod_{i = 1}^l \sigma(z - q_i')}
}\).
\end{itemize}

\item[(c)] Suppose that 
\(\,p_1+\cdots+p_l\equiv 
\frac{1}{2}\omega_1 + \omega_2\pmod{\Lambda'}\).
There exist elements
\(p_1',\ldots,p_l';\,q_1',\ldots,q_l'\) in \(\CC\)
satisfying the following conditions.
\begin{itemize}

\item[(c1)] \(p_i'\equiv p_i \pmod{\Lambda'}\) and 
\(q_i'\equiv q_i\pmod{\Lambda'}\) for \(i=1,\ldots, l\),

\item[(c2)] \(\sum_{i=1}^l p_i' 
=\frac{1}{2}\omega_1+\omega_2\)

\item[(c3)] \(q_i' = p_i' + \omega_2\) for \(i=1,\ldots, l-1\),
and \(q_l' = p_l' - \omega_2 - \omega_1\),

\item[(c4)] \(\displaystyle{h(z) 
= \frac{\sigma^l(z) \sigma^l(z - \omega_2)}{\prod_{i =
1}^{l} \sigma(z - p_i') \prod_{i = 1}^l \sigma(z - q_i')}
}\).

\end{itemize}
\end{itemize}
\end{lemmass}

\begin{proof}
Clearly (b1) and (b3) implies that 
\[\sum_{i=1}^m p_i'+ \sum_{i=1}^m q_i' = l\cdot \omega_2
= \sum_{i=1}^m p_i'+ \sum_{i=1}^m q_i',\]
therefore (b4) follows from (b1)--(b3).
by \ref{subsubsec:properties_sigma}\,(iii).
Similarly (c1)--(c3) implies (c4).

Because \(p_i+\omega_2\equiv q_i \pmod{\Lambda'}\) for each \(i\),
the condition (\ref{sum-constraint}) implies that
\(\sum_{i=1}^l p_i \equiv m\cdot \frac{\omega_1}{2} + n \cdot
\omega_2\) for integers \(m,n\in \{0,1\}\).
By \ref{subsubsec:properties_sigma}\,(iii).
Let 
\begin{itemize}
\item \(p_1'=p_1,\ldots, p_{l-1}'=p_{l-1}, p_{l}'=p_{l}\),
\item \(q_1'=p_1+\omega_2, \ldots, q_{l-1}'=p_{l-1}+\omega_2\) and
\item \(q_l=p_l-m\omega_1+(1-2n)\omega_2\).
\end{itemize}
By \ref{subsubsec:properties_sigma}\,(iii) the equality
\(\displaystyle{h(z) 
= \frac{\sigma^l(z) \sigma^l(z - \omega_2)}{\prod_{i =
1}^{l} \sigma(z - p_i') \prod_{i = 1}^l \sigma(z - q_i')}
}\) holds. All that remains is use the transformation law
(\ref{sigma-law}) to see whether 
\(h(z+\omega_2)=-h(z)\).
\medbreak

There are only four possibilities for the pair $(m, n)$, 
namely
\begin{equation*}
\mbox{$(m, n) =$\quad (i) (0, 0),\quad (ii) (1, 0),\quad 
(iii) (0,1), \quad (iv) (1, 1)}.
\end{equation*}
One verifies by direction calculations 
that \(h(z+\omega_2)=h(z)\) for all \(z\) if \(m=0\),
while \(h(z+\omega_2)=-h(z)\) for all \(z\) if \(m=1\).
For instance when \((m,n)=(0,0)\), \(h(z+\omega_2)\)
and \(h(z)\) differ by 
the factor of automorphy
\begin{equation*}
\frac{(-1)^l \exp ({l \eta_2 z})}{(-1)^l \exp \left[{\eta_2 \sum_{i = 1}^l
(z - p_i)}\right]} = 1,
\end{equation*}
Meaning that \(h(z+\omega_2)=h(z)\).
When \((m,n)=(1,0)\),
\(h(z+\omega_2)\) and \(h(z)\) differ by the factor
\begin{equation*}
\frac{(-1)^{l - (l + 1)} \exp (l \eta_2 z)}{\exp \left[\eta_2 \sum_{i =
1}^{l - 1} (z - p_i) + (\eta_2 - \eta_1)(z - p_l -
\frac{1}{2}\omega_1) + \eta_1(z - p_l - \frac{1}{2}\omega_1)\right]}
\end{equation*}
which is $-1$ since $\sum p_i = \frac{1}{2}\omega_1$. 
The other two cases are checked similarly. 
We have proved lemmas \ref{half-cong} and \ref{lemmass:nontriv}.
\end{proof}

%

\subsubsection{}
In order to construct type I solutions from the elliptic function
$g$, we need to find all the other constraints imposed on its poles.
\medbreak

Let \(p_1,\ldots,p_l\) be points of \(\CC\) such that
\(\bigcup_{i=1}^{l} p_i+\Lambda'\) are the simple zeroes of the
developing map \(f\) and
\(\bigcup_{i=1}^{l} p_i+\omega_2+\Lambda'\) 
are the simple poles of \(f\).
Let \(P_i:= p_i\ \textup{mod}\,\Lambda'\) and let 
\(Q_i:=p_i+\omega_2\ \textup{mod}\,\Lambda'\) for \(i=1,\ldots, l\).
We know that 
\[P_1,\ldots, P_l, Q_1,\ldots, Q_l\] are \(2l\) distinct 
points of \((\CC\smallsetminus \Lambda)/\Lambda'
=E'\smallsetminus\{0_{E'},[\omega_2]_{E'}\}\); equivalently
\[
p_i - p_j \not\in \Lambda \quad \forall\, i\neq j,\ 1\leq i,j\leq l.
\]
We also know that 
\[\sum_{i=1}^{\ell} p_i \equiv \frac{\omega_1}{2}
\pmod{\Lambda}\] according to Lemma \ref{lemmass:nontriv}.
By \ref{subsubsec:properties_sigma}\,(iv) we know that there
exists a constant \(A\in \CC^{\times}\) such that
\begin{equation} \label{unshift2}
g(z) 
= A\cdot 
\frac{\sigma^l(z)\cdot \sigma^l(z - \omega_2)}{\prod_{i =
1}^{l} \sigma(z - p_i) \cdot \prod_{i = 1}^l \sigma(z - q_i)},
\end{equation}
where \(q_1,\ldots,q_l\) are elements of \(\CC\) such that
\begin{equation}\label{conditions-qi}
q_i\equiv p_i+\omega_2\pmod{\Lambda'}\ \ \forall\,i,
\quad\textup{and}\quad
\sum_{i=1}^l p_i + \sum_{i=1}^l q_i = l\omega_2.
\end{equation}
Notice that the residue of $g(z)$ at $z = p_j$ is
given by $Ar_j$ for \(j=1,\ldots, l\), where
\begin{equation}\label{expr-rj}
r_j = \frac{\sigma^l(p_j)\cdot \sigma^l(p_j - \omega_2)}{\prod_{i =
1,\ne j}^{l} \sigma(p_j - p_i)\cdot \prod_{i = 1}^l \sigma(p_j - q_i)}
\qquad \textup{for}\ j=1,\ldots,l.
\end{equation}
It is immediate from \ref{subsubsec:properties_sigma}\,(ii) that
the formula (\ref{expr-rj}) for \(r_j\) is independent of the
choice of \(q_1,\ldots, q_l\) satisfying 
(\ref{conditions-qi}), with \(p_1,\ldots, p_l\) fixed,
and also independent of the choice of 
\(p_1, \ldots, p_{j-1}, p_{j+1},\ldots, p_{l}\) in their respective
congruence classes modulo \(\Lambda'\) when 
the \(q_i\)'s and \(p_1 + \cdots+ p_{j-1}+ p_{j+1}+\cdots +p_{l}\)
are fixed. 
One checke by a routine calculation that the right hand side of the
formula (\ref{expr-rj}) remains the same when 
\(p_j\) is replaced by \(p_j+\alpha\) and \(q_j\) is replaced by
\(\alpha\) for any element \(\alpha\in \Lambda'\).
So the right hand side of the formula (\ref{conditions-qi}) 
is a meromorphic function of 
\((P_1,\ldots, P_l)\in E'\times\cdots\times E'\).

\begin{lemmass}
Let \(p_1,\ldots, p_l\) be elements of \(\CC\) such that
\(\bigcup_{i=1}^{l} p_i+\Lambda'\) are
the zeroes of the developing map \(f\) 
and \(\bigcup_{i=1}^{l} p_i+\omega_2+\Lambda'\) are the poles of \(f\).
Let \(q_1,\ldots, q_l\) be elements of \(\CC\) satisfying the conditions
in \textup{(\ref{conditions-qi})}.
Then 
\begin{equation} \label{residue-I}
r_1 = r_2 = \cdots =  r_l.
\end{equation}
where the non-zero complex numbers \(r_1,\ldots,r_l\) 
are defined by \textup{(\ref{expr-rj})} and 
the elements \(q_1,\ldots,q_l\in\CC\) appearing in the formula
\textup{(\ref{expr-rj})} satisfies the conditions in
\textup{(\ref{conditions-qi})}.
\end{lemmass}

\begin{proof}
Write $g = f'/f$ as in (\ref{unshift2}) for a suitable
constant \(A\in \CC^{\times}\).
Then $Ar_j = 1$ for $j = 1, \ldots, l$, hence
\(r_1=r_2=\cdots=r_l\).
\end{proof}

\begin{propositionss} \label{g-odd-u-even}
Let \(p_1,\ldots, p_l\) be elements of \(\CC\) with
the following properties.
\begin{itemize}
\item[(i)] \(\sum_{i=1}^l p_i\equiv \omega_1/2 \pmod{\Lambda}\),
\item[(ii)] \(p_i-p_j\not\equiv 0 \pmod{\Lambda}\) whenever \(i\neq j\),
and
\item[(iii)] the residue equalities \textup{(\ref{residue-I})} hold, where
\(r_1,\ldots, r_l\) are defined by \textup{(\ref{expr-rj})} and the
elements
\(q_1,\ldots, q_l\in \CC\) satisfy the conditions in
\textup{(\ref{conditions-qi})}.
\end{itemize}
Let $h$ be an elliptic function on $E'$ defined by
\textup{(\ref{unshift})}.
Let \(A:=r_1^{-1}\) and let \(g_1:=A\cdot h\).
\begin{itemize}
\item[(a)] 
\(s_1=\cdots =s_l = -r_1 =\cdots = - r_l\), where 
\begin{equation}
s_j=\frac{\sigma^l(q_j)\cdot\sigma^l(q_j-\omega_2)}
{\prod_{i=1}^{l}\sigma(q_j-p_i)\cdot \prod_{i=1,\,\neq
    j}^l\sigma(q_j-q_i) } \quad\textup{for}\ \ j=1,\ldots,l.
\end{equation}
Consequently the residue of the simple pole \(P_i\)
(respectively \(Q_i\)) of the meromorphic function 
\(A\cdot h\) on \(E'\) is equal to \(1\)
(respectively \(-1\)).
Here \(P_i:=p_i\ \textup{mod}\,\Lambda'\in E'\) 
and \(Q_i:=q_i\ \textup{mod}\,\Lambda'\in E'\) for \(i=1,\ldots, l\).

\item[(b)]
If \(h\) is an odd function, then the following statements hold.
\begin{itemize}

\item[(b1)] The subset \(\,\{P_1,\ldots, P_l\}\subset E'\smallsetminus
\{0_{E'}, [\omega_2]_E'\}\,\) is stable under the involution of
\(E'\) induced by ``multiplication by \(-1\)''.

\item[(b2)] Exactly one of \(P_1, \ldots, P_l\) is a two-torsion
point of \(E'\); this point is either 
\(\frac{\omega_1}{2}\ \textup{mod}\,\Lambda'\) or
\(\frac{\omega_1}{2}+\omega_2\ \textup{mod}\,\Lambda'\).

\item[(b3)] $l$ is an odd integer. 

\end{itemize}

\item[(c)] Conversely suppose that the condition \textup{(b1)}
is satisfied, or equivalently conditions all \textup{(b1)--(b3)} hold.
Then \(h\) is an odd function; i.e.\ \(h(-z)=-h(z)\,\) for all 
\(z\in\CC\).

\item[(d)] Assume \(h\) is an odd function, or equivalently
that conditions \textup{(b1)--(b3)} hold.
\begin{itemize}
\item[(d1)] There exists a normalized type I developing map \(f_1\) of a solution
of \textup{(\ref{3-1})} with parameter \(\rho = 4\pi l\) such that
\({f_1'}/{f_1}= g_1\).
\item[(d2)] \(f_1\) and \(-f_1\) are the
only normalized  type I developing map whose logarithmic
derivative is \(g_1\).  
\item[(d3)] \(f_1\) is an even function, i.e.\ \(f_1(-z)=f_1(z)\)
for all \(z\in \CC\).
\end{itemize}
\end{itemize}
\end{propositionss}

\begin{proof}
We know that \(h(z+\omega_2)= - h(z)\) for all \(z\in \CC\) by Lemma
\ref{half-cong}\,(a). So 
the statement (a) follows from the assumption that
\(r_1=\cdots=r_l\).
\smallbreak

The set \(\{P_1,\ldots, P_l\}\) 
is the set of all (simple) poles with residue \(r_1\) of the
meromorphic differential \(h dz\) on \(E'\).  The 
assumption that \(h\) is odd means that \(h dz\) is invariant
under ``multiplication by \(-1\)'', so the statement (b1) follows.
The statement (b2) follows because of assumption (i).
The statement (b3) follows from (a) and (b2).
\smallbreak

Suppose that (b1)--(b3) hold.  
Let \(n=(2l-1)/2\). After renumbering the \(p_i\)'s 
we may assume that \(P_{n+1}=-P_1, P_{n+2}=-P_2,\ldots,
P_{2n}=-P_n\) and \(P_{l}=[\omega_2]_{E'}\).
According to \ref{subsubsec:properties_sigma}\,(iii), we have
\begingroup\makeatletter\def\f@size{8.5}\check@mathfonts
\begin{equation}\label{hz-odd}
h(z)=\frac{\sigma^{2n+1}(z)\cdot \sigma^{n+1}(z-\omega_2)\cdot
\sigma^n(z+\omega_2)}{\left[\prod_{i=1}^n\sigma(z\!-\!p_i)\cdot
  \sigma(z\!+\!p_i)\right]
\cdot \left[\prod_{i=1}^n\sigma(z\!-\!p_i\!-\!\omega_2)\cdot
\sigma(z\!+\!p_i\!+\!\omega_2)\right]
\cdot
\sigma(z\!-\!\frac{\omega_1}{2})\cdot\sigma(z\!+\!\frac{\omega_1}{2}
\!-\!\omega_2)
}
\end{equation}\endgroup
Using the fact that \(\sigma(z)\) is an odd function, we get
\begin{equation*}
\begin{split}
\frac{h(-z)}{h(z)}&=
\frac{\sigma(z+\omega_2)\cdot \sigma(z-\frac{\omega_1}{2})
\cdot \sigma(z+\frac{\omega_1}{2}-\omega_2)}{\sigma(z-\omega_2)
\cdot \sigma(z+\frac{\omega_1}{2}) \cdot 
\sigma(z-\frac{\omega_1}{2}+\omega_2)}\\
&=(-1)\cdot e^{\eta_2\cdot z}\cdot e^{\eta_1\cdot z}\cdot
e^{(\eta_1-\eta_2)\cdot z} = -1
\end{split}
\end{equation*}
by the transformation law for the \(\sigma\)-function.
We have proved (c).

Assume again that (b1)--(b3) hold, so that 
\(l=2n+1\) is odd and \(h(z)\) is an odd function.
Then 
\(g_1(z)\) is an odd meromorphic function on
\(E'\) which has simple poles with residue \(1\) at
\(P_1,\ldots, P_{2n+1}\) and has simple poles with residue \(-1\)
at \(Q_1,\ldots, Q_{2n+1}\).
From the proof of 
\ref{g-odd-u-even}\,(d) may and so assume that 
\(p_1=p_{n+i}\) for \(i=1,\ldots, n\),
and \(p_n=\frac{\omega_1}{2}\),
\(q_n =\frac{\omega_1}{2}+\omega_2\), so that
\(h(z)\) is given by equation (\ref{hz-odd}).
For each of the \(4n+2\) poles of \(g_1(z)\, dz\), the
integral along a sufficiently small circle around the
pole is \(\pm 2\pi\).  Hence the line integral
\(\,\displaystyle{\int_0^z\,g_1(w)\,dw}\,\) is
well-defined as an element of \(\CC/2\pi\sqrt{-1}\ZZ\)
and the function
\[
f_2(z)=\exp {\int_0^z g_1(w)\, dw}
\]
is a well-defined meromorphic function on \(\CC\)
with simple poles at points in the union
\(\,\bigcup_{i=1}^{2n+1}q_i+\Lambda'\,\) of \(\Lambda'\)-cosets,
simple zeros at points in \(\,\bigcup_{i=1}^{2n+1}p_i+\Lambda'\), 
neither zero nor pole elsewhere on \(\CC\).
In particular \(f_2(z)\) is holomorphic and non-zero
at points of \(\Lambda\). Notice that \(f_2(z)=f_2(-z)\) for all 
\(z\in \CC\) because \(g_1\) is odd.

The fact that \(\,g_1(w) dw\) is invariant under translation by
\(\omega_1\) 
implies that 
\[
f_2(z+\omega_1)= \int_{0}^{\omega_1}g_1(w)\,dw\cdot f_2(z)
\qquad\forall\,z\in\CC\,.
\]
Similarly the fact that \(\,g_1(w+\omega_2) =-g_1(w)\) implies
that 
\[
f_2(z+\omega_2)\cdot f_2(z)= \int_0^{\omega_2} g_1(w)\, dw.
\]
To prove (d) it suffice to show that
\begin{equation}\label{period=piI}
\int_0^{\omega_1}g_1(w)\,dw \equiv \pi\sqrt{-1}\pmod{2\pi\sqrt{-1}\ZZ},
\end{equation}
for then \(f_1(z)= \sqrt{f_2(\omega_2)^{-1}}\cdot f_2(z)\) 
will be a normalized developing map of type I (for a solution
of equation (\ref{3-1}) with \(\rho=2n+1\)),
for either of the two square roots of 
\(f_2(\omega_2)^{-1}\). Clearly these are the only two
normalized developing maps of type I whose logarithmic derivatives
are equal to \(g_1\).

To compute the integral \(\,\int_0^{\omega_1}g_1(w)\,dw\) 
modulo \(2\pi\sqrt{-1}\ZZ\), let 
\(C_{\epsilon}\) be the path from \(0\) to \(\omega_2\), obtained
from the oriented line segment
\(\overrightarrow{0\omega_2}\)  from \(0\) to \(\omega_2\) near by replacing
the \(\epsilon\)-neighborhood of each pole of \(g_1(w)dw\) 
by the half circle of radius \(\epsilon\) to the right of 
\(\overrightarrow{0\omega_2}\), for all sufficiently small 
\(\epsilon>0\). 
Clearly the integral \(\,\int_{C_\epsilon}g_1(w)\,dw\,\)
is independent of \(\epsilon\).
Write \(C_{\epsilon}\) the union
of the small half circles and the ``straight part''
\(C_{\epsilon}'\) of \(C_{\epsilon}\).
Let \(m_1\) (respectively \(m_2)\) be the number of poles 
of \(g_1\) with residue
\(1\) (respectively \(-1\)) on 
the line segment \(\overrightarrow{0\omega_2}\).

The fact that \(g_1(w)\,dw\) is invariant under
multiplication by \(-1\) implies that
the integral of \(g_1(w)\) over \(C_{\epsilon}'\) is \(0\),
so \(\,\int_0^{\omega_1}g_1(w)\,dw\,\) converges to 
\((m_1-m_2)\cdot \pi\sqrt{-1}\) as \(\epsilon\to 0^{+}\).
In other words 
\[\,\int_{C_\epsilon}g_1(w)\,dw =(m_1-m_2)\cdot \pi\sqrt{-1}\,\]
for all (sufficiently small) \(\epsilon>0\).
On the other hand the assumptions (b1) and (b2) tells us that
\(m_1-m_2\) is an odd integer.  We have proved the statements 
(d1)--(d3).
\end{proof}

\subsection{Type II scaling families and blow-up points}

In type II, it follows from (\ref{case-type-II}) that $g = f'/f$
is an elliptic function on $E$. From \S \ref{0-constraint}, 
$g$ has zero only at $z = 0$. Thus by Lemma \ref{order},
\begin{equation} \label{II-g}
g(z) = A\,\frac{\sigma^l(z)}{\prod_{i = 1}^l \sigma(z - p_i)}
\end{equation}
for $p_i$'s being simple zeros/poles of $f$ with $\sum p_i = 0$.
Now the Weierstrass function $\sigma$ is with respect to $E$. Also
the points $p_i$'s are unique up to elements in $\Lambda$ as long
as the constraint $\sum p_i = 0$ is satisfied.

\begin{propositionss} \label{non-exist}
For $\rho = 4\pi l$ with $l$ being odd, there are no type II,
i.e.~blow-up, solutions to the mean field equation
\begin{equation*}
\triangle u + e^u = \rho\delta_0 \quad \mbox{on $E$}.
\end{equation*}
\end{propositionss}

\begin{proof}
If there is a solution $u$ with developing map $f$, then $g =
f'/f$ is elliptic on $E$ with residues at $p_i$, $i = 1, \ldots,
l$, being $\pm 1$. Since $l$ is odd, the sum of residues of $g$ is
non-zero, which contradicts to the classical fact that the sum of
residues of an elliptic function must be zero.
\end{proof}

Therefore we may set $l = 2n$. Let $p_1, \ldots, p_n$ be zeros and
$p_{n + 1}, \ldots, p_{2n}$ be poles of $f$. The residue of $g$ at $z
= p_j$ is given by $Ar_j$ with
\begin{equation} \label{residue-eq1}
r_j = \frac{\sigma^l(p_j)}{\prod_{i = 1, \ne j}^l \sigma(p_j - p_i)}.
\end{equation}
Then we have equations
\begin{equation} \label{residue-eq2}
r_1 = \cdots = r_n = -r_{n + 1} = \cdots = -r_{2n}.
\end{equation}

Recall that
\begin{equation*}
f(z) = f(0) \exp \int_0^z g(w)\,dw.
\end{equation*}
\begin{lemmass}
In order for $f$ to verify (\ref{case-type-II}), it is equivalent
to require that the periods integrals are purely imaginary:
\begin{equation*}
\int_{L_i} g(z)\,dz \in i\,\mathbb{R}, \quad i = 1, 2.
\end{equation*}
\end{lemmass}

Another characteristic feature for type II is that any solution
must exist in an one parameter scaling family of solutions. To see
this, notice that if $f$ is a developing map of solution $u$ then
$e^{\lambda}f$ also satisfies (\ref{case-type-II}) for any
$\lambda \in \mathbb{R}$. In fact $e^\lambda f$ is a developing
map of $u_\lambda$ defined by (\ref{u-lambda}):
\begin{equation*}\label{3-11}
u_\lambda(z) = \log
\frac{8 e^{2\lambda}|f'(z)|^2}{(1+e^{2\lambda}|f(z)|^2)^2}
\end{equation*}
and it is clear that $u_\lambda$ is a scaling family of solutions
of (\ref{3-1}).

Let $z_0$ be a zero of $f$. We know that $z_0 \not\equiv 0$ and
$f'(z_0) \ne 0$. Thus
\begin{equation*}
u_\lambda(z_0) \sim 2\lambda \to +\infty \quad \mbox{as} \quad
\lambda \to +\infty
\end{equation*}
while if $f(z) \ne 0$ then
\begin{equation*}
u_\lambda(z) \sim -2\lambda \to -\infty \quad \mbox{as} \quad
\lambda \to +\infty.
\end{equation*}
Points like $z_0$ are referred as \emph{blow-up} points. 

Thus as $\lambda \to +\infty$, the blow-up set of $u_\lambda$ 
consists of the zeros of $f$. Similarly, as $\lambda \to -\infty$, 
the blow-up sets of $u_\lambda$ consists of the poles of $f$.\smallskip

\begin{remarkss}
In general it is very hard to solve the residue equations 
(\ref{residue-I}) (for type I) and (\ref{residue-eq2}) (for type II) directly, 
though some simplest cases had been treated in \cite{LW, LW2} 
for $\rho = 4\pi$, $8\pi$ and $12\pi$. 
\end{remarkss}

\section{Type I solutions: Evenness and algebraic integrability} \label{I-int}
\setcounter{equation}{0}

Let $\rho = 4\pi l$, $l \in \mathbb N$. 
Let $u$ be a type I solution and $f$ be a developing map of $u$. 
In this section we will prove Theorem \ref{thm-type I} 
stated in the introduction. 
Proposition \ref{non-exist} proves that 
if $l$ is odd then the solution is of type I. 
We will start by proving the converse in Theorem \ref{even-II}, 
i.e., if the solution is of type I then $l$ must be odd. 
At the same time the evenness of $u$ is deduced.

\subsection{The evenness of solutions}

Recall the logarithmic derivative
\begin{equation*}
g = (\log f)' = \frac{f'}{f}
\end{equation*}
which is elliptic on $E' = \mathbb{C}/\Lambda'$ with
$\Lambda' = \mathbb{Z}\omega_1 + \mathbb{Z}2\omega_2$. For the
ease of notations we will use $\omega_1' = \omega_1$ and
$\omega_2' = 2\omega_2$. In the following all the elliptic
functions are with respect to the torus $E'$.

Since $g$ has zero at $z = 0$ of order $l$, it also has zero of
order $l$ at $z = \omega_2$. There are no other zeros hence it has
simple poles at $p_1, \ldots, p_l$ and $q_1, \ldots, q_l$ where
$p_i$'s are simple zeros of $f$ and $q_i$'s are simple poles of
$f$ modulo $\Lambda'$. Thus we may assume that
\begin{equation*}
q_i = p_i + \omega_2, \quad i = 1, \ldots, l.
\end{equation*}

From
\begin{equation*}
f(z) = f(0) \exp \int_0^z g(w)\, dw,
\end{equation*}
the residues of $g$ are $1$ at $p_i$'s and $-1$ at $q_i$'s. Thus
we may write $g$ as
\begin{equation} \label{g-zeta}
g(z) = \sum_{i = 1}^l (\zeta(z - p_i) - \zeta(z - p_i - \omega_2))
+ c
\end{equation}
By (\ref{g:I}), it is easily seen that $c = l\eta_2/2$.

There are also other useful equivalent forms of $g$:
\begin{equation*}
\begin{split}
g(z) &= \frac{1}{2} \sum_{i = 1}^l (2\zeta(z - p_i) - \zeta(z - p_i -
\omega_2) - \zeta(z - p_i + \omega_2)) \\
 &= -\frac{1}{2} \sum_{i = 1}^l \frac{\wp'(z - p_i)}{\wp(z - p_i) - e_2}\\
 &= -\frac{1}{2} \sum_{i = 1}^l \frac{d}{dz} \log (\wp(z - p_i) - e_2)
\end{split}
\end{equation*}
by the addition formula.

\begin{remarkss}
The middle formula says that up to a constant $g(z)$ is the sum of
slopes of the \(l\) lines from the point 
$(\wp(\omega_2), \wp'(\omega_2)) = (e_2, 0)$ to the points
$(\wp(z - p_i), \wp'(z - p_i))$ of the torus $E'$ under the
standard cubic embedding into $\mathbb{C}^2 \cup \{\infty\}$,
for \(i=1,\ldots, l\).
\end{remarkss}

The only constraint remained is the zero order of $g$ at $z = 0$. Namely
\begin{equation*}
0 = g(0) = g'(0) = \cdots = g^{(l - 1)}(0).
\end{equation*}

The proof starts by noticing that
\begin{equation*}
2g(0) = \sum \frac{\wp'(p_i)}{\wp(p_i) - e_2} =: \sum s(p_i)
\end{equation*}
is the first (degree one) symmetric polynomial of the slops
$s(p_i)$. It is reasonable to expect that some of the higher
derivatives $g^{(m)}(0)$ are also higher degree symmetric
polynomials of slops. The expectation turns out to be true only
for $m$ even and for odd degree polynomials:

\begin{propositionss} \label{even-odd}
The even order differentiation $g^{(2j)}(0)$, $j = 0, \ldots,
[\frac{l - 1}{2}]$ from a basis of the odd degree symmetric
polynomials in $s_i$'s up to degree $l$ for $l$ being odd and up
to degree $l - 1$ for $l$ being even.
\end{propositionss}

\begin{proof}
Consider the slop function
\begin{equation} \label{slope}
\begin{split}
s(z) &= \frac{d}{dz} \log (\wp(z) - e_2) = \frac{\wp'(z)}{\wp(z) - e_2}\\
&= -2\zeta(z) + \zeta(z + \omega_2) + \zeta(z - \omega_2) \\
&= -2(\zeta(z) - \zeta(z - \omega_2) - \eta_2/2).
\end{split}
\end{equation}
By differentiating the last equation, we get
\begin{equation} \label{s'}
\begin{split}
\frac{1}{2}s'(z) &= \wp(z) - \wp(z - \omega_2) \\
&= \wp(z) - e_2 - \frac{\mu}{\wp(z) - e_2}
\end{split}
\end{equation}
where we have used the half period formula with
\begin{equation*}
\mu = (e_1 - e_2)(e_3 - e_2) = e_1 e_3 - (e_1 + e_3)e_2 + e_2^2
= 2e_2^2 + e_1 e_3.
\end{equation*}
Also
\begin{equation*}
\begin{split}
\frac{1}{2}s'' &= \wp' + \frac{\mu\wp'}{(\wp - e_2)^2} \\
&= s \Big(\wp - e_2 + \frac{\mu}{\wp - e_2}\Big).
\end{split}
\end{equation*}
(Notice the variations on signs with (\ref{s'}).) Then we have

\begin{lemmass} \label{lemma-s''}
The slope satisfies the ODE:
\begin{equation} \label{s''}
s'' = \frac{1}{2} s^3 - 6e_2 s.
\end{equation}
\end{lemmass}

\begin{proof}
We will compute $s''$ in a different way, namely
\begin{equation} \label{s'-2}
s' = \frac{\wp''}{\wp - e_2} - \frac{\wp' \wp'}{(\wp - e_2)^2}
= \frac{6\wp^2 - \frac{1}{2}g_2}{\wp - e_2} - s^2.
\end{equation}
It is elementary to see
\begin{equation*}
6\wp^2 - \frac{g_2}{2} = 6(\wp - e_2)^2 + 12 e_2(\wp - e_2) + 6e_2^2
- \frac{g_2}{2}
\end{equation*}
and
\begin{equation*}
6e_2^2 - \frac{g_2}{2} = 6e_2^2 + 2(e_1 e_2 + e_3 e_2 + e_1 e_3)
= 2(2e_2^2 + e_1 e_3) = 2\mu.
\end{equation*}
Thus (\ref{s'-2}) becomes
\begin{equation*}
s' = 12 e_2 - s^2 + 6(\wp - e_2) + \frac{2\mu}{\wp - e_2}.
\end{equation*}
Then
\begin{equation*}
\begin{split}
s'' &= -2s s' + 6\wp' - \frac{2\mu s}{\wp - e_2} \\
&= -24 e_2 s + 2s^3 - 12s (\wp - e_2) - \frac{4\mu s}{\wp - e_2}
+ 6s(\wp - e_2) - \frac{2\mu s}{\wp - e_2} \\
&= -24e_2 s + 2s^3 - 6s \Big(\wp - e_2 + \frac{\mu}{\wp - e_2}\Big) \\
&= -24e_2 s + 2s^3 - 3s'',
\end{split}
\end{equation*}
where the last equality is by (\ref{s''}). The lemma follows.
\end{proof}

To proceed to higher even derivatives, we notice that
\begin{equation} \label{s^k}
(s^k)'' = (ks^{k - 1}s')' = k(k - 1)s^{k - 2} (s')^2 + k s^{k - 1}s''.
\end{equation}
By (\ref{s'}) and (\ref{s''}),
\begin{equation*}
\begin{split}
(s')^2 &= 4 \Big( \wp - e - \frac{\mu}{\wp - e}\Big)^2 \\
&= 4 \Big( \wp - e + \frac{\mu}{\wp - e}\Big)^2 - 16\mu
= \Big(\frac{s''}{s}\Big)^2 - 16\mu
\end{split}
\end{equation*}
which is an even degree polynomial in $s$ of degree 4 by Lemma
\ref{lemma-s''}. Thus $(s^k)''$ is odd in $s$ of degree $k + 2$ if
$k$ is odd. By induction we then have that $s^{(2j)}$ is a degree
$2j + 1$ odd polynomial in $s$.

The proposition now follows easily from
\begin{equation}
2g^{(2j)}(0) = \sum_{i = 1}^l s^{(2j)}(p_i)
\end{equation}
and general facts on symmetric polynomials.
\end{proof}

Now we are ready to prove
\begin{theorem} \label{even-II}
Let $\rho = 4\pi l$. If the developing map $f$ 
satisfies the type I relation \textup{(\ref{case-type-I})}, 
then $l$ is odd. Furthermore $g(-z) = -g(z)$ and $u(-z) = u(z)$. 
\end{theorem}

\begin{proof}
Consider the polynomial
\begin{equation*}
S(x) = \prod_{i = 1}^l (x - s(p_i)).
\end{equation*}
By Proposition \ref{even-odd}, the relations
\begin{equation*}
0 = g(0) = g''(0) = \cdots = g^{(2[\frac{l - 1}{2}])}(0)
\end{equation*}
lead to the vanishing of all odd symmetric polynomials of
$s(p_i)$'s in the expansion of $S(x)$.

If $l = 2n$, then $S(x)$ consists of only even degrees and its
roots $s(p_i)$ must appears in pairs. Without loss of generality
we may assume that
\begin{equation} \label{even-pm}
s(p_1) = -s(p_{n + 1}), s(p_2) = -s(p_{n + 2}), \ldots, s(p_n) =
-s(p_{2n}).
\end{equation}

Notice that the slope equation
$$
\frac{\wp'(a)}{\wp(a) - e_2} = s(a) = -s(b) = -\frac{\wp'(b)}{\wp(b) - e_2}
$$
leads to $b = -a$ or $b = a + \omega_2$. To see this, notice that
under the cubic embedding $z \mapsto (\wp(z), \wp'(z))$, $s(a)$ is
slope of the line $\ell_a$ connecting the images of $z = \omega_2$
and $z = a$, with the unique third intersection point being $z =
-a  - \omega_2$ and $s(-a - \omega_2) = s(a)$. Thus the slope
function defines a branched double cover
$$
s: E' \to \mathbb{P}^1(\mathbb{C}).
$$
(From (\ref{s'}), it has 4 branch points given by $\wp(z) = e_2
\pm \sqrt{\mu}$.)

In particular the line with slope $-s(a) = s(-a)$ and passing
through $(e_2, 0)$ must be  $\ell_{-a} \equiv \ell_{a +
\omega_2}$. That is, $b = -a$ or $b = a + \omega_2$ as claimed.

In our case (\ref{even-pm}), we must conclude $p_{n + 1} = -p_1$
since $p_1 + \omega_2 = q_1$ can not appear in $p_i$'s. In the
same way we conclude that
\begin{equation} \label{pm}
p_i = -p_{i + n}, \quad i = 1, \ldots, n.
\end{equation}
In particular $\sum p_i = 0$. But this violates $\sum p_i \equiv
\frac{1}{2}\omega_1$ modulo $\Lambda'$ (which follows from $g(z +
\omega_2) = -g(z)$ in Lemma \ref{half-cong}), hence $l$ is odd.

For $l = 2n + 1$, $S(x)$ is a polynomial in odd degrees only. In
particular there is a root $x = 0$ of $S(x)$ and we may assume
that $s(p_{2n + 1}) = 0$ (namely $p_{2n + 1} =
\frac{1}{2}\omega_1$ or $\frac{1}{2}(\omega_1' + \omega_2') =
\frac{1}{2}\omega_1 + \omega_2$). 

Consider the polynomial $S(x)/x$ in pure even degrees, then in
exactly the same manner as above we conclude that (\ref{pm}) still
holds and
$$
S(x) = x \prod_{i = 1}^n (x - s(p_i))(x + s(p_i)).
$$

It is clear that now $g(-z) = -g(z)$. 
Then $f(-z) = f(z)$, which implies that $u$ is an even function. 
\end{proof}

\subsection{The polynomial system}

The remaining statements in Theorem \ref{thm-type I} 
which have not been proved yet are that these $p_1, \ldots, p_n$ 
are determined by polynomial equations in
$\wp(p_i)$'s.

Philosophically this follows easily from (\ref{g-zeta}) and
(\ref{s'}). Indeed it is clear that the odd order derivatives of
$g$ at $z = 0$ will involve only rational expressions with
denominator being powers of $\wp(p_i) - e_2$ and with at most even
derivatives $\wp(z)^{(2j)}(p_i)$ in the numerator (all expressions
in $-p_i$ are transformed into expressions in $p_i$). The latter
can be written into polynomials in $\wp(p_i)$ 
and thus the polynomial system is obtained.

\begin{proof}[Proof of Theorem \ref{thm-type I}]

To write down the complete set of polynomial equations explicitly, recall
\begin{equation} \label{g-der}
\begin{split}
g(z) &= \sum_{i = 1}^l (\zeta(z - p_i) - \zeta(z - p_i - \omega_2) - \eta_2/2), \\
-g'(z) &= \sum_{i = 1}^l (\wp(z - p_i) - \wp(z - p_i - \omega_2)),\\
-g^{(m + 1)}(z) &= \sum_{i = 1}^l (\wp^{(m)}(z - p_i) 
- \wp^{(m)}(z - p_i - \omega_2)) \quad \forall m \in \mathbb{Z}_{\ge 0}, 
\end{split}
\end{equation}
and the half period formula (let $\tilde \wp(p) = \wp(p + \omega_2)$) 
\begin{equation*}
\tilde \wp = e_2 + \frac{\mu}{\wp - e_2}
\end{equation*}
where $\mu = (e_1 - e_2)(e_3 - e_2)$. 
Equivalently $(\wp - e_2)(\tilde \wp - e_2) = \mu$.

In the proof of Theorem \ref{even-II}, 
the even order derivatives $g^{(2j)}(0) = 0$, $j = 0, \ldots, n$, 
leads to the evenness of solutions. 
We will show that the remaining odd order differentiations 
$g^{(2j + 1)}(0) = 0$, $j = 0, \ldots, n - 1$, leads to the desired polynomial system. 

To calculate $g^{(2j + 1)}(0)$, we first notice that
\begin{lemmass} \label{P-der}
For $k \in \mathbb{N}$, $(\wp^k)''$ is a degree $k + 1$ polynomial 
in $\wp$. Indeed
$$
(\wp^k)'' = 2k(2k + 1) \wp^{k + 1} - \frac{g_2}{2} 
k(2k - 1) \wp^{k - 1} - k(k - 1)g_3 \wp^{k - 2}. 
$$
\end{lemmass}

\begin{proof}
Since $(\wp^k)' = k \wp^{k - 1}\wp'$, we get
\begin{equation*}
(\wp^k)'' = k(k - 1)\wp^{k - 2}(\wp')^2 + k\wp^{k - 1}\wp''.
\end{equation*}
The lemma follows from the cubic relations.
\end{proof}

Now we set $x_i = \wp(p_i)$, $\tilde x_i = \tilde \wp(p_i) 
= \wp(p_i + \omega_2)$ for $i = 1, \ldots, n$. 
It is clear that $(x_i - e_2)(\tilde x_i - e_2) = \mu$ for all $i = 1, \ldots, n$.

During the following computations, we assume that 
$p_{2n + 1} = \tfrac{1}{2} \omega_1$ and $p_{n + i} = -p_i$ 
for $i = 1, \ldots, n$. 
For the other case $p_{2n + 1} = \tfrac{1}{2} \omega_1 + \omega_2$, 
we could replace $f$ by $1/f$ to reduce to the former case, 
since $f$ and $1/f$ give rise to the same solution $u$.

For $j = 0$ we have from (\ref{g-der}) that
\begin{equation*}
-g'(0) = 2 \sum_{i = 1}^n x_i + e_1 - 2\sum_{i = 1}^n \tilde x_i - e_3 = 0.
\end{equation*}
This is the degree one equation ($m = 1$) with $c_1 = -\tfrac{1}{2}(e_1 - e_3) \ne 0$.

For $j = 1$, since 
\begin{equation*}
-g''' = \sum_1^l \wp'' - \sum_1^l \tilde \wp'' 
= 6\sum_1^l \wp^2 - 6\sum_1^l\tilde \wp^2,
\end{equation*}
the equation $g'''(0) = 0$ becomes
\begin{equation*}
\sum_{i = 1}^n x_i^2 - \sum_{i = 1}^n \tilde x_i^2 = -\tfrac{1}{2}(e_1^2 - e_3^2).
\end{equation*}
This is the degree two equation ($m = 2$) with $c_2 = -\tfrac{1}{2}(e_1^2 - e_3^2)$.

The general case follows from Lemma \ref{P-der}. 
Suppose that $g^{(2j + 1)}(0) = 0$ gives rise to a new polynomial relation 
$\sum_{i = 1}^n x_i^j - \sum_{i = 1}^n \tilde x_i^j = c_j$. 
A further double differentiation increases the degree of the
polynomial in $\wp$ by one, hence it gives rise to 
a new relation 
$\sum_{i = 1}^n x_i^{j + 1} - \sum_{i = 1}^n \tilde x_i^{j + 1} = c_{j  + 1}$, 
with the universal constant $c_{j + 1}$ 
being determined by $c_1, c_2, g_2, g_3$ recursively.

Therefore, we conclude that 
$x_i = \wp(p_i)$, $\tilde x_i = \wp(p_i + \omega_2)$, $i = 1, \ldots, n$, 
satisfy the polynomial system:
\begin{equation*}
\begin{split}
\sum_{i = 1}^n x_i^j - \sum_{i = 1}^n \tilde x_i^j = c_j, \quad j = 1, \ldots, n,\\
(x_i - e_2)(\tilde x_i - e_2) = \mu, \quad i = 1, \ldots, n,
\end{split}
\end{equation*}
which is easily seen to be equivalent to the system (\ref{system-I}).

Conversely, any solution of the polynomial system gives rise to 
a function $g$ which satisfies
$$
g^{(j)}(0) = 0, \quad j = 0, 1, \ldots, 2n.
$$
From $g$, the developing map $f$ is then constructed 
by Proposition \ref{g-odd-u-even}.
\end{proof}

\begin{remarkss}
In the next section we will prove that except for a finite set of tori, 
the mean field equation (\ref{Liouville-eq}) has exactly $n + 1$ 
solutions for $\rho = 4\pi l$ with $l = 2n + 1$. 
This implies that, except for those tori, 
the above polynomial system has exactly $n + 1$ solutions 
up to permutation symmetry by $S_n$. Equivalently it has $(n + 1)!$ solutions.

Since $c_j(\tau)$'s are all holomorphic in $\tau$, 
solutions $(x_i(\tau), \tilde x_i(\tau))$ of the polynomial system, 
hence the developing map $f(z; \tau)$, 
should then depend on $\tau$ holomorphically. 
It is not so obvious how to prove the holomorphic dependence 
of $f(z; \tau)$ in the moduli space of tori by other methods.
\end{remarkss}

\begin{example}
For $\rho = 4\pi$, $l = 1$ and $n = 0$. 
Then $p_1 = \tfrac{1}{2}\omega_1$. 
The polynomial system is empty and the solution $u$ is unique. 
This was first proved in \cite{LW}.
\end{example}

\begin{example} \label{12pi}
Consider the case $\rho = 12\pi$, i.e.~$l = 3$ and $n = 1$. Let
$p_1 = a$. $p_2 = -a$ and $p_3 = \frac{1}{2}\omega_1$. Then the
equation $g'(0) = 0$ becomes
\begin{equation*}
2\Big((\wp(a) - e_2) - \frac{\mu}{\wp(a) - e_2}\Big) + (e_1 - e_3) = 0.
\end{equation*}
That is, we get a degree 2 polynomial in $\wp(a)$:
\begin{equation*}
(\wp(a) - e_2)^2 + \tfrac{1}{2}(e_1 - e_3) (\wp(a) - e_2) - \mu = 0
\end{equation*}
and then
\begin{equation*}
\wp(a) = e_2 + \tfrac{1}{4}(e_3 - e_1) \pm \tfrac{1}{4} \sqrt{(e_3 -
e_1)^2 + 16(e_1 - e_2)(e_3 - e_2)}.
\end{equation*}

These are exactly the solutions obtained in \cite{LW2} via a
different method. In particular there are precisely two solutions
of the mean field equation on any torus $E$ with 
non-zero discriminant $(e_3 - e_1)^2 + 16(e_1 - e_2)(e_3 - e_2) \ne 0$ 
for the double cover $E'$, and with $\rho = 12\pi$. 
The case with zero discriminant will be discussed in Example \ref{ex-pn}.
\end{example}

\begin{example} \label{20pi}
Consider the case $\rho = 20\pi$, i.e.~$l = 5$ and $n = 2$. 
The full set of polynomial equations in $x_i$'s and $\tilde x_i$'s is given by 
\begin{equation*}
\begin{split}
x_1 + x_2 - \tilde x_1 - \tilde x_2 &= c_1  = -\tfrac{1}{2}(e_1 - e_3),\\
x_1^2  + x_2^2 - \tilde x_1^2 - \tilde x_2^2 &= c_2  = -\tfrac{1}{2}(e_1^2 - e_3^2),\\
(x_1 - e_2)(\tilde x_1 - e_2) &= \mu,\\
(x_2 - e_2)(\tilde x_2 - e_2) &= \mu.
\end{split}
\end{equation*}

Now the number of solutions $N'_n$ (here $n = 2$) 
for $x_1. x_2, \tilde x_1, \tilde x_2$ can be calculated by 
the Bezout theorem to be 
$N'_2 = 1 \times 2 \times 2 \times 2 - r_2^\infty = 8 - r_2^\infty$ 
where $r_2^\infty$ is the number of solutions at $\infty$, 
counted with multiplicity, of the projectivized 
system of polynomial equations. The projective system is 
\begin{equation*}
\begin{split}
X_1  + X_2 - \tilde X_1 - \tilde X_2 &= c_1 X_0,\\
X_1^2  + X_2^2 - \tilde X_1^2 - \tilde X_2^2 &= c_2 X_0^2,\\
(X_1 - e_2 X_0)(\tilde X_1 - e_2 X_0) &= \mu X_0^2,\\
(X_2 - e_2 X_0)(\tilde X_2 - e_2 X_0) &= \mu X_0^2.
\end{split}
\end{equation*}
And the infinity solutions are given by setting $X_0 = 0$:
\begin{equation*}
\begin{split}
X_1 + X_2 = \tilde X_1 + \tilde X_2, 
&\qquad X_1^2 + X_2^2 = \tilde X_1^2 + \tilde X_2^2,\\
X_1 \tilde X_1 = 0, &\qquad X_2 \tilde X_2 = 0.
\end{split}
\end{equation*}
This shows that $\{X_1, X_2\} = \{\tilde X_1, \tilde X_2\}$. 
Since these four variables are not all zero, it is easy to see 
that there are precisely two solutions given by
\begin{equation*}
\begin{split}
P_1:\quad X_1 = 0 = \tilde X_2, \quad X_2 = \tilde X_1 \ne 0,\\
P_2:\quad X_2 = 0 = \tilde X_1, \quad X_1 = \tilde X_2 \ne 0.
\end{split}
\end{equation*}

It remains to compute the multiplicity of $P_1$ and $P_2$. 
Consider $P_1$ first. Since it is in the chart 
$\tilde U_1 := \{ \tilde X_1 \ne 0\}$, in terms of 
$y_i = X_i/\tilde X_1$, $i = 1, 2$, $\tilde y_2 = \tilde X_2/\tilde X_1$ 
and $y_0 = X_0/\tilde X_1$, $P_1$ has coordinates 
$(y_0, y_1, y_2, \tilde y_2) = (0, 0, 1, 0)$ 
and the system at point $P_1$ reads as $f_i = 0$, $i = 1, \ldots, 4$, where
\begin{equation*}
\begin{split}
f_1 &= y_1 + y_2 - 1 - \tilde y_2 - c_1 y_0,\\
f_2 &= y_1^2 + (y_2 - 1)^2 +2(y_2 - 1) - \tilde y_2^2 - c_2 y_0^2,\\
f_3 &= (y_1 - e_2 y_0)(1 - e_2 y_0) - \mu y_0^2 = y_1 + \cdots,\\
f_4 &= (y_2 - e_2 y_0)(\tilde y_2 - e_2 y_0) - \mu y_0^2 
= (y_2 - 1)\tilde y_2 + \tilde y_2 + \cdots.
\end{split}
\end{equation*}
From these expressions, the appearance of degree one monomial 
in each $f_i$ shows that the local analytic coordinates 
$(y_0, y_1, y_2 - 1, \tilde y_2)$ at the point 
$P_1$ can be replaced by $f_1, f_3, f_2, f_4$ accordingly, 
and thus the multiplicity is one. 
Indeed $P_1 = (0, 0, 1, 0)$ is a simple point of $\{ f_i = 0 \}$ 
by computing the Jacobian
\begin{equation*}
\det \frac{\p (f_1, f_2, f_3, f_4)}{\p (y_0, y_1, y_2, \tilde y_2)}
(0, 0, 1, 0) = e_1 - e_3 \ne 0.
\end{equation*}  
Similarly the multiplicity at $P_2$ is one. 
Thus $r_2^\infty = 2$ and $N_2' = 8 - 2 = 6$. 

Since any reordering of $p_i$'s leads to the same solution, 
also it is easy to see that for generic tori 
we do not have any solution with $x_1 = x_2$, so finally
\begin{equation*}
N_2 = N_2'/2! = 3 = 2 + 1.
\end{equation*} 
\end{example}

\begin{remark} \label{KL-Lin}
The above method can be extended to the case $n = 3$, $\rho = 28\pi$ 
to show that $N_3 = 4$ since in this case the infinity solutions 
are still zero dimensional. It fails for $n \ge 4$ 
since positive dimensional intersections at infinity do occur 
and excess intersection theory is needed. 
The cases $n = 4$ and $n = 5$ were recently settled in \cite{KL} 
where the infinity solutions are one dimensional.
\end{remark}

%
%
%
%
%
%

\section{Lam\'e for type I: Finite monodromies} \label{monodromy-I}
\setcounter{equation}{0}

\noindent
In this section we prove Theorem \ref{thm-K4} (c.f.~Theorem \ref{K4-thm}).

\subsection{From mean field equations to Lam\'e}

The second order equation
\begin{equation} \label{Lame-2}
L_{\eta, B}\, w := w''(z) - (\eta(\eta + 1) \wp(z) + B) w(z) = 0
\end{equation}
is known as the Lam\'e equation with two parameters \(\eta\) and
\(B\); the parameter $\eta$ is called the \emph{index} and $B$ is the 
called the ``\emph{accessary parameter}''.

\subsubsection{}
Recall that for any two linearly independent solutions $w_1$ and $w_2$
of a general second order ODE $w'' = I w$,
the Schwarzian derivative 
$$S(h)= \frac{h'''}{h'}-\tfrac{3}{2}
\left(\frac{h''}{h'}\right)^2$$
of $h = w_1/w_2$ satisfies $S(h) = -2I$, hence for any two
linear independent local solutions \(w_1, w_2\) of the Lam\'e equation 
(\ref{Lame-2}) we have
$$
S(\tfrac{w_1}{w_2}) = -2(\eta(\eta + 1) \wp(z) + B).
$$
Conversely if \(h_1\) is meromorphic function with
\(\,S(h_1)=-2(\eta(\eta + 1) \wp(z) + B\),
then \(\,S(h_1)=S(\tfrac{w_1}{w_2})\) for a chosen pair of linearly
independent solutions \(w_1, w_2\) of (\ref{Lame-2}), 
therefore \(\,h_1\,\) is equal to a linear fractional transformation
of \(\tfrac{w_1}{w_2}\), or equivalently there exists a pair of
linearly independent solutions \(w_3, w_4\) of (\ref{Lame-2})
such that \(h_1={w_3}/{w_4}\).

\subsubsection{}
Suppose that \(u\) is a solution o f the mean field equation
\begin{equation} \label{MFE-rho}
\triangle u + e^u = \rho\delta_0
\end{equation}
on a flat torus $E = \mathbb C/\Lambda$, $\Lambda 
= \mathbb Z \omega_1 + \mathbb Z \omega_2$,
and \(f\) is a developing map of \(u\) on a covering space
of the punctured torus \(\,E\smallsetminus\{0\}\).
Locally $u$ is expressed in $f$ via
\begin{equation*}
u(z) = \log \frac{8|f'(z)|^2}{(1 + |f(z)|^2)^2}.
\end{equation*} 
Let \(\,\eta:={\rho}/{8\pi}\).
By (\ref{3-3}), we have
\begin{equation} \label{S-der}
\begin{split}
S(f) &:= \frac{f'''}{f'} - \frac{3}{2} \Big( \frac{f''}{f'} \Big)^2 \\
&= u_{zz} - \frac{1}{2} u_z^2 = -2\eta(\eta + 1) \frac{1}{z^2} + O(1),
\end{split}
\end{equation}
where the last equality follows from the asymptotic expansion 
\[u \sim 4\eta \log |z|\] 
at $z = 0$ and that $u$ is smooth outside $z = 0$ in $E$.
The expression of $S(f)$ in $u$ shows that it is a meromorphic
function on \(E\), which is holomorphic outside \(\{0\}\) and
its polar part at \(z=0\) is given by the last
expression in (\ref{S-der}).
Therefore there exists a constant \(B = B(E, \eta, u)\) such that
\begin{equation} \label{S(f)}
S(f)=u_{zz} - \tfrac{1}{2} u_z^2 = -2(\eta(\eta + 1) \wp(z) + B).
\end{equation} 
It follows that there exists two linearly independents solutions
$w_1$ and $w_2$ of the Lam\'e equation (\ref{Lame-2})
with accessary parameter \(B(E,\eta,u)\) such that 
\(\,f = w_1/w_2\).

\subsubsection{}
%
The Lam\'e equation (\ref{Lame-2})
had been studied in the classical literature
in two special cases, very extensively in case when
the index $\eta$ is a positive integer, and
somewhat less so in the case when the index \(\eta\)
is a  half-integer, i.e.\ $2\eta = 2n + 1$ is an odd positive integer.
We have seen in the previous sections that the
former case corresponds to type II solutions 
while the latter case is for type I solutions. 
The main objective of this section is to prove that for any
odd positive integer \(2n+1\), on all but a finite
number of isomorphism classes of elliptic curves,
there are precisely $n + 1$ solutions to the mean field equation
\(\,\triangle u+ e^u = 4\pi(2n+1)\delta_0\).

The following theorem is due to 
Brioschi \cite{Brioschi}, Halphen \cite[pp.\,471--473]{Halphen} 
and Crawford \cite{Crawford} in the late nineteenth century;
see for \cite{Crawford} for a complete proof.
See also \cite[pp.\,162--164]{Poole} for a succinct presentation of
Halphen's transformation and Crawford's procedure for analyzing 
Brioschi's solution; c.f.\ \cite[p.\,570]{Whittaker}.

\begin{theorem} 
\label{BH-poly}
Let \(n\) be a non-negative integer.
\begin{itemize}
\item[\textup{(a)}] There exists a 
monic polynomial 
$\,p_n(B; \Lambda)=p_n(B, g_2(\Lambda), g_3(\Lambda))$ 
of degree $n + 1$ in $B$ with coefficients
in \(\,\ZZ[\tfrac{g_2(\Lambda)}{4}; \
\tfrac{g_3(\Lambda)}{4}]\,\) such that 
the Lam\'e equation \(\,\displaystyle{L_{n + 1/2, B}\, w = 0}\,\) 
on \(\CC/\Lambda\)
has all solutions free from logarithm at $z = 0$ if and only if 
$p_n(B) = 0$.
This polynomial \(\,p_n(B, g_2, g_3)
\in \ZZ[\tfrac{1}{2}][B, g_2, g_3]\,\) is homogeneous
of weight \(n\) if \(B, g_2, g_3\) are given weights
\(1, 2, 3\) respectively.

\item[\textup{(b)}] For any lattice \(\Lambda\)
outside a finite subset \(S_n\) of homothety classes of
lattices in \(\CC\), 
the polynomial \(p_n(B;\Lambda)\) has \(n + 1\) distinct 
roots.
\end{itemize}
\end{theorem} 

\begin{proof}
The logarithm-free solutions of the Lam\'e equation 
\(\,\displaystyle{L_{n + 1/2, B}\, w = 0}\,\) were first 
discovered by Brioschi \cite[p.\,314]{Brioschi}, but the 
underlying structure are more transparently exhibited
using Halphen's transformation \cite[p.\,471]{Halphen}
as carried out in detail by Crawford \cite{Crawford}.
The statement (a) is proved in \cite{Crawford};
see also \cite[p.\,164]{Poole} for a presentation of
Crawford's proof.
A slightly different proof of (a) following the same train of
ideas can be found in \cite[p.\,26--28]{Baldassari}.
\smallbreak

Crawford's proof provides a recursive formula for
\(\,p_n(B;\Lambda)\).  
When \(\Lambda\) is of the form 
\(\ZZ+\sqrt{-1}a\ZZ\) with \(a\in \RR_{>0}\),
this recursive formula also produces a 
\emph{Sturm sequence} starting with 
\(\,p_n(B)\), therefore
\(\,p_n(B;\Lambda)\,\) has \(n+1\) distinct
real roots; see
\cite[p.\,94]{Crawford}.%
\footnote{The statement that 
\(p_n(B)\) has \(n+1\) distinct real roots was proved in 
\cite[p.\,94]{Crawford} under the condition that the 
\(x\)-coordinates 
of the three non-trivial two-torsion points,
\(e_i=\wp(\omega_i/2;\Lambda)\) for \(i=1, 2, 3\),
are real numbers.  This is the case when 
the lattice \(\Lambda\) is of the form 
\(\Lambda_{\tau}\) with \(\tau\in\sqrt{-1}\RR_{>0}\).}
This implies that the discriminant of the polynomial
\(\,p_n(B;\Lambda)\,\), which is a modular form 
for \(\,\textup{SL}_2(\ZZ)\), is not identically \(0\).
The statement (b) follows. 
See \S \ref{rem:sturm-seq} for remarks on Sturm's theorem
used in Crawford's proof.
\end{proof}

\begin{remarkss}
We will give an alternative proof 
of part (a) of Theorem \ref{BH-poly} in \S \ref{alt-proof-BH},
which is essentially local near $z = 0$. 
Our proof not only provides a new construction of the polynomial $p_n(B)$, 
it also generalizes to the case with multiple singular sources. 
This generalization will be presented in a later work; c.f.~\cite{CLW2}.
\end{remarkss}

\subsection{\bf Remark on Sturm's theorem.}\enspace
\label{rem:sturm-seq}
Crawford's proof in \cite[p.\,94]{Crawford} that 
the polynomial
\(\,p_n(B;\Lambda)\,\) has \(n+1\) distinct real roots 
for rectangular tori
uses a fact closely related to Sturm's theorem on
real roots of polynomials over \(\RR\),
not found in standard treatment
of this topic, such as \cite[11.3]{vdWaerden} 
and \cite[5.2]{Jacobson}.\footnote{This fact must 
be familiar to all educated 
scientists in the late nineteenth and early twentieth
century, often used freely without comments in 
mathematical writings at the time.
This is the case for the proofs 
in \cite[p.\,557]{Whittaker}
and \cite[p.\,163]{Poole} for the existence of
\(2m+1\) distinct real roots of the polynomial \(\,l_m(B)\) 
corresponding to \(2m+1\) Lam\'e functions for 
the equation \[\,\displaystyle{\frac{d^2w}{dz^2}
- (m(m+1)\wp(z;\Lambda)+B)w = 0}\,\]
when \(\Lambda = \ZZ+\sqrt{-1}a\ZZ\,\) for some 
\(a\in\RR_{>0}\) and \(m\in \NN_{>0}\).
However this then-well-known fact is no longer part of the general 
education for mathematicians today.}
We have been able to find only one reference of this fact,
as a ``starred exercise'' in \cite[p.\,149
ex.\,30]{Uspenski}. 
In Proposition \ref{prop:sturm_thm} below
we provide a mild generalization of the usual form of
Sturm's theorem for the convenience of the readers.
Its corollary \ref{cor:sturm}
is equivalent to 
\cite[p.\,149 ex.\,30]{Uspenski}.

\begin{definitionss}\label{def:sturm_seq}
A sequence of non-zero polynomials 
\[f_0(x), f_1(x), \ldots, f_m(x)\in\RR[x]\]
is a \emph{Sturm sequence} on \((a,b]\) if the following two properties
hold.
\smallbreak
\noindent
(i) \(f_m(x)\) is either positive definite or negative definite on
\((a,b]\).
\smallbreak
\noindent
(ii) Suppose that \(\xi\in (a,b]\) and \(f_i(\xi)=0\) for some \(i\)
with \(1\leq i\leq m-1\).
Then \(f_{i-1}(\xi)\) and \(f_{i+1}(\xi)\) have opposite signs
(in the sense that either they are both non-zero with opposite signs,
or are both zero.\footnote{The latter possibility is ruled out by 
condition (i).})
\end{definitionss}

\begin{remark*} There is an extra
condition in the conventional definition of a Sturm sequence: 
\emph{\(\,f_1(\xi)\,\) and \(\,f_0(\xi)\,\) have the same sign 
for every root \(\xi\) of \(f_0(x)\) in \((a,b]\)}.
This condition has been dropped in 
Definition \ref{def:sturm_seq} above.
\end{remark*}

\subsubsection{\bf Definition.}
Let \(f_0(x), f_1(x), \ldots, f_m(x)\) be a Sturm sequence.
\smallbreak
\noindent
(1) For every real number \(\xi\), define \(\,\sigma(\xi)\,\) to be the total number
of changes of signs in the sequence 
\((f_0(\xi^+), f_1(\xi^+),\ldots, f_{m-1}(\xi^+), f_m(\xi))\).\footnote{Here we used
\(f_i(\xi^{+})\) to make sure that each term has a well-defined sign.
In view of condition (ii), we could have used
the sequence \((f_0(\xi^{+}), f_1(\xi),\ldots, f_{m-1}(\xi), f_m(\xi))\) 
in the definition, suppress
zeros when counting the number of variations of signs in it.}
\smallbreak
\noindent
(2) Define a \(\{-1, 0, 1\}\)-valued
``local index'' function
\(\epsilon_{f_0(x)}\) on \(\RR\) attached to
a real polynomial \(f_0(x)\in \RR[x]\) as follows.  
\begin{itemize}
\item Suppose that \(f_0(\xi)=0\)\footnote{\(f_1(\xi)\neq 0\)
if \(f_0(\xi)=0\), by (i) and (ii).} and
\(\textup{mult}_{x=\xi}\,f_0(x)\) is \emph{odd}.\footnote{For a zero 
\(\xi\) of \(f_0(x)\), the sign of
\(f_0(x)\) changes when \(x\) moves across \(\xi\) if and only if
\(\textup{mult}_{x=\xi}\,f_0(x)\)
is odd.} 
Define
\[
\epsilon_{f_0(x)}(\xi):=\left\{
\begin{array}{ll}1 &\textup{if}\ \  f_0(\xi^+)\  \textup{and}  \ f_1(\xi) \
\textup{have the same sign}
\\  
-1 &\textup{if}\ \ f_0(\xi^+)\ \textup{and}  \ f_1(\xi) \
\textup{have opposite signs}
\end{array}
\right.
\]

\item \(\epsilon_{f_0(x)}(\xi)=0\)
if \(\textup{mult}_{x=\xi}\,f_0(x)\) is even.
In particular \(\epsilon_{f_0(x)}(\xi)=0\) if \(f_0(\xi)\neq 0\).

\end{itemize}

\noindent
(3) Define \(Z_{f_0(x)}((a,b])\in \ZZ\) by
\[
Z_{f_0(x)}((a,b]):=\sum_{\xi\in (a,b]}\,\epsilon_{f_0(x)}(\xi).
\]
This number \(Z_{f_0(x)}((a,b])\) counts the number of zeros of \(f_0(x)\) 
with odd multiplicity with a signed weight given by \(\epsilon_{f_0(x)}\).
It can be thought of as some sort of ``total Lefschetz number'' 
for \(f_0(x)\vert_{(a,b]}\).

\begin{propositionss}
\label{prop:sturm_thm}
Let \(f_0(x),f_1(x),\ldots,f_m(x)\) be a Sturm sequence
on \((a,b]\).  Then 
\[Z_{f_0(x)}((a,b])=\sigma(a)-\sigma(b),\] i.e.\
\(\sigma(a)-\sigma(b)\) is the number of zeros of \(f_0(x)\) in the half-open
interval \((a,b]\) with \emph{odd} multiplicity, counted with the 
sign \(\epsilon_{f_0(x)}\).
\end{propositionss}

\begin{proof}
Condition (ii) ensures that crossing a zero in \([a,b)\)
of any of the internal members \(f_1(x),\ldots, f_{m-1}(x)\) of the Sturm chain
makes no contribution to changes of \(\sigma(\xi)\).  
Each time a zero \(\xi_0\) of \(f_0(x)\) with odd multiplicity
is crossed, \(\sigma(\xi)\) decreases by \(\epsilon_{f_0(x)}(\xi)\) as 
\(\xi\) moves from the left of \(\xi_0\) to its right.
On the other hand, moving across a zero of \(f_0(x)\) with even multiplicity
does not change the value of \(\sigma\). 
So the \(\sigma(b)-\sigma(a)\) is equal to the total number of zeros of \(f_0(x)\) in
\((a,b]\) with odd multiplicity, counted with the sign \(\epsilon_{f_0(x)}\).
\end{proof}

\begin{corollaryss} \label{cor:sturm}
Let \(f_0(x),\ldots, f_m(x)\) be a Sturm sequence on
\((a,b]\). Let \(n\in \NN\) be a non-negative integer. 
If \(\sigma(a)-\sigma(b)= \pm n \) and 
\(f_0(x)\) has at most \(n\) distinct real roots in \((a,b]\),
then \(f_0(x)\) has exactly \(n\) distinct real roots in the
half-open interval \((a,b]\).
In particular if \(a=-\infty, b=\infty\), \(\textup{deg}(f_0(x))=n\)
and \(\sigma(-\infty)-\sigma(\infty)= \pm n \), then 
\(f_0(x)\) has \(n\) distinct real roots.
\end{corollaryss}

\subsection{\bf A proof of  Theorem \ref{BH-poly}\,(a).}\enspace
\label{alt-proof-BH}
Let's start with any $f$ as the quotient of two independent solutions 
of Lam\'e equation $L_{n + 1/2, B}\,w = 0$ at $z = 0$ 
and consider $v(z) = \log f'(z)$. It is readily seen that
\begin{equation*}
v'' - \frac{1}{2} (v')^2 = \Big(\frac{f''}{f'}\Big)' 
- \frac{1}{2} \Big(\frac{f''}{f'}\Big)^2 = S(f).
\end{equation*}
We remark that the function $v$ satisfies the similar equation 
as $u$ in (\ref{S-der}), but $v$ is analytic in nature 
while $u$ is only a real function.

The indicial equation at $z = 0$ is given by 
$\lambda^2 - \lambda - \eta(\eta + 1) 
= (\lambda - (\eta + 1))(\lambda + \eta) = 0$. 
If there are logarithmic solutions, the fundamental solutions are given as
\begin{equation} \label{fund-sol}
w_1(z) = z^{\eta + 1} h_1(z), \quad w_2(z) = \xi w_1(z) \log z + z^{-\eta} h_2(z),
\end{equation}
where $\xi \ne 0$ and $h_1, h_2$ are holomorphic 
and non-zero at $z = 0$. But then 
$$
f = \frac{a w_1 + b w_2}{c w_1 + d w_2}
$$
is easily seen to be logarithmic as well if $ad - bc \ne 0$, 
thus the Lam\'e equation has no logarithmic solutions at $z = 0$ 
if and only if we have one nontrivial solution quotient $f$ 
to be logarithmic free at $z = 0$. 

Now suppose that the Lam\'e equation has no solutions 
with logarithmic term. 
Let $f$ be a ratio of two independent solutions. 
Without lose of generality, we may assume that $f$ is 
regular at $0$. Since $\eta = n + \tfrac{1}{2}$ and
\begin{equation} \label{S-eqn}
S(f) = -2((n + \tfrac{1}{2})(n + \tfrac{3}{2}) \wp(z) + B),
\end{equation}
to require that $f $ is logarithmic free at $z = 0$ 
is equivalent to that 
$f(z) = c_0 + c_{2n + 2} z^{2n + 2} + \cdots$ near $z = 0$ with $c_0 \ne 0$. 

Recall that
$$
\wp(z) = \frac{1}{z^2} + \sum_{k \ge 1} (2k + 1) G_{k + 1} z^{2k}
$$
where $G_{k} = \sum_{\omega \in \Lambda^*} 1/\omega^{2k}$ is 
the standard Eisenstein series of weight \(2k\) for
\(\textup{SL}_2(\ZZ)\). 
It is customary to write $g_2 = 60 G_2$ and $g_3 = 140 G_3$. 
It is also well known that all $G_k$'s are 
expressible as polynomials in $g_2$, $g_3$.

We will show that the solvability of the Schwarzian equation 
(\ref{S-eqn}) for $f$ being of the proposed form is equivalent to 
that $B$ satisfies $p_n(B) = 0$ for some universal polynomial 
$p_n(B, g_2, g_3)$ of degree $n + 1$. Indeed, 
$$
v = \log f' = \log c_{2n + 2} (2n + 2) + (2n + 1) \log z + \sum_{j \ge 1} d_j z^j.
$$
For convenience we set $e_j = (j + 1) d_{j + 1}$ for $j \ge 0$ and then
$$
v' = \frac{2n + 1}{z} + \sum_{j \ge 0} e_j z^j
$$

The degree $z^{-1}$ terms in
\begin{equation*} 
v'' - \frac{1}{2} (v')^2 =  -2((n + \tfrac{1}{2})(n + \tfrac{3}{2}) \wp(z) + B)
\end{equation*}
match by our choice. There is no $z^{-1}$ term in the RHS 
shows that $e_0 = 0$. Then the constant terms give 
$e_1 - \tfrac{1}{2} 2 (2n + 1) e_1 = -2B$, 
i.e.~$n e_1 = B$. For $n = 0$, we must conclude $B = 0$. 
Thus we set $p_0(B) = B$.

Similarly, for $j \ge 1$, the degree $j$ terms in the LHS give
$$
(j + 1) e_{j + 1} - \tfrac{1}{2} 2 (2n + 1) e_{j + 1} 
- \tfrac{1}{2} \sum_{i = 1}^{j - 1} e_i e_{j - i}.
$$ 
Since there is no odd degree terms in the RHS, 
by considering $j = 1, 3, 5, \ldots$ 
we first conclude inductively that $e_i = 0$ for $i$ even.

Next we consider degree $j = 2, 4, 6, \ldots$ terms inductively. 
Write $E_k = e_{2k - 1}$ for $k \ge 1$. Then $j = 2k$ leads to 
\begin{equation} \label{BH}
2(k - n) E_{k + 1} -\tfrac{1}{2} \sum_{i = 1}^k E_i E_{k + 1 - i} 
= -2(n + \tfrac{1}{2})(n + \tfrac{3}{2}) (2k + 1) G_{k + 1}. 
\end{equation}
We have just seen that $n E_1 = B$. If we assign degree $k$ to $G_k$, 
then (\ref{BH}) shows inductively that $E_k = E_k(B, g_2, g_3)$ 
is a degree $k$ polynomial in $B$ which is homogeneous 
in $B, g_2, g_3$ of degree $k$ up to $k \le n$.    

Now put $k = n$ in (\ref{BH}), the first term vanishes and we must have
$$
\tilde p_n(B, g_2, g_3) 
:= \sum_{i = 1}^n E_i E_{n + 1 - i} - 
8(n + \tfrac{1}{2})^2(n + \tfrac{3}{2}) G_{n + 1}
$$ 
vanishes too. Up to a multiplicative constant, 
this $\tilde p_n(B)$ is the degree $n + 1$ polynomial in $B$ 
we search for. Indeed, by our inductive construction through (\ref{BH}), 
the leading coefficients $c_n$ of $\tilde p_n(B)$ 
depends only on $n$. 
Hence $p_n(B, g_2, g_3) := c_n^{-1} \tilde p_n(B, g_2, g_3)$ 
is monic in $B$ and homogeneous of degree $n + 1$ in $B$, $g_2$, $g_3$. 

Conversely, if $\tilde p_n(B) = 0$, then $E_1, \ldots, E_n$ 
can be solved by (\ref{BH}) up to $k = n - 1$. 
For $k = n - 1$, $\tilde p_n(B) = 0$ is equivalent to (\ref{BH}) 
at $k = n$. By assigning any value to $E_{n + 1}$, 
we can use (\ref{BH}) for $k \ge n + 1$ to find $E_j$, $j \ge n + 2$. 
Thus this $f$ is a solution to the Schwarzian equation (\ref{S-eqn}) 
and is free from logarithmic terms. The proof is complete.
\qed

\begin{remarkss}
Notice that $E_{n + 1} = e_{2n + 1}= (2n + 2) d_{2n + 2}$ is a free parameter. 
All $E_k$'s are determined by $B$ and $E_{n + 1}$. 
For any $B$ with $p_n(B) = 0$, the three constants $c_0$, $c_{2n + 2}$ 
and $E_{n + 1}$ provide the three dimensional freedom for $f$ 
due to the freedom of ${\rm SL}_2(\mathbb C)$ action on $f$.
\end{remarkss}

\begin{remarkss}
We have seen that the type I developing map $f(z)$ is even. 
This also follows form our proof of Theorem \ref{BH-poly} 
since we do not assume the a priori evenness during the proof.
\end{remarkss}
\bigbreak

%
To apply Theorem \ref{BH-poly} to study mean field equations 
for $\rho = 4\pi(2n + 1)$, the essential point is the following theorem.

\begin{theorem} [$=$ Theorem \ref{thm-K4}] \label{K4-thm}
Let \(n\) be a non-negative integer.
The projective monodromy group of the Lam\'e equation 
$L_{n + (1/2), B}\, w = 0$ is isomorphic to 
Klein's four-group $(\ZZ/2\ZZ)^2$ if and only if there
exists two meromorphic solutions \(w_1, w_2\) on \(\CC\) of
the above Lam\'e equations such that
\(\,\tfrac{w_1}{w_2}\,\) is a type I developing map of
a solution of the mean field equation 
\(\,\triangle u+ e^u=4\pi(2n+1) \delta_0\).
Moreover, each such value of the
accessary parameter $B$ with the above property
gives rise to exactly one type I solution.
\end{theorem}

\begin{proof}
Let $u$ be a type I solution of the mean field equation
\(\,\triangle u+e^u=\rho\,\delta_0\) on \(\CC/\Lambda\) and
let \(f\) be a normalized developing map of \(u\) 
satisfying the type I transformation rules
(\ref{case-type-I}). We know from Theorem \ref{even-II}
that there exists a non-negative integer \(n\) such that
\(\,\rho=4\pi(2n+1)\), and we have seen that there exists
a complex number \(B\) such that
the Schwarzian derivative \(S(f)\)
of \(f\) is equal to 
\(\,-2\big(n+\tfrac{1}{2})(n+\tfrac{3}{2})\wp(z;\Lambda)+B\big)\).
Then local solutions of the Lam\'e equation
\(L_{n + 1/2, B}\, w = 0\) are free of logarithmic solutions,
and there exists two solutions \(w_1, w_2\) over \(\CC\)
such that \(\,f=\tfrac{w_1}{w_2}\).  
The projective monodromy group of the equation
\(L_{n + 1/2, B}\, w = 0\) is canonically isomorphic to 
the monodromy group of the meromorphic function
\(\tfrac{w_1}{w_2}\), which is a Klein four group.
We have proved the ``only if'' part
of Theorem \ref{thm-K4}.
\medbreak

Conversely, suppose that 
the projective monodromy group of a Lam\'e equation
\(L_{n + 1/2, B}\, w = 0\) is a Klein-four group.
Then all local solutions of this Lam\'e equation are free of
logarithmic singularities, and there are
for two linearly independent solutions $w_1, w_2$ of this equation
which are meromorphic functions over \(\CC\). It is easy to
check from basic theory of linear ODE's with regular singularities
that the holomorphic
map \(\,\frac{w_1}{w_2}: \CC\to \PP^1(\CC)\,\) has no critical 
point outside \(\Lambda\), and has multiplicity \(2n+2\) at
points of \(\Lambda\).

Let \(\rho:\Lambda\to \textup{GL}_2(\CC)\) be the monodromy
representation of the differential equation \(L_{n + 1/2, B}\, w = 0\)
attached to the basis \(w_1, w_2\) of solutions of
\(L_{n + 1/2, B}\, w = 0\).
Let \(\bar{\rho}:\Lambda\to \textup{PSL}_2(\CC)\)
be the composition of \(\rho\) with the canonical projection
\(\textup{GL}_2(\CC)\to \textup{PSL}_2(\CC)\).
Because \(\textup{PSU}(2)\) is a maximal compact subgroup of
\(\textup{PSL}_2(\CC)\), the finite subgroup
\(\,\textup{Im}(\bar{\rho})\,\)
of \(\textup{PSL}_2(\CC)\) is a conjugate of a subgroup
of \(\textup{PSU}(2)\),
i.e.\ there exists an element
\(S_1\in\textup{GL}_2(\CC)\) such that 
\(S_1\!\cdot\! \textup{Im}(\bar\rho)\!\cdot\!
S_1^{-1}\subset \textup{PSU}(2)\).
By Corollary \ref{cor:comm_homo}, there exists an element
\(S_2\in\textup{PSU}(2)\) such that 
\(
S_2\!\cdot \!S_1\!\cdot\! \bar{\rho}(\omega_1)
\!\cdot\! S_1^{-1}\!\cdot\! S_2^{-1}
\) and
\(
S_2\!\cdot\! S_1\!\cdot\! \bar{\rho}(\omega_2)
\!\cdot\! S_1^{-1}\!\cdot\! S_2^{-1}
\) are the image in \(\textup{PSU}(2)\) of
\scalebox{0.8}{\(\begin{pmatrix}\sqrt{-1}&0\\0&-\sqrt{-1}
\end{pmatrix}
\)}
and \scalebox{0.8}{\(\begin{pmatrix}0&\sqrt{-1}\\ \sqrt{-1}&0
\end{pmatrix}
\)} respectively.
Write \(\,S_2\cdot S_1=\begin{pmatrix}a&b\\c&d\end{pmatrix}\).
Then \(f_1:=\displaystyle{\frac{aw_1+ cw_2}{bw_1+dw_2}}\) is a developing
map of a solution of \(\,\triangle u+ e^u= 4\pi(2n+1)\delta_0\),
by Lemma \ref{lemma:n-s-cond-develop},
and it is normalized of type I by construction. We have proved the
``if'' part of Theorem \ref{thm-K4}.
The uniqueness assertion in the last sentence of Theorem \ref{thm-K4}
is clear from the correspondence we have established,
between solutions of 
the mean field equation \(\,\triangle u+ e^u= 4\pi(2n+1)\delta_0\)
and Lam\'e equations
\(L_{n + 1/2, B}\, w = 0\) such that no solution has logarithmic
singularity.
\end{proof}

\begin{corollaryss}\label{cor:odd-lame-nber-soln}
On any flat torus \(\CC/\Lambda\), the mean field equation 
\(\,\triangle u+ e^u = 4\pi (2n+1)\,\delta_0\)
at most $n + 1$ solutions. 
It has exactly $n + 1$ solutions 
except for a finite number of conformal isomorphism classes 
of flat tori.
\end{corollaryss}

\begin{proof}
This is an immediate consequence of theorems \ref{BH-poly}, 
\ref{K4-thm} and \ref{thm-K4}.
\end{proof}

\begin{corollaryss}
For $\eta = n + \tfrac{1}{2}$, the monodromy group $M$ 
of $L_{n + 1/2, B}\, w = 0$ on an elliptic curve \(\CC/\Lambda\)
is finite if and only if it corresponds to 
a type I solution of the mean field equation 
\(\,\triangle u+ e^u = 4\pi (2n+1)\delta_0\,\) on \(\CC/\Lambda\)
as in Theorem \textup{\ref{K4-thm}}.
\end{corollaryss}

\begin{proof}
It was shown in \cite[Thm.\,2.3]{BW} that 
the monodromy group of the Lam\'e equation $L_{n + (1/2), B}\, w = 0$
is finite if and only if no solution of
$L_{n + (1/2), B}\, w = 0$ has logarithmic singularity,
and if so 
the projective monodromy group $L_{n + (1/2), B}\, w = 0$ is 
isomorphic to \((\ZZ/2\ZZ)^2\).
\end{proof}

\begin{example} \label{ex-pn}
By (\ref{BH}), it is easy to determine $p_n(B)$. For example, 
\begin{equation*}
\begin{split}
p_1(B) &= B^2 - \tfrac{3}{4} g_2, \\
p_2(B) &= B^3 - 7g_2 B + 20 g_3.
\end{split}
\end{equation*}

For $\rho = 12\pi$, the two solutions to the mean field equation collapse 
to the same one precisely when $p_1(B)$ has multiple roots. 
This is the case if and only if $g_2 = 0$, which means that $\tau = e^{\pi i/3}$.

To see this from Example \ref{12pi} is a little bit trickier. 
We may solve 
\[(e_3 - e_1)^2 +
16(e_1 - e_2)(e_3 - e_2) = 0
\] in terms of the modular function
\begin{equation*}
\lambda(\tau') = \frac{e_3 - e_2}{e_1 - e_2}
\end{equation*}
where $\tau' = \omega_2'/\omega_1' =
2\omega_2/\omega_1 = 2\tau$. A simple calculation leads to
\begin{equation*}
\frac{(\lambda - 1)^2}{\lambda} = -16, \quad \mbox{i.e.} \quad
\lambda^2 + 14 \lambda + 1 = 0.
\end{equation*}
Then the corresponding $j$ invariant is
\begin{equation*}
j(\tau') := 2^8 \frac{(\lambda^2 - \lambda + 1)^3}{\lambda^2
(\lambda - 1)^2} = -2^8 15^3 \frac{\lambda}{(\lambda - 1)^2} = 2^4
3^3 5^3.
\end{equation*}


In general it would be difficult to determine $\tau'$ from $j$. 
Fortunately the value $j = 2^4 3^3
5^3$ appears in the famous list of elliptic curves with complex
multiplications (see e.g.~\cite{Husemoller})
and it is known that
\begin{equation*}
\tau' \equiv \sqrt{-3} \pmod{{\rm SL}_2(\mathbb{Z})}.
\end{equation*}
Take $\tau$ in the fundamental region, then there is a unique choice of $\tau$, 
namely $\tau =
\frac{1}{2}(1 + \sqrt{-3}) = e^{\pi \sqrt{-1}/3}$, which gives rise to
\[2\tau = 1 + \sqrt{-3} \equiv \sqrt{-3} = \tau'
\pmod{\textup{SL}_2(\ZZ)}.\]
\end{example}

\section{Singular Liouville equations with 
$\rho=4\pi$ and modular forms} \label{mod-form}
\setcounter{equation}{0}

\noindent
In \S \ref{I-int}, we discussed how to find all type I solutions
by solving a system of polynomial equations which depends holomorphically on
the moduli parameter \(\tau\) of the torus
\(E_{\tau}=\CC/\Lambda_{\tau}=\CC/(\ZZ\!+\!\ZZ\tau)\),
where \(\tau\) varies in the upper-half plane \(\mathbb H\).
In this section, we consider the simplest case 
\begin{equation} \label{4pi}
\triangle u + e^u = 4\pi \delta_0 \quad \mbox{in $E_\tau$},
\end{equation}
and show that certain modular forms of level \(4\) are naturally 
to the solutions of (\ref{4pi}) as \(\tau\) varies.
The general case with multiple singular sources will be considered 
in a subsequent work.

\subsection{\bf Notation.}\enspace \label{notation-K4}
\begin{itemize}
\item Let \(\HH\) be the upper-half plane.  The group \(\textup{SL}_2(\RR)\)
  operates transitively on \(\HH\) through the usual formula
 \(\displaystyle{\begin{pmatrix}a&b\\c&d\end{pmatrix}\cdot \tau 
= \frac{a\tau+ b}{c\tau+d}}\).

\item Let \(\,j(\gamma,\tau)\,\) be the \(1\)-cocycle of
  \(\textup{SL}_2(\RR)\) for its action on \(\HH\), defined by
\(\,\displaystyle{j(\gamma;\tau)=c\tau+d
}\) for any \(\,\gamma=\begin{pmatrix}a&b\\c&d\end{pmatrix}\in
\textup{SL}_2(\RR)\,\) and any \(\tau\in \HH\).

\item Denote by \(K_4\) the subgroup of 
\(\textup{PSU}(2)\subset \textup{PSL}_2(\CC)\) 
isomorphic to \(\,(\ZZ/2\ZZ)^2\),
consisting of the image in \(\textup{PSU}(2)\) of the four
matrices
\[\begin{pmatrix}1&0\\0&1\end{pmatrix},
\begin{pmatrix}\sqrt{-1}&0\\0&-\sqrt{-1}\end{pmatrix},
\begin{pmatrix}0&\sqrt{-1}\\\sqrt{-1}&0\end{pmatrix}
\ \textup{and}\ 
\begin{pmatrix}0&1\\-1&0\end{pmatrix}.
\]
We know from Lemma \ref{lemma:mono_grpthy}\,(1b) that the
centralizer subgroup of \(K_4\) in \(\textup{PSL}_2(\CC)\)
is equal to itself.

\item Let \(\textup{N}(K_4)\) be the normalizer subgroup of \(K_4\)
in \(\textup{PSU}(2)\), which is also equal to the
normalizer subgroup of \(K_4\) in \(\textup{PSL}_2(\CC)\).
We know from Lemma \ref{lemma:mono_grpthy}\,(1d) that
\(\textup{N}(K_4)\) is a semi-direct product of \(K_4\) with \(S_3\)
and \(K_4\) is isomorphic to \(S_4\).
Moreover the conjugation action induces an isomorphism
from \(\textup{N}(K_4)/K_4\) to the permutation group of the three
non-trivial elements of \(K_4\).

\end{itemize}

\begin{proposition}\label{prop:extra-normalized}
 \ \textup{(a)}\, For any \(\tau\in \HH\), there exists a 
unique normalized developing map 
\(\,f(z;\tau)\,\) for the unique solution 
\(u(z)\) of the equation \textup{(\ref{4pi})} which 
has the following properties.
\begin{equation}\label{I:zero}
\textup{ord}_{z=a}\,f(z,\tau)=0\quad \forall a\not\equiv
\tfrac{1}{2}\pmod{\Lambda}
\end{equation}
\begin{equation}\label{I:mult_0}
\tfrac{d}{dz}f(z;\tau)\big\vert_{z=0}=0,\quad
\tfrac{d^2}{dz^2}f(z;\tau)\big\vert_{z=0}\in \CC^{\times}.
\end{equation}
\begin{equation}\label{I:simple_else}
{The\ holomorphic\ map\ }\,
f(z;\tau):\CC\to \PP^1(\CC)\ \
{is\  etale\ outside\ } \Lambda_{\tau}.
\end{equation}
\begin{equation} \label{I:tau}
f(z + 1; \tau) = -f(z), \quad f(z + \tau; \tau) = 1/f(z; \tau)
\quad\forall\, z\in\CC.
\end{equation}
\begin{equation} \label{I:even}
f(-z; \tau) = f(z; \tau) \quad 
\forall\, z\in\CC.
\end{equation}
\begin{equation} \label{I:-1}
f(\tfrac{1}{2}\tau; \tau) = 1 \quad
\forall\, z\in\CC.
\end{equation}
\begin{equation}\label{I:order-1half}
\textup{ord}_{z=1/2}\,f(z;\tau)=1\quad 
\end{equation}
\begin{equation}\label{I:pole-tau}
\textup{ord}_{z=(1/2)+\tau}\,f(z;\tau)=-1\quad
\end{equation}
%
\medbreak

\noindent
\textup{(b)}\, The function \(f(z;\tau)\) in 
\textup{(a)} is characterized by properties 
\textup{(\ref{I:zero})},
\textup{(\ref{I:tau})}, \textup{(\ref{I:-1})}  
and \textup{(\ref{I:order-1half})}, i.e.\
if \(h(z)\) is a meromorphic function on
\(\CC\) which satisfies \textup{(\ref{I:zero})},
\textup{(\ref{I:tau})}, \textup{(\ref{I:-1})}  
and \textup{(\ref{I:order-1half})}
for an element \(\tau\in\HH\), then
\(\,h(z)=f(z;\tau)\) for all \(z\in\CC\).
\medbreak

\noindent
\textup{(c)} The function \(f(z;\tau)\) can be expressed
in terms of Weierstrass elliptic functions:
\begin{equation}\label{explicit-formula-1}
\begin{split}
f(z;\tau)&=-e^{\scalebox{0.7}{\(\tfrac{1}{4}\)}\eta(\tau;\,\Lambda_{\tau})\cdot(1+\tau)}
\cdot\frac{\sigma(\tfrac{z}{2}\!-\!\tfrac{1}{4};\Lambda_{\tau})
\cdot \sigma(\tfrac{z}{2}\!+\!\tfrac{1}{4};\,\Lambda_{\tau})}{
\sigma(\tfrac{z}{2}\!-\!\tfrac{1}{4}\!-\!\tfrac{\tau}{2};\Lambda_{\tau})
\cdot \sigma(\tfrac{z}{2}\!+\!\tfrac{1}{4}\!+\!\tfrac{\tau}{2};\Lambda_{\tau})}
\\
&= -
e^{\scalebox{0.7}{\(\tfrac{1}{4}\)}\eta(\tau;\,\Lambda_{\tau})\cdot(1+\tau)}
\!\cdot\!
\frac{\sigma^2(\tfrac{1}{4};\Lambda_{\tau})}{
\sigma^2(\tfrac{1}{4}\!+\!\tfrac{\tau}{2};\Lambda_{\tau})}
\!\cdot\! \frac{\wp(\tfrac{z}{2};\Lambda_{\tau})
-\wp(\tfrac{1}{4};\Lambda_{\tau})}{\wp(\tfrac{z}{2};\Lambda_{\tau})
-\wp(\tfrac{1}{4}\!+\!\tfrac{\tau}{2};\Lambda_{\tau})}
\\
&= \frac{\wp(\scalebox{0.9}{\(\tfrac{\tau}{4}\)};\,\Lambda_{\tau}) -
\wp(\scalebox{0.9}{\(\tfrac{1}{4}+\tfrac{\tau}{2}\)};\,\Lambda_{\tau})}{
\wp(\scalebox{0.9}{\(\tfrac{\tau}{4}\)};\,\Lambda_{\tau}) -
\wp(\scalebox{0.9}{\(\tfrac{1}{4}\)};\,\Lambda_{\tau})}
\cdot
\frac{\wp(\tfrac{z}{2};\Lambda_{\tau})-
\wp(\tfrac{1}{4};\Lambda_{\tau})}{
\wp(\tfrac{z}{2};\Lambda_{\tau})-
\wp(\tfrac{1}{4}+\tfrac{\tau}{2};\Lambda_{\tau})
}
\end{split}
\end{equation}

\end{proposition}

\begin{proof}
{\bf (a)} For any $\tau \in \mathbb H$, we have proved that in \S\ref{I-int}
that equation (\ref{4pi}) 
has a unique solution $u(z; \tau)$, $z \in E_\tau$ and there
exists a normalized type I developing map \(f(z;\tau)\) for \(u(z;\tau)\).
Because the centralizer subgroup of \(K_4\) in \(\textup{PSU}(2)\)
is \(K_4\) itself, normalized type I developing maps consists are
of the form \(\gamma\cdot f\) with \(\gamma\in K_4\).
Properties \textup{(\ref{I:zero})}--(\ref{I:even})
are satisfied by all 4 normalized developing maps.
The first part of (\ref{I:tau}) and (\ref{I:-1}) implies that
\(f(z;\tau)\) has either a zero or a pole at \(z=\frac{1}{2}\).
Changing \(f_1\) to 
\(\,\scalebox{0.7}{\(\begin{pmatrix}0&\sqrt{-1}\\
    \sqrt{-1}&0\end{pmatrix}\)}\!\cdot\! f\,\) 
if necessary, we may assume that \(f_1(\tfrac{1}{2};\tau)=0\).
Then \(f(z;\tau)\) has a simple zero at \(z=\tfrac{1}{2}\) by
(\ref{I:zero}), and properties
(\ref{I:order-1half})--(\ref{I:pole-tau})
hold for \(f\).
Similarly properties (\ref{I:zero}), (\ref{I:tau}) and (\ref{I:even} 
for \(f\) imply that
\(f(\tfrac{\tau}{2};\tau)=\pm 1\).
Changing \(f\) to \(\,\scalebox{0.7}{\(\begin{pmatrix}\sqrt{-1}&0\\
    0&-\sqrt{-1}\end{pmatrix}\)}\!\cdot\! f\,\)
if necessary, we have produced a normalized 
developing map satisfying (\ref{I:zero})--(\ref{I:pole-tau}).
\medbreak

\noindent
{\bf (b)} Suppose that \(h(z)\) is a meromorphic function which satisfies
properties \textup{(\ref{I:zero})},
\textup{(\ref{I:tau})}, \textup{(\ref{I:-1})}  
and \textup{(\ref{I:order-1half})}.
Then \(h(z)\) descends to a meromorphic function on
\(\CC/2\Lambda_{\tau}\) which has simple zeros at
\(\pm\tfrac{1}{2}\, \textup{mod}\,2\Lambda_{\tau}\),
simples poles at
\(\pm\tfrac{1}{2}+\tau\,\textup{mod}\,2\Lambda_{\tau}\)
and no zeros or poles elsewhere
just like \(f(z,\tau)\).
Therefore \(h(z)= c \cdot f(z;\tau)\) for some \(c\in\CC^{\times}\).
This constant \(c\) is equal to \(1\) by (\ref{I:order-1half}).
\medbreak

\noindent
{\bf (c)} For the first equality 
in (\ref{explicit-formula-1}),
it suffices to show that the function
\[
e^{\scalebox{0.7}{\(\tfrac{1}{4}\)}\eta(\tau;\,\Lambda_{\tau})\cdot(1+\tau)}
\!\cdot\!
\frac{\sigma(\tfrac{z}{2}\!-\!\tfrac{1}{4};\Lambda_{\tau})
\cdot \sigma(\tfrac{z}{2}\!+\!\tfrac{1}{4};\,\Lambda_{\tau})}{
\sigma(\tfrac{z}{2}\!-\!\tfrac{1}{4}\!-\!\tfrac{\tau}{2};\Lambda_{\tau})
\cdot \sigma(\tfrac{z}{2}\!+\!\tfrac{1}{4}\!+\!\tfrac{\tau}{2};\Lambda_{\tau})}
\]
satisfies conditions \textup{(\ref{I:zero})},
\textup{(\ref{I:tau})}, \textup{(\ref{I:-1})}  
and \textup{(\ref{I:order-1half})}
according to (b).
The properties \textup{(\ref{I:zero})},
\textup{(\ref{I:tau})} and \textup{(\ref{I:order-1half})} 
follows quickly from the tranformation law for the
Weierstrass \(\sigma\)-function
\(\sigma(z;\Lambda_{\tau})\) and the
fact that the entire function
\(\sigma(z;\Lambda_{\tau})\) has simple zeros at points of 
\(\Lambda_{\tau}\) does not vanish elsewhere. 
The condition (\ref{I:-1}) is equivalent to 
\[\frac{\sigma(\tfrac{\tau}{4}\!-\!\tfrac{1}{4};\Lambda_{\tau})
}{
\sigma(-\tfrac{3\tau}{4}\!-\!\tfrac{1}{4};\Lambda_{\tau})}
= - e^{-\scalebox{0.7}{\(\tfrac{1}{4}\)}\eta(\tau;\,\Lambda_{\tau})\cdot(1+\tau)}
,\]
which follows from the transformation law of 
\(\sigma(z;\Lambda_{\tau})\) with respect to the element
\(\tau\in\Lambda_{\tau}\).
We have proved that the first equality 
\[
f(z;\,\tau)=
e^{\scalebox{0.7}{\(\tfrac{1}{4}\)}\eta(\tau;\,\Lambda_{\tau})\cdot(1+\tau)}
\!\cdot\!
\frac{\sigma(\tfrac{z}{2}\!-\!\tfrac{1}{4};\Lambda_{\tau})
\cdot \sigma(\tfrac{z}{2}\!+\!\tfrac{1}{4};\,\Lambda_{\tau})}{
\sigma(\tfrac{z}{2}\!-\!\tfrac{1}{4}\!-\!\tfrac{\tau}{2};\Lambda_{\tau})
\cdot \sigma(\tfrac{z}{2}\!+\!\tfrac{1}{4}\!+\!\tfrac{\tau}{2};\Lambda_{\tau})}
\]
in (\ref{explicit-formula-1}).  The second equality in
(\ref{explicit-formula-1}) follows from the classical formula
\begin{equation}\label{wp-sigma-formula}
\wp(u;\Lambda)-\wp(v;\Lambda)
=-\frac{\sigma(u+v;\Lambda)\cdot \sigma(u-v;\Lambda)}{
\sigma^2(u;\Lambda)\cdot\sigma^2(v;\Lambda)}\,.
\end{equation}
The last equality in (\ref{explicit-formula-1}) is equivalent to
\begin{equation}
\frac{\wp(\scalebox{0.9}{\(\tfrac{\tau}{4}\)};\,\Lambda_{\tau}) -
\wp(\scalebox{0.9}{\(\tfrac{1}{4}+\tfrac{\tau}{2}\)};\,\Lambda_{\tau})}{
\wp(\scalebox{0.9}{\(\tfrac{\tau}{4}\)};\,\Lambda_{\tau}) -
\wp(\scalebox{0.9}{\(\tfrac{1}{4}\)};\,\Lambda_{\tau})}
= -
e^{\scalebox{0.7}{\(\tfrac{1}{4}\)}\eta(\tau;\,\Lambda_{\tau})\cdot(1+\tau)}
\!\cdot\!
\frac{\sigma^2(\tfrac{1}{4};\Lambda_{\tau})}{
\sigma^2(\tfrac{1}{4}\!+\!\tfrac{\tau}{2};\Lambda_{\tau})}
\,,
\end{equation}
which is easily verified using (\ref{explicit-formula-1})
and the transformation law of the Weierstrass \(\sigma\)-function
\(\sigma(z;\Lambda_{\tau})\) with respect to the lattice
\(\Lambda_{\tau}\).
We have proved part (c) of Proposition \ref{prop:extra-normalized}.
\end{proof}

\begin{proposition}\label{prop:transf_4pi}
Let \(\,f(z;\tau)\) be the developing map specified
in Proposition \textup{\ref{prop:extra-normalized}}.
\smallbreak

\noindent \textup{(a)} There exists a unique group homomorphism
\(\,\psi:\textup{SL}_2(\ZZ)\to \textup{N}(K_4)\,\) 
such that
\begin{equation}\label{transf-law-N}
f\big(j(\gamma,\tau)^{-1}\!\cdot\! z;\,\gamma\cdot \tau\big)
= \psi(\gamma)\!\cdot\! f(z;\tau)\qquad
\forall\,z\in\CC,\ \forall\,\tau\in\HH.
\end{equation}
Here \(\,\psi(\gamma) \!\cdot\! f(z;\tau)
=\frac{a f(z;\tau) +b}{-\bar{b}f(z;\tau)+\bar{a}}\)
if \(\,\psi(\gamma)\,\) is 
the image of
\scalebox{0.8}{\(\begin{pmatrix}a&b\\-\bar{a}&\bar{b}\end{pmatrix}\)}
in \(\,\textup{PSU}(2)\).
\medbreak

\noindent
\textup{(b)} The homomorphism \(\psi\) is \emph{surjective}.
The kernel \(\,\textup{Ker}(\psi)\,\) of 
\(\psi\) is equal to the subgroup of
\(\textup{SL}_2(\ZZ)\) generated by \(\pm \textup{I}_2\)
and the principal congruence subgroup \(\,\Gamma(4)\,\)
of level \(4\), consisting of all \(\,\gamma\in\textup{SL}_2(\ZZ)\)
with \(\,\gamma\equiv \textup{I}_2\pmod{4}\).
The inverse image \(\,\psi^{-1}(K_4)\,\) of \(K_4\) under \(\psi\)
is the principal congruence subgroup \(\,\Gamma(2)\).
(In other words \(\,\psi\,\) induces an isomorphism
\(\,\textup{SL}_2(\ZZ/4\ZZ)/\{\pm \textup{I}_2\}
\xrightarrow{\sim} \textup{N}(K_4)\), and
also an isomorphism 
\(\,\textup{SL}_2(\ZZ/2\ZZ)
\xrightarrow{\sim}\textup{N}(K_4)/K_4\cong S_3\).)
\medbreak

\noindent
\textup{(c)} \(\,\displaystyle{\psi\big(\scalebox{0.8}{$\begin{pmatrix}1&1
\\0&1\end{pmatrix}$}\big)
= the\ image\ in\ \textup{PSU}(2)\ 
of\ the\ unitary\ matrix \scalebox{0.7}{$\begin{pmatrix}e^{\pi\sqrt{-1}/4}&0\\
0&e^{-\pi\sqrt{-1}/4}
\end{pmatrix}$}}
\), and
\(\,\displaystyle{\psi\big(
\scalebox{0.8}{$\begin{pmatrix}0&1
\\-1&0\end{pmatrix}$}\big)
= the\ image\ in\ \textup{PSU}(2)\ of\ \tfrac{\sqrt{-1}}{\sqrt{2}}
\scalebox{0.9}{$\begin{pmatrix}-1&1\\
1&1
\end{pmatrix}$}}\).

\end{proposition}

\begin{proof}
{\bf (a)} It is easily checked that for each
\(\gamma\in\textup{SL}_2(\ZZ)\), 
\(f\big(j(\gamma,\tau)^{-1}\!\cdot\! z;\,\gamma\cdot \tau\big)\)
is a developing map of the unique solution
of (\ref{4pi}), and that for each \(\omega\in\Lambda_{\tau}\) 
we have
\(f\big(j(\gamma,\tau)^{-1}\!\cdot\! z+\omega;\,\gamma\cdot \tau\big)
=\pm f(z;\tau)^{\pm 1}\). 
Since \(f\big(j(\gamma,\tau)^{-1}\!\cdot\! z+\omega;
\,\gamma\cdot \tau\big)\) and \(f(z;\tau)\) are developing
maps for the same solution of (\ref{4pi}), there exists an unique
element \(\psi(\gamma)\in \textup{PSU}(2)\) such that
the equality (\ref{transf-law-N}) holds. 
The fact that \(f\big(j(\gamma,\tau)^{-1}\!\cdot\! z+\omega;\,\gamma\cdot \tau\big)
=\pm f(z;\tau)^{\pm 1}\) for each \(\omega\in \Lambda_{\tau}\) 
means that \(\psi(\gamma)\in \textup{N}(K_4)\).  
\smallbreak

For all \(\gamma_1, \gamma_2\in \textup{SL}_2(\ZZ)\), we have
\begin{equation*}
\begin{split}
f(j(\gamma_1\gamma_2, \tau)^{-1}z;\,\gamma_1\gamma_2 \!\cdot\!\tau)
&= f(j(\gamma_1, \gamma_2\!\cdot\! \tau)^{-1}\cdot
j(\gamma_2, \tau)^{-1}\!\cdot\! z;\,
\gamma_1\!\cdot\! (\gamma_2\!\cdot\! \tau))\\
&=\psi(\gamma_1)\!\cdot\!
f(j(\gamma_2, \tau)^{-1}\!\cdot\! z;\, \tau)
\\
&=\psi(\gamma_1)\!\cdot\!\psi(\gamma_2)\cdot
f(z;\,\tau),
\end{split}
\end{equation*}
therefore
\(\psi(\gamma_1\gamma_2)=\psi(\gamma_1)\!\cdot\psi(\gamma_2)\).
We have proved statement (a).
\medbreak

\noindent
{\bf (b)} We get from (a) that for any 
\(\gamma = \scalebox{0.8}{$\begin{pmatrix}a&b\\c&d\end{pmatrix}$}\)
in \(\textup{SL}_2(\ZZ)\) we have
\begin{equation}\label{transf-torsion}
f(j(\gamma,\tau)^{-1}z \!+\!u\,\gamma\!\cdot\! \tau
\!+\!v;\,\gamma\!\cdot\!\tau)
= \psi(\gamma)\!\cdot\!f(z\!+\!(ua\!+\!vc)\tau \!+\! 
(ub\!+\!vd);\,\tau)
\end{equation}
for all \(u,v\in\QQ\).
For any given \(\gamma\in \Gamma(2)\),  we have
\[\,(ua\!+\!vc)\tau \!+\! 
(ub\!+\!vd)\equiv u\tau+v \ (\textup{mod}\, 2\Lambda_{\tau})\quad 
\forall\, (u,v)\in\ZZ^2,
\]
so the equality (\ref{transf-torsion}) for all \((u,v)\in\ZZ^2\)
implies that
\(\,\psi(\gamma)\,\) commutes with every element of
\(K_4\).  Hence \(\psi(\gamma)\in K_4\) for any \(\gamma\in
\Gamma(2)\).
\medbreak

Suppose that \(\gamma\in \Gamma(4)\).  Then 
\[\,(ua\!+\!vc)\tau \!+\! 
(ub\!+\!vd)\equiv u\tau+v \ (\textup{mod}\,2\Lambda_{\tau})\quad
\forall\, (u,v)\in\tfrac{1}{2}\ZZ^2,
\] and
the equality (\ref{transf-torsion}) with \(z=0\) implies that
\[
(\psi(\gamma)\!\cdot\!f)(u\tau+v;\,\tau)
= f(u\gamma\!\cdot\!\tau+v;\,\gamma\!\cdot\!\tau)
\quad\forall\, (u,v)\in\tfrac{1}{2}\ZZ^2.
\]
Because we already know that \(\psi(\gamma)\in K_4\), the last
equality implies that \(\psi(\gamma)=\textup{I}_2\).
We have proved that \(\Gamma(4)\subset \textup{Ker}(\psi)\).
\medbreak

Suppose that \(\gamma=\scalebox{0.8}{$\begin{pmatrix}a&b\\c&d
  \end{pmatrix}$}\in \textup{Ker}(\psi)\).
As before we have
\begin{equation}\label{ker-psi}
f(j(\gamma,\tau)^{-1}z \!+\!u\,\gamma\!\cdot\! \tau
\!+\!v;\,\gamma\!\cdot\!\tau)
= f(z\!+\!(ua\!+\!vc)\tau \!+\! 
(ub\!+\!vd);\,\tau)
\end{equation}
for all \((u,v)\in \QQ^2\).
The transformation law (\ref{I:tau}) for \(f(z;\tau)\)
and the above equality for \((u,v)\in\ZZ^2\) imply that
\[
(ua+vc)\tau+(ub+vd)\equiv u\tau+v\ (\textup{mod}\,2\Lambda_{\tau})
\quad \forall\,(u,v)\in \ZZ^2,
\]
therefore \(\gamma\in \Gamma(2)\).
The equation (\ref{ker-psi}) with \(z=0\) and \(u,v\in
\tfrac{1}{2}\ZZ\) tells us that
\[
f(u\,\gamma\!\cdot\! \tau
\!+\!v;\,\gamma\!\cdot\!\tau)
= f((ua\!+\!vc)\tau \!+\! 
(ub\!+\!vd);\,\tau)\qquad
\forall\, (u,v)\in \tfrac{1}{2}\ZZ^2.
\]
The properties (\ref{I:tau}), {(\ref{I:-1})}  
and (\ref{I:order-1half}) imply that
\[
(ua+vc)\tau+(ub+vd)\equiv \pm(u\tau+v)\ (\textup{mod}\,2\Lambda_{\tau})
\quad \forall\,(u,v)\in \tfrac{1}{2}\ZZ^2,
\]
therefore \(\,\gamma\in \{\pm\textup{I}_2\}\!\cdot\!\Gamma(4)\).
We have proved the statement (b).
\medbreak

\noindent
{\bf (c)} We want to compute \(\,T:=\psi\Big(
\scalebox{0.8}{$\begin{matrix}1&1\\0&1\end{matrix}$}\Big)\)
and 
\(S:=\psi\Big(
\scalebox{0.8}{$\begin{matrix}0&1\\-1&0\end{matrix}$}\Big)\).
The defining relation for \(T\) is
\begin{equation}\label{rel-S}
f(z;\,\tau+1)=T\cdot f(z;\tau) \quad \forall\,z\in\CC,\ \forall\,\tau\in\HH.
\end{equation}
Substituting \(z\) by \(z+\omega\) in (\ref{rel-S})
with \(\omega\in \Lambda_{\tau}\) gives us two equalities
\[
T\cdot \scalebox{0.7}{$\begin{pmatrix}\sqrt{-1}&0\\ 0&-\sqrt{-1}\end{pmatrix}$}
= \scalebox{0.7}{$\begin{pmatrix}\sqrt{-1}&0\\ 0&-\sqrt{-1}\end{pmatrix}$}
\cdot T
\ \ \ \textup{and}\ \ \
\scalebox{0.7}{$\begin{pmatrix}0&\sqrt{-1}\\ \sqrt{-1}&0\end{pmatrix}$}
\cdot T
= T\cdot \scalebox{0.7}{$\begin{pmatrix}0&-1\\ 1&0\end{pmatrix}$}
\]
in \(\textup{PSU}(2)\).
A easy computation with the above equalities reveals that
\(\,\gamma\in 
\scalebox{0.7}{$\begin{pmatrix}e^{\pi\sqrt{-1}/4}&0\\ 
0&e^{-\pi\sqrt{-1}/4}
\end{pmatrix}$}\cdot K_4\), i.e.
\(\displaystyle
f(z;\,\tau+1)=\pm\sqrt{-1}\cdot f(z;\,\tau)^{\pm 1}\).
Since \(f(\tfrac{1}{2};\,\tau)=f(\tfrac{1}{2};\,\tau+1)=0\), 
the possibilities narrow down to
\(\displaystyle
f(z;\,\tau+1)=\pm\sqrt{-1}\cdot f(z;\,\tau)\).
It remains to determine the sign, which amounts to
computing \(\,f(\tfrac{\tau+1}{2};\,\tau)\) 
From the first equality in (\ref{explicit-formula-1}) 
we get
\begin{equation*}
\begin{split}
f(\tfrac{\tau+1}{2};\tau)
&=-e^{\scalebox{0.7}{\(\tfrac{1}{4}\)}\eta(\tau;\,\Lambda_{\tau})\cdot(1+\tau)}
\cdot\frac{\sigma(\tfrac{\tau}{4};\Lambda_{\tau})
\cdot \sigma(\tfrac{\tau+2}{4};\,\Lambda_{\tau})}{
\sigma(-\tfrac{\tau}{4};\Lambda_{\tau})
\cdot
\sigma(\tfrac{3\tau+2}{4};\Lambda_{\tau})}
\\
&=e^{\scalebox{0.7}{$\tfrac{1}{4}$}[\eta(\tau;\,\Lambda_{\tau})-
  \eta(1;\,\Lambda_{\tau})]}
\\
&= -\sqrt{-1}
\end{split}
\end{equation*}
The second equality in the displayed equation above follows from the 
transformation law of the Weierstrass \(\sigma\)-function,
while the last equality follows from the Legendre relation
\(\,\eta(1;,\Lambda_{\tau})\tau - \eta(\tau;\,\Lambda_{\tau})
= 2\pi\sqrt{-1}\).
We conclude that 
\[\,f(z;\,\tau +1)=\sqrt{-1}\!\cdot\! f(z;\,\tau),\]
which gives the formula for 
\(T=\psi\big(\scalebox{0.7}{$\begin{pmatrix}
1&1\\0&1\end{pmatrix}$}\big)\).
\bigbreak

Finally let's compute \(S\). The defining
relation for \(S\) is
\begin{equation*}
f\big(-\tfrac{1}{\tau}\,z;\, -\tfrac{1}{\tau})
= S\cdot f(z;\,\tau)
\end{equation*}
The functional equation for \(\,z \to z+\omega\,\) 
with \(\,\omega\in\Lambda_{\tau}\,\) gives us two equalities
\begin{equation*}
S\cdot \scalebox{0.8}{$\begin{pmatrix}
1&0\\0&-1\end{pmatrix}$}
= \scalebox{0.8}{$\begin{pmatrix}
0&1\\1&0\end{pmatrix}$}\cdot S
\quad \textup{and}\quad
\scalebox{0.8}{$\begin{pmatrix}
1&0\\0&-1\end{pmatrix}$}\cdot S
=S\cdot \scalebox{0.8}{$\begin{pmatrix}
0&1\\1&0\end{pmatrix}$}
\end{equation*}
in \(\textup{PSU}(2)\).
A straight-forward computation gives four possible solutions
of \(S\), which translates into
\begin{equation*}
f(-\tfrac{1}{\tau}\,z;\, -\tfrac{1}{\tau})
= \pm \frac{f(z;\,\tau)+1}{f(z;\,\tau)-1}
\ \ \ \textup{or}\ \
\pm \frac{f(z;\,\tau)-1}{f(z;\,\tau)+1}
\end{equation*}
The requirement that \(\,f(-\tfrac{1}{2\tau};\,-\tfrac{1}{\tau})=1\)
eliminates two possibilities: the two \(\pm\) signs above
are both \(-1\).
The requirement that \(\,f(\tfrac{1}{2};\,-\tfrac{1}{\tau})=0\,\)
shows that
\[\,f\big(-\tfrac{1}{\tau}\,z;\, -\tfrac{1}{\tau})
=-\frac{f(z;\,\tau)-1}{f(z;\,\tau)+1}.
\]
We have proved the statement (c).
\end{proof}

Recall that the quotient \(\Gamma(4)\backslash\HH\) has a natural structure
as (the \(\CC\)-points of) a smooth affine algebraic curve \(Y(4)\).
The compactified modular curve \(X(4)\)
is the smooth compactification \(X(4)\) of \(Y(4)\).
As a topological space \(X(4)\) is naturally identified as
the quotient by \(\Gamma(4)\) of \(\HH\sqcup \PP^1(\QQ)\); the topology of 
the latter is described in \cite[p.\,10]{Shimura-redbook}.
The complement \(\,X(4)\smallsetminus Y(4)\,\), called the cusps
of \(X(4)\), is a set with \(6\) elements naturally identified with 
\(\Gamma(4)\backslash \PP^1(\QQ)\), or equivalently
the set \(\PP^1(\ZZ/4\ZZ)\) of \(\ZZ/4\ZZ\)-valued points
of the scheme \(\PP^1\) over \(\ZZ\).
It is well-known that
\(X(4)\) has genus zero; c.f.\ \cite[(1.6.4),\,p.\,23]{Shimura-redbook}.

The general discussion in \S\ref{I-int}, of which the present
situation is the special case \(\rho=4\pi\), implies that
the function \(\,\tau\mapsto f(0;\tau)\,\) is 
holomorphic on \(\HH\).
Proposition \ref{prop:transf_4pi} implies that
the holomorphic function \(\,\tau\mapsto f(0;\tau)\,\) on
\(\HH\) descends to a holomorphic function \(h_{X(4)}\) on the open
modular curve \(Y(4)\). 
The next corollary says that \(h_{X(4)}\) is a 
\emph{Hauptmodul} for \(X(4)\).

\begin{corollary}\label{cor:hauptm}
\textup{(a)} The holomorphic function \(h_{X(4)}\) on \(Y(4)\)
is a meromorphic function
on \(X(4)\) which defines a biholomorphic isomorphism \(h_{X(4)}^{\ast}\)
from \(X(4)\) to \(\PP^1(\CC)\).
This isomorphism \(h_{X(4)}^{\ast}\) is equivariant with
respect to \(\psi\), for the action of
\(\textup{SL}_2(\ZZ)/\{\pm\,\textup{Id}_2\}\!\cdot\! \Gamma(4)\)
on \(X(4)\) and the action of 
\(\textup{N}(K_4)\) on \(\PP^1(\CC)\). 
\smallbreak

\noindent
\textup{(b)} We have explicit formulas for \(h_{X(4)}(\tau)=f(0,\tau)\):
\begin{equation}\label{explicit-haupm}
h(\tau)
=
 -e^{\scalebox{0.7}{\(\tfrac{1}{4}\)}\eta(\tau;\,\Lambda_{\tau})\cdot(1+\tau)}
\!\cdot\!
\frac{\sigma^2(\tfrac{1}{4};\Lambda_{\tau})}{
\sigma^2(\tfrac{1}{4}\!+\!\tfrac{\tau}{2};\Lambda_{\tau})}
= \frac{\wp(\scalebox{0.9}{\(\tfrac{\tau}{4}\)};\,\Lambda_{\tau}) -
\wp(\scalebox{0.9}{\(\tfrac{1}{4}+\tfrac{\tau}{2}\)};\,\Lambda_{\tau})}{
\wp(\scalebox{0.9}{\(\tfrac{\tau}{4}\)};\,\Lambda_{\tau}) -
\wp(\scalebox{0.9}{\(\tfrac{1}{4}\)};\,\Lambda_{\tau})}
\end{equation}

\noindent
\textup{(c)} The isomorphism \(\,h_{X(4)}^{\ast}\,\) sends (the image
of) the standard cusp ``\(\infty\cdot \sqrt{-1}\)'', that is the 
point \(\infty\in \PP^1(\QQ)\), to the point
\(0\in \PP^1(\CC)\).

\end{corollary}

\begin{proof}
The formula (\ref{explicit-haupm}) in (b) follows immediately from 
the formulas (\ref{explicit-formula-1}) for \(\,f(z;\tau)\).
\smallbreak

There are two ways to see that \(h_{X(4)}\) is a meromorphic function
on \(X(4)\). One can use either of the two formulas in (b) and 
classical results on Weierstrass elliptic functions.
The other way is to use Picard's theorem: we know that 
\(\,f(0,\tau)\neq 0\,\) for all \(\tau\in \HH\).  
The \(\psi\)-equivariance of \(\,f(0,\tau)\,\) implies that
\[\,f(0,\tau)\not\in\{\pm1, \pm\sqrt{-1}, 0,\infty\}\quad
\forall\,\tau\in \HH,
\]
for the set \(\,\{\pm1, \pm\sqrt{-1}, 0,\infty\}\,\) is the orbit
of \(\,\textup{N}(K_4)\,\) on \(\PP^1(\CC)\). 
So \(\,h_{X(4)}\,\) cannot have essential singularities at any of 
the cusps.
\smallbreak

From the meager information in the previous paragraph we can already
conclude that the holomorphic map \(h_{X(4)}\) from \(X(4)\) to
\(\PP^1(\CC)\) has degree \(1\): Because \(X(4)\) has exactly \(6\)
cusps, the map \(h_{X(4)}\) is totally ramified over the six
points of \(\{\pm1, \pm\sqrt{-1}, 0,\infty\}\subset \PP^1\),
and the Hurwitz formula forces the degree of \(h_{X(4)}\) to be \(1\).
\smallbreak

The fact that \(h_{X(4)}\) sends the standard
cusp \(\infty\cdot\sqrt{-1}\) to \(0\) can be seen by an easy
computation, using the formula (\ref{explicit-haupm})
and the \(q\)-expansion of the Weierstrass \(\wp\)-function
\begin{equation*}
\tfrac{1}{(2\pi\sqrt{-1})^2}\,
\wp(z;\Lambda_{\tau})
=\frac{1}{2}+ \frac{q_z}{(1-q_z)^2}
+\sum_{m,n\geq 1} n\,q_{\tau}^{mn}(q_z^n+q_z^{-n})
-\sum_{m,n\geq 1} n\,q_{\tau}^{mn}
\end{equation*}
in the range \(\,\vert q_{\tau}\vert < \vert q_{z}\vert
< \vert q_{\tau}\vert^{-1}\), where
\(\,q_{\tau}=e^{2\pi\sqrt{-1}\tau}\) and 
\(\,q_{z}=e^{2\pi\sqrt{-1}z}\) for \(\tau\in\HH\) and \(z\in \CC\).
\end{proof}

\begin{remarkss} \label{rem:modunit}
(a) The fact that \(\,f(0;\,\tau)\,\) is a Hauptmodul for the principal
congruence subgroup \(\Gamma(4)\) is classical; 
see \cite[p.\,176]{Schoeneberg}.
We have not been able to locate in the literature the transformation formula 
in Proposition \ref{prop:transf_4pi}, but formula is not difficult to
prove starting from the formula
(\ref{explicit-formula-1}) for \(\,f(z;\tau)\). 
Perhaps the only new thing here is the phenomenon that 
type I solutions of the Liouville equation 
\(\,\triangle u+ e^{u}= 4\pi(2n+1)\delta_0\,\) on elliptic curves
produce modular forms for \(\,\Gamma(4)\) in an organized way.
\smallbreak

\noindent
(b) Clearly the function \(\,f(0;\,\tau)\,\) is a \emph{modular unit}
in the sense that it is a unit in the integral closure of 
\(\CC[\,j\,]\) in \(\CC(X(4))\), where \(\,j\,\) is the \(j\)-invariant
and \(\CC(X(4))\) is the function field of the modular curve
\(X(4)\) over \(\CC\). 
It turns out that \(\,f(0;\,\tau)\,\) is actually a unit of the 
integral closure of \(\QQ[\,j\,]\) in the function field of the
modular curve \(\,X(4)_{_{\QQ(\sqrt{-1})}}\,\) over \(\,\QQ(\sqrt{-1})\);
see \cite[Thm.\,1, p.\,189]{K-Lang}.
\end{remarkss}


\begin{corollary}\label{modular-form-1stcase}
For \(k=0,1,2,\ldots\in \NN\), let \(\,a_k(\tau)\) be the holomorphic
function on \(\HH\) defined by 
\begin{equation}\label{dev-modular-forms}
f(z;\tau)=\sum_{k=0}^{\infty}\,a_k(\tau)\,z^k,\qquad
z\in\CC,\ \tau\in \HH,
\end{equation}
where \(\,f(z;\tau)\,\) is the developing map in 
Proposition \textup{\ref{prop:extra-normalized}}.
For each \(k\geq 0\), \(a_k(\tau)\) is a holomorphic function
on \(\HH\) and defines a modular form of weight
\(k\) for the congruence subgroup \(\Gamma(4)\) in the sense that
\[
a_k(\gamma\cdot \tau)= j(\gamma,\tau)^k\cdot a_k(\tau)
\qquad \forall \gamma\in \Gamma(4).
\]
Moreover \(a_k(\tau)\) is meromorphic at the cusps of 
\(X(4)\) for every \(k\in \NN\).
\end{corollary}

\begin{proof}
The transformation formula (\ref{dev-modular-forms}) follows
immediately
from Proposition \ref{prop:transf_4pi}.
The fact that \(\,f(z;\tau)\,\) is holomorphic on \(\CC\times\HH\)
implies that \(\,a_k(\tau)\,\) is a holomorphic function on \(\HH\).
The last assertion that \(\,a_k(\tau)\,\) is meromorphic at the cusps
is most easily seen from the explicit formula (\ref{explicit-formula-1})
for \(\,f(z;\tau)\).
\end{proof}

\subsection{Generalization to \(\rho=(2n+1)4\pi\).}
\label{subsec:gen_higher_odd}
The considerations leading to the 
transformation formula (\ref{transf-law-N}) with respect to
\(\textup{SL}_2(\ZZ)\) for the normalized developing map
\(f(z;\tau)\) for the unique solution of 
\(\,\triangle u+e^{u}=4\pi\cdot \delta_0\,\) on \(\CC/\Lambda_{\tau}\)
specified in Proposition \ref{prop:extra-normalized}
can be extended for all type I cases.
In \ref{def:ram_covers}--\ref{cor:gen-psi} below 
we formulate the basic geometric structures which lead to
a generalization of (\ref{transf-law-N}),
and ends with an unsolved irreducibility and monodromy question 
in \ref{question:irred}.

\begin{definitionss}\label{def:ram_covers}
Let \(n\) be a non-negative integer.
\begin{itemize}
\item[(1)] Let \(\,\mathbb{M}_n\,\) be the set of all pairs
\((u(z), \tau)\), where \(\tau\in \HH\) and \(u(z)\) is a solution
of the mean field equation
\(\,\triangle u + e^u = 4(2n+1)\pi\cdot \delta_0\,\) on
the elliptic curve \(\CC/\Lambda_{\tau}\).
\smallbreak

\noindent
Let \(\,\pi_n:\mathbb{M}_n\to \HH\,\) be the map which sends
a typical element \((u(z),\tau)\) in \(\mathbb{M}_n\) to 
the point \(\tau\) of the upper-half plane \(\HH\).
Note that \(\,\pi_n\,\) is surjective according to
Theorem \ref{thm-type I}.
%
\smallbreak

\item[\textup{(2)}] Let \(\,\mathbb{D}_n\,\) be the set of all 
\(\,(f(z),\tau)\), where \(\tau\) is an element of \(\HH\)
and \(f(z)\) is a developing map of a solution of
\(\,\triangle u + e^u = 4(2n+1)\pi\cdot \delta_0\,\) on
the elliptic curve \(\CC/\Lambda_{\tau}\) whose monodromy group
is equal to the standard Klein's four subgroup 
\(K_4\subset\textup{PSU}(2)\) in the notation of \ref{notation-K4}.
\smallbreak

\noindent
Let \(\,p_n:\mathbb{D}_n\to \mathbb{M}_n\,\) be the map
which sends a typical element \((f(z),\tau)\in \mathbb{D}_n\)
to the element
\(\,\big(\log\tfrac{8 \vert f'(z)\vert^2}{(1+\vert f(z)\vert^2)^2},
\tau\big)\) of \(\mathbb{M}_n\),
and let \(\tilde{\pi}_n=\pi_n\circ p_n:\mathbb{D}_n\to \HH_n\)
be the natural projection map which sends
each element \((f(z),\tau)\in \mathbb{D}_n\) to \(\tau\).
\smallbreak

\item[\textup{(3)}] Let \(\,\mathbb{D}'_n\,\) be the subset
of \(\mathbb{D}_n\) consisting of all pairs 
\(\,(f(z),\tau)\in \mathbb{D}'\) 
such that 
and \(f(z)\) is a \emph{normalized} type I developing map
for an element \((u(z),\tau)\in \mathbb{M}_n\) satisfying
the monodromy condition that 
\(\,f(z+1)=-f(z)\) and \(\,f(z+\tau)=f(z)^{-1}\) for all 
\(z\in \CC\).
Let \(\,p'_n:\mathbb{D}'\to \mathbb{M}_n\,\) be the restriction 
to \(\mathbb{D}'_n\) of \(p_n\), and
let \(\tilde{\pi}_n'=\pi\circ p'_n: \mathbb{D}'\to \HH_n\) 
be the restriction to \(\mathbb{D}'_n\) of \(\pi_n\).
\smallbreak

\item[\textup{(4)}] Define \(\,\phi_n: \mathbb{M}_n\to 
\mathbb{D}'_n\,\) be the map which sends a typical element
\((u(z), \tau)\in \mathbb{M}_n\) to the element
\((f(z),\tau)\in \mathbb{D}'_n\) such that
\(f(\tfrac{1}{2})=0\) and \(f(\tfrac{\tau}{2})=1\).

This map \(\,\phi_n\,\) is well-defined because for each 
normalized type I developing map \((f(z),\tau)\in \mathbb{D}'_n\),
we have 
\[f(\tfrac{1}{2})= 0\,\ \textup{or}\,\ \infty,\qquad
f(\tfrac{\tau}{2})=\pm 1;\]
these four possibilities are permuted simply 
transitively by the action of \(K_4\) through fractional
linear transformations.

\end{itemize}

\end{definitionss}

\begin{lemmass}\label{lemma:ram_covers}
Let \(\mathbb{M}_n, \mathbb{D}_n, \mathbb{D}'_n\) be as in Definition
\textup{\ref{def:ram_covers}} above.
\begin{itemize}
\item[\textup{(1)}] Each of the three sets 
\(\mathbb{M}_n, \mathbb{D}_n\)  and \(\mathbb{D}'_n\) has
a natural structure as a one-dimensional complex manifold 
such that the maps
\(\pi_n:\mathbb{M}_n\to \HH\),
\(\tilde{\pi}_n:\mathbb{D}_n\to \HH\) and
\(\tilde{\pi}'_n:\mathbb{D}'_n\to \HH\) are 
finite surjective holomorphic maps.
Moverover there exists a discrete subset 
\(R_n\subset \HH\) which stable
under the natural action of the modular group \(\textup{SL}_2(\ZZ)\) on \(\HH\) 
with \(\,\vert \textup{SL}_2(\ZZ)\backslash R_n\vert <\infty\),
such that \(\pi_n, \tilde{\pi}_n\) and \(\tilde{\pi}'_n\)
are unramified over the complement \(\HH\smallsetminus R_n\)
of \(R_n\). 

\item[\textup{(2)}] The action of the finite group \(\,\textup{N}(K_4)\,\) 
on \(\mathbb{D}_n\) via linear fraction transformations is
holomorphic, making \(\,p_n:\mathbb{D}_n\to \mathbb{M}_n\,\) an
unramified Galois cover with group \(\textup{N}(K_4)\).
Similarly the map \(\,p_n':\mathbb{D}'_n\to \mathbb{M}_n\,\) is a 
holomorphic unramified Galois cover for the action of the
standard Klein's four group \(K_4\) in \(\textup{PSU}(2)\).

\item[\textup{(3)}] The map \(\phi_n: \mathbb{M}_n\to \mathbb{D}'_n\)
is a holomorphic section of \(\,p'_n:\mathbb{D}'_n\to \mathbb{M}_n\).
Consequently \(\,p_n: \mathbb{D}_n\to \mathbb{M}_n\,\) is a
trivial Galois cover with group \(\,\textup{N}(K_4)\) and
\(\,p'_n: \mathbb{D}'_n\to \mathbb{M}_n\,\) is a
trivial Galois cover with group with group \(K_4\).

\end{itemize}
\end{lemmass}

\begin{proof}
The statement (1) for \(\,\pi_n:\mathbb{M}_n\to \HH\,\) follows from
theorems \ref{thm-type I} and \ref{BH-poly}.
The part of statement (1) for \(\tilde{\pi}_n:\mathbb{D}_n\to \HH\) and
\(\tilde{\pi}'_n:\mathbb{D}'_n\to \HH\) is a consequence of
theorems \ref{thm-type I} and \ref{BH-poly} and the 
group-theoretic lemmas \ref{lemma:mono_grpthy} and
\ref{cor:comm_homo}.

The action of \(\,\textup{N}(K_4)\,\) on \(\mathbb{D}_n\) is easily
seen to be continous and is simply transitive on every fiber
of \(\,p_n:\mathbb{D}_n\to \mathbb{M}_n\). The first part of statement
(2) follows.  The second part of (2) is proved similarly.

The fact that \(p'_n\circ \phi_n = \textup{id}_{\mathbb{M}_n}\) is
immediate from the definition.  It is not difficult to see that
\(\phi_n\) is continuous, which implies that \(\phi_n\) is
holomorphic.
The first statement in (3) is proved; the rest of (3) follows.
\end{proof}

\begin{definitionss}\label{typeII-SL2action}
Define compatible actions of the modular group \(\,\textup{SL}_2(\ZZ)\,\) on 
\(\,\mathbb{M}_n\,\) and \(\,\mathbb{D}_n\,\) as follows.
For any element \(\,\gamma\in \textup{SL}_2(\ZZ)\),
any element \((u(z),\tau)\in \mathbb{M}_n\), 
and any element \((f(z),\tau)\in \mathbb{D}_n\)
such that \(\,p_n((f(z),\tau))=(u(z),\tau)\), 
\begin{itemize}

\item  \(\gamma\) sends 
\((u(z),\tau)\in \mathbb{M}_n\) to the element
\[\big(u(j(\gamma,\tau)\!\cdot\!z)+\log(\vert j(\gamma,\tau)\vert^2),\
\gamma\!\cdot\tau\big) \in \mathbb{M}_n,
\]

\item and \(\gamma\) sends 
\((f(z),\tau)\in \mathbb{D}_n\) to the element
\[
\big(f(j(\gamma,\tau)\!\cdot\!z),\ \gamma\!\cdot\!\tau\big)
\in \mathbb{D}_n.
\]

\end{itemize}
It is easy to check that \(\,p_n:\mathbb{D}_n\to\mathbb{M}_n\,\) is
\(\textup{SL}_2(\ZZ)\)-equivariant, i.e.\
\[\,p_n\big(\gamma\cdot (f(z),\tau)\big)
=\gamma\cdot p_n\big((f(z),\tau)\big)\,\] for every 
\(\,\gamma\in \textup{SL}_2(\ZZ)\),
and every element \((f(z),\tau)\in \mathbb{D}_n\).

\end{definitionss}

\begin{lemmass}\label{lemma:modular-action-genTeich}
\textup{(1)} The actions of \(\,\textup{SL}_2(\ZZ)\,\) on
\(\mathbb{D}_n\) and \(\mathbb{M}_n\) defined in
\textup{\ref{typeII-SL2action}} are holomorphic.
\smallbreak

\textup{(2)} The holomorphic maps
\[p_n:\mathbb{D}_n\to \mathbb{M}_n,\ \
\pi_n:\mathbb{M}_n\to \HH\ \ \ \textup{and}\ \ \
\tilde{\pi}_n=\pi_n\circ p_n: \mathbb{D}_n\to\HH\]
are 
equivariant for the \(\textup{SL}_2(\ZZ)\)-actions on
\(\,\mathbb{D}_n, \mathbb{M}_n\) and \(\HH\).
\smallbreak 

\textup{(3)} The actions of \(\,\textup{SL}_2(\ZZ)\,\)
and \(\,\textup{N}(K_4)\,\) on \(\mathbb{D}_n\) commute.
\smallbreak

\textup{(4)} The submanifold \(\,\mathbb{D}'_n\subset
\mathbb{D}_n\) is stable under the action of 
the principal congruence subgroup 
\(\Gamma(2)\) of level \(2\) in \(\,\textup{SL}_2(\ZZ)\).

\end{lemmass}

\begin{proof}
That the action of \(\,\textup{SL}_2(\ZZ)\) on \(\mathbb{D}_n\)
and \(\mathbb{M}_n\) is continuous can be verified without
difficulty, from which (1) follows. 
The proofs of statements (2)--(4) are easy and omitted.
\end{proof}

\begin{corollaryss}\label{cor:gen-psi}
Suppose that \(\,\mathbb{M}_n\,\) is \emph{connected}.
\begin{itemize}
\item[\textup{(a)}] There exists a group homomorphism
\[
\psi_n: \textup{SL}_2(\ZZ)\to \textup{N}(K_4)
\]
such that
\[
\phi_n\big(\gamma\cdot (f(z),\tau)\big)
= \psi_n(\gamma)\cdot \phi_n\big((f(z),\tau)\big)
\]
for all \(\gamma\in \textup{SL}_2(\ZZ)\) and all elements
\((f(z),\tau)\in \mathbb{M}_n\).

\item[\textup{(b)}] The homomorphism \(\,\psi_n\,\) in \textup{(a)}
satisfies
\[\,\Gamma(4)\subseteq\textup{Ker}(\psi_n)\quad
\textup{and}\quad
\psi_n(\Gamma(2))\subseteq K_4.
\]
\end{itemize}
\end{corollaryss}

\subsubsection{\bf Questions.}\label{question:irred}
(a) \emph{Is \(\,\mathbb{M}_n\) connected?}\footnote{We think the
  answer is very likely ``yes'', but we don't have a proof.}
\smallbreak

\noindent
(b) \emph{Suppose that \(\,\mathbb{M}_n\,\) is connected.  
What is the Galois group of the ramified cover
\(\,\pi_n:\mathbb{M}_n\to \HH_n\)?
Is it the symmetric group \(S_n\)?}

\section{Type II solutions: Evenness and Green's functions} \label{II-even}
\setcounter{equation}{0}
  
In this section we give the proof of Theorem \ref{thm-type II}
concerning type II solutions. By Proposition \ref{non-exist}, we may
assume that $\rho = 8n \pi$ ($l = 2n$). Let $u$ be a solution of
(\ref{Liouville-eq}), 
and $f$ be a developing map of $u$. 
We recall that $u_\lambda$ in (\ref{u-lambda}) 
is a one parameter family of solutions of (\ref{Liouville-eq}).

\subsection{Evenness of solutions for $\rho = 8n\pi$}

\begin{theorem}\label{thm:g_even}
There is a unique even solution within each normalized type II family of
solutions of the singular Liouville equation 
$\triangle u + e^u = 8n\pi\, \delta_0$ on a torus
$E$, 
$n$, where \(n\) is a positive integer. 
In other words for any
normalized type II developing map \(f\) of a solution \(u\) of the
above equation, there exists a
unique \(\lambda\in \RR\) such that
the solution
\[
u_{\lambda}(z)= \log\frac{e^{2\lambda}\vert
  f'(z)\vert^2}{(1+e^{2\lambda} \vert f(z)\vert^2)^2}
\]
of the same equation satisfies
\(\,u_{\lambda}(-z)=u_{\lambda}(z)\) \(\,\forall\,z\in \CC\).
\end{theorem}

\begin{proof}
Let \(f\) be a normalized type II developing map of a solution
\(u\) of (\ref{Liouville-eq}). 
It is enough to show that 
there exists a unique 
\(\lambda\in\RR\) such that
\(f_{\lambda}(z) := e^{\lambda}\cdot f(z)\) satisfies
\( f_{\lambda}(-z) = c/ f_{\lambda}(z)\) for a constant $c$ with $|c| = 1$.
\smallbreak

Let \(g:=f'/f\), the logarithmic derivative of \(f\); it is
a meromorphic function on \(E=\CC/\Lambda\) because
\(f\) is normalized of type II.
It suffice to show that \(g\) is even, for then
\begin{equation*}
f(-z) = f(0) \exp \int_0^{-z} g(w)\, dw = \frac{f(0)^2}{f(z)},
\end{equation*}
and the unique solution of $\lambda$ is given by $\lambda = -\log |f(0)|$.
\medbreak

We know from Lemma \ref{lemma_zero-pole_g} that
\(f\) is a local unit at points of \(\Lambda\) and 
\(g\) is a meromorphic function on \(E\) which has a zero
of order \(2n\) at \(0\in E\), no other zeros 
and \(2n\) simple poles on \(E\).
Moreover the residue of \(g\) is equal to \(1\) at
\(n\) of the simple poles of \(g\), and equal to \(-1\) 
at the other \(n\) simple poles.

Denote by \(P_1,\ldots, P_n\) the \(n\) simples poles of \(g\)
with residue \(1\) on \(E=\CC/\Lambda\), corresponding to zeros of the developing
map \(f\), and let \(Q_1,\ldots, Q_n\) be the \(n\)-simple
poles of \(g\) with residue \(-1\), corresponding to 
simple poles of \(f\).
Let \(p_1,\ldots, p_n\in \CC\) be representatives of \(P_1,\ldots, P_n
\in\CC/\Lambda\); similarly let 
\(q_1,\ldots, q_n\in \CC\) be representatives of \(Q_1,\ldots, Q_n\in 
\CC/\Lambda\).
%
%
The condition on the poles of \(g\) allows
us to express $g$ in terms of the Weierstrass \(\zeta\)-function:
\begin{equation}\label{formula:g}
g(z) = \sum_{i = 1}^n \zeta(z - p_i) - \sum_{i = 1}^n \zeta(z -
q_i) + \sum_{i=1}\zeta(p_i) - \sum_{i=1}\zeta(q_i)
\end{equation}
for a unique constant $c$,
because \(g(z)- \sum_{i = 1}^n \zeta(z - p_i) + \sum_{i = 1}^n \zeta(z -
q_i)\) is a meromorphic holomorphic function on \(E\).
Of course the constant \(c\) is completely determined by
the elements \(p_1,\ldots, p_n; q_1,\ldots,q_n\in\CC\):
\[c=\sum_{i=1}\zeta(p_i) - \sum_{i=1}\zeta(q_i).\]
It remains to analyze the condition that 
\(g(z)\) has a zero of order \(2n\) at \(0\in E\),
i.e. 
\begin{equation} \label{relation}
0= - g^{(r)}(0)= \left. -\frac{d^r g}{dz^r}\right\vert_{z=0}
=\sum_{i = 1}^n \wp^{(r-1)}(p_i) - \sum_{i = 1}^n \wp^{(r-1)}(q_i)
\end{equation}
for \(r=1,\ldots, 2n-1\) because
\begin{equation*}
-g^{(r)}(z) = \sum_{i = 1}^n \wp^{(r - 1)}(z - p_i) - \sum_{i =
1}^n \wp^{(r - 1)}(z - q_i).
\end{equation*}
Only the conditions that \(g^{(2s+1)}(0)=0\) for \(s=0,1,\ldots,n-1\)
will be used for the proof of Theorem \ref{thm:g_even}.  
The vanishing of the even-order derivatives of \(g\) will be
explored in the proof of Theorem \ref{thm:hecke-system}.
\medbreak


By Lemma \ref{lemma:symm_wp} below, relations
(\ref{relation}) for \(r=1, 3, 5,\ldots, 2n-1\) impies that
the sets
\(\{\wp(p_1),\ldots,\wp(p_n)\}\) and \(\{\wp(q_1),\ldots,\wp(q_n)\}\)
are equal as sets with multiplicities.
Because \(P_1,\ldots, P_n;\,Q_1,\ldots, Q_n\) are \(2n\) distinct
points on \(E\), it follows that 
\(\wp(p_i)\neq \wp(p_j)\) whenever \(i\neq j\)
and 
\(\{Q_1,\ldots, Q_n\}=\{-P_1,\ldots, -P_n\}\) as
subsets of \(E\smallsetminus\{0\}\) with \(n\) elements.
From the expression \eqref{formula:g} of \(g\) and the fact that
\(\zeta(z)\) is an even function on \(\CC\) one sees that 
\(g(z)\) is even. Theorem \ref{thm:g_even} is proved modulo the
elementary Lemma \ref{lemma:symm_wp}.
\end{proof}

We record the following statements from the proof of Theorem \ref{thm:g_even}.
\begin{corollary}\label{cor:typeIIzeropole}
Let $f$ be a normalize type II developing map of a solution $u$ of
the equation $\triangle u + e^u = 8n\pi\, \delta_0$ for 
a positive integer \(n\).
Then the zeros $p_i$'s of $f$ modulo $\Lambda$ correspond to \(n\) elements 
$P_1, \ldots, P_n \in E\smallsetminus\{0\}$ and the
poles $q_i$'s of \(f\) modulo $\Lambda$ correspond to
\(n\) elements
$Q_1, \ldots, Q_n \in E\smallsetminus\{0\}$.
Moreover the following statements hold.
\begin{itemize}
\item[(a)]
\(\,\displaystyle \{Q_1, \ldots, Q_n\} = \{-P_1, \ldots, -P_n\}\,\)
as subsets of \(E\smallsetminus\{0\}\) with \(n\) elements.

\item[(b)] \(\wp'(p_i)\neq 0\),
or equivalently \(P_i\) is not a \(2\)-torsion point of \(E\),
for \(i=1,\ldots, n\).

\item[(c)] \(\wp(p_i)\neq \wp(p_j)\) for any \(i, j=1,\ldots, n\)
such that  \(i\neq j\).

\end{itemize}
\end{corollary}

%

\begin{lemma}\label{lemma:symm_wp} 
\textup{(a)} For each positive integer
\(j\), there exists a polynomial \(h_j(X)\in\CC[X]\)
of degree \(j+1\) such that
\[
\wp^{(2j)}(z):=\Big(\frac{d}{dz}\Big)^{\!2j}\wp(z)
=h_j(\wp(z))
\]
as meromorphic functions on \(\CC\).
\smallbreak

\textup{(b)}
For every symmetric polynomial \(P(X_1,\ldots,
X_n)\in\CC[X_1,\ldots,X_n]\), 
there exists a polynomial
\(Q(W_1,\ldots,W_n)\in \CC[W_1,\ldots,W_n]\) such that
\begingroup\makeatletter\def\f@size{10}\check@mathfonts
\[
P(\wp(z_1),\ldots, \wp(z_n))
=Q\Big(\sum_{i = 1}^n \wp(z_i), \sum_{i = 1}^n \wp^{(2)}(z_i), 
\sum_{i = 1}^n \wp^{(4)}(z_i), \ldots,
\sum_{i = 1}^n \wp^{(2n-2)}(z_i)\Big)
\]
\endgroup
as meromorphic functions on \(\CC^n\).
\end{lemma}

\begin{proof}
Taking the derivative of the Weierstrass equation 
\[\wp'(z)^2=4\wp(z)^3- g_2\wp(z) -g_3,\]
and divide both sides by \(2\wp'(z)\), we get
\[
\wp^{(2)}(z)=6 \wp(z)^2 - \frac{1}{2}\,g_2.
\]
An easy induction shows that 
\(\wp^{(2j)}(z)\) is equal to \(h_j(\wp(z))\) for a polynomial
\(h_j(X)\in \CC[X]\) of degree \(j+1\) and the coefficient
\(a_j\) of \(X^{j+1}\) is positive.
In fact one sees that \(\,\displaystyle{a_j=(2j+1)!}\,\) when one 
compares the coefficient of \(z^{-2j-2}\) in the Laurent series
expansion of \(\wp^{(2j)}(z)\) and \(h_j(\wp(z))\).
We have proved the statement (a).

The statement (b) follows from the fact that every
symmetric polynomial in \(\CC[X_1,\ldots,X_n]\) is a
polynomial of the Newton polynomials
\(p_1(X_1,\ldots, X_n)$, $\ldots$, $p_n(X_1,\ldots,X_n)\),
where \(\,p_j(X_1,\ldots,X_n):=\sum_{i=1}^n X_i^j\,\) for
\(j=1,\ldots, n\).
\end{proof}

\subsection{Green/algebraic system for $\rho = 8n\pi$}
\label{subsec:green-alg}

Let $u$ be a type II even solution of the singular Liouville equation
$\triangle u + e^u = 8n\pi\, \delta_0$ on a torus
$E=\CC/\Lambda$, where \(n\) is a positive integer.
As before let $f$ be a normalized developing map of $u$ and let $g =
(\log f)' = f'/f$. 
Let $p_1, \ldots, p_n\in \CC$ be the simple zeros of
$f$ modulo \(\Lambda\) 
and let $-p_1, \ldots, -p_n$ be the simple poles of \(f\)
modulo \(\Lambda\) as in Corollary \ref{cor:typeIIzeropole}.
Let \(\,P_1=p_1\,\textup{mod}\,\Lambda,\ldots, 
P_n=p_n\,\textup{mod}\,\Lambda\,\); they are
exactly the blow-up points of the scaling family \(\,u_{\lambda}\,\)
of solutions of $\triangle u + e^u = 8n\pi\, \delta_0$.

\subsubsection{Two approaches to the configuration of blowup-up points}\enspace
\label{subsubsec:approaches}

We will investigate the constraints on the
configuration of the points \(P_1, \ldots, P_n\)
with two approaches outlined below.
The plans are executed in \S \ref{thm:hecke-system}
and \S \ref{analytic-aspect} respectively.
\medbreak

\noindent
{\bf A.} In the first approach
we have the relations \eqref{relation}
for \(r=2, 4, \ldots, 2n-2\) and also the
condition that the period integrals of the meromorphic
differential \(\,g(z)\,dz\,\) on \(E\) are purely imaginary;
the latter comes directly from the assumption that 
the image of the monodromy of the type II developing map
lies in the diagonal maximal torus of \(\textup{PSU}(2)\).
The first \(n-1\) conditions translates into a system of polynomial
equations in the coordinates 
\(\wp(p_1),\ldots,\wp(p_n),\wp'(p_1),\ldots,\wp'(p_n)\)
of the \(n\) points \(P_1,\ldots, P_n\):
\begin{equation} \label{Alg-eqn-1}
\wp'(p_1) \wp^r(p_1) + \cdots + \wp'(p_n) \wp^r(p_n) = 0, \quad \forall\,r
= 0,\ldots, n - 2,
\end{equation}
while the monodromy constraint becomes
\begin{equation} \label{G-eqn}
\sum_{i=1}^n\,\frac{\partial G}{\partial z}(p_i) =0
\end{equation}
where \(G\) is the Green's function on \(E\) as in 
\eqref{eqn-green}.
\smallbreak

The method used in this approach also shows that the
\(n\) equations (\ref{Alg-eqn-1}) and (\ref{G-eqn}) are
also sufficient: if \(P_1=p_1\,\textup{mod}\,\Lambda,\ldots,
P_n=p_n\,\textup{mod}\,\Lambda\)
are \(n\) elements in \(E\) with distinct \(x\)-coordinates
\(\wp_1(p_1),\ldots, \wp(p_n)\) 
satisfying equations (\ref{Alg-eqn-1}) and (\ref{G-eqn})
and none of \(P_1,\ldots, P_n\) is a \(2\)-torsion point of \(E\),
then there exists a normalized type II developing map
\(f\) for a solution of $\triangle u + e^u = 8n\pi\, \delta_0$
such that \(p_1,\ldots, p_n\) is a set of representatives
of the zeros of \(f\) modulo \(\Lambda\).
\smallbreak

\noindent
{\bf B. } In the second approach, 
results on blow-up solutions of 
a mean field equation on a Riemann surface provides the
following constraints 
\begin{equation} \label{nGz-1}
n \frac{\partial G}{\partial z}(p_i) 
= \sum_{1\leq j \leq n,\, j \ne i}\, \frac{\partial G}{\partial z}(p_i- p_j), 
\qquad\textup{for}\
\ i = 1, 2, \ldots, n,
\end{equation}
on the blow-up points \(P_1,\ldots,P_n\).
%

\subsubsection{}
We will see in \ref{analytic-aspect} that 
the system of equations \eqref{nGz-1} 
is equivalent to 
the combination of (\ref{G-eqn}) and the following system 
of equations
\begin{equation} \label{ratio-eq-1}
\sum_{j \ne i} \frac{\wp'(p_i) + \wp'(p_j)}{\wp(p_i) - \wp(p_j)} = 0, 
\qquad \textup{for}\  i = 1, \ldots, n. 
\end{equation}
Moreover for elements \((P_1,\ldots, P_n)\!\in\! (E\!\smallsetminus\! E[2])^n\) 
with \emph{distinct} \(\,x\)-coordinates 
\(\wp(p_1),\ldots,\wp(p_n)\), 
the two systems of equations (\ref{Alg-eqn-1}) and
(\ref{ratio-eq-1}) are equivalent.\footnote{For every \(m\in \NN\),
  \(E[m]\) denote the subgroup of \(m\)-torsion points on \(E\).} 
This means that among elements of the subset 
\(\textup{Bl}_n\subset E^n\) consisting of all
\(n\)-tuples \((P_1,\ldots, P_n)\in (E\!\smallsetminus\!\{0\})^n\) satisfying the 
constraints (\ref{G-eqn}) for blow-up points,
those satisfying the non-degeneracy condition
\begin{equation}\label{x-nondegen}
\wp'(p_i)\neq 0\ 
\quad\textup{and}\quad
\wp(p_i)\neq \wp(p_j)\ \ \textup{whenever}\ i\neq j
\quad \forall 1\leq i,j\leq n
\end{equation}
are indeed blow-up points of the scaling family
\(u_{\lambda}(z)\) of a type II solution of
the singular Liouville equation
\(\triangle u + e^u = 8n\pi\, \delta_0\) on \(E\).

\bigbreak

We recall some properties about period integrals
and Green's functions in lemmas \ref{lemma:period_int}--\ref{lemmass:green-zeta}
below, before returning to the first approach
outlined in \ref{subsubsec:approaches}.

\begin{lemmass}\label{lemma:period_int}
For any \(y\in\CC\) and any \(\omega\in \Lambda=\textup{H}_1(E;\ZZ)\), 
the \(\omega\)-period of the meromorphic differential 
\(\,\displaystyle{\frac{\wp'(y)\,dz}{\wp(z)-\wp(y)}}\,\) on \(E=\CC/\Lambda\)
is given by 
\begin{equation}\label{period-of-wp}
\int_{L_{\omega}} \frac{\wp'(y)}{\wp(z)-\wp(y)}dz
\equiv 2\omega\cdot \zeta(y) - 2\eta(\omega)\cdot y
\pmod{2\pi\sqrt{-1}\ZZ}.
\end{equation}
Here \(L_{\omega}:[0,1]\to \CC\) is any piecewise smooth path on
\(\CC\) such that \(L_{\omega}(1)-L_{\omega}(0)=\omega\)
and \(\wp(z)\neq \wp(y)\) for all \(z\in L_{\omega}([0,1])\).
\end{lemmass}

\begin{proof}
This is a reformulation of \cite[Lemma\,2.4]{LW2}.
Note that meromorphic differential \({\frac{\wp'(y)\,dz}{\wp(z)-\wp(y)}}\)
on \(E\) has poles at \(0\) and \(\pm y\pmod \Lambda\), with
residues \(0\) and \(\pm 1\) respectively,
therefore the period integral
\(\,I_{\omega}(y):=\int_{L_{\omega}} \frac{\wp'(y)\,dz}{\wp(z)-\wp(y)}\,\)
is well-defined modulo \(\,2\pi\sqrt{-1}\ZZ\).
The addition formula for \(\wp(z)\) 
gives
\[{\frac{\wp'(y)\,dz}{\wp(z)-\wp(y)}}
={\frac{\wp'(z)\,dz}{\wp(z)-\wp(y)}}
-2\zeta(z+y;\Lambda)dz+ 2\zeta(z;\Lambda)dz + 2\zeta(y;\Lambda)dz
\]
The lemma follows after an easy calculation, using 
the functional equation for \(\zeta(z;\Lambda)\)
and the fact that
\(\frac{d}{dz}\log\sigma(z) =\zeta(z)\) and; 
see \cite[Lemma\,2.4]{LW2} for details.

Alternatively, one computes 
\[\frac{d}{dy}I_{\omega}(y)= \int_{L_{\omega}} (2 \wp(z+y)-2\wp(y))dz
=-2\wp(y)+2\eta(\omega)
\]
and determine the constant of integration
up to \(2\pi\sqrt{-1}\ZZ\)  by evaluation at
\(2\)-torsion points of \(E\).
\end{proof}

\begin{lemmass}\label{lemmass:green-zeta}
Let \(G\) be the Green's function on \(E=\CC/\Lambda\) as in 
\textup{(\ref{eqn-green})}.
The following formulas hold.
\begin{equation}\label{green-sigma}
G(z)=
-\frac{1}{2\pi}\,
\log \left\vert   \Delta(\Lambda)^{\frac{1}{12}} \cdot
e^{-z\,\eta(z;\Lambda)/2}\cdot
\sigma(z;\Lambda)
\right\vert \qquad \textup{on}\ \ E,
\end{equation}
\begin{equation}\label{green-zeta}
-4\pi\,\frac{\partial G}{\partial z}(z)
= \zeta(z;\Lambda) - \eta(z;\Lambda)
\qquad \forall\,z\in\CC.
\end{equation}
In the first formula \textup{(\ref{green-sigma})}, 
\(\eta(z;\Lambda)\) is the quasi-period and
\(\Delta(\Lambda)\) is the non-zero cusp
form of weight \(12\) for \(\textup{SL}_2(\ZZ)\) given by the 
formula
\[
\Delta(\Lambda)= g_2(\Lambda)^3-27g_3(\Lambda)^2
=\frac{(2\pi\sqrt{-1})^{12}}{(\omega_2)^{12}}\cdot
 q_{\tau}\cdot\prod_{m=1}^{\infty}(1-q_{\tau}^n)^{24},
\]
where \(q_{\tau}=e^{2\pi\sqrt{-1} \tau}\), \(\tau =
\omega_2/\omega_1\) with
\(\,\textup{Im}(\tau)>0\).


\end{lemmass}

\begin{remarkss} \label{rem:green-formulas}
(a) An equivalent form of 
\ref{lemmass:green-zeta}\,(a) is
\[
G(z; \Lambda_{\tau})= -\frac{1}{2\pi}\,
\log\left\vert
e^{-\frac{\pi\,\textup{Im}(z)^2}{\textup{Im}(\tau)}}
\cdot 
\Delta(\CC/\Lambda_{\tau})^{-\frac{1}{12}}
\cdot \theta[{\scriptscriptstyle \myatop{1/2}{1/2}}](z;\tau)
\right\vert
\]
for the Green's function \(G(z;\Lambda_{\tau})\) on the elliptic curve
\(\CC/\Lambda_{\tau}\), where \(\tau\) is an
element of the upper-half plane and 
\(\Lambda_{\tau}=\ZZ+\ZZ\!\cdot\!\tau))\). 
Here we have used the general
notation for theta functions with characteristics
\[\theta[{\scriptscriptstyle \myatop{a}{b}}](z;\tau)
:=\sum_{m\in\ZZ}
e^{\pi\sqrt{-1}\,\tau (m+a)^2}\cdot 
e^{2\pi\sqrt{-1}\,(m+a)(z+b)}
.\]
The equivalence of the two formulas follows from the formulas
\[
\sigma(z,\Lambda_{\tau})=-\frac{1}{\pi}\,e^{\eta(1,\Lambda_{\tau})}\,
\cdot\frac{\theta[{\scriptscriptstyle \myatop{1/2}{1/2}}](z;\tau)}{
\theta[{\scriptscriptstyle \myatop{0}{0}}](z;\tau)
\theta[{\scriptscriptstyle \myatop{0}{1/2}}](z;\tau)
\theta[{\scriptscriptstyle \myatop{1/2}{0}}](z;\tau)}
\]
and
\[
\theta[{\scriptscriptstyle \myatop{0}{0}}](z;\tau)
\theta[{\scriptscriptstyle \myatop{0}{1/2}}](z;\tau)
\theta[{\scriptscriptstyle \myatop{1/2}{0}}](z;\tau)
= 2\,\eta_{_\textup{Dedekind}}(\tau)^3
= 2\,\Big(q_{\tau}^{\frac{1}{24}}\prod_{m=1}^{\infty}(1-q_{\tau}^m)\Big)^3.
\]
(b) The function 
\(Z(z;\Lambda):=\zeta(z;\Lambda)-\eta(z;\Lambda)
=-4\pi\frac{\partial G}{\partial z}\,\) 
appeared in \cite[p.\,452]{Hecke};
we will call it the \emph{Hecke form}.
For any integers \(a,b\) and \(N\geq 1\) such that 
\(\textup{gcd}(a, b, N)=1\) and \((a,b)\not\equiv (0,0)\pmod N\),
\(Z(\frac{a}{N}\!\!+\!\! \frac{b}{N}\tau;\,\ZZ+\ZZ\tau)\) 
is a modular form of weight one 
and level \(N\),
equal to the Eisenstein series
\begin{equation*}
\begin{split}
&-\left. N\!\cdot E_1^N(\tau,s;a,b)\right\vert_{s=0} \\
&\quad =-\left. N\!\cdot \textup{Im}(\tau)^s\cdot 
\sum'_{(m,n)\equiv (a,b)\,\textup{mod}\, N}
(m\tau+n)^{-1}\cdot |m\tau+n|^{-2s}\right\vert_{s=0}.
\end{split}
\end{equation*}
\noindent
See \cite[p.\,475]{Hecke}.
\end{remarkss}

\noindent
\emph{Proof of Lemma} \ref{lemmass:green-zeta}.\enspace
The formula (\ref{green-sigma}) is proved in \cite[II,\,\S5]{Lang-Arakelov}.
The equivalent formula in \ref{rem:green-formulas}\,(a) is proved
in \cite[p.\,417--418]{Faltings:cal}.
See also \cite[\S7]{LW} and \cite[\S2]{LW2}.
The formula (\ref{green-zeta}) follows from (a) by an easy computation.
\qed

\begin{theorem}\label{thm:hecke-system}
Let \(n\) be a positive integer.
Let \(\,P_i=p_i\,\textup{mod}\,\Lambda\,\), \(i=1,\ldots,n\)
be \(n\) distinct points on \(E=\CC/\Lambda\) such that
\(\{P_1,\ldots, P_n\}\cap \{-P_1,\ldots, -P_n\}=\emptyset\).
In other words 
\(\wp(p_1),\ldots, \wp(p_n)\) are mutually \emph{distinct}
and none of the \(P_i\)'s is a \(2\)-torsion point of \(E\).
There exists a normalized type II developing map \(f\) for a solution \(u\) of
\(\,\triangle u + e^u = 8\pi n\delta_0\,\) on \(E\)
such that \(f(p_1)=\cdots=f(p_n)=0\) 
if and only if
\begin{equation} \label{partial-G-eqn}
\sum_{i=1}^n\,\frac{\partial G}{\partial z}(p_i) =0
\end{equation}
and
\begin{equation} \label{Alg-eqn}
\wp'(p_1) \wp^r(p_1) + \cdots + \wp'(p_n) \wp^r(p_n) = 0 \quad 
\textup{for}\ \ r = 0,\ldots, n - 2.
\end{equation}
\end{theorem}
\noindent
Notice that (\ref{partial-G-eqn}) is the same as (\ref{G-eqn})
and (\ref{Alg-eqn}) is the same as (\ref{Alg-eqn-1}).

\begin{proof}
We use the notation in the proof of Theorem \ref{thm:g_even}
and continue with the argument there.
The logarithmic derivative \(\,g' = {f'}/{f}\,\) 
of a normalized type II developing map 
has simple poles at the \(2n\) points \(\pm P_1,\ldots, \pm P_n\)
and is holomorphic elsewhere on \(E\).
Moreover the residue of \(g\) is \(1\) \(P_i\) and
is \(-1\) at \(-P_i\) for each \(i\).  Therefore
\begin{equation} \label{g-sum}
g(z) = \frac{\wp'(p_1)}{\wp(z) - \wp(p_1)} + \cdots +
\frac{\wp'(p_n)}{\wp(z) - \wp(p_n)}.
\end{equation}
because \(g(0)=0\).  
We know two more properties of \(g\): (a) \(g(z)\) has a zero of
order \(2n\) at \(z=0\), and (b) for any \(\omega\in \Lambda\) 
and any piecewise smooth path
\(L_{\omega}:[0,1]\to \CC\) such that \(L(1)-L(0)=\omega\) and
\[L([0,1])\cap \left[\big({\textstyle \bigcup_{i=1}^n p_i+\Lambda}\big) \cup
\big({\textstyle \bigcup_{i=1}^n -p_i+\Lambda}\big)\right]=\emptyset ,\]
we have
\[\int_{L_{\omega}} g\,dz\in \sqrt{-1}\RR
.\]
To see what property (a) means, we expand \(g(z)\) at \(z=0\)
as a power series in \(\wp(z)\):
\begin{equation*}
g(z) = \sum_{j = 1}^n \frac{\wp'(p_j)}{\wp(z)\,(1 - \wp(p_j)/\wp(z))} \\
= \sum_{m=0}^{\infty} \left(\sum_{j=1}^n
  \wp'(p_j)\,\wp(p_j)^m \right)\cdot
\wp(z)^{-m-1}
\end{equation*}
Because  \(g\) has exactly \(2n\) simple poles and is holomorphic
elsewhere on \(E\), we see that 
\(\textup{order}_{z=0}\, g(z)= 2n\) if and only if
all \(n-1\) equations in (\ref{Alg-eqn}) hold.
\smallbreak

We know that 
\(\,\displaystyle{\eta(z;\Lambda) y - z \eta(y;\Lambda)\equiv 0
\pmod{\sqrt{-1}\RR}}\,\)
for all  \(y, z\in \CC\)
because the left-hand side is \(\RR\)-bilinear and we know
from the Legendre relation that 
the statement holds when \(y, z\) are both in \(\Lambda\).
By Lemma \ref{lemma:period_int}, 
\begin{equation*}
\begin{split}
\int_{L_{\omega}} g(z)\,dz\
&\equiv 2 \sum_{j=1}^n \big(\omega \zeta(p_j) -  \eta(\omega) p_j\big)
\pmod{2\pi\sqrt{-1}\,\ZZ}
\\
&\equiv 2 \omega \cdot \sum_{j=1}^n \big(\zeta(p_j) - \eta(p_j)\big)
\pmod{\sqrt{-1}\RR}
\end{split}
\end{equation*}
for all \(\omega\in \Lambda\).
Therefore property (b) holds for \(g\) given by (\ref{g-sum})
if and only (\ref{partial-G-eqn}) holds.
We have proved the ``only if'' part of Theorem \ref{thm:hecke-system}.

Conversely suppose that equations (\ref{Alg-eqn}) and (\ref{partial-G-eqn})
hold. We have seen that the meromorphic function 
\(g(z)\) given by (\ref{g-sum}) has a zero of order \(2n\) at \(z=0\)
and the period integrals of \(\,g\,dz\,\) are all purely imaginary.
Therefore \(\,\displaystyle{f(z)=\exp\int_0^z g(w)\,dw}\,\) is
a type II developing map for a solution of the singular
Liouville equation \(\,\triangle u + e^u = 8\pi n\,\delta_0\).
\end{proof}

\begin{remarkss}
The property (a) that the order of the meromorphic function
\[\frac{\wp'(p_1)}{\wp(z) - \wp(p_1)} + \cdots +
\frac{\wp'(p_n)}{\wp(z) - \wp(p_n)}\] on \(E\) at \(z=0\) is equal to 
\(\,2n\,\)
is also equivalent to: \(\exists\, C\in\CC^{\times}\)
such that
\begin{equation} \label{poly=const}
\sum_{j = 1}^n \wp'(p_j) \prod_{i \ne j} (\wp(z) - \wp(p_i)) = C.
\end{equation}
\end{remarkss}
\subsection{Analytic approach to the configuration of the blow-up set}
\label{analytic-aspect}

\subsubsection{}
We may also study the set $\{p_1, \ldots, p_n\}$ from the analytic
point of view. As we have already seen, $\{p_i\}$ also represents the
blow-up set of 
the family of solutions $u_\lambda$ as $\lambda \to \infty$. 
The equations to determine the position of blow-up points 
are fundamentally important in the study of 
\emph{bubbling solutions of semi-linear equations} such as mean field equations, 
Chern--Simons--Higgs equation, Toda system in two dimension, 
or scalar curvature equation in higher dimensions. 
Hence we will derive these equations from the analytic aspect.
\medbreak

We recall the definition of blow-up points for a sequence of solutions 
$u_k$, $k \in \mathbb N$, to the mean field equation 
\begin{equation}\label{Liouville-varible-strength}
\triangle u_k + e^{u_k} = \rho_k\, \delta_0 \qquad \mbox{on $E$}
\end{equation}
with possibly varying singular strength $\rho_k$ such that 
$\rho_k \to \rho = 8\pi n$ for some $n \in \mathbb N$. 
If $\rho_k = 8\pi n$ for all $k$, this goes back to the situation
$u_\lambda$ in (\ref{u-lambda}) as has been discussed. 
In general it is also important to consider blow-up phenomenon 
from a sequence of solutions $u_k$ with $\rho_k \to \rho$. 
(It is known that if $\rho \not\in 8\pi \mathbb N$ then there is 
no blow-up phenomenon \cite{CL0}.) 

\begin{definitionss} \label{bp-pt}
A subset
$S = \{P_1, \ldots, P_m\} \subset E=\CC/\Lambda$ is called the 
\emph{blow-up set} of the sequence of solutions $(u_k)_{n\in\NN}$ 
of (\ref{Liouville-varible-strength}) with \(\rho_k\to 8\pi n\)
if for all $i$
\begin{equation*}
u_k(P_i) \to +\infty \quad \mbox{as} \quad
k \to \infty,
\end{equation*}
while if $P \not\in S$ then
\begin{equation*}
u_k(P) \to -\infty \quad \mbox{as} \quad k \to \infty.
\end{equation*}
Points $P_i$ in the blow-up set are called \emph{blow-up points} of
the sequence of solutions \((u_k)\). 
\end{definitionss}

It is also shown in \cite{CL0} that $m = n$ and the configuration of 
the blow-up points $\{P_1, \ldots, P_n\}$ 
satisfies the following equations:
\begin{equation} \label{nGz}
n G_z(P_i) = \sum_{j = 1, \ne i}^n G_z(P_i - P_j), \qquad i = 1, 2, \ldots, n,
\end{equation}
where \(z\) is the coordinate for \(\CC\) and
\(\,G_z=\frac{\partial G}{\partial z}\).
Notice that the system of equations (\ref{nGz}) 
is the same as the equations (\ref{nGz-1}).
Summing the \(n\) equations in (\ref{nGz}) from $i = 1, \ldots, n$, we get 
\begin{equation}\label{G-eqn-2}
\sum_{i = 1}^n G_z(P_i) = 0,
\end{equation}
since $\frac{\partial G}{\partial z}$ is an odd function. 
The last equation (\ref{G-eqn-2}) is the same as the Green equation (\ref{G-eqn})
and (\ref{partial-G-eqn}).

\begin{lemmass}
Let \(\{P_1,\ldots, P_n\}\) be a set of \(n\) mutually distinct points
in \(E\smallsetminus \{0\}=\CC/\Lambda \smallsetminus \{[0]\}\),
and let \(p_1,\ldots, p_n\) be elements of \(\CC\) such that
\(\,P_i=[p_i]:=p_i\,\textup{mod}\,\Lambda\,\) for \(i=1,\ldots, n\).
The system of equations \textup{(\ref{nGz})} for the set \(\{P_1,\ldots,
P_n\}\) is equivalent to the combination of the 
Green equation \textup{(\ref{G-eqn-2})}
and the following system equations
\begin{equation} \label{zeta-eq}
\sum_{1\leq j\leq n,\ j \ne i} 
\big(\zeta(p_i - p_j;\Lambda) + \zeta(p_j;\Lambda) 
- \zeta(p_i;\Lambda)\big) = 0, \qquad i = 1, \ldots, n.
\end{equation}
\end{lemmass}
\noindent Notice that for each \(i\), the sum
in the left-hand side of (\ref{zeta-eq})
is independent
of the choice of representatives \(p_1,\ldots, p_n\in \CC\) of
\(P_1,\ldots, P_n\in \CC/\Lambda\).

\begin{proof}
We have seen that the \(n\) equations in (\ref{nGz}) implies the Green equation
(\ref{G-eqn}). It suffices to show that under (\ref{G-eqn-2}),
the system of equations (\ref{nGz}) is equivalent
to the system of equations (\ref{zeta-eq}).

We know from \eqref{green-zeta} that 
\(\,G_z(P_i)=\zeta(p_i;\Lambda)-\eta(p_i;\Lambda)\,\)
for each \(i\). So the Green equation (\ref{G-eqn-2})
means that \(\sum_{i=1}^n\,\zeta(p_i;\Lambda) = 
\sum_{i=1}^n\,\eta(p_i;\Lambda)\). For each \(i\)
the \(i\)-th equation in (\ref{nGz}) becomes
\begin{equation*}
n\cdot\big[\zeta(p_i;\Lambda) - \eta(p_i;\Lambda)\lambda]
= \sum_{1\leq j\leq i,\ j\neq i}
\big(\zeta(p_i-p_j;\Lambda) - \eta(p_i;\Lambda) + \eta(p_j;\Lambda)
\big)
\end{equation*}
which is equivalent to the \(i\)-th equation 
in (\ref{zeta-eq}) because
\(\sum_{i=1}^n\,\zeta(p_i;\Lambda) = 
\sum_{i=1}^n\,\eta(p_i;\Lambda)\).
\end{proof}

%

\begin{remarkss}
Part of the condition for the blow-up set $\{P_1, \ldots, P_n\}$ 
of a sequence of solutions \((u_k)\) as in Definition \ref{bp-pt}
is that
\begin{equation}
P_i \ne P_j \quad \mbox{for} \quad i \ne j,
\end{equation}
instead of the stronger property
\begin{equation} \label{p-p}
\{P_1, \ldots, P_n\} \cap \{-P_1, \ldots, -P_n\} = \emptyset
\end{equation}
which is satisfied when \(\{p_1,\ldots, p_n\}\) are zeros of a 
normalized developing
map of a solution of \(\,\triangle u + e^{u}=8\pi n\,\delta_0\).
\end{remarkss}

\subsection{}{\bf Equivalence of algebraic systems (\ref{Alg-eqn}) and
  (\ref{zeta-eq}) under (\ref{p-p})}

\subsubsection{}
In light of Theorem \ref{thm-type II}, the analytic discussion in \S
\ref{analytic-aspect} suggests 
that
the 
system of equations (\ref{G-eqn}) + (\ref{Alg-eqn}) 
may be \emph{equivalent} to the system of equations
(\ref{G-eqn}) + (\ref{zeta-eq})
under the constraint that  $P_i \ne  P_j$ whenever $i \ne j$ 
and $P_i \ne -P_j$ for all $i, j$. 
\smallbreak

Since the Green equation (\ref{G-eqn}) is the only non-holomorphic
equation shared by both systems, one might optimistically ask
\begin{quotation}
\emph{Are the two holomorphic systems of $n - 1$ equations} (\ref{Alg-eqn}) 
and (\ref{zeta-eq}) \emph{equivalent?}
\end{quotation}
Note that the sum of equations in (\ref{zeta-eq}) is zero,
hence we may remove one equation from it.

\subsubsection{}
To answer this question, we recall the addition formula
\begin{equation*} 
\frac{1}{2} \frac{\wp'(z) - \wp'(u)}{\wp(z) - \wp(u)} = \zeta(z +
u) - \zeta(z) - \zeta(u).
\end{equation*}
Thus (\ref{zeta-eq}) with additional constraint $p_i \ne \pm p_j$ for $i \ne j$ is equivalent to
\begin{equation} \label{ratio-eq}
\sum_{j \ne i} \frac{\wp'(p_i) + \wp'(p_j)}{\wp(p_i) - \wp(p_j)} = 0, 
\qquad i = 1, \ldots, n. 
\end{equation}

Denote by $(x_i, y_i) = (\wp(p_i), \wp'(p_i))$. As points on $E$ they are related by the defining cubic curve equation $y_i^2 = p(x_i) = 4x_i^3 - g_2 x_i - g_3$. Then (\ref{ratio-eq}) and (\ref{Alg-eqn}) can be written as the following systems respectively:
\begin{equation} \label{(I)}
\sum_{j = 1,\, \ne i}^n \frac{y_i + y_j}{x_i - x_j} = 0, \quad i = 1, \ldots, n,
\end{equation} 
where $x_i \ne x_j$ for $i \ne j$ is imposed, and
\begin{equation} \label{(II)}
\sum_{i = 1}^n x_i^l y_i = 0, \quad l = 0, \ldots, n - 2.
\end{equation}
Both systems appear to be linear in $y_i$'s, 
and in fact we can prove their equivalence even without the elliptic curve equations:

\begin{propositionss} \label{eq-2-sys}
For a given set of mutually distinct elements 
$x_1, \ldots,x_n\in\CC$, the linear systems of equations 
\begin{equation}\label{(I)-1}
\sum_{1\leq j \leq n,\ j\ne i} \frac{Y_i + Y_j}{x_i - x_j} = 0, 
\quad \forall\,i = 1, \ldots, n
\end{equation}
and
\begin{equation} \label{(II)-1}
\sum_{i = 1}^n x_i^l\cdot Y_i = 0, \quad \forall\,l = 0, \ldots, n - 2
\end{equation}
in variables $Y_1, \ldots, Y_n$ are equivalent.
\end{propositionss}

\begin{proof}
The system (\ref{(I)-1}) corresponds to the $n \times n$ matrix
\begin{equation*}
A_n = \begin{pmatrix}
\sum\limits_{k = 2}^n \frac{1}{x_1 - x_k} & \frac{1}{x_1 - x_2} 
& \frac{1}{x_1 - x_3} & \cdots & \frac{1}{x_1 - x_n} 
\\ \frac{1}{x_2 - x_1} & \sum\limits_{k = 1, \ne 2}^n \frac{1}{x_2 -  x_k} 
& \frac{1}{x_2 - x_3} & \cdots & \frac{1}{x_2 - x_n} 
\\ \vdots & \vdots & \ddots & & \vdots 
\\ \frac{1}{x_n - x_1} & \frac{1}{x_n - x_2} 
& \cdots & \cdots & \sum\limits_{k = 1}^{n - 1} \frac{1}{x_n - x_k}
\end{pmatrix},
\end{equation*}
that is, $A_n = (a_{ij}) \in M_n(\mathbb Q(x_1, \ldots, x_n))$, 
where $a_{ij} = \frac{1}{x_i - x_j}$ if $j \ne i$, 
and $a_{ij} = \sum_{k = 1, \ne i}^n \frac{1}{x_i -x_k}$ if $j = i$, 
which is the sum of all the other entries in the same row. 
Note that the sum of all rows in $A_n$ is the zero row vector. 
In particular, $\det A_n = 0$.

The system (\ref{(II)-1}) corresponds to the $(n - 1) \times n$ matrix
\begin{equation*}
B_n := \begin{pmatrix}
1 & \cdots & 1 \\ x_1 & \cdots & x_n \\\vdots & \ddots & \vdots 
\\ x_1^{n - 2} & \cdots & x_n^{n - 2}
\end{pmatrix}.
\end{equation*}
Let $b = (b_1, \ldots, b_n)$ where $b_i$ is the determinant of the 
$(n - 1) \times (n - 1)$ minor of $B_n$ without the $i$-th column. 
Then $b_i$'s are given by the Vandermonde determinant:
\begin{equation*} 
b_i = \prod_{1 \le l < k \le n;\, l, k \ne i} (x_k - x_l).
\end{equation*} 
Let $C_n$ be the $(n - 1)\times n$ matrix consisting of the 
first $n - 1$ rows of $A_n$, and define $c = (c_1, \ldots, c_n)$ similarly. 

We want to prove $c \ne 0$, which implies that ${\rm rank}\,A_n = n -1$ 
and the kernel of $A_n$ is spanned by $c$. Then the equivalence of 
these two linear systems simply means that $b$ and $c$ 
are proportional to each other. 

We claim that
\begin{equation} \label{c_i}
c_i = \frac{(-1)^{n + i} (n - 1)!}{\prod_{k \ne i} (x_k - x_i)}, \qquad i = 1, \ldots, n.
\end{equation}

Due to symmetry, it is enough to consider the case $i = n$. 
We will show that the order of $c_n$ along the divisor $x_k - x_l$ 
is non-negative for all $k, l \ne n$. This will imply that $c_n$ 
is a constant times $\prod_{k = 1}^{n - 1} (x_k - x_n)^{-1}$. 

Again by symmetry, it is enough to check the case $k = 1$, $l = 2$. 
The only terms which may contribute poles along $x_k - x_l$ 
are $a_{11}$, $a_{12}$, $a_{21}$ and $a_{22}$. 
If we subtract the second row by the first row, 
and then add the resulting first column into the second column, 
we get the following $(n - 1) \times (n - 1)$ matrix
\begin{equation*}
\begin{pmatrix}
\frac{1}{x_1 - x_2} + r & -r & * & \cdots & * \\
r & (x_1 - x_2) * & * & \cdots & * \\
* & (x_1 - x_2) * & * & \cdots  & * \\
\vdots & \vdots & \vdots & \ddots & \vdots \\
* & (x_1 - x_2) * & * & \cdots & *
\end{pmatrix},
\end{equation*}
where $r = \sum_{j = 3}^n \frac{1}{x_1 - x_j}$, as well as 
all entries labeled by $*$, does not have pole along the divisor $x_1- x_2$. 
This shows that $\det\, (a_{ij})_{1 \le i, j  \le n - 1}$ has
non-negative order 
along $x_1 - x_2$. So there exists an element \(d_n\in \CC\) such that
\begin{equation} \label{d_n}
c_n = \frac{d_n}{\prod_{k < n} (x_k - x_n)}. 
\end{equation} 

By Lemma \ref{dn} below we have  $d_n = (n - 1)! \ne 0$. 
Then (\ref{c_i}) holds and we have $c \ne 0$. Now we note that
$$
b_i = \frac{(-1)^{n - i} \prod_{1 \le l < k \le n} (x_k -x_l)}{\prod_{k \ne i} (x_k - x_i)} 
= \frac{1}{(n - 1)!} \prod_{1 \le l < k \le n} (x_k - x_l) c_i,
$$
i.e.~$b$ is parallel to $c$. Hence the equivalence is proved.
\end{proof}

\begin{lemmass} \label{dn}
The constant $d_n$ in \textup{(\ref{d_n})} is $(n - 1)!$.
\end{lemmass}

\medbreak
\noindent
We offer two proofs.
\smallbreak
\noindent{\bf The first/analytic proof.}\enspace
%
It is easy to see that $d_1 = 1$. 
If we may show that $d_n = (n - 1) d_{n - 1}$ for $n \ge 2$ then we
are done. We observe that
\begin{equation*}
\frac{d_n}{\prod_{2 \le k < n} (x_k - x_n)} 
= \lim_{x_1 \to \infty} \frac{d_n x_1}{\prod_{k < n} (x_k - x_n)}
\end{equation*}
which, by the definition of $c_n$ and (\ref{d_n}), is equal to 
\begin{equation*}
\lim_{x_1 \to \infty}
\begin{vmatrix}
\sum\limits_{k = 2}^n \frac{x_1}{x_1 - x_k} & \frac{x_1}{x_1 - x_2} 
& \frac{x_1}{x_1 - x_3} & \cdots & \frac{x_1}{x_1 - x_{n - 1}} 
\\ \frac{1}{x_2 - x_1} & \sum\limits_{k = 1, \ne 2}^n \frac{1}{x_2 -  x_k} 
& \frac{1}{x_2 - x_3} & \cdots & \frac{1}{x_2 - x_{n - 1}} 
\\ \vdots & \vdots & \ddots & & \vdots 
\\ \frac{1}{x_{n - 1} - x_1} & \frac{1}{x_{n - 1} - x_2} 
& \cdots & \cdots & \sum\limits_{k = 1, \ne n - 1}^{n} \frac{1}{x_{n - 1} - x_k}
\end{vmatrix}.
\end{equation*}
Since 
$$
\lim_{x_1 \to \infty} \sum_{k = 2}^n \frac{x_1}{x_1 - x_k} = n - 1,
$$
by evaluating the limit, the determinant becomes 
\begin{equation*}
\begin{split}
&(n - 1) \begin{vmatrix}
\sum\limits_{k = 3}^n \frac{1}{x_2 - x_k} & \frac{1}{x_2 - x_3} 
& \frac{1}{x_2 - x_4} & \cdots & \frac{1}{x_2 - x_{n - 1}} 
\\ \frac{1}{x_3 - x_2} & \sum\limits_{k = 2, \ne 3}^n \frac{1}{x_3 -  x_k} 
& \frac{1}{x_3 - x_4} & \cdots & \frac{1}{x_3 - x_{n - 1}} 
\\ \vdots & \vdots & \ddots & & \vdots 
\\ \frac{1}{x_{n - 1} - x_2} & \frac{1}{x_{n - 1} - x_3} 
& \cdots & \cdots & \sum\limits_{k = 2, \ne n - 1}^{n} \frac{1}{x_{n - 1} - x_k}
\end{vmatrix} \\
&= (n - 1)\frac{d_{n - 1}}{\prod_{2 \le k < n} (x_k - x_n)}.
\end{split}
\end{equation*}
Thus $d_n = (n - 1) d_{n - 1}$ as expected. \qed
\medbreak

\noindent
{\bf The second/algebraic proof.}\enspace
It is enough to consider the specialization 
$x_i = \zeta^i$ for $i = 1, \ldots, n$, where $\zeta = e^{2\pi i/n}$ 
is the $n$-th primitive root of unity. 
Let $A' = (a_{ij}')$ be the specialized matrix and $A'' = (a_{ij}'')$ 
the Hermitian matrix with
$$
a_{ij}'' := \zeta^i a_{ij}' = \begin{cases} 
\displaystyle \frac{1}{1 - \zeta^{j - i}} &\mbox{if $i \ne j$}, 
\\ \displaystyle \sum_{k = 1}^{n - 1} \frac{1}{1 - \zeta^k} 
= \frac{n - 1}{2}&\mbox{if $i = j$}. \end{cases}
$$
Here the diagonal entries $a_{ii}'' = \tfrac{1}{2}(n - 1)$ follows from the fact that 
$$
\frac{1}{1 - \zeta^k} + \frac{1}{1 - \zeta^{n - k}} 
= \frac{1 - \zeta^{n - k} + 1 - \zeta^k}{1 - \zeta^k - \zeta^{n - k} + 1}= 1.
$$ 

Let $V$ be the underlying vector space of the group ring 
$\bar {\mathbb Q}[\mathbb Z/n\mathbb Z] 
= \bigoplus_{j \in \mathbb Z/n\mathbb Z} \bar{\mathbb Q}\cdot [j]$. 
Then $(A'')^t$ is the matrix representation of the following operator 
${\bf T}$ on $V$ with respect to the basis $[\bar 1], [\bar 2], \cdots [\bar n]$:
$$
{\bf T} = \frac{n - 1}{2} + \sum_{j = 1}^{n - 1} \frac{1}{1 - \zeta^j} [j]. 
$$
We put the Hermitian inner product on $V$ so that $[i]$'s are
orthonormal. 
It is easy to diagonalize $A''$. Indeed, for $a \in \mathbb Z/n \mathbb Z$, let
$$
x_a := \sum_{i \in \mathbb Z/n \mathbb Z} \zeta^{-ia} [i] 
\in \bar {\mathbb Q}[\mathbb Z/n\mathbb Z].
$$
Then $V$ is also the orthogonal direct sum of the one dimensional
subspaces $V_a := \bar{\mathbb Q} \cdot x_a$. 
It is easily seen that $[j] \cdot x_a = \zeta^{ja} x_a$. 
Hence $x_a$'s are eigenvectors of ${\bf T}$ with eigenvalues
$$
\lambda_a = \frac{n - 1}{2} + \sum_{j = 1}^{n - 1} \frac{\zeta^{ja}}{1 - \zeta^j}.
$$

In fact $\lambda_a = a - 1$ for $a = 1, \ldots, n$. 
To see this, we rewrite $\lambda_a$ as
$$
\lambda_a = (n - 1) - \sum_{j = 1}^{n - 1} \frac{1 - \zeta^{ja}}{1 - \zeta^j} 
= (n - 1) - \sum_{j = 1}^{n - 1}\sum_{k = 0}^{a - 1} \zeta^{jk}.
$$
By changing the order of summation, for $k = 0$ we get $n - 1$, 
while for $k = 1, \ldots, a - 1$ we get $\sum_{j = 1}^{n - 1}\zeta^{jk} = -1$. 
Hence $\lambda_a = a - 1$ as expected.

The diagonalization in terms of matrices reads as
$$
CA'' = 
{\scalebox{0.7}{
\(\begin{pmatrix} 
0 \\ &1 \\ &&\ddots \\ &&& n - 1\end{pmatrix}\) 
}}C,
$$
where the $a$-th row vector of $C = (z_{ij})_{1 \le i, j \le n}$,
$z_{ij} := \zeta^{-ij}$ corresponds to $x_a$.

Now we work on $\Lambda^{n - 1}V$ and $\Lambda^{n - 1}{\bf T} 
\in {\rm End}(\Lambda^{n - 1}V)$. For a square matrix $B$, 
$\Lambda^{n - 1}B = {\rm adj}(B)^t$ is the "non-transposed" cofactor matrix. 
It has the covariant property that $\Lambda^{n - 1}B_1 B_2 
= (\Lambda^{n - 1} B_1) (\Lambda^{n - 1} B_2)$. We find
$$
\Lambda^{n - 1}A'' = \Lambda^{n - 1}C^{-1} 
{\scalebox{0.7}{\(\begin{pmatrix} (n - 1)! 
\\ &0 \\ &&\ddots \\ &&& 0\end{pmatrix} \)}}
\Lambda^{n - 1}C.
$$
Hence
\begin{equation} \label{Ani}
(\Lambda^{n - 1}A'')_{ni} = (n - 1)! (\Lambda^{n - 1}C^{-1})_{n1}(\Lambda^{n - 1}C)_{1i}.
\end{equation}

To compute the right hand side, from $C.\bar C^t = nI_n$ 
and $C.(\Lambda^{n - 1}C)^t = (\det C) I_n$, we get 
$$
\Lambda^{n - 1}C = n^{-1}(\det C) \bar C.
$$ 
Also $C^{-1} = n^{-1} \bar C^t$. The same reasoning implies that
$$
\Lambda^{n - 1}C^{-1} = n^{-(n - 1)} \Lambda^{n - 1}\bar C^t 
= n^{-n} (\det\bar C^t) C^t = (\det C)^{-1} C^t.
$$
In particular, (\ref{Ani}) becomes
$$
(\Lambda^{n - 1}A'')_{ni} = \frac{(n - 1)!}{n} \zeta^i.
$$

By definition of $c_i$, the equation (\ref{d_n}) for $d_n$ 
specialized to $x_i = \zeta^i$ reads as 
(notice that $\prod_{j = 1}^{n - 1}(1 - \zeta^j) = n$)
$$
(-1)^{n + i}(\Lambda^{n - 1}A')_{ni} 
= \frac{(-1)^{n + i}d_n}{\prod_{k \ne i}(\zeta^k - \zeta^i)} 
= \frac{(-1)^{n + i} d_n \zeta^i (-1)^{n - 1}}{n}.
$$
Since $(-1)^{n + i}(\Lambda^{n - 1}A')_{ni} 
= (-1)^{n + i}(-1)^{n - 1} (\Lambda^{n - 1}A'')_{ni}$, 
the above two expressions lead to $d_n = (n - 1)!$.
\qed

\begin{remark*}
We found the algebraic proof first which gives the value $d_n = (n -1)!$. 
The shorter and more elementary analytic proof came much later which
was inspired by the factorial nature of $d_n$. Then we were informed
by Y.~Zarhin that Lemma \ref{dn} appeared in \cite[\S1]{Rees},
with a different proof.
\end{remark*}

\begin{corollaryss} \label{Xn-in-Yn}
For $P_1, \ldots, P_n \in E$ satisfying $P_i \ne P_j$ for $i \ne j$
and \(P_i\neq - P_j\) for all \(i, j\), 
the system of equations \textup{(\ref{Alg-eqn})} is equivalent to 
\textup{(\ref{ratio-eq})}, 
hence also equivalent to \textup{(\ref{zeta-eq})}.
\end{corollaryss}

\begin{remarkss} \label{rem:deg-solutions}
We record two easy observations about the 
system of linear equations
\begin{equation}\label{sys-pow-eq}
  \sum_{i=1}^n s_i^l\cdot Y_i=0, \qquad l=0,\ldots, n-2
\end{equation}
with \(s_1,\ldots, s_n\) in \(\CC\).
\smallbreak
\noindent
(a) If \(s_1,\ldots,s_n\) are mutually distinct, and
\((y_1,\ldots, y_n)\) is a solution of (\ref{sys-pow-eq})
in which one \(y_i\) is \(0\), then 
all \(y_j\)'s are equal to zero.
\smallbreak
\noindent
(b) If \(s_1,\ldots, s_n\) are not mutually distinct,
then (\ref{sys-pow-eq}) only has trivial solutions
in the following sense:
We have a set
\(\{t_1,\ldots, t_m\}\) consisting of mutually distinct numbers
such that \(\{s_1,\ldots, s_n\}=\{t_1,\ldots,t_m\}\).
Suppose that \((y_1,\ldots, y_n)\) is a solution of (\ref{sys-pow-eq}),
let \(z_j:=\sum_{\,\textup{all}\ i\ \textup{s.t.}\ s_i=t_j}y_i\) for \(i=1,\ldots, m\).
Then the system of linear equations 
for \(y_1,\ldots, y_n\) becomes 
\[\sum_{j=1}^m s_j^l\cdot z_j=0\qquad \textup{for}\ l=0,\ldots, n-2,\]
and 
\(z_1=\cdots=z_m=0\) by the non-vanishing of the Vandermonde 
determinant.

\end{remarkss}

\section{Lam\'e for type II: Characterizations of $X_n$ and $Y_n$} 
\label{X_n}
\setcounter{equation}{0}

\subsection{An overview for this section}

\subsubsection{}
In \S \ref{II-even}, we have proved that for each positive integer
\(n\),
for every solution \(u\) of the mean field equation
\begin{equation}\label{eq-liouville-sec6}
\triangle u+e^u=8\pi n\cdot \delta_0\qquad \textup{on}\ \ \CC/\Lambda
\end{equation}
there exists a set $a = \{a_1, \ldots, a_n\}$
of \(n\) complex numbers which satisfies (\ref{G-eqn}), 
(\ref{Alg-eqn}) and (\ref{p-p}) such that
\begin{equation} \label{fa}
f(z) = f_a(z) := \prod_{i = 1}^n \exp \int_0^z
\frac{\wp'(a_i)}{\wp(w) - \wp(a_i)}\,dw,
\end{equation}
is a normalized type II developing map of \(u\).
Moreover every set $a = \{a_1, \ldots, a_n\}$ of complex numbers
satisfying conditions (\ref{partial-G-eqn}), (\ref{Alg-eqn}) and (\ref{p-p}) gives rise to a solution 
of the above mean field equation.

\subsubsection{}
In this section we will leave the Green equation (\ref{partial-G-eqn}) alone
and consider those $a = \{a_1, \ldots, a_n\}$ satisfy only the
equations (\ref{Alg-eqn}) under the constraint (\ref{p-p}), 
that is, we consider $a$ in the set $X_n$ defined in (\ref{X-n}) 
in the introduction. We would like to characterize $a \in X_n$ 
in terms of certain Lam\'e equations.

\subsubsection{}
We will make use of the following addition formulas freely:
\begin{equation} \label{add-1}
\frac{\wp'(z)}{\wp(z) - \wp(u)} = \zeta(z - u) + \zeta(z + u) -
2\zeta(z),
\end{equation}
\begin{equation} \label{add-1b}
\frac{\wp'(u)}{\wp(z) - \wp(u)} = \zeta(z - u) - \zeta(z + u) +
2\zeta(u),
\end{equation}
\begin{equation} \label{add-2}
\frac{1}{2} \frac{\wp'(z) - \wp'(u)}{\wp(z) - \wp(u)} = \zeta(z +
u) - \zeta(z) - \zeta(u),
\end{equation}
\begin{equation} \label{add-3}
\frac{1}{4} \left(\frac{\wp'(z) - \wp'(u)}{\wp(z) -
\wp(u)}\right)^2 = \wp(z + u) + \wp(z) + \wp(u).
\end{equation}

\begin{definitionss}\label{def:f_a}
Let \(\Lambda\) be a cocompact lattice in \(\CC\).
Let \(n\geq 1\) be a positive integer.
Let \(\,[a]=\{[a_1],\ldots,[a_n]\}\,\) be an unordered list of \(n\)
elements in \(\,(\CC/\Lambda)\smallsetminus \{[0]\}\),
possibly with multiplicity.
Define a meromorphic function \(f_{[a]}(z)\) on \(\CC\) by
\begin{equation} \label{f_a}
f_{[a]}(z)=f_{[a]}(z;\Lambda):=\prod_{i = 1}^n \exp \int_0^z (\zeta(w - a_i) 
- \zeta(w + a_i) + 2\zeta(a_i))\,dw.
\end{equation}
where \(a_i\) is a representative in \(\CC\) of \([a_i]\) for
each \(\,i=1,\ldots,n\).

Note that \(\,f_{[a]}\,\) 
depends only on 
the element \(\,\{[a_1],\ldots, [a_n]\}\,\) of 
the symmetric product \(\,\textup{Sym}^n\big(\CC/\Lambda \smallsetminus
\{[0]\}\big)\,\) and not on the choice of
representatives \(a_i\in [a_i]\).
Because \(\zeta(z;\Lambda) = \tfrac{d}{dz}\log \sigma(z;\Lambda)\),
we get from (\ref{add-1b}) an equivalent definition 
\begin{equation} \label{f-sigma}
f_{[a]}(z):= 
{(-1)^n}\cdot e^{2z\sum_{i = 1}^n \zeta(a_i)} \cdot
\prod_{i = 1}^n \frac{\sigma(z - a_i)}{\sigma(z + a_i)}\,.
\end{equation}
Note also that \(\,f_{[a]}(0)=1\,\) 
and \(\,f_{[a]}(-z)\!\cdot\!f_{[a]}(z)=1\) for
all \(z\). 
\end{definitionss}

\begin{definitionss}\label{def:w_a}
Let \(a=\{a_1,\ldots,a_n\}\) be an unordered list of elements
of \(\CC\smallsetminus \Lambda\).
The Hermite-Halphen ansatz function \(w_a(z)\) attached
to the list \(a\) is the meromorphic function on \(\CC\) defined by
\begin{equation} \label{w_a}
w_a(z)=w_a(z;\Lambda)
: = e^{z \sum \zeta(a_i)} \prod_{i = 1}^n \frac{\sigma(z - a_i)}{\sigma(z)}\,.
\end{equation}
\end{definitionss}

\begin{remark*} 
(a) In classical literature the functions \(w_a(z)\) arise as explicit 
solutions of the Lam\'e equation 
\begin{equation} \label{Lame-n}
w'' = \big(n(n + 1) \wp(z) + B\big) w;
\end{equation}
see \cite[I--VII]{Hermite}, \cite[p.\,495--497]{Halphen} and also 
\cite[\S23.7]{Whittaker}.

(b) Clearly we have
\[
f_{[a]}(z)=\frac{w_a(z)}{w_{-a}(z)}\,,
\]
where \(-a\) is the list \(\{-a_1,\ldots, -a_n\}\) and
\([a]\) is the list \(\{[a_1],\ldots,[a_n]\}\).

(c) If \(b=\{b_1,\ldots, b_n\}\) is a list such that \(b_i-a_i\in \Lambda\)
for all \(i=1,\ldots, n\), then
\(\,\frac{w_b}{w_{a}}\in \CC^{\times}\), a non-zero constant.
\end{remark*}

\begin{lemmass}\label{lemma:f_a}
If a list \(\,[a]=\{[a_1],\ldots, [a_n]\}\,\) of
\(n\) elements of \((\CC/\Lambda)\smallsetminus\{[0]\}\) 
satisfies \textup{(\ref{partial-G-eqn})}, 
\textup{(\ref{Alg-eqn})} and the non-degeneracy condition
\textup{(\ref{p-p})}, then
there exists a constant \(B=B_{[a]}\) such that
the Schwarzian derivative of $f_{[a]}$ satisfies
$$
S(f_{[a]}) = -2\big(n(n + 1) \wp(z;\Lambda) + B_{[a]}\big).
$$
\end{lemmass}
\begin{proof}
By Theorem \ref{thm:hecke-system}, \(f_{[a]}\) 
is a normalized developing map for the
mean field equation \textup{(\ref{eq-liouville-sec6})}, 
and the assertion follows from \textup{(\ref{S(f)})}.
\end{proof}

\subsubsection{}
The constant \(B_{[a]}\) in Lemma \ref{lemma:f_a} can be evaluated by
a straightforward computation; the answer is
\begin{equation} \label{B-value}
B_{[a]} = (2n - 1) \sum_{i = 1}^n \wp(a_i;\Lambda).
\end{equation}
On the other hand there is a proof of the formula (\ref{B-value}) 
for \(B_{[a]}\) without resorting to messy computations,  
via Lam\'e's differential
equation (\ref{Lame-n})
because $f_{[a]}$ can be written as the ratio of two 
linearly independent solutions 
of (\ref{Lame-n}). The idea is this:
use the Hermite--Halphen ansatz functions \(w_a(z)\)
to find solutions to Lam\'e equations,
and the constant $B$ can be computed from 
the ansatz solutions $w_a$. 
Then $f_{[a]}a = w_a/w_{-a}$ has the expected 
Schwarzian derivative 
by ODE theory. 

We take this approach since 
it requires less computation to prove the formula (\ref{B-value})
for \(B\), and
it leads to a characterization of the set
$Y_n$ defined in (\ref{Yn1}) 
as the set of all unordered lists  
\(\,a=\{[a_1],\ldots, [a_n]\}\) of \(n\) elements
in \((\CC/\Lambda)\smallsetminus\{[0]\}\)
such that $w_a(z;\Lambda) $ satisfies a Lam\'e equation
(\ref{Lame-n}) for some \(B\in\CC\),
see Theorem \ref{lame}. 
We then move back to characterize the set 
\(X_n\) defined in (\ref{X-n})
as the set of all \(a\)'s such that
${\rm ord}_{z = 0} f'_{[a]}(z) = 2n$, which is the highest possible value
of \(\,\textup{ord}_{z=0}f_{[a]}(z)\); see
Theorem \ref{poly-eqn}. This leads to the important consequence that 
for $a \in Y_n$, $a \not\in X_n$ if and only if $a = -a$, 
and a characterization of $X_n$ via the Schwarzian
derivative. 


The following result is known in the literature, see e.g.~\cite{Halphen}. 
We reproduce it here for the sake of completeness.

\begin{theorem} [Characterization of $Y_n$] \label{lame}
Let \(n\geq 1\) be a positive integer.
Let $a = \{a_1, \ldots, a_n\}$ be an unordered list 
of \(n\) elements in \(\CC\smallsetminus \Lambda\).
Let $w_a$ be defined as in \textup{(\ref{w_a})}. 
Let \([a]\) be the unordered list \(\{[a_1],\ldots,[a_n]\}\), where
\([a_i]:=a_i\,\textup{mod}\,\Lambda \in \CC/\Lambda\) for each \(i\).
\begin{itemize}
\item[\textup{(1)}] There exists a constant \(B\in \CC\)
such that the
meromorphic function $w_a$ on \(\CC\)
satisfies the Lam\'e equation \textup{(\ref{Lame-n})}
if and only if the following conditions hold.
\begin{itemize}
\item \([a_i]\neq [a_j]\) whenever \(i\neq j\), and

\item the $a_i$'s satisfy 
\begin{equation} \label{zeta-eqn}
\sum_{j \ne i} \big(\zeta (a_i - a_j;\Lambda) 
- \zeta(a_i;\Lambda) + \zeta (a_j;\Lambda)\big) = 0, 
\qquad i = 1, \ldots, n.
\end{equation}
\end{itemize}
In other words the necessary and sufficient
condition is that $a$ is a point of the variety 
\(Y_n\) in the notation of  
\textup{(\ref{X-n})}.

\item[\textup{(2)}] If the system of equations 
\textup{(\ref{zeta-eqn})} holds
for \(a\), then \(w_a\) satisfies the Lam\'e equation
\textup{(\ref{Lame-n})}
whose accessary parameter \(B\) of the
equation is 
$$
B = B_{[a]}= (2n - 1) \sum_{i = 1}^n \wp(a_i;\Lambda).
$$
\end{itemize}
\end{theorem}

\begin{proof}
If there are two indices \(i_1\neq i_2\) such that
\([a_{i_1}]=[a_{i_2}]\), then \(w_{a}(a_{i_1})=
w_a'(a_{i_1})=0\).
If \(w_a(z)\) is a solution of the second order linear ODE
\textup{(\ref{Lame-n})}, then all 
higher derivatives of \(w_a(z)\) vanish at
\(z=a_{i_1}\), so \(w_a(z)\) is identically zero, a 
contradiction. We have shown that the \([a_i]\)'s must be
mutually distinct if \(w_a(z)\) is a solution of (\ref{Lame-n}).
\smallbreak

The logarithmic derivative
\begin{equation*}
\frac{w_a'(z;\Lambda)}{w_a(z;\Lambda)} 
= \sum_i \big(\zeta(a_i;\Lambda) 
+ \zeta(z - a_i;\Lambda) - \zeta(z\Lambda) \big)
\end{equation*}
of \(w_a\) is an elliptic function on \(\CC/\Lambda\). 
Applying \(\,\frac{d}{dz}\) again, we get
\begin{equation} \label{w''/w}
\begin{split}
\frac{w_a''}{w_a} &= \Big(\frac{w_a'}{w_a}\Big)' +
\Big(\frac{w_a'}{w_a}\Big)^2 \\
&= \sum \big(\wp(z) - \wp(z - a_i)\big) + \sum \big(\zeta(a_i) + \zeta(z - a_i)
- \zeta(z)\big)^2 \\
&\quad + \sum_{i \ne j} \big(\zeta(a_i) + \zeta(z - a_i) -
\zeta(z)\big)\big(\zeta(a_j) + \zeta(z - a_j) - \zeta(z)\big)
\\
&= 2n \wp(z) + \sum_i \wp(a_i)\\
&\quad + \sum_{i \ne j} \big(\zeta(a_i) + \zeta(z - a_i) -
\zeta(z)\big)\big(\zeta(a_j) + \zeta(z - a_j) - \zeta(z)\big)
\end{split}
\end{equation}
where we have used the consequence
\begin{equation*}
\big(\zeta(a_i) + \zeta(z - a_i) - \zeta(z)\big)^2 = \wp(z) + \wp(a_i) +
\wp(z - a_i).
\end{equation*}
of (\ref{add-2}) and (\ref{add-3}) to add up 
the first two sums after the second equality sign 
in (\ref{w''/w}) to get the last expression of 
\(\,\frac{w_a''}{w_a}\) in (\ref{w''/w}).

The sum in the last line of (\ref{w''/w}) is an elliptic function on
\(\CC/\Lambda\)
with a double pole 
at $z = 0$ with Laurent expansion \(\,\frac{n^2}{z^2}+O(1)\);
denote this function by \(F_{a}(z)\).
Therefore \(\,w_a\,\) satisfies a Lam\'e equation
\textup{(\ref{Lame-n})} for some \(B\in\CC\) if and only if
\(F_a(z)\) has no pole outside
of \([0]\in \CC/\Lambda\).

Suppose \([a_{i_0}]\) appears in the list \(\,[a]=\{[a_1],\ldots,[a_n]\}\,\)
\(r\) times with \(r\geq 2\)
for some \(i_0\in\{1,\ldots,n\}\).  Then \(F_a(z)\) has a double pole
at \(z=a_{i_0}\), where it has a Laurent expansion
\[\,F_a(z)={r(r-1)}(z-a_{i_0})^{-2} +
O\big((z-a_{i_0})^{-1}\big).\]
We have shown that if \(F_a(z)\) is holomorphic outside \(\Lambda\),
then \([a_i]\neq [a_j]\) whenever \(i\neq j\).

Under the assumption that \([a_1],\ldots, [a_n]\) are mutually
distinct, 
the function \(F_a(z)\) 
is holomorphic on \((\CC/\Lambda)\smallsetminus \{[z_1],\ldots,[a_n]\}\)
and has at most simple poles at \([z_1],\ldots,[z_n]\). 
Therefore \(F_a(z)\) is holomorphic outside \(\Lambda\) 
if and only if its residue at $z = a_i$ is zero for \(i=1,\ldots, n\),
which means that
\begin{equation*}
\sum_{j \ne i} \big(\zeta(a_j;\Lambda) 
+ \zeta(a_i - a_j;\Lambda) - \zeta(a_i;\Lambda)\big) = 0,
\quad \forall\, 1 \le i \le n.
\end{equation*}
This proves the statement (1) of Theorem \ref{lame}. 
\medbreak

We know that there a constants \(B_1\in\CC\) such that
\begin{equation} \label{mixed-term}
F_a(z)=n(n-1)\wp(z;\Lambda) + B_1,
\end{equation}
because 
\(F_a(z)\) is holomorphic on \(\CC/\Lambda \smallsetminus\{[0]\}\) and
its Laurent expansion at \(z=0\) is 
\(\,n(n-1)\cdot z^{-2}+ O(1)\).
%
To determine $B_1$, we need to compute its Laurent expansion
at \(z=0\) modulo \(O(z)\).
From
\[\zeta(z - a_i;\Lambda) 
= -\zeta(a_i;\Lambda) - \wp(a_i;\Lambda)z + O(z^2),
\]
we get
\begin{equation*}
\begin{split}
F_{a}(z)
&=\sum_{i \ne j} \Big(-\frac{1}{z} - \wp(a_i;\Lambda)z + O(z^2)\Big)
\Big(-\frac{1}{z} - \wp(a_j;\Lambda)z + O(z^2)\Big)\\
&= n(n - 1)\frac{1}{z^2} + 2(n - 1)\sum_i \wp(a_i;\Lambda) + O(z).
\end{split}
\end{equation*}
In particular $B_1 = 2(n - 1)$. 
From (\ref{w''/w})
we get
\begin{equation*}
\frac{w_a''}{w_a} = n(n + 1)\wp(z;\Lambda) 
+ (2n - 1)\sum_{i = 1}^n \wp(a_i;\Lambda).
\end{equation*}
We have proved the statement (2).
\end{proof}

\begin{remarkss}
(a) Clearly that the necessary and sufficient condition in
Theorem \ref{lame}\,(1), which defines the variety \(Y_n\),
depends only on the list \([a]=\{[a_1],\ldots,[a_n]\}\) of
elements in \((\CC/\Lambda)\smallsetminus\{[0]\}\) determined by \(a\).

(b) It is also clear that a list \(a=\{a_1,\ldots, a_n\}\) 
satisfies the condition in \ref{lame}\,(1) if and only
if the list \(-a=\{-a_1,\ldots, -a_n\}\) does.

\end{remarkss}

\begin{proposition}\label{prop:spectral_Yn}
Let \(a=\{a_1,\ldots,a_n\}\) be an unordered list
of elements in
\(\CC\smallsetminus \Lambda\), \(n\geq 1\).
\begin{itemize}
\item[\textup{(1)}] The function \(w_a(z)\) is a
common
eigenvector for the translation action by 
elements of \(\Lambda\):
\begin{equation*}
\frac{w_a(z+\omega)}{w_a(z)}
=e^{\omega\cdot \sum_{i=1}^n\zeta(a_i;\Lambda) - 
\eta(\omega;\Lambda)\cdot \sum_{i=1}^n a_i}\qquad
\forall \omega\in\Lambda.
\end{equation*}
This ``eigenvalue package'' attached to \(w_a\) is the homomorphism
\[
\chi_{{a}}:\Lambda\to \CC^{\times},
\quad 
\omega\mapsto
e^{\omega\cdot \sum_{i=1}^n\zeta(a_i;\Lambda) - 
\eta(\omega;\Lambda)\cdot \sum_{i=1}^n a_i}\quad
\forall\omega\in \Lambda,\]
which depends only on the list \([a]=\{[a_1],\ldots,[a_n]\}\). 
\smallbreak 

\item[\textup{(2)}]
If \(w_a\) satisfies a Lam\'e equation
\textup{(\ref{Lame-n})}, then so does
\(w_{-a}\)

\item[(3)] For any unordered list
\(b=\{b_1,\ldots,b_n\}\) 
of elements in
\(\CC\smallsetminus \Lambda\), the functions 
\(w_a\) and \(w_{b}\) are linearly dependent
if and only if 
either \([b]=[a]\) or
\([b]=[-a]\), where \([-a]\) is the unordered list
\(\{[-a_1],\ldots, [-a_n]\}\) of elements
in \(\CC/\Lambda\).

\item[(4)] The 
homomorphisms \(\chi_{{a}}\) and \(\chi_{-a}\)
are equal if and only if
%
there exists an element
\(\omega\in \Lambda\) such that
\begin{equation}\label{cond-central-mono}
\sum_{i=1}^n\zeta(a_i;\Lambda)=\frac{\eta(\omega;\Lambda)}{2}
\quad\textup{and}\quad
\sum_{i=1}^n a_i=\frac{\omega}{2},
\end{equation}
in which case \(\,\textup{Im}(\chi_a)\subseteq
\{\pm 1\}\).

\item[(5)] Suppose that \(w_a\) and \(w_{-a}\) are two solutions
of a Lam\'e equation \textup{(\ref{Lame-n})}, and
\([a]\neq [-a]\).  Then 
\(\,\chi_a\neq \chi_{-a}\).  Moreover
\(\CC\cdot w_a\) and  \(\CC\cdot w_{-a}\) are characterized
by the monodromy representation of \textup{(\ref{Lame-n})}
as the two one-dimensional subspaces of solutions which are
stable under the monodromy.
%
%

\end{itemize}

\end{proposition}

\begin{proof}
The statements (1) is immediate from the transformation formula
for the Weierstrass \(\sigma\) function.
The statements (2) and (3) are obvious and easy respectively. 
The statement (4) is a consequence of the Legendre relation
for the quasi-periods.

Suppose that \([a]\neq [-a]\)
 and \(\chi_a=\chi_{-a}\).
By (3) the monodromy representation 
of the Lam\'e equation (\ref{Lame-n})
is isomorphic to 
the direct sum \(\chi_a\oplus\chi_{b}\), and
the character \(\chi_a\) has order at most \(2\)
by (4). Consider the algebraic form 
\begin{equation}\label{eq-lame-alg-n}
p(x) \frac{d^2y}{dx^2}+\frac{1}{2}p'(x) \frac{dy}{dx}
-\big(n(n+1)x+B\big)y=0
\end{equation}
of the Lam\'e equation (\ref{Lame-n}).
The monodromy group \(M\) of (\ref{eq-lame-alg-n})
contains the monodromy group (\ref{Lame-n})
as a normal subgroup of index at most \(2\),
therefore \(M\) is a finite abelian group of
order dividing \(4\). In particular the
monodromy representation of the algebraic Lam\'e equation
(\ref{eq-lame-alg-n}) is completely reducible.
However one knows from \cite[Thm.\,4.4.1]{Waall} or  
\cite[Thm.\,3.1]{BW} that 
the monodromy representation of (\ref{eq-lame-alg-n})
is not completely reducible, a contradiction.
We have proved the first part of (5).
The second part of (5) follows from the
first part of (5).
\end{proof}

\begin{proposition}\label{cor:Ba_and_a}
Suppose that \([a]=\{[a_1],\ldots, [a_n]\}\)
and \([b]=\{[b_1],\ldots, [b_n]\}\) are two
points of \(\,Y_n\), \(n\geq 1\).
If 
\(\sum_{i=1}^n \wp(a_i;\Lambda)
=\sum_{i=1}^n \wp(b_i;\Lambda),
\)
the either \([a]=[b]\) or 
\([a]=[-b]\).

\end{proposition}

\begin{proof}
Pick representatives \(a_i\in [a_i]\) and \(b_i\in [b_i]\)
for each \(i=1,\ldots,n\).
Suppose that \([b]\neq [a]\) and \([b]\neq [-a]\).
The functions \(w_a(z)\) and \(w_b(z)\) are linearly independent
by Proposition \ref{prop:spectral_Yn}\,(3) because
\([b]\neq [a]\), and they satisfy
the same Lam\'e differential equation because
\(B_{[a]}=B_{[b]}\). 
By either \cite[Thm.\,4.4.1]{Waall} or  
\cite[Thm.\,3.1]{BW}, 
that image of the monodromy representation of the Lam\'e equation
\(\frac{d^2w}{d^2z}- 
\big(n(n+1) \wp(z;\Lambda)+B_{[a]}\big)w=0
\) is not contained in \(\CC^{\times}\textup{I}_2\), for 
otherwise the monodromy group of the algebraic form of
the above Lam\'e equation on \(\PP^1(\CC)\) is 
contained in the product of \(\CC^{\times}\textup{I}_2\)
with a subgroup of order two in \(\textup{GL}_2(\CC)\).
So \(\CC\cdot w_a\) and \(\CC\cdot w_b\) are the two 
distinct common eigenspaces of the monodromy representation
of the above Lam\'e equation on \(\CC/\Lambda\).
If follows that \(\CC\cdot w_{-a}=\CC\cdot w_{a}\) and 
\(\CC\cdot w_{-b}=\CC\cdot w_{-b}\), i.e. \([a]=[-a]\) 
and \([b]=[-b]\).  Therefore the monodromy group 
of the above Lam\'e equation divides \(4\),
and the monodromy of the algebraic form of the same
Lam\'e equation divides \(8\), which again contradicts
\cite[Thm.\,4.4.1]{Waall} and   
\cite[Thm.\,3.1]{BW}.
\end{proof}



\begin{theorem} [Characterization of $X_n$ by 
$\textup{ord}_{z=0}\,f_{[a]}'(z)$] 
\label{poly-eqn} 
Let \(n\geq 1\) be a positive integer.
Let $a = \{[a_1], \ldots, [a_n]\}$ be an unordered list 
of \(n\) non-zero points on the elliptic curve \(\CC/\Lambda\).
Let \(a_1,\ldots, a_n\) be representatives of 
\([a_1],\ldots, [a_n]\) in \(\CC\smallsetminus \Lambda\).
\begin{itemize}
\item[\textup{(0)}] \(f_{[a]}\) is a constant if and only if 
\([a]=[-a]\), where \([-a]\) is  the unordered list
\(\{[-a_1],\ldots, [-a_n]\}\) of \(n\) non-zero elements 
in the elliptic curve \(\CC/\Lambda\).

\item[\textup{(1)}]
If 
\([a]\neq [-a]\), 
then ${\rm ord}_{z = 0}\,f_{[a]}'(z) \le 2n$.

\item[\textup{(2)}]
Assume that \([a]\neq [-a]\). Then
\(\textup{ord}_{z=0}f_{[a]}'(z)=2n\) if and only if 
%
the Weierstrass coordinates 
\((\wp(a_i;\Lambda), \wp'(a_i;\Lambda))\)
of 
\([a_1],\ldots,[a_n]\) in
\(\CC/\Lambda\) satisfy the following
system of polynomial equations.
\begin{equation} \label{Xn}
\sum_{i = 1}^n \wp'(a_i;\Lambda)\cdot 
\wp^{k}(a_i;\Lambda) = 0,\quad\textup{for}\ 
\ k = 0, \ldots, n - 2.
\end{equation}
Moreover if the above equivalent conditions hold, then
\begin{itemize}
\item \(\wp'(a_i;\Lambda) \ne 0\) for all \(i=1,\ldots, n\), and
\item \(\wp(a_i;\Lambda) \ne \wp(a_j;\Lambda)\) whenever \(i\neq j\).
\end{itemize}
In other words \([a]\) is a point of the variety \(X_n\)
defined in \textup{(\ref{X-n})}.
\end{itemize}
\end{theorem}

\begin{proof}
The divisor \(\textup{div}(f_{[a]})\) 
of the meromorphic function \(f_{[a]}\) on \(\CC\) is 
stable under translation by \(\Lambda\), and
\(\,\textup{div}(f_{[a]})\,\textup{mod}\,\Lambda\,\) is the formal sum
(or \(0\)-cycle)
\[\sum_{i=1}^n [a_n] - \sum_{i=1}^n [-a_n]\] of points of \(\CC/\Lambda\).
So if \(f_{[a]}\) is a constant, the above formal sum is \(0\),
meaning that \([a]=[-a]\).  Conversely if \([a]=[-a]\), then
\[\sum_{i=1}^n\zeta(a_i;\Lambda)= -\sum_{i=1}^n\zeta(a_i;\Lambda)=0\]
and \(f_{[a]}\) is equal to the constant function \(1\).
We have proved statement (0).
\medbreak

Let $(x_i, y_i) := (\wp(a_i;\Lambda), \wp'(a_i;\Lambda))$. We have
\begin{equation} \label{log-der}
f_{[a]} ' = f_{[a]}\cdot\sum_{i = 1}^n
\frac{\wp'(a_i;\Lambda)}{\wp(z;\Lambda) - \wp(a_i;\Lambda)}
\end{equation}
and
\begin{equation*}
\sum_{i = 1}^n \frac{y_i}{\wp(z;\Lambda) - x_i} 
= \sum_{i = 1}^n \frac{y_i\, \wp(z;\Lambda)^{-1}}{1 - x_i \,\wp(z;\Lambda)^{-1}} 
= \sum_{k = 1}^\infty \Big( \sum_{i = 1}^n y_i x_i^k \Big)\, 
\wp(z;\Lambda)^{-k-1}.
\end{equation*}
We conclude that
\begin{itemize}
\item \(\textup{ord}_{z=0}\, f_{[a]} (z) > 2n \) if and only if
\(\sum_{i=1}^n y_i x_i^k=0\) for \(k=0, \ldots, n-1\),

\item \(\textup{ord}_{z=0}\, f_{[a]} (z) = 2n \)
if and only if $\sum_{i = 1}^n y_i x_i^k = 0$ 
for $0 \le k \le n - 2$ while $\sum_{i = 1}^n y_i x_i^{n - 1} \ne 0$. 
\end{itemize}
\bigbreak

Suppose that \(\,\textup{ord}_{z=0}\,f_{[a]} '(z)>2n\), i.e.\
\(\sum_{i=1}^n y_i x_i^k=0\) for \(k=0, \ldots, n-1\).
If \(x_1,\ldots, x_n\) are distinct, we get from the non-vanishing
of the Vandermonde determinant that \(y_1=\cdots=y_n=0\),
meaning that \([a_1],\ldots, [a_n]\) are all \(2\)-torsion points.
That contradicts the assumption that \([a]\neq [-a]\).
So we know that \(x_1,\ldots, x_n\) are not all distinct.
Apply the argument in Remark \ref{rem:deg-solutions}\,(b):
let \(\{s_1,\ldots, s_m\}=\{x_1,\ldots,x_n\}\) as sets without
multiplicity, let 
\[z_j:=\sum_{\,\textup{all}\ i\ \textup{s.t.}\ s_i=t_j}y_i
\qquad \textup{for}\ \ i=1,\ldots m,\]
and we have \(z_1=\cdots =z_m=0\).
Note that for each \(j\), the \(y_i\)'s which appear in the
sum defining \(z_j\) differ from each other at most by a sign
\(\pm 1\), so that the sum \(z_j\) is either \(0\) or 
a non-zero multiple of a \(y_i\).
Note that we have cancelled a number of pairs \(([a_{i_1}], [-a_{i_2}])\)
in forming the reduced system of equations
\[\sum_{j=1}^m z_j s_j^k=0\qquad \textup{for}\ \ k=0, \ldots, n-1.\]
That \(z_1=\cdots=z_m=0\) means that, after removing 
a number of pairs \(([a_{i_1}], [-a_{i_2}])\) from the
unordered list \([a]\), we are left with another
unordered list \([b]\) with \([b]=[-b]\).
So again we have \([a]=[-a]\), a contradiction with proves 
the statement (1). The first part of statement (2) follows.
\medbreak

It remains to prove the second part of (2).
We are assuming that \([a]\neq [-a]\) and
\(\sum_{i=1}^n y_i x_i^k=0\) for \(k=0, \ldots, n-2\).
If there exists \(i_1, i_2\) between \(1\) and \(n\)
such that \(\,x_{i_1}=x_{i_2}\), the same argument in the previous
paragraph produces a contradiction that \([a]=[-a]\).
Therefore \(x_1,\ldots, x_n\) are mutually distinct.
If there is an \(i_0\) such that \(y_{i_0}=0\), then
\(y_1=\cdots=y_n=0\) by Remark \ref{rem:deg-solutions}\,(a),
contradicting the assumption that \([a]\neq [-a]\).
\end{proof}
%
%
%
%

We have seen in Proposition \ref{Xn-in-Yn} that
\(X_n\subset Y_n\), where 
\(X_n\) is defined in (\ref{X-n}) and \(Y_n\) is defined in
(\ref{Yn1}).
The following proposition, which is a consequence of
Theorem \ref{poly-eqn}, describes the complement of \(X_n\) in \(Y_n\).

\begin{proposition} \label{XnYn}
Let \([a]=\{[a_1],\ldots,[a_n]\}\) be a point of \(Y_n\), i.e.\
\([a_i]\neq [0]\) for each \(i\), 
\([a_i]\neq [a_j]\) whenever \(i\neq j\) and the equations
\textup{(\ref{zeta-eqn})} hold.
Then \([a]\in X_n\) if and only if \([a]\neq [-a]\).
\end{proposition}

\begin{proof}
The ``only if'' part is  part of the definition of \(X_n\).
Assume that $[a] \ne [-a]$.  We mush show that $a \in X_n$. 
We know from Theorem \ref{lame} and \ref{poly-eqn}\,(0)
that $w_a$ and $w_{-a}$ 
are linearly independent solutions of the Lam\'e equation (\ref{Lame-n}). 
If $a \not\in X_n$, then the lists
\([a]\) and \([-a]\) have common members.
So either (A) there exists two indices \(i_1, i_2\) such that
\(i_1\neq i_2\) and \([a_{i_1}]=[-a_{i_2}]\), or (B) there exists
an index \(i_3\) such that \([a_{i_3}]=[-a_{i_3}]\).
We start with a non-canonical process to reduce the length of the list
\([a]\) while keeping the associated functions \(w_a\) and \(w_{-a}\)
unchanged up to some non-zero constants:
First remove all \(a_i\)'s such that \([a_i]=[-a_i]\) from the list
\(a\).  In the resulting reduced list, remove
both \(a_{i_1}\) and \(a_{i_2}\) from the list if \(i_1\neq i_2\) and
\([a_{i_1}]=[-a_{i_2}]\). Keep doing so until we get a sublist
\(b=\{b_1,\ldots, b_m\}\) of \(a\) such that 
\([b]\) and \([-b]\) have no common members, \(m<n\),
and there exists a non-zero constant \(C\in\CC^{\times}\) such that
\[
f_{[b]}=\frac{w_b}{w_{-b}}= C\cdot \frac{w_a}{w_{-a}} = C\cdot f_{[a]}.
\]
%
The Schwarzian derivative \(S(f_{[a]})\) satisfies
\begin{equation} \label{Sfn}
S(f_{[a]}) = -2(n(n + 1) \wp(z) + B_{[a]}).
\end{equation}
Let 
\[2\eta := {\rm ord}_{z = 0}\,f_{[a]} '(z)
.\]
Theorem \ref{poly-eqn}\,(1) tells us that
\[2\eta = {\rm ord}_{z = 0}\,f_{[b]}'(z)\leq 2m<2n.\]
But then
$$
S(f_{[a]}) = \frac{f_{[a]} '''}{f_{[a]} '} 
- \frac{3}{2} \Big( \frac{f_{[a]} ''}{f_{[a]} '} \Big)^2 
= -2\eta(\eta + 1) \frac{1}{z^2} + O(1),
$$
which contradicts \eqref{Sfn}. 
Therefore $[a] \in X_n$.
\end{proof}

The characterization of $X_n$ in terms of 
the Schwarzian derivative follows similarly:

\begin{corollary} [Characterization of $X_n$ by $S(f)$] \label{thm-S(f)}
Let \(n\geq 1\) be a positive integer. Let $a_1, \ldots, a_n$ 
be complex numbers in \(\CC\smallsetminus \Lambda\),
let $[a]$ be the unordered list \(\{[a_1], \ldots, [a_n]\}\)
and let \([-a]\)  be the unordered list \(\{[-a_1], \ldots, [-a_n]\}\).
Then $[a] \in X_n$ if and only if 
\([a]\neq [-a]\) and
\begin{equation} \label{S(f)-formula}
S(f_{[a]}) = -2(n(n + 1)\wp(z;\Lambda) 
+ (2n - 1)\sum_{i = 1}^n \wp(a_i;\Lambda)).
\end{equation}
\end{corollary}

\begin{proof}
If $[a] \in X_n$ then \([a]\neq [-a]\) by definition,
and $a \in Y_n$ because \(X_n\subset Y_n\). 
It follows that $f_{[a]} = w_a/w_{-a}$ is a quotient of two linearly
independent solutions of a Lam\'e equation
(\ref{Lame-n}) and the formula 
for $S(f_a)$ follows from Theorem \ref{lame} and the standard ODE theory.

Conversely, if (\ref{S(f)-formula}) holds then 
${\rm ord}_{z = 0}\,f_a'(z) = 2n$. 
Hence $a \in X_n$ by Theorem \ref{poly-eqn}.
\end{proof}

\begin{remark}
We would like to point out the striking similarity between the 
solution $w_a$ to the Lam\'e equation and the defining power series of 
\emph{complex elliptic genera} in the Weierstrass form 
studied in \cite{Wang}. 
In a certain context of \emph{topological} field theory, 
complex elliptic genera serves as the \emph{genus one partition function}. 
In contrast to it, the mean field equation studies local yet 
very precise \emph{analytic} behavior of a genus one curve. 
It would be very interesting to uncover a good reason that will 
account for the similarity between these two theories.
\end{remark}

\section{Hyperelliptic geometry on $\bar X_n$} \label{hyper-ell-str}
\setcounter{equation}{0}

We have seen in Proposition \ref{cor:Ba_and_a} that
the fibers of the map \(\pi:Y_n\to \CC\) which sends
a typical point of \(Y_n\) represented by an 
unordered list \([a]= \{[a_1],\ldots, [a_n]\}\) of \(n\) elements in
\(\CC/\Lambda\smallsetminus \{[0]\}\) 
to \[\,\pi([a])=B_{[a]}=(2n-1)\cdot\sum_{i=1}^n\wp(a_i;\Lambda)\]
are exactly the orbits of the involution
\(\,\iota:[a]\mapsto [-a]\,\) on \(Y_n\).
We have also seen the complement \(Y_n\smallsetminus X_n\)
of \(X_n\subset Y_n\) is the set of all points of \(Y_n\) fixed
by the involution \(\iota\).
In turn the fact that both \(X_n\) and \(Y_n\) are locally the locus
of common zeros of \(n-1\) holomorphic functions on 
\(n\)-dimensional complex manifolds suggest that 
\(X_n\) and \(Y_n\) are both one-dimensional.
The fact that there exists a two-to-one map
from \(X_n\to \CC\) suggest that \(X_n\) is the unramified 
locus of a possibly singular hyperelliptic curve,
\(Y_n\) is a partial compactification of \(X_n\), 
\(\,\iota\) is the hyperelliptic involution on \(Y_n\),
and \(\pi:Y\to \CC\) is the hyperelliptic projection.
The entire section is devoted to the proof of Theorem
\ref{hyp-ell-thm}, which asserts that the above guesses are indeed
true, and provides more detailed information about this hyperelliptic curve.
Due to its fundamental importance, we shall give two different
treatments of this result, one based on the theory of ordinary 
differential equations 
and another one based on purely algebraic method. 

The analytic method continues the train of ideas 
in \S\ref{X_n}, that points of \(Y_n\) corresponds to
the ansatz solutions of Lam\'e equations with fixed index \(n\)
but varying accessory parameters.
With the analytic method it is easier to show that
the closure \(\bar{X}_n\) of \(X_n\) in
the \(n\)-th symmetric product ${\rm Sym}^n E = E^n/S_n$
of \(E=\CC/\Lambda\) is \(Y_n \cup \{\infty\}\),
where ``$\infty$'' stands for the point
\(\{0, \ldots, 0\}\) of \(\textup{Sym}^n E\).
Moreover one gets a recursively defined
polynomial \(\ell_n(B)\) of degree 
\(2n+1\) in \(B\), whose roots are the image of 
the ramification points of \(\pi:Y_n\to \CC\),
i.e.\ fixed points of the involution \(\iota\) on \(Y_n\).
With the algebraic method one gets not only the same polynomial
\(\ell_n(B)\),
but also an explicit regular function \(C\) on \(\bar{X}_n\) such
that $C^2 = \ell_n(B)$. 
In particular, $\bar X_n$ has arithmetic genus $n$. 
The algebraic method also allows us to analyze 
the limiting equations at $\infty$ and prove that  
\(\infty\) is a smooth point of $\bar X_n$.


We emphasize that a priori there is no definite reason that 
the compactification of $Y_n$ in ${\rm Sym}^n E$ should agree 
with the \emph{projective hyperelliptic model} of $X_n$
defined in
\ref{subsubsec:compare-setup}.e.
Such an identification is one of the main statements we will
establish; see Proposition \ref{prop:use-purity}.



\subsection{Review of linear second order ODE}
\label{subsec:review-2ndsymm}

The 
starting point of this section 
is the following simple 
well-known observation on a second order ODE 
\begin{equation}\label{gen-2ndorder}
w'' = I w  
\end{equation}
Recall that the Wronskian
\begin{equation} \label{WronskianC}
C := \begin{vmatrix} w_1 & w_2 
\\ w_1' & w_2' \end{vmatrix} = w_1 w_2' - w_2 w_1'
=w_1 w_2\cdot \frac{d}{dz}\log \frac{w_1}{w_2}
\end{equation}
of two linearly independent solutions $w_1, w_2$ 
is obviously a non-zero constant since $C' = 0$. 
If the product $X = w_1 w_2$ is easier to get hold of, 
then we may express the solutions
$w_1, w_2$ in terms of $C$ and $X$: we have
$$
\frac{X'}{X} = \frac{w_1'}{w_1} + \frac{w_2'}{w_2}, \qquad \frac{C}{X} 
= \frac{w_2'}{w_2} - \frac{w_1'}{w_1},
$$
hence
$$
\frac{w_1'}{w_1} = \frac{X' - C}{2X}, 
\qquad \frac{w_2'}{w_2} = \frac{X' + C}{2X}.
$$
In particular, 
\begin{equation} \label{ODE-sol}
w_1 = X^{1/2} \exp \Big(-\frac{C}{2} \int \frac{dz}{X}\Big), \qquad 
w_2 = X^{1/2} \exp \Big(\frac{C}{2} \int \frac{dz}{X}\Big).
\end{equation}

From 
$$
\Big(\frac{X' + C}{2X}\Big)' = \Big(\frac{w_2'}{w_2}\Big)' 
= \frac{w_2''}{w_2} - \Big(\frac{w_2'}{w_2}\Big)^2 
= I - \frac{(X' + C)^2}{4X^2},
$$
one finds easily that
\begin{equation} \label{C2}
C^2 = X'^2 - 2X'' X + 4I X^2.
\end{equation}
Differentiating (\ref{C2}) we see that the product 
\(X=w_1\cdot w_2\) of any two solutions 
\(w_1, w_2\) of the equation (\ref{gen-2ndorder})
satisfies
\begin{equation} \label{sym-2}
X''' - 4I X' - 2I' X = 0.
\end{equation}
This third order ODE is known as the 
\emph{second symmetric power} of 
the equation (\ref{gen-2ndorder})
and 
can easily be derived directly.
In this way, (\ref{C2}) is simply the first integral of (\ref{sym-2}) 
with integration factor $-2X$.

Conversely suppose we have a non-trivial solution $X$ of the
second symmetric square (\ref{sym-2}) of (\ref{gen-2ndorder}).
Then \(\,X'^2 - 2X'' X + 4I X^2\,\) 
is a constant\footnote{Of course if this constant
is zero, then the functions
\(w_1, w_2\) defined by (\ref{ODE-sol}) are linearly dependent.}
and the constant $C$ is defined up to sign by (\ref{C2}), 
so we get a pair of functions $w_1, w_2$ 
defined by (\ref{ODE-sol}). 
It can be checked easily using (\ref{C2})
that $w_1$ and $w_2$ are indeed solutions of the
equation (\ref{gen-2ndorder}).


\begin{remark}
The product $\,X = w_1\!\cdot\! w_2\,$ has appeared implicitly 
in our previous discussions, in the sense that
there exists a developing map for a solution of
the mean field equation \(\,\triangle u+e^u=8\pi n\,\delta_0\,\)
whose logarithmic derivative is equal to \(-C/X\)
for some non-zero constant \(C\). 
To see this, recall that for any given solution of
the above mean field equation on \(\CC/\Lambda\), there
exist two 
independent solutions of a Lam\'e equation
\[\,\frac{d^2 w}{dz^2}- (n(n+1) \wp(z;\Lambda)+B)w=0\,\]
such that \(\,w_1/w_2=:\!f\,\) is a developing map of \(u\).
Then
\begin{equation} \label{g=-C/X}
g := \frac{f'}{f} 
= \frac{(w_1'w_2 - w_1 w_2')/w_2^2}{w_1/w_2} 
= \frac{-C}{w_1w_2} = \frac{-C}{X}.
\end{equation}
\end{remark}
\bigbreak

We start with the analytic approach of Hermite-Halphen;
c.f.\
\cite[p.\,499]{Halphen} and
{\cite[\S 23.7]{Whittaker}}.

\begin{theorem} \label{hyper}
Let \(n\geq 1\) be a positive integer.
\begin{itemize}
\item[(i)] There exist polynomials
\(s_1(B), s_2(B), s_3(B),\ldots, s_n(B)\) in \(B\)
with coefficients in \(\QQ[g_2(\Lambda), g_3(\Lambda)]\)
with the following properties:
\begin{itemize}
\item for every element \([a]=\{[a_1],\ldots,[a_n]\}\in Y_n\)
and any \(i=1,\ldots, n\), the
\(i\)-th elementary polynomial of
\(\,\{\wp(a_1;\Lambda),\ldots, \wp(a_n;\Lambda)\}\,\)
is equal to \(s_i(B_{[a]})\), 
where \(B_{[a]}:=(2n-1)\sum_{i=1}^n \wp(a_i;\Lambda)\).

\item \(s_1(B)= (2n-1)^{-1} B\), and
\(s_i(B)\) is of degree \(i\) with leading coefficients
in \(\QQ^{\times}\) for \(i=2,\ldots, n\).

\item \(s_i(B)\) is homogenous of weight \(\,i\,\) for
\(i=2,\ldots, n\) if 
\(B, g_2, g_3\) are given weights \(1, 2, 3\)
respectively.

\end{itemize}

\item[(ii)] The fibers of the map
$\pi: Y_n \to \mathbb C$ defined by 
$$\pi:\{(x_i, y_i)\}_{i = 1}^n \mapsto
B_{[a]}=(2n-1)\,\sum_{i = 1}^n x_i$$ 
are orbits of the involution \(\,\iota: [a]\mapsto [-a]\,\) on 
\(Y_n\). In other words
\[\pi^{-1}(B_{[a]})=\{[a], [-a]\}\quad \forall\,[a]\in Y_n.\]
The subset \(X_n\subset Y_n\) is the complement in \(Y_n\) of the
fixed point set \((Y_n)^{\iota}\) of the involution \(\iota\);
it is also equal to the subset of \(Y_n\) consisting of all
elements \([a]=\{[a_1],\ldots,[a_n]\}\in Y_n\)
such that
\(\,\wp'(a_i;\Lambda)\neq 0\) for all \(i\)
and \(\wp(a_i;\Lambda)\neq \wp(a_j;\Lambda)\) for all
pairs \((i,j)\) with \(i\neq j\) and \(1\leq i,j\leq n\).
Moreover \(X_n\) is a locally closed
smooth one-dimensional complex submanifold 
of \(\,\textup{Sym}^n(\CC/\Lambda)\).

\item[(iii)] The set \((Y_n)^{\iota}=Y_n\!\smallsetminus\! X_n\) 
is a finite subset of
\(Y_n\) with at most \(2n+1\) elements.
Up to \(\CC^{\times}\) the set of all ansatz functions 
\(\,w_a(z)\,\) with \([a]\in Y_n^{\iota}\)
coincides 
with the set of all Lam\'e functions 
of index \(n\).
In other words \(Y_n^{\iota}=Y_n\smallsetminus X_n\) is in natural
bijection with the set of all Lam\'e functions of index \(n\)
up to non-zero constants.

\item[(iv)] The closure \(\bar X_n\) of \(X_n\) 
in \(\textup{Sym}^n(\CC/\Lambda)\) consists of
\(Y_n\) and a single ``point at infinity''
\([0]^n:= \{[0], \ldots, [0]\}\):
$\bar X_n = Y_n \cup \{[0]^n\}$. 

\item[(v)] 
The map \(\pi:Y_n\to \CC\) extends to a surjective continuous map
\(\bar\pi: \bar X_n\to \PP^1(\CC)\) which sends
the point \([0]^n\) to \(\infty\in\PP^1(\CC)\).
\end{itemize}
\end{theorem}

\subsubsection{}
{\scshape Proof of Theorem} \ref{hyper}\,(i).\enspace
Consider the Weierstrass equation 
$y^2 = p(x) = 4 x^3 - g_2 x - g_3$, 
where $(x, y) = (\wp(z), \wp'(z))$, and we set 
$(x_i, y_i) = (\wp(a_i), \wp'(a_i))$ for 
$[a]=\{[a_1],\ldots,[a_n]\} \in Y_n$. 
Pick \(a_i\in [a_i]\) for \(i=1,\ldots, n\).
Consider the following pair of ansatz solutions  
$\Lambda_a (z)$, $\Lambda_{-a} (z)$
the Lam\'e equation 
\begin{equation}\label{lame-sec7}
\frac{d^2w}{dz^2}- \big(n(n+1)\wp(z;\Lambda)+B\big)w=0,
\end{equation}
where
\begin{equation} \label{solu-lambda}
\Lambda_a (z) := \frac{w_a (z)}{\prod_{i = 1}^n \sigma(a_i)} 
= e^{z\sum \zeta(a_i)} \prod_{i = 1}^n 
\frac{\sigma(z - a_i)}{\sigma(z) \sigma(a_i)}.
\end{equation}
Let $\tilde{X}_{[a]}(z) 
= \Lambda_a(z) \Lambda_{-a}(z)$ be the product of this pair. 
Note that if $[a] \not\in X_n$ then $[a] = [-a]$ by 
Proposition \ref{XnYn} and $\Lambda_a = \Lambda_{-a}$. 
From the addition theorem we have 
\begin{equation} \label{X-poly}
\begin{split}
\tilde{X}_{a}(z) &= (-1)^n \prod_{i = 1}^n \frac{\sigma(z + a_i;\Lambda) 
\sigma (z - a_i;\Lambda)}{\sigma(z;\Lambda)^2 \sigma(a_i;\Lambda)^2} 
\\
&= (-1)^n \prod_{i = 1}^n (\wp(z;\Lambda) - \wp(a_i;\Lambda))
= X(\wp(z;\Lambda))
\end{split}
\end{equation}
That is \(\tilde{X}_{[a]}(z)=X_{[a]}(\wp(z;\Lambda))\), 
where 
\[X_{[a]}(x) = (-1)^n \prod_{i = 1}^n (x - \wp(a_i;\Lambda)).\]
a polynomial of degree \(n\) in the variable \(x\).

We know that $\tilde{X}_{[a]}(z)$ 
satisfies the second symmetric power of the Lam\'e equation
(\ref{lame-sec7})
\begin{equation} \label{sym2-z}
\frac{d^3 \tilde{X}}{dz^3} - 4 (n (n + 1) \wp(z;\Lambda) + B) 
\frac{d \tilde{X}}{dz} -2n(n + 1) \wp'(z,\Lambda) \tilde{X}(z) = 0,
\end{equation}
it is thus a \emph{polynomial solution} in the variable $x$, to 
the algebraic form
\begin{equation} \label{sym2-lame}
p(x) \frac{d^3X}{dx^3} 
+ \frac{3}{2} \frac{dp}{dx}\cdot \frac{d^2 X}{dx^2}
- 4((n^2 + n - 3)x + B) \frac{d X}{dx} - 2n (n + 1) X = 0,
\end{equation}
of (\ref{sym2-z}),
where \(p(x)\) is the cubic polynomial
\[p(x)=4 x^3- g_2(\Lambda)x-g_3(\Lambda)\] in the Weierstrass equation for
\(\CC/\Lambda\).
As a result, $X_{[a]}(x) $ will be determined by $B$ 
and certain initial conditions.

Indeed, write 
$X_{[a]}(x) = (-1)^n (x^n - s_1 x^{n - 1} + \cdots + (-1)^n s_n)$, 
then (\ref{sym2-lame}) translates into a 
linear recursive relation for each $\mu = n - 1, \ldots, 0$,
where we set $s_0 = 1$ by convention:
\begin{equation} \label{recursive}
\begin{split}
0 &= 2(n - \mu)(2\mu + 1)(n + \mu + 1) 
s_{n - \mu} - 4(\mu + 1) B s_{n - \mu - 1} 
\\
&\quad + \tfrac{1}{2} g_2 (\mu + 1) (\mu + 2) (2\mu + 3) s_{n - \mu - 2} \\
&\qquad - g_3 (\mu + 1) (\mu + 2) (\mu + 3) s_{n - \mu - 3}.
\end{split}
\end{equation}
Since $B = (2n - 1) s_1$, the initial relation for $\mu = n - 1$ is
automatic. 
Let \(s_1(B):=(2n-1)^{-1} B\).  
The recursive relations (\ref{recursive}) 
with \(s_i\) substituted by \(s_i(B)\) define 
polynomials \(s_2(B),\ldots, s_n(B)\in \QQ[g_2, g_3]\)
which satisfy the first condition listed in Theorem
\ref{hyper}\,(i).
Moreover we see from the recursive relations that
$s_j$ is a polynomial of degree $j$ in $B$, 
and it is homogenous of weight \(n\) if \(B, g_2, g_3\) 
are given weights \(1, 2, 3\) respectively, for
\(j=1,\ldots, n\),
We have proved Theorem \ref{hyper}\,(i). \qed

\subsubsection{}
{\scshape Proof of Theorem }\ref{hyper}\,(ii).\enspace
The first sentence of Theorem \ref{hyper}\,(ii) is a restatement of
Proposition \ref{cor:Ba_and_a}.  
We give another proof here more in line with the proof of (i).
Suppose now we have two element
\[[a]=\{[a_1],\ldots,[a_n]\},\  
[b]=\{[b_1],\ldots,[b_n]\} \in Y_n
\] such that 
\[\sum_{i=1}^n \wp(a_i;\Lambda) = \sum_{i=1}^n \wp(a_i';\Lambda).\]
Let \(X_{[a]}(x)\) be the polynomial in \(x\) of degree \(n\) 
such that 
\(X_{[a]}(\wp(z;\Lambda)) = \Lambda_a(z)
\Lambda_{-a}(z)\); similarly let \(X_{[b]}\) be the polynomial 
of degree \(n\) such that
$X_{[b]}(\wp(z;\Lambda)) = \Lambda_{b}(z) \Lambda_{-b}(z)$.
Then $X_{[a]}$ and $X_{[b]}$ both satisfy the same equation 
(\ref{sym2-lame}), 
and we get from the recursive relations (\ref{recursive}) that
$X_{[a]} = X_{[b]}$, i.e.\ 
\begin{equation*}
\prod_{i = 1}^n (x - \wp(a_i;\Lambda)) 
= \prod_{i = 1}^n (x - \wp(b_i;\Lambda)).
\end{equation*}
Therefore \(\{\wp(a_1;\Lambda),\ldots, \wp(a_n;\Lambda)\}
= \{\wp(b_1;\Lambda),\ldots, \wp(b_n;\Lambda)\}\) as unorder\-ed lists.
\medbreak

We claim that either \([a]=[b]\) or \([a]=[-b]\) as unordered lists. 
Otherwise after renumbering the \(a_i\)'s and the \(b_i\)'s,
there exist integers \(r, s\geq 1\) with \(r+s\leq n\) such that
the following hold:
\begin{itemize}

\item[(i)] \([a_i]=[b_i]\in \tfrac{1}{2}\Lambda/\Lambda\) for 
all \(i\) such that \(r+s+1\leq i\leq n\),

\item[(ii)] \([a_i]\not\in \tfrac{1}{2}\Lambda/\Lambda\) and 
\([b_i]\not\in \tfrac{1}{2}\Lambda/\Lambda\), for 
all \(i\) such that \(1\leq i\leq r+s\),

\item[(iii)] \([a_i]=[-b_i]\) for \(i=1,\ldots, r\),

\item[(iv)] \([a_i]=[b_i]\) for \(i=r+1,\ldots, r+s\).

\item[(v)] \([a_i]\neq [-a_j]\) if \(i\neq j\) 
and \(1\leq i,j\leq r\).

\end{itemize}
We know from Theorem \ref{poly-eqn}\,(0) that
\(w_a(z)\) and \(w_b(z)\) are linearly independent
because \([a]\neq [\pm b]\), and they satisfy the same
Lam\'e equation with index \(n\) because \(B_{[a]}=B_{[b]}\).
So the Schwarzian derivative of 
\(\,w_a/w_b\,\) is 
\[S(w_a/w_b)=2\big(n(n+1)\wp(z;\Lambda)+B_{[a]}\big).\]
On the other hand conditions (i)--(iv) 
tells us that 
\(w_a/w_b\) is equal to a non-zero constant times the function
\(f_{[c]}= w_{c}/w_{-c}\), where \([c]=\{c_1,\ldots, c_r\}\),
so \(S(f_{[c]})=2\big(n(n+1)\wp(z;\Lambda)+B_{[a]}\big)\).
The condition (v) above tells us that \([c]\neq [-c]\), 
so we get from Proposition \ref{poly-eqn}\,(1)
that
\(\,\textup{ord}_{z=0}\,f_{[c]}\leq 2r\leq 2n-2\),
which implies that 
\(S(f_{[c]})\neq 2\big(n(n+1)\wp(z;\Lambda)+B_{[a]}\big)\).
We have proved the first sentence of (ii):
if \(B_{[a]}=B_{[b]}\), \([a], [b]\in Y_n\), then
either \([a]=[b]\) or \([a]=[-b]\).

The second sentence of Theorem \ref{hyper}\,(ii) is the content of 
Proposition \ref{XnYn}.
The argument below provides a different proof and
also the rest of the statement (ii) at the same time.
Suppose that \([a]=\{[a_1],\ldots, [a_n]\) is a given
point of \(Y_n\).
As in \S \ref{subsec:review-2ndsymm}, we know that
\begin{equation}\label{C22}
\begin{split}
\big(\tfrac{d}{dz} X_{[a]}(\wp(z;\Lambda))\big)^2
-2\,  X_{[a]}(\wp(z;\Lambda))\, \tfrac{d^2 }{dz^2}X_{[a]}(\wp(z;\Lambda))
\\
+4\left(n(n+1)\wp(z;\Lambda)+B_{[a]}\right)
X_{[a]}(\wp(z;\Lambda))^2
\end{split}
\end{equation}
is a constant because its derivative vanishes identically;
write this constant as \(C^2\).
This constant \(C^2\) can be evaluated by plugging in \(z=a_i\) 
in equation (\ref{C22}),
for any
\(i\) with \(1\leq i\leq n\): 
\[
C^2=
\Big(\frac{dX_{[a]}}{dx}(\wp(a_i; \Lambda))\cdot \wp'(a_i;\Lambda)\Big)^2
\]
for each \(i=1,\ldots, n\).

Suppose that \(C^2=0\).
We know from \S \ref{subsec:review-2ndsymm}
that \(w_a\) and \(w_{-a}\) are linearly dependently, therefore
\([a]=[-a]\).  The above argument also tells us that 
\(\pi^{-1}(B_{[a]})\) is the singleton 
\(\{[a]\}\).
In this case \(w_a(z)\) is a Lam\'e function of 
index \(n\): up to \(\CC^{\times}\) it is a square root of 
\(X_{[a]}(\wp(z;\Lambda))\), a polynomial of degree \(n\) 
in \(\wp(z;\Lambda)\). 
We also see that \([a]\not\in X_n\), because for each \(i\)
we know that either \(\wp'(a_i;\Lambda)=0\) or
\(\wp(a_i;\Lambda)\) is a multiple root of \(X_{[a]}(x)\).

On the other hand, suppose that  \(C^2\neq 0\). 
Then \[\wp'(a_i;\Lambda)\neq 0\quad
\textup{and}\quad \frac{dX_{[a]}}{dx}(\wp(a_i; \Lambda))\neq 0
\quad\textup{for}\ \ i=1,\ldots,n.
\]
Therefore \([a]\neq [-a]\), and \([a], [-a]\in X_n\).
After making a choice of a square root \(C\) of
\(C^2\), one can ``pick out'' \([a]\) from the pair \(\{[a], [-a]\}\) 
using
\(C\) and \(\{\wp(a_1;\Lambda),\ldots, \wp(a_n;\Lambda)\}\), by
\begin{equation}\label{chooseC}
C=\frac{dX_{[a]}}{dx}(\wp(a_i; \Lambda))\cdot \wp'(a_i;\Lambda)
\end{equation}
The above formula shows that the map \(\pi:Y_n\to \CC\) is 
a local isomorphism near \([a]\) and \([-a]\).
The procedure reviewed in \S \ref{subsec:review-2ndsymm} tells us
that the pair of ansatz functions
\(w_a\), \(w_{-a}\) are determined up to \(\CC^{\times}\) by
\(X_{[a]}(x)\), so \(\pi^{-1}(B_{[a]}) =\{[a],[-a]\}\) in this case.
We have proved Theorem \ref{hyper}\,(ii).
\qed
%
%
%
%
%
%

\begin{remarkss}\label{rem:dicho}
The proofs of  Theorem \ref{hyper}\,(i) and (ii) employed the 
method in \cite[pp.\,498--500]{Halphen} and 
\cite[pp.\,570--572]{Whittaker} which gives
a recursive formula for the product of a pair of
ansatz solutions \(w_a\) and \(w_{-a}\) in terms of
the auxiliary parameter \(B_{[a]}\), then bootstrap
to find the ansatz pair \(w_a, w_{-a}\). 

The ansatz solutions parametrized by \(Y_n\)
are eigenfunctions for the 
translation action of the lattice \(\Lambda\),
and they are also eigenfunction for the 
differential operator
\(\,\frac{d^2}{dz^2}-n(n+1)\wp(z;\Lambda)\).
In this sense \(Y_n\) can be regarded as the
\emph{spectral curve} of this
second order differential operator.
\smallbreak

Theorem \ref{hyper}\,(v) asserts that for every 
\(B\in \CC\) there exist an element \([a]\in Y_n\) 
such that \(B_{[a]}=B\). We discuss the dichotomy
whether \(\pi:Y_n\to \CC\) is ramified above \(B\)
from the perspective of the translation action of 
\(\Lambda\) on the solution space of the 
Lam\'e equation \(L_{n,B_{[a]}}\).
\smallbreak
\noindent
{\bf 1.} {\scshape Case}  \([a]\in X_n\), 
equivalently \([a]\neq [-a]\).
\enspace In this case \(\CC\!\cdot\! w_a\) and
\(\CC\!\cdot\! w_{-a}\) are one-dimensional spaces of 
solutions of the same Lam\'e equation \(\,L_{n,B_{[a]}}\,\) 
but their
eigenvalue packages for the 
translation action by \(\Lambda\) are different,
hence the ansatz solutions \(\CC\!\cdot\! w_a\) and
\(\CC\!\cdot\! w_{-a}\) are \emph{intrinsic} to
the Lam\'e equation \(\,L_{n,B_{[a]}}\).
\smallbreak
\noindent
{\bf 2.} {\scshape Case} \([a]\not\in X_n\), equivalently
\([a]=[-a]\). \enspace
The assumption that \([a]\not\in X_n\) is equivalent
to \([a]=[-a]\).  We have seen in the proof of 
\ref{hyper}\,(ii) that up to a non-zero constant
\(w_a(z)\) is a square root of a polynomial
of \(\wp(z;\Lambda)\); in other words
\(w_a\) is a \emph{Lam\'e function}.
In this case the action of \(\Lambda\)
on the space of solutions of the Lam\'e 
equation \(\,L_{n,B_{[a]}}\,\)
is not diagonalizable, and
the Lam\'e functions \(\CC\!\cdot\! w_a\) are 
the only \(\Lambda\)-eigenfunctions 
among the space of solutions of the Lam\'e equation
\(\,L_{n,B_{[a]}}\).
\end{remarkss}

\subsubsection{}\label{rem:disc_ell_n}
{\scshape Proof of Theorem \ref{hyper}}\,(iii).\enspace
We have seen in the last paragraph of Remark \ref{rem:dicho}
that \((Y_n)^{\iota}\) is in natural bijection with
the set of all Lam\'e functions of index \(n\) up
to \(\CC^{\times}\).
One knows from classical literature that
there exists a polynomial 
\(\ell_n(B)\in \QQ[g_2(\Lambda),g_3(\Lambda)]\) of degree \(2n+1\)
in the variable \(B\), explicitly defined by
recursion, whose roots are precisely the
\(B_{[a]}\)'s with \([a]\in (Y_n)^{\iota}\);
see Theorem B in \S\ref{intro}. 
Theorem \ref{hyper}\,(iii) follows.
\smallbreak

The definition of this polynomial 
\(\ell_n(B)\) will be reviewed in the proof of
Theorem \ref{alg-hyper}.
It is known that 
\(\ell_n(B)\) has \(2n+1\) distinct real roots when
the lattice has the form 
\(\Lambda= \ZZ+ \sqrt{-1}t\ZZ\) for some \(t\in\RR_{>0}\).
This fact is stated on line 13,  page 221 in Liouville's letter \cite{liouville_1846}, where 
Liouville said that one can use Sturm's method 
to prove that the polynomial \(\ell_n(B)\),
written as \(R(B)\) in \emph{loc.\ cit.}, 
has \(2n+1\) (real) roots and therefore there are
\(2n+1\) Lam\'e functions. 
The proof in \cite[pp.\,471--476]{Halphen}
goes through a change of variable
\(u=2v\), \(y=\wp'(v;\Lambda)\), which has the advantage that
every Lam\'e function is rationally expressible in terms of 
\(\wp(v,\Lambda)\)
and \(\wp'(v;\Lambda)\);
this proof is sketched in \cite[p.\,163]{Poole}.
In \cite[\S23.41]{Whittaker} Lam\'e functions ``of the third kind''
(in the case when \(n\) is even) is discussed, 
with the other three cases left as exercises.
Sturm's method, in the form of Corollary \ref{cor:sturm},
was used in all references above.
\qed

%
%

\subsubsection{}
{\scshape Proof of Theorem \ref{hyper}}\,(iv)--(v).
\enspace

Suppose that \([a] = \{[a_1], \ldots, [a_n]\}\) is a
point of \(\bar X_n\).
By definition there exists a sequence 
\([a]_k = \{[a_{k, 1}], \ldots, [a_{k, n}]\}\in X_n\)
which converges
as $k \to \infty$. 
Let $B_k = B_{[a]_k}= (2n - 1) \sum_{i = 1}^n \wp(a_{k, i};\Lambda)$. 
Then 
\(B := \lim_{k \to \infty} B_k\)
exists as an element of  \(\mathbb C \cup\{\infty\}=\PP^1(\CC)\).
Let $X_k(x)=X_{[a]_k}(x)$ be the polynomial of degree
\(n\) in \(x\) determined by $B_k$ 
through the recursive relation (\ref{recursive}). 
We discuss the two cases (a) \(B\in \CC\) 
and (b) \(B=\infty\in\PP^1(\CC)\).
\smallbreak

\noindent
(a) Suppose that $B \ne \infty$.  By (\ref{recursive}) 
the coefficients $s_j(B_k)$ of $X_k(x)$ are 
uniformly bounded for $j = 1, \ldots, n$. 
Thus the roots $x_{k, i} = \wp(a_{k, i};\Lambda)\in \CC$ of $X_k(x)$ 
are uniformly bounded as well. This implies that 
\(a_i\not\in \Lambda\) for \(i=1,\ldots, n\),
$w_{a_k} \to w_a$, 
\([a]_k\to [a]\) in \(\textup{Sym}^n(\CC/\Lambda)\),
and $w_a$ is a non-trivial solution of the Lam\'e
equation  $w'' = (n (n + 1) \wp(z;\Lambda) + B) w$. 
Notice that 
we must have $a_i \ne a_j$ whenever $i \ne j$.
For otherwise \(w_a(z)\) has multiplicity
at least \(2\) at \(z=a_i\), which implies that
\(w_a\) is identically zero, a contradiction.
We conclude that \([a]\in Y_n\).
\smallbreak

\noindent
(b) Suppose that $B = \infty$. 
We claim that $[a_{k, i}] \to [0]$ for all $i=1,\ldots,n$. 
Change the variable to \(t=x^{-1}\). Look at the polynomial
\begin{equation} \label{NB}
\begin{split}
Y_k(y)&:=s_n(B_{k})^{-1}\cdot y^n\cdot  X_{[a]_k}(y^{-1})
\\ 
&=y^n-\tfrac{s_{n-1}(B_k)}{s_n(B_k)} y^{n-1}+\cdots
 + (-1)^{n-1} \tfrac{s_1(B_k)}{s_n(B_k)}
+(-1)^n\tfrac{1}{s_n(B_k)}
\end{split}
\end{equation}
whose roots are 
\(\{\wp(a_{k,1};\Lambda)^{-1},\ldots,\wp(a_{k,n};\Lambda)^{-1}\}\).
The assumption that \(B_k\to \infty\) as \(k\to \infty\) tells us that 
\(Y_k(y)\to y^n\) as \(k\to \infty\),
which implies that all the roots $\wp(a_{k, i})^{-1}$ of $Y_k(y)$
go to \(0\) as \(k\to\infty\).
Therefore $a_{k, i} \to 0$ as \(k\to \infty\) for all \(i=1,\ldots,n\).
\smallbreak

Combining the cases (a) and (b), we draw the following conclusions.
\begin{itemize}
\item The map
\(\pi_n\vert_{X_n}:X_n\to \CC\) extends to a continuous map
\(\bar{\pi}_n: \bar X_n\to \PP^1(\CC)\). 
\item The inverse image \(\bar{\pi}_n^{-1}(\infty)\)
of the point \(\infty\in\PP^1(\CC)\) under \(\bar\pi\)
consists of a single point \([0]^n=\{[0],\ldots,[0]\}\).

\item The inverse image \(\bar{\pi}^{-1}(\CC)\) of
\(\CC\) under \(\bar{\pi}\) is contained in \(Y_n\).
In other words \(\bar{X}_n\smallsetminus\{[0]^n\}
\subseteq Y_n\).
\item Because \(\bar{\pi}_n\) is compact by definition, and we already
know that \(\pi_n(X_n)\) contains the completment of
a finite subset of \(\CC\), therefore
\(\bar{\pi}_n\) is surjective. 
\end{itemize}
We have proved Theorem \ref{hyper}\,(v) and half of (iv).
\smallbreak

To complete the proof of Theorem \ref{hyper}\,(iv), we need to 
show that \((Y_n)^{\iota} \subset \bar{X}_n\).
Let \([a]\in (Y_n)^{\iota}\) be a given element of 
\(Y_n\smallsetminus X_n\).
We know that there exists an element \([b]\in \bar{X}_n\)
such that \(B_{[b]}=B_{[a]}\), and have seen \([b]\in Y_n\).
Theorem \ref{hyper}\,(ii) tells us that either \([b]=[a]\)
or \([b]=[-a]\).  In either case we
conclude that \([a]=[-a]=[b]\in \bar{X}_n\).
We have proved Theorem \ref{hyper}.
\qed


\begin{corollaryss}
Let \(\left([a]_k\right)_{k\in \NN} = 
\left(\{[a_{k,1}],\ldots,[a_{k,n}]\}\right)_{k\in \NN}\)  be a sequence
of elements in \(X_n\) indexed by \(\NN\).
If there exists an \(i\) between \(1\) and \(n\) such that
\([a_{k,i}] \to [0]\) in \(\CC/\Lambda\),
then $[a_{k,i}] \to [0]$ in \(\CC/\Lambda\) for all \(i=1,\ldots,n\).
\end{corollaryss}


\begin{remarkss} \label{HH-ansatz}
The proof of Theorem \ref{hyper}\,(i) has appeared in 
\cite[p.\,499-500]{Halphen}.
The proof in  
{\cite[\S 23.7]{Whittaker}} is essentially the same,
except that \(X_{[a]}(x)\) is expressed as a polynomial in
\(x-e_2\) and recursive formula was given for the
coefficients of powers of \(x-e_2\).
We may compare our argument with {\cite[\S 23.7]{Whittaker}} 
on such a polynomial solution $X$ to (\ref{sym2-lame}), 
which is indeed the origin where $X(x)$ and $\Lambda_a$ 
were first found during our study. 
Let $X = \sum_{r = 0}^\infty c_r (x - e_2)^{n - r}$ 
be a solution of (\ref{sym2-lame}) in descending power. 
Since $p(x) = 4(x - e_1) (x - e_2) (x - e_3)$, 
there is a recursive formula for $c_r$:
\begin{equation*}
\begin{split}
&4r (n + \tfrac{1}{2} - r) (2n + 1 - r) c_r \\
&\quad = (n + 1 - r) 
\big(12 e_2 (n - r) (n - 2 - r) - 4e_2 (n^2 + n - 3) - 4B\big) c_{r -  1} 
\\
& \qquad - 4(n + 1 - r) (n + 2 - r) 
(n + \tfrac{3}{2} - r) (e_1 - e_2) (e_2 - e_3) c_{r - 2}.
\end{split}
\end{equation*}
The above recursive formula is slightly different 
from (\ref{recursive}). 
Given $c_0 = 1$ and $c_1$, we can solve 
$c_2, \ldots, c_n$ and express them as polynomials in \(c_1\)
and \(B\).
The recursive formula forces 
$c_{n + 1} = \cdots = c_{2n} = 0$. 
The next coefficient $c_{2n + 1}$ appears as another 
``free parameter'',
and the coefficients of higher order terms are expressed 
as polynomials of \(c_1\).
The condition that \(X(x)\) is a polynomial is that 
$c_l = 0$ for all $l \ge n$. 
Thus $X(x)$ is a polynomial solution, 
which is determined by $c_1$ and $B$.

From  $(-1)^n \Lambda_a \Lambda_{-a} 
= \prod_{i = 1}^n ((x - e_2) + (e_2 - x_i))
= \sum_{r = 0}^n c_r (x - e_2)^{n - r}$, 
we see that
$c_1 = \sum_{i = 1}^n (e_2 - x_i) 
= ne_2 - \sum_{i = 1}^n x_i = ne_2 - B$. 
Hence $X$ is a polynomial in $B$.
\end{remarkss}

\begin{remarkss}
The statement of Theorem \ref{hyper}\,(iv), that 
$\bar X_n = Y_n \cup \{\infty\}$, 
does not seem to have appeared in the literature,
but this fact must be known as it follows quickly from
the method of recovering the ansatz pair \(w_a, w_{-a}\)
from their product. 
The behavior of $X_n$ or $Y_n$ at $B = \infty$ is important, 
which will be discussed in
Proposition \ref{B=infty}). 
\end{remarkss}

%
We would like to rephrase Theorem \ref{hyper} in purely algebraic terms 
without appealing to solutions of Lam\'e equations. It is given below, 
whose proof uses system (\ref{Alg-eqn}) instead of (\ref{ratio-eq}). 
\begin{theorem} \label{alg-hyper}
Let \(n\geq 1\) be a positive integer.
Let \(s_1,\ldots,s_n\in \QQ[g_2, g_3][B]\) be the polynomials 
in Theorem \textup{\ref{hyper}\,(i)} defined recursively by
the relation
\textup{(\ref{recursive})}.
\begin{itemize}
\item[(1)] The space $X_n$ admits a natural projective compactification 
$\hat X_n$ as a possibly singular, 
hyperelliptic curve defined by the following equation in $(B, C)$:
\begin{equation*}
\begin{split}
C^2 &= \ell_n(B, g_2(\Lambda), g_3(\Lambda)) \\
&= 4B\, s_n^2 + 4g_3(\Lambda)\, s_{n - 2}\, s_n 
- g_2(\Lambda)\, s_{n - 1}\, s_n - g_3(\Lambda)\, s_{n - 1}^2.
\end{split}
\end{equation*}

\item[(2)] The discriminant \(\textup{disc}_B(\ell_n(B))
\in\QQ[g_2,g_3]\)
in the variable \(B\)
of the polynomial \(\ell_n(B)\) is
a non-zero polynomial in two variables \(g_2, g_3\);
it is homogeneous of weight \(2n(2n+1)\) if
\(g_2, g_3\) are given weights \(2\) and \(3\)
respectively.
In other words \(\textup{disc}_B(\ell_n(B))
\in \QQ[g_2(\Lambda), g_3(\Lambda)]\) is
a non-zero 
modular form of weight \(4n(2n+1)\)
for the full modular group
\(\textup{SL}_2(\ZZ)\), holomorphic
on \(\HH\) and also on the cusps.

\end{itemize}
\end{theorem}

\begin{remark*}
The polynomial \(\ell_n(B,g_2,g_3)\) has degree \(2n+1\) 
in the variable \(B\); it is homogeneous of weight \(2n+1\)
in \(B, g_2\) and \(g_3\) when \(B, g_2, g_3\) are given
weights \(1, 2, 3\) respectively.
The projective curve $\hat X_n$ has arithmetic genus \(n\);
it is smooth unless \(\textup{disc}_B(\ell_n(B))(\Lambda)=0\).
\end{remark*}

\begin{proof}
Let $p(x) = 4x^3 - g_2 x - g_3$
and let  $q(x) = \prod_{j = 1}^n (x - x_j)$. 
The set $X_n$ is defined by 
\(2n-1\) polynomial equations
\[y_i^2 = p(x_i)\ \ \forall\,i=1,\ldots,n
\quad\textup{and}\quad
\sum_{i = 1}^n x_i^k\, y_i = 0
\ \ \forall\,k=0,1,\ldots, n-2
\] 
in the \(2n\) variables \(x_1,\ldots, x_n; y_1,\ldots, y_n\)
and \(\,n(n+1)/2\,\) inequalities
\[x_i \ne x_j \ \ \forall i\neq j,\ 1\le i, j\leq 1 \quad
\textup{and}\quad
y_i \ne 0\ \
\forall i=1,\ldots, n.
\]

Applying Cramer's rule to the $n - 1$ linear equations 
$\sum_{i = 1}^n x_i^k\, y_i = 0$ in $y_i$'s, 
we conclude that there is a constant\footnote{This ``constant''
\(C\) depends on \(n\) and the lattice \(\Lambda\).}
$C \in \mathbb C^\times$ such that
\begin{equation} \label{(0)}
y_i = \frac{C}{\prod_{j \ne i} (x_i - x_j)}, \qquad i = 1, \ldots, n.
\end{equation}
Since $q'(x_i) = \prod_{j \ne i} (x_i - x_j)$, we get
\begin{equation} \label{C^2}
p(x_i) q'(x_i)^2 = C^2, \quad i = 1, \ldots, n,
\end{equation}
and so $q(x) | (p(x) q'(x)^2 -C^2$). 
This implies that there are $h_1, \ldots, h_n, a, b \in \mathbb C$ such that
\begin{equation} \label{(1)}
\frac{p(x) q'(x)^2 - C^2}{q(x)^2} = \sum_{i = 1}^n \frac{h_i}{(x - x_i)} + ax + b.
\end{equation}
It is easy to compute (e.g.~using power series expansion of $p, q$ at $x_i$)
\begin{equation} \label{(2)}
\begin{split}
h_i &= {\rm Res}_{x = x_i} p(x) \frac{q'(x)^2}{q(x)^2} 
- {\rm Res}_{x = x_i} \frac{C^2}{q(x)^2} \\
&= p'(x_i) + p(x_i) \frac{q''(x_i)}{q'(x_i)} 
+ C^2 \frac{q''(x_i)}{q'(x_i)^3} \\
&= p'(x_i) + 2p(x_i) \frac{q''(x_i)}{q'(x_i)}.
\end{split}
\end{equation}
From (\ref{(1)}) we get
\begin{equation} \label{(3)}
p(x) q'(x)^2 - C^2 = \sum_{i = 1}^n h_i \frac{q(x)^2}{(x - x_i)} + (ax + b) q^2(x). 
\end{equation}
Comparing coefficients of $x^{2n + 1}$ and $x^{2n}$ on both sides, we get
\begin{equation} \label{(4)}
a = 4n^2, \qquad b = 8n \sum_{i = 1}^n x_i = 8n\, s_1
\end{equation}
(recall that $q(x) = x^n - s_1\, x^{n - 1} + \cdots + (-1)^n s_n$ 
and so $s_1 = \sum_{i = 1}^n x_i$). 

Now, in a similar and easier manner, we write
\begin{equation} \label{(5)}
\begin{split}
\frac{p'(x) q'(x)}{q(x)} 
&= \sum_{i = 1}^n \frac{p'(x_i)}{(x - x_i)} + (12n x + 12 s_1), \\
\frac{p(x) q''(x)}{q(x)} 
&= \sum_{i = 1}^n \frac{p(x_i)}{(x - x_i)} 
\frac{q''(x_i)}{q'(x_i)} + (4n(n - 1) x + 8(n - 1)s_1).
\end{split}
\end{equation}
Then (\ref{(2)}), (\ref{(3)}), (\ref{(4)}) and (\ref{(5)}) lead to 
\begin{equation} \label{int-eq}
p q'^2 - p' q' q -2 p q'' q + 4\big(n(n + 1) x + (2n -1) s_1\big) q^2 - C^2 = 0.
\end{equation}

One more differentiation gives
\begin{equation*}
\begin{split}
0&=p' q'^2 + 2p q' q'' - p'' q' q 
- p' q'' q - p' q'^2 - 2p' q'' q - 2p q''' q - 2p q'' q' \\
&\qquad + 4n(n + 1) q^2 + 8 (n(n + 1) x + (2n - 1) s_1) q q' \\
& = -2q\big(p q''' + 
\tfrac{3}{2} p' q''  - 4((n^2 + n - 3)x + B)q' - 2n(n + 1) q\big),
\end{split}
\end{equation*}
which is $(-2q)$ times the linear ODE (\ref{sym2-lame}), 
and so the same recursive relation (\ref{recursive}) 
shows that $q$ is determined by $s_1$.

Suppose we have two different points 
\(\underline{x}=\{(x_1, y_1),\ldots, (x_n,y_n)\}\) 
\(\underline{x}'=\{(x_1', y_1'), \ldots, (x_n', y_n') \}\) in 
$X_n$ such that 
$\pi_n(\underline{x})=\sum_{i=1}^n x_i 
= \sum_{i=1}^n x_i'=\pi_n(\underline{y})$, by rearrangement 
we have $x_i = x_i'$ for all $i$ 
and then $y_i' = \pm y_i$ for all $i$. 
If $y_i = y_i'$ for some $i$, then by (\ref{(0)}), 
$$
\frac{C'}{\prod_{j \ne i}(x_i' - x_j')} 
= \frac{C}{\prod_{j \ne i}(x_i - x_j)},
$$
which implies that $C = C'$ and $y_j = y_j'$ for all $j$, 
a contradiction. Hence $y_i = -y_i'$ for all $i$.
We have shown that if two different points 
\(\underline{x}, \underline{x}'\)
of \(X_n\) have the same
image in \(\CC\) under the map \(\pi_n\), then 
\(\underline{x}'=\iota(\underline{x})\), where
\(\iota\) is the involution on \(X_n\) defined by
``multiplication by \(-1\)'' on \(\CC/\Lambda\).

The constant terms in formula (\ref{int-eq}) leads to
\begin{equation} \label{ln-poly}
C^2 = \ell_n(B) = 4B\, s_n^2 + 4g_3\, s_{n - 2}\, s_n 
- g_2\, s_{n - 1}\, s_n - g_3\, s_{n - 1}^2,
\end{equation}
where $s_k = s_k(B)$ is a polynomial of degree $k$ 
and $B = (2n - 1) s_1$. Thus $\deg \ell_n = 2n + 1$. 
Equation (\ref{ln-poly}) 
provides a natural algebraic (hyperelliptic) compactification 
$\hat X_n$ of $X_n$.

To make this precise, we show that $X_n$ is mapped 
onto those $B \in \mathbb C$ with $C^2 = \ell_n(B) \ne 0$. 
Indeed we define $s_k$ by $s_k(B)$ and $x_i$'s 
by $q(x) = x^n - s_1 x^{n - 1} + \cdots + (-1)^n s_n 
= \prod_{i = 1}^n (x - x_i)$. 
Then (\ref{int-eq}) holds and by substituting $x = x_i$
 we get $p(x_i) q'(x_i)^2 = C^2$ as in (\ref{C^2}). 

If $C \ne 0$, we get $p(x_i) \ne 0$ and $q'(x_i) \ne 0$ 
which give the non-degenerate conditions. Now we define 
$$
y_i := \frac{C}{q'(x_i)} 
= \frac{C}{\prod_{j \ne i} (x_i - x_j)} \ne 0, \qquad i = 1, \ldots, n.
$$
Then $y_i^2 = p(x_i)$ and $\{(x_i, y_i)\}$ solves 
the system of equations 
\[\sum_{i = 1}^n x_i^k\, y_i = 0 \qquad k = 0, \ldots, n - 2,\] 
hence gives rise to a point in $X_n$. 

If $C= 0$, we have either $p(x_i) = 0$ or $q'(x_i) = 0$
for all \(i=1,\ldots, n\).
Let $x_i = \wp(a_i)$. In the former case $a_i = -a_i$ 
is a half period and $y_i = 0$. In the latter case $x_i = x_j$ 
for some $j \ne i$. 
Notice that $\{(x_i, y_i)\}$ still satisfies the equations 
$\sum_{i = 1}^n x_i^k y_i = 0$ for $k = 0, \ldots, n - 2$ 
since they define a closed set. 
Now the same argument in the proof of Theorem \ref{poly-eqn} 
shows that 
\([a]=[-a]\), where \([a]=\{[a_1],\ldots,[a_n]\}\).

If $B \to \infty$, 
then the first \(n\) elementary 
symmetric polynomials for the unordered list 
\(x_1^{-1},\ldots, x_n^{-1}\)
all go to \(0\), because
the \(i\)-th elementary symmetric polynomial
in \(x_1^{-1},\ldots, x_n^{-1}\) is
\(\,\frac{s_{n-i}}{s_n}\) for \(i=1,\ldots, n\).
Since $x_i = \wp(a_i)$, we get $a_i \to 0$ for all $i$. 
That is, $\bar\pi^{-1}(\infty) = (0, \ldots, 0)$.

We have proved Theorem \ref{alg-hyper}\,(1) at this point.
The statement of Theorem \ref{alg-hyper}\,(2) is a consequence
of the second paragraph of \ref{rem:disc_ell_n}
in the proof of Theorem \ref{hyper}\,(iii).
There we recalled that 
for a rectangular lattice 
\(\Lambda_{\tau}\) with
\(\tau\!\in\! \sqrt{-1}\,\RR_{>0}\), 
the polynomial \(\ell_n(B;\Lambda_{\tau})\) in \(B\) has \(2n+1\)
distinct real roots, and gave references for this fact.
Clearly this fact implies that the
discriminant of \(\,\ell_n(B)\,\) is not identically zero.
Theorem \ref{alg-hyper}\,(2) follows.
\end{proof}

\begin{exampless}
For $n = 1$, $s_0 = 1$, $s_1 = B$ and then
$$
C^2 = \ell_1(B) = 4 B^3 - g_2 B - g_3
$$
which is exactly the equation for $E$, since $\bar X_1 \cong E$.

For $n = 2$, $s_0 = 1$, $s_1 = \tfrac{1}{3} B$, 
$s_2 = \tfrac{1}{9} B^2 - \tfrac{1}{4} g_2$, and then
\begin{equation*}
\begin{split} 
C^2 = \ell_2(B) &= \tfrac{4}{81} B^5 - \tfrac{7}{27} g_2 B^3 
+ \tfrac{1}{3} g_3 B^2 + \tfrac{1}{3} g_2^2 B - g_2 g_3 \\
&= \tfrac{1}{81} (B^2 - 3g_2) (4 B^3 - 9 g_2 B + 27 g_3).
\end{split}
\end{equation*}
In terms of $s_1$, it is 
$C^2 = \ell_2(3s_1) = (3 s_1^2 - g_2)(4s_1^3 - g_2 s_1 + g_3)$.

For $n = 3$, $s_0 = 1$, $s_1 = \tfrac{1}{5} B$, $s_2 
= \tfrac{2}{75} B^2 - \tfrac{1}{4} g_2$, $s_3 
= \tfrac{1}{3^2 5^2} B^3 - \tfrac{1}{3\cdot 5} g_2 B 
+ \tfrac{1}{4} g_3$, and then
\begin{equation*}
\begin{split}
&C^2 = \ell_3(B) \\
&= \tfrac{1}{2^2 3^4 5^4} B(16B^6 - 504 g_2 B^4 + 2376 g_3 B^3 \\
&\qquad + 4185 g_2^2 B^2 - 36450 g_2 g_3 B + 91125 g_3^2 - 3375 g_2^3) \\
&=  s_1(\tfrac{500}{81} s_1^6 
- \tfrac{70}{9}g_2 s_1^4 + \tfrac{22}{3} g_3 s_1^3 
+ \tfrac{31}{12}g_2^2 s_1^2 - \tfrac{9}{2} g_2 g_3 s_1 
+ \tfrac{9}{4} g_3^2 - \tfrac{1}{12} g_2^3).
\end{split}
\end{equation*}
\end{exampless}

\begin{remark*} The referee has kindly informed us that
the curve \(\,C^2=\ell_2(B)\,\) appeared in 
\cite[p.\,63]{enolski-its} as a hyperelliptic curve 
\(\hat{C}\) whose affine coordinates \((z, w)\) are related
to \((B, C)\) here by
\(z=-B\) and \(w=\sqrt{-1}\frac{9\,C}{2}\).

The paper \cite{enolski-its} is based on the general construction
of spectral curves \(\Gamma_n\) 
introduced in \cite[\S1]{Krichever:ellip_KP}.
Note that the factors \(\,e^{\zeta(a_i)z}\frac{\sigma(z-a_i)}{\sigma(z)}\,\)
of the ansatz function \(\,w_{a}(z)\,\) in Definition \ref{def:w_a}
appeared in \cite[p.\,284]{Krichever:ellip_KP} up to a factor
\(-\sigma(a_i)\): the function \(\Phi(x,\alpha)\) in \cite{Krichever:ellip_KP} 
is \(\,-e^{\zeta(\alpha)z}\frac{\sigma(z-\alpha)}{\sigma(z)\sigma(\alpha)}\).

Explicit examples of Riemann surfaces associated to (finite gap) Lam\'e
potentials and Treibich-Verdier potentials can be found in 
\cite{Treibich-Verdier}.
\end{remark*}

\begin{remarkss}[Meaning of the parameter $C$] \label{MeaningC}
We have introduced the same notation $C$ in various places. 
Indeed they are all equivalent: 
The constant $C$ in (\ref{(0)}) coincides with the constant $C$ 
in (\ref{poly=const}) by setting 
$w = x_i = \wp(a_i)$ in (\ref{poly=const}). 
It also coincide with the Wronskian $C$ defined in (\ref{WronskianC}) 
up to sign by comparing (\ref{g=-C/X}) with the expression 
of $g(z)$ in (\ref{g-sum}) (using (\ref{poly=const}) and
(\ref{X-poly})). 
These equivalences allow us to study the 
hyperelliptic curve $Y_n$ from both the analytic 
and algebraic point of views at the same time. 
\end{remarkss}

\begin{remarkss}[Relation to KdV theory] \label{KdV}
There are several methods for computing compute $\ell_n(B)$ in the literature, 
e.g.~\cite{Dahmen, Waall}. 
It is also interesting to note that the hyperelliptic curve $\hat X_n$ 
also appears in the study of KdV equations, where it is known as 
the \emph{spectral curve}.

Indeed in KdV theory,  a differential operator $P_{2n + 1}$ 
of order $2n + 1$ is constructed by 
$$
P_{2n + 1} = \sum_{l = 0}^n (f'_{n - l}(z) - \tfrac{1}{2} f_{n - l}(z)) L^l, 
$$
where $L = -d^2/dz^2 + u(z)$, $f_0(z) = 1$ and $f_k(z)$ 
satisfies the recursive relation
\begin{equation} \label{(*)}
f'_{k + 1} = -\tfrac{1}{4} f'''_k + u f'_k + \tfrac{1}{2} u' f_k, 
\qquad k = 0, 1, 2, \cdots.
\end{equation}
Using the recursion (\ref{(*)}), we have
$$
[P_{2n + 1}, L] = 2f_{n + 1}'.
$$
A potential $u(z)$ is called a stationary solution to an $n$-th 
KdV hierarchy equation if $f'_{n + 1} = 0$. Let 
$$
F(z; E) = \sum_{l = 0}^n f_{n - l}(z) E^l.
$$
Then $F(z; E)$ satisfies 
\begin{equation} \label{(**)}
F''' - 4(u - E) F' - 2u' F = 0.
\end{equation}
Conversely, if $F(z; E)$ is a monic polynomial in $E$ of degree $n$ 
and satisfies (\ref{(**)}), write $F(z; E) = \sum_{l = 0}^n f_{n -  l}(z) E^l$. 
Then $f_k(z)$ satisfies (\ref{(*)}) with $f_k = 0$ for $k \ge n + 1$. 
By integrating (\ref{(**)}), we obtain
\begin{equation} \label{(***)}
\tfrac{1}{2} F'' F - \tfrac{1}{4} (F")^2 - (u - E) F^2 = R_{2n + 1}(E),
\end{equation}
where $R_{2n + 1}(E)$ is independent of $z$ and is a monic polynomial 
in $E$ of degree $2n + 1$. The spectral curve for the potential $u$,
if \(u\) is a stationary solution of the \(n\)-th KdV hierarchy, 
is by definition the hyperelliptic curve 
$$
y^2 = R_{2n + 1}(E);
$$
it parametrizes one-dimensional eigenspaces of the
commutator subring of the differential operator \(L\)
in the space of ordinary differential operators.

If $u(z) = n(n + 1) \wp(z)$ is the Lam\'e potential and $B = -E$, 
then (\ref{(**)}) is identical to (\ref{sym2-z}) with $F(z; E) =X(z)$. 
As we have seen already, $X(z)$ is also a polynomial in $\wp(z)$. 
By using $x = \wp(z)$, (\ref{int-eq}) is identical to (\ref{(***)}) 
(with $C^2 = 4 R_{2n + 1}(E)$ and $E = -B$). By this adjustment, 
the curve (\ref{ln-poly}) is identical to the spectral curve in KdV theory. 
For more details see \cite[Ch.1\,\S2]{GH}.
\medbreak

The Lam\'e potential is a very special type of
\emph{finite gap potentials}. 
There is an extensive literature. The readers may consult 
\cite{Ince, Dubrovin, IM, Novikov, MM, Krichever:alggeom_nonlineareq, Krichever:ellip_KP, Krichever:nonlinear_eqn_ellip_curve, Shiota, Treibich-Verdier, Krichever:ellip1994, Smirnov1994, Smirnov2002, Smirnov2006}.

The Lam\'e potential is also a special case of
\emph{Picard potential} \cite{GW2};
the system of equations (\ref{zeta-eq}) (i.e.~equations for $Y_n$ (\ref{Yn})) 
appeared in {\cite[(3.8) in p.82]{GW}}. 
According to \cite[Rmk.\,3.3]{GW}, that was the first time 
after \cite{Burkhardt} when (\ref{zeta-eq}) 
reappeared in mathematical publications. 
However in a comment in {\cite[p.83]{GW}} the authors 
said that \emph{the conditions} (\ref{zeta-eq}) 
\emph{appear to be too difficult to be handled directly}, 
so they turned to develop another method to compute the spectral curve.
\end{remarkss}



The following proposition arises from the study of the process $B \to
\infty$. 
When $x_i \to \infty$, we have $y_i \to \infty$ too. 
Asymptotically $(x_i, y_i) \sim (t_i^2, 2 t_i^3)$ 
hence $\sum x_i^k y_i \sim 2\sum t_i^{3 + 2k}$. 
The uniqueness of $\bar\pi^{-1}(\infty)$ suggests the uniqueness 
of solutions of the limiting equations up to permutations. 
It turns out to be true and can be proved along the similar reasoning as above.

\begin{proposition} \label{B=infty}
Consider the following system of $n - 1$ homogeneous equations 
in $\mathbb P^{n - 1}(\CC)$ ($n \ge 2$) with coordinates $t_1, \ldots, t_n$:
\begin{equation}\label{eqn-at-infty}
\sum_{i = 1}^n t_i^{2k + 1} = 0, \qquad k = 1, 2, \ldots, n - 1,
\end{equation}
subject to the non-degenerate conditions $\prod_{i = 1}^n t_i \ne 0$ 
and $\prod_{i < j} (t_i + t_j) \ne 0$. 
Then the solution exists uniquely up to permutations.
\end{proposition}

\begin{proof}
When $B \to \infty$, by either (\ref{NB}) or (\ref{C^2}) 
we see that all $t_i$'s have the same order $|B|^{1/2}$. 
Since the polynomial system in $t_i$'s comes from the leading order
terms of the original system $\sum x_i^k y_i = 0$, 
by passing to a subsequence if necessary, in the limit $B \to \infty$ 
we get a point $[t_1: \cdots: t_n] \in \mathbb P^{n - 1}$ 
solving the limiting equations. 
In fact $[t] \in \mathbb P(T_0(\bar X_n)) 
\subset \mathbb P(T_0 ({\rm Sym}^n E))$.

However a more careful argument is needed to verify the 
nondegeneracy
conditions. 
We recall that for $a \in X_n$, $\wp(a_i)$'s are the roots 
of the polynomial $X(x)$ where the coefficients $s_j(B)$'s 
satisfy the recursive relation (\ref{recursive}). 
Thus $\wp(a_i)/B$ tends to the roots of the limiting polynomial $X_\infty$:
\begin{equation*}
X_\infty(x) = x^n - \bar s_1 x^{n - 1} + \bar s_2 x^{n - 2} + \cdots + (-1)^n \bar s_n,
\end{equation*}
where we set $\bar s_0 = 1$ and
\begin{equation} \label{limit-s}
\bar s_k = \frac{2(n - k + 1)}{k(2n - 2k + 1)(2n - k + 1)} \bar s_{k -  1}, 
\qquad k = 1, \ldots, n.
\end{equation}
To prove $(\wp(a_i) - \wp(a_j))/B \not\to 0$ as $B \to \infty$ is
equivalent to 
showing that $X_\infty$ has $n$ distinct roots, a statement which 
does not seem to be obvious. Instead, we use (\ref{(0)}) in its analytic form
\begin{equation} \label{C-expr}
C = \wp'(a_i) \prod_{j \ne i} (\wp(a_i) - \wp(a_j)).
\end{equation} 
Obviously $|C| \sim |B|^{n + 1/2}$ and 
$|\wp'(a_i)| \sim |\wp(a_i)|^{3/2} \sim |B|^{3/2}$. 
Thus if there is some $j$ such that $|\wp(a_i) - \wp(a_j)| = o(1) |B|$ 
as $|B| \to \infty$ then (\ref{C-expr}) yields 
$$
|B|^{n + 1/2} \sim |C| \le o(1) |B|^{n + 1/2},
$$
which is a contradiction. Therefore we have 
$$
\lim_{B \to \infty} \frac{\wp(a_i)}{B} \ne \lim_{B \to \infty} \frac{\wp(a_j)}{B}
$$
for $i \ne j$. Now we write $(\wp(a_i), \wp'(a_i)) = (x_i, y_i) \sim
(t_i^2, 2t_i^3)$. 
Then the leading term of $\sum_i x_i^k y_i$ is $2\sum_i t_i^{2k + 3}$ 
for $k = 0, \ldots, n - 2$. 
By passing to $B \to \infty$, 
the limit of $t_i/|B|^{1/2}$ (still denoted by $t_i$) then satisfies
$$
\sum_{i = 1}^n t_i^{2k + 1} = 0, \qquad 1 \le k \le n - 1,
$$
and $t_i + t_j \ne 0$ for any $i, j$. This proves the existence of solutions.

The remaining task is to prove the uniqueness. 
While it may be possible to prove this by working harder on the
asymptotic equations, 
however, owing to its elementary nature, 
we will offer a purely elementary argument 
using only basic algebra and divisions. 

Before we proceed, notice that the loci $\prod t_i = 0$ 
or $\prod_{i <  j} (t_i + t_j) = 0$ 
provide positive dimensional solutions to the system. 
Thus it is crucial to analyze the non-degenerate conditions. 
By a Vandermonde-like determinant argument, 
it is easy to see that under the assumption $t_i \ne 0$ for all $i$, 
we have $t_i \ne -t_j$ for all $i \ne j$ 
if and only if $t_i^2 \ne t_j^2$ for all $i \ne j$.

Let $q(t) = \prod_{j = 1}^n (t - t_j) = \sum_{i = 0}^n (-1)^i s_{i}
t^{n - i}$, 
where $s_i$ is the $i$-th elementary symmetric polynomial in $t_j$'s, 
and $p_l = \sum_{i = 1}^n t_i^l$ being the 
Newton symmetric polynomial for all $l \ge 0$. Then
\begin{equation*}
\frac{q'(t)}{q(t)} + \frac{q'(-t)}{q(-t)} 
= \sum_{i = 1}^n \frac{1}{t - t_i} - \frac{1}{t + t_i} 
= 2 \sum_{m \ge 1} p_{2m - 1} t^{-2m}.
\end{equation*}
The conditions $p_3  = p_5 = \cdots = p_{2n - 1} = 0$ 
imply that (comparing degrees)
\begin{equation*}
\frac{q'(t)}{q(t)} + \frac{q'(-t)}{q(-t)} 
= \frac{2p_1}{t^2} + \frac{(-1)^n 2p_{2n + 1}}{t^2 q(t) q(-t)}.
\end{equation*} 
Denote by $u = t^2$, $u_i = t_i^2$, $G(u) = q(t) q(-t) = \prod_{i =
  1}^n (u - u_i)$, 
this then could be regarded as an equality in $\mathbb C(u)$ as
\begin{equation*}
\sum_{i = 1}^n \frac{t_i}{u - u_i} 
= \frac{p_1}{u} + \frac{(-1)^n p_{2n + 1}}{u G(u)}.
\end{equation*}
From now on we denote $' = d/du$, then 
\begin{equation*}
t_i = {\rm Res}_{u = u_i} = \frac{(-1)^n p_{2n + 1}}{u_i G'(u_i)}.
\end{equation*}
In particular, $u_i^3 G'(u_i)^2 = C^2$ is independent of $i$, 
where $C = (-1)^n p_{2n + 1}$. So $G(u) \mid u^3 G'(u)^2 - C^2$ 
and we may perform division to write
\begin{equation*}
\frac{u^3 G'(u)^2 - C^2}{G(u)^2} 
= \sum_{i = 1}^n \frac{h_i}{u - u_i} + n^2 u + 2n \tau_1
\end{equation*}
for some $h_i \in \mathbb C$ and $\tau_1 = \sum_{i = 1}^n u_i$. 
Using series expansion in $u - u_i$ we calculate
\begin{equation*}
{\rm Res}_{u = u_i} \frac{u^3 G'(u)^2 - C^2}{G(u)^2} 
= 3u_i^2 + \frac{2u_i^3 G''(u_i)}{G'(u_i)^2},
\end{equation*}
hence there are $a, b \in \mathbb C$ such that
\begin{equation*}
\frac{u^3 G'(u)^2 - C^2}{G(u)^2} 
= \frac{3u^2 G'(u)}{G(u)} + \frac{2u^2 G''(u)}{G(u)} + a u + b.
\end{equation*}
By division again, it is clear that 
$$
2u^3 G''(u) = 2(n(n - 1) u + 2(n - 1) \tau_1) G(u) + \cdots
$$ 
and $3u^2 G'(u) = 3(n u + \tau_1) G(u) + \cdots$. Hence
\begin{equation*}
a = -n(n + 1), \qquad b = -(2n + 1) \tau_1,
\end{equation*}
\begin{equation*}
u^3 G'^2 - 3u^2 G' G - 2u^2 G'' G - (au + b)G^2 = C^2.
\end{equation*}
Differentiation and simplification lead to $-G(u)$ times the equation
\begin{equation*}
2u^3 G''' + 9u^2 G'' - 2((n^2 + n - 3) u + (2n - 1)\tau_1) G' -n(n + 1) G = 0.
\end{equation*}

Write $G(u) = \sum_{i = 0}^n (-1)^{n - i} \tau_{n - i} u^i$ (so
$\tau_0 = 1$ 
and $\tau_k = 0$ if $k < 0 $ or $k > n$ by convention), 
the above linear third order ODE translates into the recursive relation:
\begin{equation*}
(i - n) (2i + 1) (i + n + 1) \tau_{n - i} 
= -2 (i + 1) (2n - 1) \tau_1 \tau_{n - i - 1}.
\end{equation*}
This rather short recursion (instead of four terms) is due to the fact 
that the ODE has an irregular singularity at $u = 0$. 
It is consistent for $i = n$ ($0 = \tau_0 \tau_{-1}$) 
and for $i = n - 1$ ($\tau_1 = \tau_1$), and then all 
$\tau_k$, $k \ge 2$, are completely determined by $\tau_1$. 
This proves the uniqueness of solution up to permutations.
\end{proof}

\begin{remark*}
The non-degeneracy conditions in Proposition \ref{B=infty} are essential: 
when \(n\geq 4\) the set of all degenerate solutions
has a natural structure as a positive dimensional algebraic variety.
\end{remark*}

\subsubsection{\bf Question.}\enspace\label{ques:arith_irred}
Let \((b_1:\ldots:b_n)\in \PP^{n-1}(\CC)\) be a non-degenerate
solution of equation (\ref{eqn-at-infty}).
Let \(K_n\) be the smallest subfield of \(\CC\) which contains
\({b_2}/{b_1},\ldots, {b_n}/{b_1}\).
Is \(\,[K_n:\QQ]=n!\,\)?\footnote{The answer is
likely ``yes'', but we don't have a proof.}


\begin{corollaryss}\label{cor:smooth_at_infty}
The curve $\bar X_n$ is smooth at the infinity point $[0]^n$.
\end{corollaryss}

\begin{proof}
The idea is that the solutions sought in Proposition \ref{B=infty} describe 
the tangent directions of $\bar X_n$ at $0^n$, in the
sense that 
the projectivized
tangent cone of $\bar X_n$ at $[0]^n$ is the affine open subset
projective spectrum of the ring
\[
\mathscr{R}=
\CC\left[t_1,\ldots,t_n \right]
\left/\left({\textstyle \sum_{i=1}^n} t_i^{2k+1}\right)_{1\leq k\leq n-1}
\right. ;
\]
associated by localization to the homogeneous element
\[\prod_{i=1}^nt_i\cdot
\prod_{1\leq i<j\leq n} (t_i+t_j)
\] of \(\mathscr{R}\).
Once we know this 
then the existence and
uniqueness statement in Proposition \ref{B=infty}
is equivalent the smoothness of $\bar X_n$ at $0^n$. 
However the above description of the projectivized
tangent cone of \(\bar{X}_n\) at \([0]^n\) is not 
self-evident from the definition of \(\bar{X}_n\) as the closure
of \(X_n\) in \(\textup{Sym}^n(\CC/\Lambda)\). So we 
proceed slightly differently.

Let \((r_1,\ldots, r_n)\) be a non-degenerate solution of the
system of equations in  Proposition \ref{B=infty}.
From the non-vanishing of the Vandermonde determinant one
sees that \(\,\sum_{i=1}^n r_i\neq 0\).
From Hensel's lemma one sees that there exists a morphism \(\alpha\)
from the spectrum of a formal power series ring
\(\CC[[t]]\) to the inverse image in \(\textup{Sym}^n(\CC/\Lambda)\)
of \(\bar{X}_n\) which sends the closed point
of \(\textup{Spec}\,\CC[[t]]\) to \([0]^n\), such that
\[
\frac{x_i}{y_i}\ \mapsto\  r_i\, t + O(t^2)\quad \forall\,i=1,\ldots,n.
\]
The condition that \(\,r_1+\cdots+r_n\neq 0\,\) tells us
that \(\alpha\) induces an isomorphism between
\(\CC[[t]]\) and the completed local ring of \(\bar{X}_n\)
at the point \([0]^n\).
\end{proof}

\subsection{Comparing the compactifications
\(\bar{X}_n\) and \(\hat{X}_n\) of \(X_n\)}
\label{subsec:compare}

\subsubsection{}\label{subsubsec:compare-setup}
At this point we have two compactifications of the smooth
algebraic curve \(X_n\). We summarize the situation.
\medbreak

\noindent
{\bf \ref{subsubsec:compare-setup}.a.}\enspace
By definition \(X_n\) is a locally closed algebraic subvariety 
of the symmetric product \(\textup{Sym}^n(\CC/\Lambda)\). 
The first compactification \(\bar{X}_n\) of \(X_n\) 
is the closure of \(X_n\) in 
\(\textup{Sym}^n(\CC/\Lambda)\).
We have seen that \(\bar{X}_n\) contains the closed subvariety
\(Y_n\) of \(\textup{Sym}^n(\CC/\Lambda\!\!\smallsetminus\{[0]\})\).
The latter variety \(Y_n\) classifies 
all ansatz solutions modulo \(\CC^{\times}\) to Lam\'e equations
of index \(n\in \NN_{>0}\). 
\medbreak

\noindent
{\bf \ref{subsubsec:compare-setup}.b.}\enspace
The map ``multiplication by \(-1\)'' on \(\CC/\Lambda\)
defines an involution \(\tilde\iota\) on \(\textup{Sym}^n(\CC/\Lambda)\).
The subvarieties \(X_n, Y_n, \bar{X}_n\) of
\(\textup{Sym}^n(\CC/\Lambda)\)
are stable under the involution \(\tilde\iota\).
The restriction of \(\tilde\iota\) to \(\bar{X}_n\) is an involution
\(\bar\iota\) on \(\bar{X}_n\).
It turned out that \(X_n\) is the complement in \(\bar{X}_n\)
of the fixed point set \((\bar{X}_n)^{\bar\iota}\) of the involution
\(\bar\iota\). 
One of the fixed points of \(\bar\iota\) is the
point \([0]^n=\{[0],\ldots, [0]\}\) of \(\textup{Sym}^n(\CC/\Lambda)\);
the rest are all in \(Y_n\).
In particular \(\bar{X}_n\smallsetminus Y_n=\{[0]^n\}\).
It is known that 
\(\,\#(\bar{X}_n)^{\bar\iota}\leq
2n+2\),
and the equality \(\,\# (\bar{X}_n)^{\bar\iota}=
2n+2\) holds for all \(\Lambda\) outside of a finite
number of homothety classes of lattices in \(\CC\).
\medbreak

\noindent
{\bf \ref{subsubsec:compare-setup}.c.}\enspace
The map \(\pi_n: Y_n\to \CC\) which sends 
a point \([a]\in Y_n\) to the accessary parameter \(B_{[a]}\)
of the Lam\'e equation satisfied by 
the ansatz function \(\,w_a\,\) is an 
algebraic morphism from \(Y_n\) to the affine line \(\mathbb{A}^1\)
over \(\CC\).
The morphism \(\pi_n: Y_n\to \mathbb{A}^1\) extends to
a morphism \(\bar{\pi}_n: \bar{X}_n\to \PP^1\). 
This projection morphism \(\bar{\pi}_n\) is compatible with the involution
\(\iota\) in the sense that \(\bar{\pi}_n=\bar{\pi}_n\circ\bar\iota\),
and \(\bar{\pi}_n(\underline{x})=\bar{\pi}_n(\underline{x}')\) for two points
\(\underline{x}, \underline{x}'\in \bar{X}_n\) if and only if
either \(\underline{x}=\underline{x}'\) or 
\(\,\bar\iota(\underline{x})=\underline{x}'\).

In particular \(\bar{\pi}_n\) induces a bijection
from the fixed point set \((\bar{X}_n)^{\bar\iota}\) to
a finite subset \(\,\bar\Sigma_n\subset \PP^1(\CC)\).
This ramification locus \(\bar\Sigma_n\) for \(\bar{\pi}_n\) 
is the disjoint union
of \(\{\infty\}\) with a finite subset \(\Sigma_n\subset \CC\).
The restriction \(\,\left.\pi_n\right\vert_{X_n}\,\) of \(\pi_n\)
to \(X_n\) makes \(X_n\) an unramified double cover
of the complement 
\(\,\mathbb{A}^1\!\smallsetminus\! \Sigma_n\)
of \(\,\Sigma_n\) in \(\mathbb{A}^1\).
\medbreak

\noindent
{\bf \ref{subsubsec:compare-setup}.d.}\enspace
The ramification locus \(\,\Sigma_n\,\) is
the set of roots of a polynomial \(\ell_n(B)\)
\[\ell_n(B)=
4B s_n^2 + 4g_3(\Lambda) s_{n - 2} s_n 
- g_2(\Lambda) s_{n - 1} s_n - g_3(\Lambda) s_{n - 1}^2
\]
of degree \(2n+1\) in the variable \(B\) 
with coefficients in \(\QQ[g_2, g_3]\),
where the polynomials \(s_n, s_{n-1}, s_{n-2}\in \QQ[g_2,g_3][B]\) are
defined recursively by equations
(\ref{recursive}), starting with \(s_0=1\) and
\(s_1=(2n-1)^{-1}B\).
The recursive relation (\ref{recursive}) implies that
\(\ell_n(B, g_2, g_3)\) is homogenous of weight \(2n+1\)
if \(g_2, g_3\) are given weights \(1, 2, 3\) respectively;
the coefficient of \(B^{2n+1}\) in \(\ell_n(B)\)
is a positive rational number.\footnote{The coefficient of 
\(B^i\) in \(s_i(B)\) is \\
\(\frac{2^i}{2n-1}\cdot 
\frac{n(n-1)(n-2)\cdots(n-i+1)}{[(2n)(2n-1)(2n-2)\cdots (2n-i+1)]
\cdot[(2n-1)(2n-3)(2n-5)\cdots(2n-2i+1)]}
\,\) for \(i=1,\ldots, n\).}

\medbreak

\noindent
{\bf \ref{subsubsec:compare-setup}.e.}
The polynomial \(\,\ell_n(B)\,\) gives rise to another
compactification \(\,\hat{X}_n\,\) of \(X_n\).
Let \(X_n^{\ast}\) be the 
zero locus of the homogeneous polynomial
\[
F_n(\hat{A},\hat{B},\hat{C})
:= \hat{C}^2\hat{A}^{n-1} - 
\hat{A}^{2n+1}\ell_n({\hat{B}}/{\hat{A}})
\]
in the projective plane $\PP^2$ with projective coordinates
\((\hat{A}:\hat{B}:\hat{C})\).
By definition \(\hat{X}_n\) is the partial desingularization of
\(X_n^{\ast}\), changing the local structure near
the singular point \((0:0:1)\) by replacing the 
structure sheaf near \((0:0:1)\) by its normal closure
in the field of fractions.
More explicitly we replace a small Zariski open neighborhood
of the point \((0:0:1)\) in \(X_n^{\ast}\)
by the corresponding open neighborhood of the curve
\[
v^2= u \cdot (u^{2n+1}\ell_n(1/u))
\]                                                                  
near \((u,v)=(0,0)\); the coordinates are related by
\[
\frac{\hat{B}}{\hat{A}}=B=\frac{1}{u},\ \
\frac{\hat{C}}{\hat{A}}=C=\frac{v}{u^{n+1}}.
\]
The natural morphism \(\hat{X}_n\to X_n^{\ast}\) is 
a homeomorphism, and is a local isomorphism outside the 
point \(\hat\infty\) which maps to the point \((0:0:1)\in 
X_n^{\ast}\).
The projective curve \(\hat{X}_n\) 
is reduced, irreducible
and has arithmetic genus
\(n\); we call it the
``hyperelliptic model'' of \(X_n\). 

We have a ``hyperelliptic involution'' \(\hat\iota\)
on \(\hat{X}_n\), given by
\[
\hat\iota\colon
\,(\hat{A}:\hat{B}:\hat{C})\,
\mapsto (\hat{A}:\hat{B}:-\hat{C})
\]
in projective coordinates.
The \(X_n\) is the complement in \(\hat{X}_n\) of the fixed point set
\(\,(\hat{X}_n)^{\hat\iota}\,\) of the hyperelliptic involution.
We also have a morphism
\(\hat\pi_n:\hat{X}_n\to \PP^1\), defined by
\(\hat\pi_n \colon (\hat{A}:\hat{B}:\hat{C})\mapsto
(\hat{A}:\hat{B})\)
over the open subset of \(\hat{X}_n\) where \(\hat{A}\)
is invertible, and 
\(\hat\pi_n\colon 
(\hat{A}:\hat{B}:\hat{C}) \mapsto
\big(\tfrac{\hat{A}}{\hat{C}}:\tfrac{\hat{B}}{\hat{C}}\big)\)
over the open subset of \(\hat{X}_n\) where \(\hat{C}\) is invertible.
One of the fixed points \(\hat\iota\) is \(\hat\infty\).
The map \(\hat\pi_n\) induces a bijection
from \((\hat{X}_n)^{\hat\iota}\) to \(\bar\Sigma_n\).

The hyperelliptic projection \(\hat\pi_n\) is compatible with
the hyperelliptic involution \(\hat\iota\) on \(\hat{X}_n\), 
in the sense that 
\(\hat\pi_n(P)=\hat\pi_n(P')\) if and only 
either \(P'=P\) or \(P'=\hat\iota(P)\),
for any two points \(P, P'\in \hat{X}_n\).
The restriction of the triple
\((\hat{X}_n,\hat\pi_n, \hat\iota)\) to the open subset
\(X_n\subset\hat{X}_n\)
is naturally identified with the 
restriction to \(X_n\) of the
triple \((\bar{X}_n,\bar\pi_n, \bar\iota)\).

\subsubsection{} \label{subsub:related-two-comp}
A natural and inevitable question is: 
\begin{quote}
\emph{Is there an isomorphism between the two triples 
\((\bar{X}_n,\bar\pi_n, \bar\iota)\)
and
\((\hat{X}_n,\hat\pi_n, \hat\iota)\) 
which extends the natural isomorphism 
between the complements of the fixed point sets 
of \(\bar\iota\) and \(\hat\iota\)?}
\end{quote}
The parallel properties of the two compactifications reviewed in
\ref{subsubsec:compare-setup} suggest that the answer is 
likely ``yes''.  
Since both \(\bar{X}_n\) and \(\hat{X}_n\) are reduced and
irreducible, to answer this question affirmatively,
we need to show that the natural identification of the
``common open'' dense subset \(X_n\) of both sides
extends to an isomorphism. 
\medbreak

We will see in Lemma \ref{lemma:mor-bar-to-hat} 
that methods in the previous
part of this section already shows that 
the identity map on \(X_n\) extends to a \emph{morphism}
\(\,\phi: \bar{X}_n\to \hat{X}_n\) of
algebraic varieties. 
That \(\phi\) is a morphism at \(\bar\infty=[0]^n\)
is a consequence of (and equivalent to) 
Corollary \ref{cor:smooth_at_infty}.
\medbreak

The following properties of the morphism
\(\phi:\bar{X}_n\to \hat{X}_n\) 
between reduced irreducible complete algebraic curves
are easily deduced from previous arguments:
\begin{itemize}
\item[(a)] \(\phi\) is bijective on points,
i.e.\ \(\phi\) is a \emph{homeomorphism}.

\item[(b)] \(\phi\) is an isomorphism over
\(X_n=\bar{X}_n\smallsetminus (\bar{X}_n)^{\bar\iota}\).

\item[(c)] \(\phi\) is an isomorphism near 
the point \(\bar{\infty}=[0]^n\).
This is a rather trivial case of Zariski's Main Theorem and
is easily verified directly.
\end{itemize}
So we are left with showing that 
\(\phi\) is a local isomorphism at each point of
\((Y_n)^{\iota}\), ramification points ``at finite distance''.

\subsubsection{}\label{subsub:use_purity}
The properties in \ref{subsub:related-two-comp}\,(a)--(c) 
of the morphism \(\,\phi:\bar{X}_n\to \hat{X}_n\,\) 
do \emph{not} formally imply that \(\,\phi\,\) is an isomorphism:
it may happen that there exists a point 
\(P\in Y_n\subset \bar{X}_n\) such that the
\emph{injection}
\[\,\phi^{\ast}:\ringO_{\hat{X}_n, \phi(P)} \to \ringO_{\bar{X}_n, P}\,\]
induced by \(\phi\)
between the stalks of the structure sheaves of 
\(\bar{X}_n\) and \(\hat{X}_n\) at \(P\) and \(\phi(P)\)
is not an isomorphism.
If this ``undesirable'' phenomenon happens at one 
ramification point \(P\in Y_n\), the arithmetic genus
of \(\bar{X}_n\) will be strictly smaller than
\(n\), the arithmetic genus of \(\hat{X}_n\).
To put it differently, the fact that 
\(\phi:\bar{X}_n\to \hat{X}_n\) is a bijective morphism
tells us that the hyperelliptic model \(\hat{X}_n\) 
can ``only be more singular'' than \(\bar{X}_n\).
\medbreak

If we can show that the arithmetic genus of \(\bar{X}_n\) is $n$,
it will follow that \(\phi\) is an isomorphism.
This approach may well be possible, but we will take
an easier route: 
We have seen that the discriminant
\(\textup{disc}(\ell_n(B))\) of \(\ell_n(B)\)
is a non-zero holomorphic modular form for the full modular group
\(\textup{SL}_2(\ZZ)\).
If the discriminant
\(\textup{disc}(\ell_n(B))\) does not vanish 
when evaluated at the lattice \(\Lambda\subset \CC\)
in question, then the polynomial \(\ell_n(B;\Lambda)\) has 
\(2n+1\) distinct roots in \(\CC\) and 
\(\hat{X}_n\) is smooth, which forces the 
morphism \(\phi:\bar{X}_n\to \hat{X}_n\) to be an isomorphism.
\smallbreak

Suppose now that the elliptic curve we are given
is \(\Lambda_{\tau_0}=\ZZ+\ZZ\tau_0\)
for some \(\tau_0\in\HH\) such that 
\(\textup{disc}(\ell_n(B))(\Lambda_{\tau_0})=0\).
The idea now is embed the given situation in a
one-parameter family such that the morphism
\(\phi\) is an isomorphism outside the
central fiber, then use purity (Hartog's theorem):
\begin{quote}
\emph{Let \(\tau\) vary in a small open disk \(D\subset \HH\)
containing \(\tau_0\) such that 
\(\textup{disc}(\ell_n(B))(\Lambda_{\tau})\neq 0\) for all 
\(\tau\in D\) and get a family of maps 
\(\left(\phi_{\tau}: \bar{X}_{n,\tau}\to
  \hat{X}_{n,\tau}\right)_{\tau\in D}\)
parametrized by \(D\).
Use the fact that \(\phi_{\tau}\) is an isomorphism
for all \(\tau\) in the punctured disk 
\(D^{\ast}=D\smallsetminus\{\tau_0\}\)
and \(\phi_{\tau_0}\) is an isomorphism outside
a finite subset of \(\bar{X}_{n,\tau_0}\) to show that
\(\phi_{\tau_0}\) itself is an isomorphism.}
\end{quote}
Details will be carried out in Proposition \ref{prop:use-purity} 

\begin{lemmass}\label{lemma:mor-bar-to-hat}
The identity map \(\,\textup{id}_{X_n}: X_n\to X_n\,\)
on \(X_n\) extends uniquely to a \emph{morphism}
\(\phi: \bar{X}_n\to \hat{X}_n\).
\end{lemmass}

\begin{proof}
Let \((x_1, y_1),\ldots, (x_n, y_n)\) be the Weierstrass
coordinates of the product \(E^n=E\times\cdots \times E\). 
The affine coordinate \(B\) of \(X_n^{\ast}\) is
given by \(B=(2n-1)\sum_{i=1}^n x_i\).
From the proofs of Theorems \ref{hyper} and \ref{alg-hyper}
we see that the other affine coordinate \(C\) of \(X_n^{\ast}\) 
can be expressed by polynomials in the \(x_i\)'s and \(y_i\)'s:
\(\,C=y_i\cdot \prod_{j\neq i} (x_i-x_j)\,\)
for all \(i=1,\ldots,n\).
It follows that the birational map \(\phi\) from
\(\bar{X}_n\) to \(\hat{X}_n\) is 
a morphism at every point of
\(\bar{X}_n\smallsetminus \{\bar\infty\}\).
Corollary \ref{cor:smooth_at_infty} implies that
\(\phi\) is a morphism at \(\bar\infty\) as well.
\end{proof}

\begin{proposition}\label{prop:use-purity}
The morphism \(\,\phi:\bar{X}_n\to \hat{X}_n\,\) is an isomorphism.
\end{proposition}

\begin{proof}
We may and do assume that the given lattice \(\Lambda\)
is \(\Lambda_{\tau_0}\) for an element \(\tau_0\in \HH\).
Let \((x_1, y_1),\ldots, (x_n, y_n)\) be the Weierstrass
coordinates of \(E^n\) as in the proof of 
Lemma \ref{lemma:mor-bar-to-hat}, where \(E=\CC/\Lambda\).
It suffices to prove the following Claim: for every monomial 
\(\,h(\underline{x},\underline{y})\)
in \(x_1,\ldots,x_n, y_1,\ldots, y_n\),
the restriction to 
\(X_n\)
of its \(\textup{S}_n\)-symmetrization
\[
\Pi_{\textup{S}_n}(h)=(n!)^{-1} 
\sum_{\sigma\in\textup{S}_n} h(\sigma(\underline{x}),
\sigma(\underline{y}))
\]
is the pull-back under \(\phi\) of a 
polynomial in \(B\) and \(C\).
\smallbreak

We have seen in the proofs of Theorems \ref{hyper} and \ref{alg-hyper}
that the restriction to \(X_n\) of 
every symmetric polynomial in \(x_1,\ldots, x_n\) can
be expressed as a polynomial in \(B\).
Because 
\[
\CC[x_1,\ldots, x_n,y_1,\ldots,y_n]
\left/ \left(
y_i^2 - 4x_i^3 + g_2 x_i + g_3\right)_{1\leq i\leq n}
\right.
\]
is a free module over 
the ring of symmetric polynomials 
\(\CC[x_1,\ldots, x_n]^{\textup{S}_n}\)
with basis given by monomials  of the form
\[
\underline{x}^{\underline{r}}\,\underline{y}^{s}
=\prod_{i=1}^n x_i^{r_i} \cdot
\prod_{i=1}^n y^{s_i}
\]
with 
\(0\leq r_i\leq n-i\) and \(0\leq s_i\leq 1\) for
all \(i=1,\ldots, n\), 
and the symmetrization operator \(\Pi_{\textup{S}_n}\)
is linear over the ring of symmetric polynomials in
\(x_1,\ldots, x_n\),
it suffices to prove the claim in the case when
\(\,h(\underline{x},\underline{y})\,\) is 
one of the above basis elements 
\(\underline{x}^{\underline{r}}\,\underline{y}^{s}\).
In principle one should be able to
show directly that the symmetrization
\(\,\Pi_{\textup{S}_n}(\underline{x}^{\underline{r}}\,
\underline{y}^{s})\,\) of 
\(\underline{x}^{\underline{r}}\,\underline{y}^{s}\)
is a polynomial in \(B\) and \(C\).
Here we take an easier way out using purity
as indicated in \ref{subsub:use_purity}.

Let \(\tau\) vary over \(\HH\), and 
consider the fiber product over \(\HH\)
of \(n\) copies of the universal elliptic curve
whose fiber over \(\tau\in \HH\) is
\(E_{\tau}= (\CC/\Lambda_{\tau})\).
Over the open subset of the \(n\)-time
fiber product of \(E\smallsetminus \{[0]\}\), 
we have affine coordinates
\(x_1, y_1,\ldots, x_n, y_n\) as before.
Let
\(\,\mathscr{X}_n\,\) 
be the relative affine spectrum over \(\HH\)
of
\begin{equation*}
\ringO_{_{\HH}}\!\!\left[\!\!
\begin{array}{c}x_1,\ldots, x_n, y_1,\ldots, y_n,
\\
(\textstyle \prod_{1\leq i<j} (x_i-x_j))^{-1}
\end{array}
\!\!\right]
\left/ \left(\!\!
\begin{array}{c}
{\textstyle \sum_{j=1}^n x_j^l\,y_j,}\ \ \
{\scriptstyle l=0,1,\ldots, n-2},
\\
y_i^2- 4 x_i^3+g_2(\Lambda_{\tau}) x_i + g_3(\Lambda_{\tau}),\ \ 
{\scriptstyle 1\leq i\leq n}
\end{array}
\!\!\right)
\right.
\end{equation*}
Let \(\,\mathscr{X}_n^{\ast}\,\) be the relative projective
curve over \(\ringO_{\HH}\) defined by the 
homogeneous polynomial
\[F_n(\hat{A},\hat{B},\hat{C})
= \hat{C}^2\hat{A}^{n-1} - 
\hat{A}^{2n+1}\ell_n({\hat{B}}/{\hat{A}})
\]
as in \ref{subsubsec:compare-setup}.e, where
\(g_2=g_2(\Lambda_{\tau})\) and
\(g_3=g_3(\Lambda_{\tau})\) in the definition of \(\ell_n(B)\).
Let \(\Phi:\bar{\mathscr{X}}_n\to \mathscr{X}_n^{\ast}\)
be the morphism extending the identity map on \(\mathscr{X}_n\).
Let \(\mathscr{U}\) be the complement in \(\mathscr{X}_n^{\ast}\)
of the set of all ramification points over
those \(\tau\)'s where the discriminant
\(\textup{disc}(\ell_n(B))\) vanishes,
so that \(\mathscr{U}\) is the complement of
the union of the section \(\bar{\infty}\) and 
a discrete set of points in \(\mathscr{X}_n^{\ast}\).
We know that \(\Phi\) is an isomorphism over \(\mathscr{U}\).
Notice that the two-dimensional variety
\(\mathscr{X}_n^{\ast}\) is \emph{normal} outside the
zero locus of \(\hat{A}\), because it is regular
in codimension one and Cohen--Macaulay (in codimension two).

For any symmetrized monomial
\(\,\Pi_{\textup{S}_n}(\underline{x}^{\underline{r}}\,
\underline{y}^{s})\,\) of 
\(\underline{x}^{\underline{r}}\,\underline{y}^{s}\)
considered earlier,
we know that its restriction to \(\mathscr{U}\)
is equal the pull-back of the restriction 
to the a regular function 
on the open subset \(\Phi(\mathscr{U})\)
of codimension at least \(2\).
By purity (or Hartog's theorem for normal
analytic spaces)
this regular function on \(\Phi(\mathscr{U})\)
extends to a regular function 
\(\,h_{\underline{r},\underline{s}}\,\) on
the complement in \(\mathscr{X}_n^{\ast}\)
of the section
\((0:0:1)\) ``at infinity''.
Restricting to the fiber over \(\tau_0\) and the Claim
follows. We have proved the \(\phi\) is an isomorphism
for every elliptic curve of the form
\(E_{\tau}=\CC/\Lambda_{\tau}\), for any element
\(\tau\) in the upper-half plane \(\HH\).
\end{proof}

\begin{remarkss}
There is a variant of the proof following the same idea,
but uses Zariski's Main Theorem instead of purity:
take the closure \(\bar{\mathscr{X}}_n\) of
\(\mathscr{X}_n\) in the \(n\)-th symmetric product
of the universal elliptic curve.
Apply Zariski's Main Theorem to the
map from the normalization of \(\bar{\mathscr{X}}_n\)
to \(\hat{\mathscr{X}}_n^{\ast}\). 
One needs to be careful when it comes to the
operation of taking the closure, for in general
this operation does not commute with the operation
of passing to a fiber.  
Details are left to the interested reader.
\end{remarkss}

\section{Deformations via blow-up sequences} \label{deformation}
\setcounter{equation}{0}

\subsection{}\label{subsec:begin_sec8}
In this section we will prove Theorem \ref{blow-up-set} 
concerning the \emph{blow-up set} of a \emph{blow-up sequence} $u_k$ 
to the mean field equation $\triangle u_k + e^{u_k} = \rho_k \delta_0$ 
with $\rho_k \to 8\pi n$ on \(E=\CC/\Lambda\).  Recall that 
the assumption that
\(u_k\) is a blow-up sequence means that
the subset \(\,S\subset E\,\) consisting of all elements \(P\in E\) such
that \(\,\lim_{k\to \infty}u_k(P)=\infty\,\) is a non-empty finite
subset of \(E\);
this subset \(\,S\,\) is called the
blow-up set of the blow-up sequence \((u_k)\).

\subsubsection{}
The following facts are known; see for instance 
\cite[Thm.\,3,\,p.\,1237]{BM}, 
\cite[p.\,1256]{LS}.
\begin{itemize}

\item[(i)] \(\lim_{k\to\infty}u_k(x)=-\infty\) for all \(x\in
  E\!\smallsetminus \!S\), uniformly on compact subsets
of \(E\!\smallsetminus \!S\), 

\item[(ii)] There exists an \(\,\big(8\pi\NN_{\geq 1}\big)\)-valued function
\(P\mapsto \alpha_P\) on \(S\) such that the limit
\(\,\lim_{k\to \infty} \left.e^{u_k}\right\vert_{E}\,\) 
converges to the measure \(\,\sum_{P\in S}\alpha_{P}\,\delta_P\,\)
on \(E\), where 
\(\delta_{P}\) denotes the delta-measure at \(P\) for all \(P\in S\).

\end{itemize}
Note that \(\,\sum_{P\in S} \alpha_P=8\pi n\)
because \(\,\int_E e^{u_k}=\rho_k\,\) for all \(k\), and this
sequence converges to \(8\pi n\) by assumption. 

Clearly the blow-up set of a blow-up sequence does not change
if we pass to a subsequence, therefore we may assume
either (1) \(\rho_k\neq 8\pi n\) for all \(k\), or
(2) \(\rho_k=8\pi n\) for all \(k\).
Theorem \ref{blow-up-set} asserts that the blow-up set \(S\) 
is an element of \(Y_n\) and \(\alpha_P=8\pi \) for all \(P\in S\);
moreover \(S\in X_n\) in case (1) and \(S\in
Y_n\!\smallsetminus\!X_n\)
in case (2).

\subsubsection{} Part (2) of the Theorem \ref{blow-up-set}, 
namely the case $\rho_k = 8\pi n$, 
follows easily from results in \S\ref{II-even}, \ref{X_n} and
\ref{hyper-ell-str}.  Suppose that \(n\) is a positive integer
\((u_k)_{k\in\NN}\) is a blow-up sequence of solutions of 
\(\,\triangle u+e^u=8\pi\delta_0\,\) on \(E=\CC/\Lambda\).
By Theorems \ref{thm:g_even}, \ref{thm:hecke-system},
Proposition \ref{eq-2-sys} and Theorem \ref{poly-eqn}, for each \(k\)
there exists an element \([a^{(k)}]\in X_n\) and 
a real number \(\lambda_k\in\RR\) such that
\[
u_k(z)=\log\frac{8\, e^{{\scalebox{0.6}{$2\lambda_k$}}}\,
\vert\, f'_{[{\scalebox{0.5}{$a^{(k)}$}}]}\,\vert^2}{
(1+e^{{\scalebox{0.6}{$2\lambda_k$}}}
\vert f_{[{\scalebox{0.5}{$a^{(k)}$}}]}(z)\vert^2)^2}
\qquad \forall\,k,
\]
where 
\(f_{[{\scalebox{0.5}{$a^{(k)}$}}]}(z)=w_{[{\scalebox{0.5}{$a^{(k)}$}}]}(z)/
w_{{[{\scalebox{0.5}{$-a^{(k)}$}}]}}(z)\) is the quotient of the 
ansatz functions as in Definitions \ref{def:f_a} and \ref{def:w_a}.
The curve \(\bar{X}_n\) being projective, after passing to 
a subsequence we may and do assume that 
the sequence \([a^{(k)}]\in X_n\) converges to a point 
\([x_0]\in \bar{X}_n\).  We claim that \([x_0]\in X_n\), i.e.\ 
\(x_0\) is not a ramification point of 
\(\bar{\pi}_n:\bar{X}_n\to \PP^1(\CC)\).
For otherwise the sequence of functions 
\(f_{[{\scalebox{0.5}{$a^{(k)}$}}]}(z)\) converges to the
constant function \(1\) uniformly on compact subsets of \(E\),
which implies that \(u_k\) cannot be a blow-up sequence, a
contradiction which proves the claim.

From the fact that \([x_0]\in X_n\) it follows that 
the sequence \(f_{[{\scalebox{0.5}{$a^{(k)}$}}]}(z)\) converges to
\(f_{[x_0]}\). Therefore the sequence \(\lambda_k\) goes to \(\infty\),
and \(S=[x_0]\in X_n\).
We have proved Theorem \ref{blow-up-set}\,(2). 
The remaining case\,(1) when $\rho_k \ne 8\pi n$ for all
$k$ will be  proved in Theorem \ref{blow-up thm} and Corollary \ref{branch-pt}.
\qed

\subsection{The setup}

Let \(n\) be a positive integer.
Write $\rho = 8\pi \eta$, $\eta \in \mathbb R^+$. 
Consider a solution $u$ be of
\begin{equation} \label{MFE-eta}
\triangle u + e^u = 8\pi \eta\,\delta_0
\end{equation}
on $E = \mathbb C/\Lambda$, $\Lambda = \mathbb Z \omega_1 + \mathbb Z \omega_2$,
where the parameter \(\eta\!\in\!\RR_{>0}\) satisfies
$|\eta - n| < \tfrac{1}{2}$. 
When the parameter \(\eta\) satisfies \(n-\tfrac{1}{2}<\eta< n+\tfrac{1}{2}\),
the topological
Leray-Schauder degree of the equation (\ref{MFE-eta}) is non-zero,
hence it has solution.
Let \((u_k)_{k\in\NN}\) be a sequence of solutions of ({MFE-eta})
with parameters \(\eta_k\) 
in the above range, and assume that \(\,\lim_{k\to\infty} \eta_k= n\).
We are interested in knowing the behavior of this sequence \((u_k)\).

\subsubsection{}
The natural map \(\CC\!\smallsetminus\!\Lambda \to
E\!\smallsetminus\!\{[0]\}\) 
is a Galois covering space with group \(\Lambda\).
We know that \(\CC\!\smallsetminus\!\Lambda \) has a universal
covering isomorphic to the upper-half plane \(\HH\).
Let \(\,z:\HH\to \CC\!\smallsetminus\!\Lambda\) be a universal
covering map; here we have abused the notation and use
the same symbol ``\(z\)'' for both the coordinate function on \(\CC\)
and this covering map.
Denote by \(\,[z]\,\) the composition of 
\(z\) with the natural projection map 
\(\CC\!\smallsetminus\!\Lambda \to
E\!\smallsetminus\!\{[0]\}=:E^{\times}\), so that 
\([z]: \HH\to E^{\times}\) is a universal covering map
of \(E^{\times}\).

\subsubsection{}
Let \(\Gamma\subset \textup{PSL}_2(\RR)\) be the discrete 
subgroup of \(\textup{PSL}_2(\RR)\) consisting of all
deck transformations of the covering map 
\([z]:\HH\to E^{\times}\) , so that
\(\Gamma\) is naturally isomorphic to the fundamental group
of \(E^{\times}\)
Let \(\Delta\subset \Gamma\) be the group of all
deck transformations of the covering map
\(z:\HH\to \CC\!\smallsetminus\!\Lambda\).
We know that \(\Delta\) is a normal subgroup of 
\(\Gamma\), and the quotient 
\(\Gamma/\Delta\) is naturally isomorphic to \(\Lambda\),
the Galois group of the Galois cover 
\(\CC\!\smallsetminus\!\Lambda\to E^{\times}\),
therefore \(\Delta\) is equal to the 
subgroup \([\Gamma,\Gamma]\) of \(\Gamma\) generated
by all commutators.
\smallbreak

The fundamental group of \(E\!\smallsetminus\!\{[0]\}\) is
a free group in two generators, so
\(\Gamma\) is a finitely generated Fuchsian subgroup
of \(\textup{PSL}_2(\RR)\).
The Fuchsian subgroup \(\Delta\subseteq \Gamma\) is not finitely
generated; it is a free group with a set of free generators
indexed by \(\Lambda\).
\smallbreak

Let \(\HH^{\ast}\) be the union of \(\HH\) and the set of all 
cusps\footnote{Recall that a \emph{cusp} with respect to \(\Gamma\)
is an element of \(\PP^1(\RR)=\RR\cup\{\infty\}\) which is 
fixed by a parabolic element in \(\Gamma\).}
with respect to \(\Gamma\), with the usual topology 
as defined in \cite[pp.\,8--10]{Shimura-redbook}
which is compatible with the \(\Gamma\)-action.
Note that \(\HH^{\ast}\) is contractible.

\begin{lemmass}
\textup{(1)} The map \(z:\HH\to \CC\!\smallsetminus\!\Lambda\) extends uniquely
to a continuous \(\Delta\)-invariant map \(z^{\ast}:\HH^{\ast}\to \CC\), 
which lifts the continuous \(\Gamma\)-invariant map
\([z]^{\ast}:\HH^{\ast}\to E\) extending
\([z]:\HH\to E^{\times}\).  
\smallbreak
\noindent
\textup{(2)} The maps \(\,z^{\ast}\) and \([z]^{\ast}\) induce homeomorphisms
\(\,\Delta\backslash \HH^{\ast}\xrightarrow{\sim} \CC\)
and \(\,\Gamma\backslash \HH^{\ast}\xrightarrow{\sim} E\) respectively.
\end{lemmass}
\begin{proof} The existence of the latter map
\([z]^{\ast}:\HH^{\ast}\to E\) is well-known, and  
the existence of the
former map \(z^{\ast}:\HH^{\ast}\to \CC\) follows from the
existence of the latter because \(H^{\ast}\) is contractible.
We have proved (1). The statement (2) follows from the
statement (1).
\end{proof}

\subsubsection{}
Choose and fix free generators \(\tilde\gamma_1, \tilde\gamma_2\) of \(\Gamma\)
such that the image of \(\tilde\gamma_i\) in \(\Gamma/[\Gamma,\Gamma]\cong 
\Lambda\) is \(\omega_i\) for \(i=1, 2\).
Let \(\mathfrak{c}_0\in \HH^{\ast}\) be the unique cusp such that
\(\,z^{\ast}(\mathfrak{c}_0)=0\) and the 
stabilizer subgroup \(\,\textup{Stab}_{\Gamma}(\mathfrak{c}_0)\)
of \(\mathfrak{c}_0\) in \(\Gamma\) is equal to the
cyclic subgroup generated by the commutator 
\([\tilde\gamma_1,\tilde\gamma_2]
:=\,\tilde\gamma_1 \tilde\gamma_2 \tilde\gamma_1^{-1}\tilde\gamma_2^{-1}\).
It follows that the inverse image 
\(\,z^{\ast}(0)\,\) of \(\,0\in \CC\,\) under 
\(\,z^{\ast}\!:\!\HH^{\ast}\to \CC\,\) is equal to
\([\Gamma,\Gamma]\!\cdot\!\mathfrak{c}_0\).


\subsubsection{}
According to Proposition \ref{prop:liouville}, for each \(k\)
there exists a meromorphic function
\(\mathbf{f}_k(\xi)\) on \(\HH\) such that
\begin{equation} \label{dev.map}
u_k\circ{z} = \log \frac{8 |\tfrac{d}{dz}
{\mathbf{f}_k}|^2}{(1 + |{\mathbf{f}_k}|^2)^2}.
\end{equation}
The Schwarzian derivative 
\[S(\mathbf{f}_k)=\frac{\tfrac{d^3}{dz^3}{\mathbf{f}_k}}{\tfrac{d}{dz}\mathbf{f}_k}-
\frac{3}{2}\left(\frac{\tfrac{d^2}{dz^2}\mathbf{f}_k}{
\tfrac{d}{dz}\mathbf{f}_k}\right)^2
\]
of $\Phi$ is equal to
\begin{equation*}
S(\mathbf{f}_k) = \frac{d^2u_k}{dz^2}\circ{z} - 
\frac{1}{2} \Big(\frac{du_k}{dz}\Big)^2 = -2(\eta_k (\eta_k + 1) \wp(z;\Lambda) + B_k)
\end{equation*}
for some constant $B_k$. 
Thus $\mathbf{f}_k$ can be written as a ratio of two independent solutions 
of the Lam\'e equation
\begin{equation*}
\frac{d^2w}{dz^2} = \big(\eta_k (\eta_k + 1) \wp(z;\Lambda) + B_k\big) w.
\end{equation*}

\subsubsection{}\label{subsub:branch-of-sol}
Choose and fix a branch of \(\log z\) on \(\HH\), i.e.\ 
a holomorphic function on \(\HH\) whose exponential is equal
to the function \(\,z:\HH\to \CC\!\smallsetminus\! \Lambda\);
we abuse the notation and denote this function again by \(\log z\).

The indicial equation of the Lam\'e equation above is given by
\begin{equation*}
\lambda^2 - \lambda - \eta_k(\eta_k + 1) 
= (\lambda - (\eta_k + 1))(\lambda + \eta_k) = 0.
\end{equation*}
The difference $\eta_k + 1 - (-\eta_k) = 2\eta_k + 1$ 
of the two roots of the above indicial equation is \emph{not} an integer 
because of the assumption that \(\vert\eta_k-n\vert<\tfrac{1}{2}\).
Hence there exist two linearly independent solutions $w_{k,1}, w_{k,2}$ on \(\HH\)
which near \(\mathfrak{c}_0\) are of the form 
\begin{equation*}
w_{k,1}= e^{(\eta_k + 1)\log z}\!\cdot\! (h_{k,1}\circ z),
\qquad 
w_{k,2}= e^{-\eta_k\log z}\!\cdot\! (h_{k,2}\circ  z),
\end{equation*}
where $h_{k,1}(z)$ and $h_{k,2}(z)$ are 
holomorphic functions
in an open neighborhood of \(z=0\) 
with the property that  $h_{k,1}(0) = 1= h_{k,2}(0)$. 
The quotient 
\begin{equation} \label{f1}
\mathfrak{f}_k:= \frac{w_{k,1}}{w_{k,2}} 
\end{equation}
is a meromorphic function \(\mathfrak{f}_k\)
on \(\HH\) such that
\begin{equation}\label{fk-lim}
\mathfrak{f}_k(\mathfrak{c}_0) := \lim_{\xi \to \mathfrak{c}_0} 
\mathfrak{f}_k(\xi) = 0
\end{equation}
and 
\begin{equation}\label{fk-monodromy}
\mathfrak{f}_k(\tilde\gamma_2^{-1}
\tilde\gamma_1^{-1}\tilde\gamma_2\tilde\gamma_1\cdot \xi)
= e^{4\pi\sqrt{-1}\eta_k}\mathfrak{f}_k(\xi)
\qquad \forall\xi\in\HH.
\end{equation}
%

\begin{lemmass}\label{lemma:Tk}
Let \(T_k\in \textup{PSL}_2(\CC)\) be the linear fractional
transformation such that 
\(\,\mathbf{f}_k= T_k\cdot \mathfrak{f}_k\).
The limit \[\,\lim_{\xi\to \mathfrak{c}_0}\mathbf{f}_k(\xi)
=:\mathbf{f}_k(\mathfrak{c}_0)\,\]
exists in \(\PP^1(\CC)\)
and is equal to \(T_k\!\cdot\! 0\),
the image of \(0\in\PP^1(\CC)\) under the 
linear fractional transformation \(T_k\).

\end{lemmass}

\begin{lemmass} \label{f(0)}
\textup{(a)} If $\mathbf{f}_k(0) = 0$, then there exists a 
constant \(A_k\in \CC^{\times}\) such that 
$\mathbf{f}_k = A_k\cdot \mathfrak{f}_k$. 
\smallbreak
\noindent
\textup{(b)}
If $\mathbf{f}_k(0) = \infty$, then there exists a 
constant \(A_k\in \CC^{\times}\) such that 
then $\mathbf{f}_k = A_k/\mathfrak{f}_k$.
\end{lemmass}

\begin{proof}
Let \(T_k\) be the linear fractional transformation
such that \(\,\mathbf{f}_k= T_k\cdot \mathfrak{f}_k\)
as in the proof of Lemma \ref{lemma:Tk}.


If $\mathbf{f}_k(\mathfrak{c}_0) = 0=\mathfrak{f}_k(0)$, then
there exists a unique element \(A_k\in\CC^{\times}\) such that
\(T_k\) is the image of \(\begin{pmatrix}A_k&0\\0&1
\end{pmatrix}\); the statement (a) follows.
If $\mathbf{f}_k(\mathfrak{c}_0)=\infty$, then there exist
a unique element \(A_k\in\CC^{\times}\) 
such that
\(T_k\) is the image of \(\begin{pmatrix}0&A_k\\1&0
\end{pmatrix}\). We have proved (b).
\end{proof}
%
%
%
%
%
%

\subsection{Normalizing \(\mathbf{f}_k\)'s through monodromy}
\subsubsection{}
Let \(\rho_{\mathbf{f}_k}:\Gamma\to \textup{PSU}(2)\) be the 
monodromy representation attached to the developing map
\(\,\mathbf{f}_k\,\) of the solution \(u_k\), defined by
\[
\mathbf{f}_k(\gamma\cdot\xi)=
\rho_{\mathbf{f}_k}(\gamma)\cdot\mathbf{f}_k(\xi)\qquad
\forall \gamma\in \Gamma,\ \forall\xi\in \HH.
\]
Note that \[\rho_{\mathbf{f}_k}(\gamma_1\cdot \gamma_2)
=\rho_{\mathbf{f}_k}(\gamma_1)\cdot \rho_{\mathbf{f}_k}(\gamma_2)
\qquad \forall\,\gamma_1,\gamma_2\in \Gamma.\]
Let 
\(S_{k,i} = \rho_{\mathbf{f}_k}(\tilde\gamma_i)\in 
\textup{PSU}(2)\) for \(i=1,2\).
Then we have
\begin{equation} \label{monodromy}
\mathbf{f}_k(\tilde\gamma_i\cdot \xi)= S_{k,i}\cdot \mathbf{f}_k(\xi)
\qquad \textup{for}\ i=1, 2,\ \forall\,\xi\in \HH.
\end{equation}
Let \begin{equation} \label{pi1-relation}
\beta_k:=S_{k,2}\cdot S_{k,1}\cdot S_{k,2}^{-1}\cdot S_{k,1}^{-1}
=\rho_{\mathbf{f}_k}([\tilde\gamma_2,\tilde\gamma_1])
\in \textup{PSU}(2).
\end{equation}

\subsubsection{}\label{subsubsec:impose-cond}
So far we have not imposed any restriction on the 
developing map \(\mathbf{f}_k\) of the solution \(u_k\)
of \(\triangle u+ e^u= \rho_k\cdot \delta_0\). 
Modifying \(\,\mathbf{f}_k\,\) by a suitable element of 
\(\textup{PSU}(2)\), we may and do assume that
\(S_{k,1}=\rho_{\mathbf{f}_k}(\tilde\gamma_1)\) lies in the
diagonal maximal torus of \(\textup{PSU}(2)\), i.e.\
there exists \(\theta_k\in \RR/2\pi\sqrt{-1}\ZZ\)
and \(a_k, b_k\in \CC\) 
with \(|a_k|^2 + |b_k|^2 = 1\) such that
\begin{equation}\label{def:Sk}
S_{k,1} = \begin{pmatrix} e^{\sqrt{-1} \cdot\theta_k} & 0 \\ 0 & e^{-\sqrt{-1}
 \cdot   \theta_k} \end{pmatrix} \quad \textup{and}
\quad S_{k,2} = \begin{pmatrix} a_k & - b_k 
\\ \overline{b_k} & \overline{a_k} \end{pmatrix},
\end{equation}
in matricial notation.
Note that we have
\begin{equation}\label{betak}
\beta_k:=[S_{k,2},S_{k,1}]=\begin{pmatrix}
a_k\overline{a_k}+e^{-2\sqrt{-1}\theta_k}b_k\overline{b_k}&
a_kb_k (-1+e^{2\sqrt{-1}\theta_k})
\\
\overline{a_k}\overline{b_k}(1-e^{-2\sqrt{-1}\theta_k})&
b_k\overline{b_k}e^{2\sqrt{-1}\theta_k}+a_k\overline{a_k}
\end{pmatrix}
\end{equation}
in \(\textup{PSU}(2)\).

\begin{lemmass}\label{lemma:fk-nonzero}
Recall that we have assumed that \(S_{k,1} \mathbf{f}_k = e^{2i\theta_k}
\mathbf{f}_k\), 
\(n-\tfrac{1}{2}<\eta_k:=\rho_k/8\pi
<n+\tfrac{1}{2}\) and \(\eta_k\neq n\).
Suppose in addition that \(n-\tfrac{1}{4}<\eta_k
<n+\tfrac{1}{4}\), then
\(\mathbf{f}_k(\mathfrak{c}_0) \in \CC^{\times}\).
\end{lemmass}

\begin{proof}
We need to show that \(\mathbf{f}_k(\mathfrak{c}_0)\) is not equal
to \(0\) nor to \(\infty\).
Suppose first that $\mathbf{f}_k(\mathfrak{c}_0) = 0$. 
By Lemma \ref{f(0)}\,(a) 
there exists a constant \(A_k\in\CC^{\times}\) such that
\(\mathbf{f}_k=A_k\cdot \mathfrak{f}_k\).
The monodromy relation (\ref{fk-monodromy}) for \(\mathfrak{f}_k\)
implies that
\[
[S_{k,2},S_{k,1}]=
\rho_{\mathbf{f}_k}([\tilde\gamma_2,\tilde\gamma_1])
=\begin{pmatrix}
e^{2\pi\sqrt{-1}\eta_k}&0\\0&e^{2\pi\sqrt{-1}\eta_k}
\end{pmatrix}
\]
in \(\textup{PSU}(2)\).
Comparing with (\ref{betak}), 
we get \(a_k\cdot b_k\cdot (e^{2\sqrt{-1}\theta_k}-1)=0\).
But we know that \(b_k\neq 0\) and \(e^{2\sqrt{-1}\theta_k}\neq 1\),
for otherwise \(S_{k,1}\) would commute with \(S_{k,2}\), 
contradicting the assumption on \(\eta_k\).
We conclude that \(a_k=0\).  In other words
becomes
\begin{equation}\label{normalized-mono-again}
\mathbf{f}_k(\tilde\gamma_1 \xi) = e^{2 i \theta_k} \f_k(\xi), 
\quad \f_k(\tilde\gamma_2 \xi) = - \frac{b_k^2}{\f_k(\xi)}
\quad\forall\,\xi\in\HH.
\end{equation}
Therefore the logarithmic derivative
\begin{equation*}
\mathbf{g}_k := \frac{d}{dz}(\log \mathbf{f}_k) 
= \frac{\tfrac{d\mathbf{f}_k}{dz}}{\mathbf{f}_k} 
\end{equation*}
of \(\mathbf{f}_k\) descends to a 
meromorphic function \(g_k\) on the elliptic curve
$E' := \mathbb C/\Lambda'$, 
where $\Lambda' = \mathbb Z \omega_1 + \mathbb Z 2\omega_2$. 
Moreover we know that \(g_k\) has a simple pole at 
\(0\,\textup{mod}\,\Lambda'\) with residue \(\,2\eta_k+1\),
and the equation \eqref{normalized-mono-again} tells us that
\(g_k\) has a simple at \(\omega_2\,\textup{mod}\,\Lambda'\) as well.
On the other hand because \(\mathbf{f}_k\) is a developing map
of a solution \(u_k\) to the equation
\(\triangle u +e^u=8\pi\eta_k\,\delta_0\), the meromorphic function
\(\mathbf{f}_k\) has multiplicity \(1\) at all points
above \([0]=0\,\textup{mod}\,\Lambda\).  
So the meromorphic function \(g_k\) on \(E'\) has two simple
poles but no zero, a contradiction.
We have proved that \(\mathbf{f}_k(\mathfrak{c}_0)\neq 0\).
%
%
\medbreak

Suppose that $\f_k(\mathfrak{c}_0) = \infty$. By Lemma
(\ref{f(0)})\,(b)
there exists a constant \(A_k\in\CC^{\times}\) such that
\(\mathbf{f}_k=A_k/\mathfrak{f}_k\).
The same argument as in the previous case shows that
the logarithmic derivative of \(\mathbf{f}_k\) descends
to a meromorphic function on \(\CC/\Lambda'\) which has
at least two simple poles but no zero, a contradiction again.
\end{proof}


\begin{lemmass}
Notation and assumptions as in \textup{\ref{subsubsec:impose-cond}}
and Lemma \textup{\ref{lemma:fk-nonzero}}.
Let \(\gamma_k\in\textup{PSU}(2)\) be the element
\begin{equation}\label{matrix-gamma}
\gamma_k:=
\begin{pmatrix}
e^{2\pi\sqrt{-1}\eta_k}&0\\0&e^{2\pi\sqrt{-1}\eta_k}
\end{pmatrix}
\end{equation}
in \(\textup{PSU}(2)\).
Let \(p_k, q_k\) be elements of \(\CC^{\times}\)
such that \[\vert p_k\vert^2+\vert q_k\vert ^2=1\quad
\textup{and}\quad
\mathbf{f}_k(\mathfrak{c}_0)=\frac{q_k}{p_k}.
\]
Let \begin{equation}\label{Tk}
T_k:=\begin{pmatrix}p_k&-q_k\\
\overline{q_k}&\overline{p_k}
\end{pmatrix}
\end{equation}
\begin{itemize}
\item[\textup{(1)}] The equality 
\begin{equation}\label{beta-matrix}
[S_{k,2}, S_{k,1}]= T_k^{-1}\cdot \gamma_k\cdot T_k
\end{equation}
holds in \(\textup{PSU}(2)\).

\item[\textup{(2)}]  Explicitly, the equality in \textup{(1)} 
means that either 
\begin{equation}\label{determine1}
\begin{split}
[1-(\vert p_k\vert^2 e^{\scalebox{0.5}{$-2\sqrt{-1}$}\eta_k}
+ \vert q_k\vert^2 e^{\scalebox{0.5}{$2\sqrt{-1}$}\eta_k})] a_k
&= \overline{p_k}q_k 
(e^{\scalebox{0.5}{$2\sqrt{-1}$}\eta_k}-e^{\scalebox{0.5}{$-2\sqrt{-1}$}\eta_k}) 
\overline{b_k}
\\
[e^{\scalebox{0.5}{$2\sqrt{-1}$}\theta_k}
-(\vert p_k\vert^2 e^{\scalebox{0.5}{$-2\sqrt{-1}$}\eta_k}
\!+\! \vert q_k\vert^2 e^{\scalebox{0.5}{$2\sqrt{-1}$}\eta_k})] b_k
&=-\overline{p_k}q_k 
(e^{\scalebox{0.5}{$2\sqrt{-1}$}\eta_k}\!-\!e^{\scalebox{0.5}{$-2\sqrt{-1}$}\eta_k})
\overline{a_k}
\end{split}
\end{equation}
or
\begin{equation}\label{determine2}
\begin{split}
[1+(\vert p_k\vert^2 e^{\scalebox{0.5}{$-2\sqrt{-1}$}\eta_k}
+ \vert q_k\vert^2 e^{\scalebox{0.5}{$2\sqrt{-1}$}\eta_k})] a_k
&= -\overline{p_k}q_k 
(e^{\scalebox{0.5}{$2\sqrt{-1}$}\eta_k}-e^{\scalebox{0.5}{$-2\sqrt{-1}$}\eta_k}) 
\overline{b_k}
\\
[e^{\scalebox{0.5}{$2\sqrt{-1}$}\theta_k}
+(\vert p_k\vert^2 e^{\scalebox{0.5}{$-2\sqrt{-1}$}\eta_k}
\!+\! \vert q_k\vert^2 e^{\scalebox{0.5}{$2\sqrt{-1}$}\eta_k})] b_k
&=\overline{p_k}q_k 
(e^{\scalebox{0.5}{$2\sqrt{-1}$}\eta_k}\!-\!e^{\scalebox{0.5}{$-2\sqrt{-1}$}\eta_k})
\overline{a_k}
\end{split}
\end{equation}
In both cases we have \(\,a_k\cdot b_k\neq 0\), for all \(k\).

\item[\textup{(3)}] If \textup{(\ref{determine1})} holds, then
\[\vert b_k\vert^2=\frac{
\vert\! \left( 1-\vert p_k\vert^2 e^{\scalebox{0.5}{$-2\sqrt{-1}$}\eta_k}
- \vert q_k\vert^2 e^{\scalebox{0.5}{$2\sqrt{-1}$}\eta_k}\right)\!\vert^2}{
\vert{p_k}\vert^2 \vert q_k \vert^2\,
\vert(e^{\scalebox{0.5}{$2\sqrt{-1}$}\eta_k}\!-\!e^{\scalebox{0.5}{$-2\sqrt{-1}$}\eta_k})
\vert^2
+\vert\! \left( 1-\vert p_k\vert^2 e^{\scalebox{0.5}{$-2\sqrt{-1}$}\eta_k}
- \vert q_k\vert^2 e^{\scalebox{0.5}{$2\sqrt{-1}$}\eta_k}\right)\!\vert^2
}\,,
\]
\[\vert a_k\vert^2=\frac{
\vert{p_k}\vert^2 \vert q_k \vert^2\,
\vert(e^{\scalebox{0.5}{$2\sqrt{-1}$}\eta_k}\!-\!e^{\scalebox{0.5}{$-2\sqrt{-1}$}\eta_k})
\vert^2
}{
\vert{p_k}\vert^2 \vert q_k \vert^2\,
\vert(e^{\scalebox{0.5}{$2\sqrt{-1}$}\eta_k}\!-\!e^{\scalebox{0.5}{$-2\sqrt{-1}$}\eta_k})
\vert^2
+\vert\! \left( 1-\vert p_k\vert^2 e^{\scalebox{0.5}{$-2\sqrt{-1}$}\eta_k}
- \vert q_k\vert^2 e^{\scalebox{0.5}{$2\sqrt{-1}$}\eta_k}\right)\!\vert^2
}\,,
\]
and 
\[
e^{\scalebox{0.5}{$2\sqrt{-1}$}\theta_k}
=-\frac{1-\vert p_k\vert^2  e^{\scalebox{0.5}{$-2\sqrt{-1}$}\eta_k}
- \vert q_k\vert^2  e^{\scalebox{0.5}{$2\sqrt{-1}$}\eta_k}}{1
-\vert p_k\vert^2  e^{\scalebox{0.5}{$2\sqrt{-1}$}\eta_k}
- \vert q_k\vert^2  e^{\scalebox{0.5}{$-2\sqrt{-1}$}\eta_k}}\,,
\]
so \(\,\vert a_k\vert^2, \vert b_k\vert^2\,\)
and \(\,e^{\scalebox{0.5}{$2\sqrt{-1}$}\theta_k}\,\)
are all determined by
\(e^{2\sqrt{-1}\eta_k}\) and \(\mathbf{f}_k(\mathfrak{c}_0)\).
In addition  \(\,a_k\!\cdot {b_k}\) is
also determined by
\(e^{2\sqrt{-1}\eta_k}\) and \(\mathbf{f}_k(\mathfrak{c}_0)\).

\item[\textup{(4)}] If \textup{(\ref{determine2})} holds, then
\[\vert b_k\vert^2=\frac{
\vert\! \left( 1+\vert p_k\vert^2 e^{\scalebox{0.5}{$-2\sqrt{-1}$}\eta_k}
+ \vert q_k\vert^2 e^{\scalebox{0.5}{$2\sqrt{-1}$}\eta_k}\right)\!\vert^2}{
\vert{p_k}\vert^2 \vert q_k \vert^2\,
\vert(e^{\scalebox{0.5}{$2\sqrt{-1}$}\eta_k}\!-\!e^{\scalebox{0.5}{$-2\sqrt{-1}$}\eta_k})
\vert^2
+\vert\! \left( 1+\vert p_k\vert^2 e^{\scalebox{0.5}{$-2\sqrt{-1}$}\eta_k}
+ \vert q_k\vert^2 e^{\scalebox{0.5}{$2\sqrt{-1}$}\eta_k}\right)\!\vert^2
}\,,
\]
\[\vert a_k\vert^2=\frac{
\vert{p_k}\vert^2 \vert q_k \vert^2\,
\vert(e^{\scalebox{0.5}{$2\sqrt{-1}$}\eta_k}\!-\!e^{\scalebox{0.5}{$-2\sqrt{-1}$}\eta_k})
\vert^2
}{
\vert{p_k}\vert^2 \vert q_k \vert^2\,
\vert(e^{\scalebox{0.5}{$2\sqrt{-1}$}\eta_k}\!-\!e^{\scalebox{0.5}{$-2\sqrt{-1}$}\eta_k})
\vert^2
+\vert\! \left( 1+\vert p_k\vert^2 e^{\scalebox{0.5}{$-2\sqrt{-1}$}\eta_k}
+ \vert q_k\vert^2 e^{\scalebox{0.5}{$2\sqrt{-1}$}\eta_k}\right)\!\vert^2
}\,,
\]
and 
\[
e^{\scalebox{0.5}{$2\sqrt{-1}$}\theta_k}
=-\frac{1+\vert p_k\vert^2  e^{\scalebox{0.5}{$-2\sqrt{-1}$}\eta_k}
+ \vert q_k\vert^2  e^{\scalebox{0.5}{$2\sqrt{-1}$}\eta_k}}{1
+\vert p_k\vert^2  e^{\scalebox{0.5}{$2\sqrt{-1}$}\eta_k}
+ \vert q_k\vert^2  e^{\scalebox{0.5}{$-2\sqrt{-1}$}\eta_k}}\,,
\]
so \(\,\vert a_k\vert^2, \vert b_k\vert^2\,\)
and \(\,e^{\scalebox{0.5}{$2\sqrt{-1}$}\theta_k}\,\)
are all determined by
\(e^{2\sqrt{-1}\eta_k}\) and \(\mathbf{f}_k(\mathfrak{c}_0)\).
In addition 
\(\,a_k\!\cdot {b_k}\) 
is also determined by
\(e^{2\sqrt{-1}\eta_k}\) and \(\mathbf{f}_k(\mathfrak{c}_0)\).


\end{itemize}

\end{lemmass}

\begin{proof} (1)
Clearly \(\,\lim_{\xi\to \mathfrak{c}_0}(T_k\cdot
\mathbf{f}_k)(\xi)=0\).  By Lemma \ref{f(0)} there exists a constant
\(c_k\in \CC^{\times}\) such that 
\(\,T_k\cdot \mathbf{f}_k= c_k\,\mathfrak{f}_k\).
We have
\[
T_k [S_{k,2}, S_{k,1}]\mathbf{f}_k(\xi)
= T_k\mathbf{f}_k([\tilde\gamma_2,\tilde\gamma_1]\,\xi)
= c_k\,\mathfrak{f}_k([\tilde\gamma_2,\tilde\gamma_1]\,\xi)
=c_k\,\gamma_k\cdot \mathfrak{f}_k(\xi)
\]
for all \(\xi\in \HH\), therefore
\[
T_k [S_{k,2}, S_{k,1}]\mathbf{f}_k=c_k\,\gamma_k\cdot \mathfrak{f}_k
= \gamma_k\cdot (c_k \mathfrak{f}_k)
= \gamma_k\cdot T_k\cdot \mathbf{f}_k.
\]
So \(T_k [S_{k,2}, S_{k,1}]=\gamma_k\cdot T_k\) in
\(\textup{PSU}(2)\).  We have proved the statement (1).
\medbreak

\noindent
(2) The equality
\([S_{k,2}, S_{k,1}]= T_k^{-1}\cdot \gamma_k\cdot T_k\)
in \(\textup{PSU}(2)\) is equivalent to the equality
\[
S_{k,1}\cdot S_{k,2}\cdot S_{k,1}^{-1}
=\pm S_{k,2}^{-1}\cdot T_k^{-1}\cdot \gamma_k\cdot T_k
\]
when both sides are regarded as elements of \(\textup{SU}(2)\)
with \(S_{k,1}, S_{k,2}\) given by (\ref{def:Sk}) and
\(T_k\) given by (\ref{Tk}).
A straightforward calculation shows that
\[
T_k^{-1}\cdot \gamma_k\cdot T_k =
\begin{pmatrix}
|p_k|^2 \, e^{\scalebox{0.5}{$2\sqrt{-1}$}\eta_k} 
+ |q_k|^2\, e^{\scalebox{0.5}{$-2\sqrt{-1}$}\eta_k} & 
-\overline{p_k}\, q_k (e^{\scalebox{0.5}{$2\sqrt{-1}$}\eta_k} 
- e^{\scalebox{0.5}{$-2\sqrt{-1}$}\eta_k}) \\
p_k \overline{q_k }( e^{\scalebox{0.5}{$-2\sqrt{-1}$}\eta_k} 
-  e^{\scalebox{0.5}{$2\sqrt{-1}$}\eta_k}) 
& |p_k|^2  e^{\scalebox{0.5}{$-2\sqrt{-1}$}\eta_k} 
+ |q_k|^2  e^{\scalebox{0.5}{$2\sqrt{-1}$}\eta_k}
\end{pmatrix}
\]
and
\[
S_{k,1}\cdot S_{k,2}\cdot S_{k,1}^{-1}
=\begin{pmatrix}
\overline{a_k}&b_k\,e^{\scalebox{0.5}{$2\sqrt{-1}\theta_k$}}
\\
-\overline{b_k}\,e^{\scalebox{0.5}{$-2\sqrt{-1}\theta_k$}}&a_k
\end{pmatrix}
\]
The statement (2) follows: the equation (\ref{determine1})
corresponds to the case \(S_{k,1}\cdot S_{k,2}\cdot S_{k,1}^{-1}
= S_{k,2}^{-1}\cdot T_k^{-1}\cdot \gamma_k\cdot T_k\),
while the equation (\ref{determine2})
corresponds to the case \(S_{k,1}\cdot S_{k,2}\cdot S_{k,1}^{-1}
= -S_{k,2}^{-1}\cdot T_k^{-1}\cdot \gamma_k\cdot T_k\).
\medbreak

The formulas in (3) and (4) for \(\,\vert a_k\vert\,\) 
and \(\,\vert b_k\vert\,\) follow from
(\ref{determine1}) and (\ref{determine2}) respectively,
through routine calculations which are omitted here.
Notice that these formulas imply that \(\cdot b_k\neq 0\).
Recall that we already know that \(a_k\neq 0\), for otherwise
\(\,S_{k,2}\,\) would commute with \(\,S_{k,1}\). 
We also know that \(p_k\cdot q_k\neq 0\), 
for \(\mathbf{f}_k(\mathfrak{c}_0)\neq 0, \infty\).
The assumptions on \(\eta_k\) implies that
\(
\left\vert \vert p_k\vert^2 e^{\scalebox{0.5}{$-2\sqrt{-1}$}\eta_k}
+ \vert q_k\vert^2 e^{\scalebox{0.5}{$2\sqrt{-1}$}\eta_k}
\right\vert <1\),
so we know from (\ref{determine1}) and (\ref{determine2})
that \(a_k\neq 0\) in both cases.
\medbreak

The formulas in (3) and (4) for \(\,e^{2\sqrt{-1}\theta_k}\,\)
follows from (\ref{determine1}) and (\ref{determine2}) respectively.
For instance if we multiply the second equation in (\ref{determine1})
by the complex conjugate of the first equation in (\ref{determine1}),
cancel out the non-zero factor \(\,\overline{a_k}b_k\,\)
on both sides, then we get the formula for 
\(\,e^{2\sqrt{-1}\theta_k}\,\) in (3).  
\medbreak

These formulas clearly show that in both cases
\(|a_k|, |b_k|\) and \(e^{2\pi\sqrt{-1}\theta_k}\,\) are determined by
\(e^{2\pi\sqrt{-1}\eta_k}\) and \(\mathfrak{f}_k(\mathfrak{c}_0)\)
and not on the choices of \(p_k\) and \(q_k\).
That 
\(a_k\cdot b_k\) is
also determined by \(e^{2\pi\eta_k}\) 
and \(\mathfrak{f}_k(\mathfrak{c}_0)\) in each 
of the two cases follows quickly from
(\ref{determine1}) and (\ref{determine2}).
\end{proof}

\bigbreak

We are ready to prove the remaining case (1) of Theorem \ref{blow-up-set}.
Let's recall the situation. 
We are given a $\,(u_k)_k $ is a blow-up sequence of solutions such that
\(\,\triangle u_k+e^u=8\pi \eta_k\,\delta_0\,\) on \(E=\CC/\Lambda\) for
each \(k\), 
$\,\lim_{k\to\infty}\eta_k=n$ with 
\(\{P_1,\ldots,P_m\}\) as the blow-up set of this 
blow-up sequence.
We may and do assume that \(n-\tfrac{1}{4}<\eta_k<n+\tfrac{1}{4}\)
and \(\eta_k\neq n\) for all \(k\).
We know from \cite{CL0} that
\(m = n\), $P_i \ne [0]$ for each \(i=1\),  
and $\lim_{k\to\infty}u_k(P) \to -\infty$ 
$P \in E\! \smallsetminus \!\{P_1,\ldots, P_n\}$, 
uniformly on compact subsets of \(\smallsetminus \!\{P_1,\ldots,
P_n\}\).


\begin{theorem}\label{blow-up thm}
There is a constant $A \in \CC^{\times}$ such that 
\(\lim_{k\to\infty}\mathbf{f}_k(\xi) \to A\)
uniformly on compact subsets of the inverse image 
\(\HH \smallsetminus [z]^{-1}(\{[0], P_1, \ldots, P_n\})\)
of \(E\!\smallsetminus\!\{[0], P_1, \ldots, P_n\}\).
Furthermore we have 
$\{P_1, \ldots, P_n\} = \{-P_1, \ldots, -P_n\}$.
\end{theorem}

\subsection{Proof of Theorem \ref{blow-up thm}} \enspace

\subsubsection{}
We first note that there exists a constant \(B\in \CC\) such that 
the Lam\'e equations
\begin{equation} \label{Lame-k}
w'' = \big(\eta_k (\eta_k + 1) \wp + B_k\big) w
\end{equation}
converge to
\begin{equation}\label{Lame-lim}
w'' = \big(n (n + 1) \wp + B_k\big) w
\end{equation}
because $\lim_{k\to\infty}\eta_k =n$ and 
$\lim_{k\to n}B_k = B$.  This is a consequence of 
the fact that $u_{k, z}$ and $u_{k, zz}$ converge uniformly 
on compact sets in $E \!\smallsetminus\! \{p_1, \ldots, p_n\}$.

\subsubsection{}
For each \(k\) choose a normalized developing map $\f_k$ of \(u_k\) as in 
\ref{subsubsec:impose-cond}, and choose
 \(p_k, q_k\in \CC^{\times}\) such that \(\,\vert p_k\vert^2 +
\vert q_k\vert^2=1\) and
\(\mathbf{f}_k(\mathfrak{c}_0)={q_k}/{p_k}\).
By Lemma \ref{f(0)} there exists for each \(k\) a constant
\(c_k\in\CC^{\times}\) such that
\begin{equation*}
\frac{p_k \f_k - q_k}{\bar q_k \f_k + \bar p_k} = c_k \mathfrak{f}_{k},
\end{equation*}
where \(\mathfrak{f}_k={w_{k,1}}/{w_{k,2}}\) is a solution of the Lam\'e equation
(\ref{Lame-k}) on the universal covering \(\HH\) of 
\(E\!\smallsetminus\!\{[0]\}\)
defined in \ref {subsub:branch-of-sol}.
Equivalently,
\begin{equation} \label{limit}
\f_k = \frac{\overline{p_k} c_k \mathfrak{f}_{k} + q_k}{
-\overline{q_k} c_k \mathfrak{f}_{k} + p_k}.
\end{equation}
The convergence of the Lam\'e equations (\ref{Lame-k})
implies that after passing to a subsequence if necessary, 
there exists solutions \(w_1, w_2\) of (\ref{Lame-lim}) on 
\(\HH\) such that
$\lim_{k\to \infty}w_{k, i} = w_i$ for \(i=1, 2\) and 
$\lim_{k\to \infty}\mathfrak{f}_{k} \to \mathfrak{f} := w_1/w_2$,
uniformly on compact set in $E$  away from the discrete set of
poles and zeros 
of $w_1$ and $w_2$. Clearly 
$$
w_i'' = (n(n + 1) \wp + B) w_i,
$$ 
and locally near $\mathfrak{c}_0$ the \(w_i\)'s can be written
in the form
\[w_1(\xi) = e^{(n+1)\log z}\cdot (h_1\circ z),
\quad
w_2(\xi) = e^{-n\log z}\cdot  (h_2\cdot \circ z)
\]
where \(h_1, h_2\) holomorphic functions in 
a neighborhood of \(\mathfrak{c_0}\) such that
\[\lim_{\xi\to\mathfrak{c}_0}  h_i(\xi) = 1\qquad\textup{for}\ i=1, 2.
\]
Most of our analysis will be based on the limiting behavior of 
(\ref{limit}) as \(k\to\infty\).

Again passing to a subsequence if necessary, we may and do assume that
there exist \(p,q\in \CC\) with $|p|^2 + |q|^2 = 1$ such that
$p_k \to p$ and $q_k \to q$ as \(k\to \infty\).
Similarly we may and do assume that 
there exists \(a, b\in\CC\) with \(\vert a\vert^2+\vert b\vert^2=1\)
such that \(a_k\to a\) and \(b_k\to b\) as 
\(k\to\infty\).
Let \(A:={q}/{p}\in\PP^{1}(\CC)\). Clearly 
\[\lim_{k\to\infty} \mathbf{f}_k(0)=\lim_{k\to\infty} q_k/p_k= A.\]
We may and do assume more over that 
the limit \(\lim_{k\to \infty}\) exists in \(\PP^1(\CC)\);
let \[c:=\lim_{k\to\infty}c_k.\]

\subsubsection{}
Our first claim is that \(c\) is either \(0\) or \(\infty\).
Suppose that $c\in\CC^{\times}$.
Let 
\begin{equation*} 
\mathbf{f}
:= \frac{\bar p c \mathfrak{f} + q}{-\bar q c \mathfrak{f} + p}.
\end{equation*}
Then
\(\,
\mathbf{f}_k(\xi) \longrightarrow \mathbf{f}(\xi)
\,
\)
for all \(\xi\) outside some discrete subset \(\Sigma\) of \(\HH\).
Hence
\begin{equation*}
u_k\circ z \to u\circ z = \log \frac{8 |f'(z)|^2}{(1 + |f(z)|^2)^2}
\end{equation*}
uniformly on compact sets outside some discrete set of \(\HH\).
This contradicts to our assumption that $u_k$ blows up as $k \to
\infty$, 
We have proved that either $c = 0$ or $c = \infty$.

\subsubsection{}\label{subsub:claim2}
Next we claim that
\[
A\in \CC^{\times}
\qquad  \textup{and}\qquad  c=0.
\]
This claim and the equation (\ref{limit}) imply that  
$\lim_{n\to\infty}\mathbf{f}_k(\xi) = A$ 
for all \(\xi\) outside of some discrete subset of \(\HH\),
and the first statement of Theorem \ref{blow-up thm} follows.

\subsubsection{}\label{subsubsec:Aneqinfty}
We will show that \(A\neq \infty\).
Suppose to the contrary that 
$A = \infty$, i.e.\ $p = 0$ and $|q| = 1$. 
\smallbreak

Our first step is to show that \(a:=\lim_{k\to\infty}a_k=0\).
Write \(\alpha_k:=e^{2\pi\sqrt{-1}\eta_k}\); clearly
\(\lim_{k\to\infty}\alpha_k=1\).
Passing to a subsequence if necessary, we may and do assume
that either (\ref{determine1}) holds for all \(k\)
or (\ref{determine2}) holds for all \(k\).
In the case when (\ref{determine2}) holds for all \(k\), taking the limit
of both sides of the first equation in
(\ref{determine2}), we see immediately that
\(\lim_{k\to\infty}a_k=0\).
In the case when (\ref{determine1}) holds for all \(k\), 
substitute \(\,|p_k|^2\,\) by \(1-|q_k|^2\) in 
first equation of (\ref{determine1}), then
divide both sides 
by
\(\,\alpha_k-\alpha_k^{-1}\,\), we get
\[
\left(\frac{1}{\alpha_k+1} + \vert q_k\vert^2\right) a_k
= \overline{p_k}\,q_k\,\overline{b_k}.
\]
Taking the limit of the above equality as \(k\to\infty\),
again we conclude that \(\,a=\lim_{k\to\infty}a_k=0\),
so \(\,|b|=1\).
%
%
\medbreak

We will analyse the two possibilities of \(c\) separately and show that
both lead to contradiction.
Suppose first that $c = 0$.  From (\ref{limit}) it is clear that
\begin{equation*}
\mathfrak{f}(\xi) \ne \infty \quad \Longrightarrow 
\quad |\f_k(\xi)| \longrightarrow \infty.
\end{equation*}
However, if we select $\xi$ so that $\mathfrak{f}(\xi) \ne \infty$ 
and $\mathfrak{f}(g_2 \xi) \ne \infty$ (such $\xi$ certainly exists),
then 
\begin{equation} \label{S2-eq}
\f_k(\tilde\gamma_2 \xi) = S_{k, 2} \f_k(\xi) 
= \frac{a_k \f_k(\xi) - b_k}{\overline{b_k} \f_k(\xi) 
+ \overline{a_k}} \longrightarrow 0,
\end{equation}
which contradicts the previous conclusion
that $|\f_k(\tilde\gamma_2 \xi)| \to \infty$.  So \(c\neq 0\).
\smallbreak

Suppose next that $c = \infty$. Again by (\ref{limit}) we have
\begin{equation*}
\mathfrak{f}(\xi) \ne 0 \quad 
\Longrightarrow \quad \f_k(\xi) \longrightarrow 0,
\end{equation*}
and convergence is uniform on compact sets outside the zeros of
$\mathfrak{f}$. But then $\f_k(\tilde\gamma_2 \xi) \to \infty$, which contradicts 
$\f_k(\tilde\gamma_2 \xi) \to 0$ 
provided that $\mathfrak{f}(\tilde\gamma_2 \xi) \ne 0$. 
We have shown that the assumption \(A=\infty\) leads to
assumption for both possible values of \(c\).
We have proved that $A \ne \infty$, i.e. \(A\in\CC\).

\subsubsection{}
Now we know that $\f_k(0) \to A$ with \(A\in\CC\).
Substituting 
$\mathfrak{f}_k\circ z = z^{2\eta_k + 1} + O(|z|^{2\eta_k +2})$ 
near \(\mathfrak{c}_0\)
into (\ref{limit}), we get
\begin{equation*}
\mathbf{f}_k\circ z 
= \frac{q_k}{p_k} + \frac{c_k}{p_k^2}\, z^{2\eta_k + 1} (1 + O(|z|)).
\end{equation*}
On the other hand we know from general facts about
blow-up sequences that 
the regular part of $u_k(z)$ at $z = 0$ tends to $-\infty$, i.e.
\begin{equation*}
\begin{split}
&\big(u_k\circ z - 4\eta_k \log |z|\big) \Big|_{z = 0} \\
&= \log \frac{8|\tfrac{d}{dz}\mathbf{f}_k|^2 |z|^{-4\eta_k}}{
(1 + |\mathbf{f}_k\circ z|^2)^2}
\Big|_{z = 0} 
= \log \frac{8 |c_k|^2 (2\eta_k + 1)^2}{|p_k|^2 (1 + |\mathbf{f}_k(0)|^2)^2} 
\longrightarrow -\infty,
\end{split}
\end{equation*}
which implies that $c_k \to c = 0$.

\subsubsection{}
We still need to exclude the possibility that $A = 0$. 
Suppose to the contrary that $A = 0$, or equivalently $|p| = 1$ 
and $q = 0$. 
The same argument used at the beginning of \ref{subsubsec:Aneqinfty}
shows that
$a = 0$ and $|b| = 1$. Since $c = 0$, by (\ref{limit}),
\begin{equation*}
\mathfrak{f}(\xi) \ne \infty \quad \Longrightarrow 
\quad \f_k(\xi) \longrightarrow 0.
\end{equation*}
Then by the expression (\ref{S2-eq}) we have $\f_k(\gamma_2 \xi) \to \infty$ 
whenever $\mathfrak{f}(\xi) \ne \infty$,
hence $\mathfrak{f}(\gamma_2 \xi) = \infty$ 
by the above implication. But as before we may select $\xi$ 
so that both $\mathfrak{f}(\xi) \ne \infty$ 
and $\mathfrak{f}(\gamma_2 \xi) \ne \infty$. 
This is a contradiction.  We have proved that \(A\in\CC^{\times}\).

\subsubsection{}
It remains to show that the blow-up set is stable under \([-1]_E\), i.e.
$\{P_1, \ldots, P_n\} = \{-P_1, \ldots, -P_n\}$. 
From (\ref{S-der}), we see that 
\[S(\mathbf{f}_k(-z)) 
= S(\mathbf{f}_k)(-z) = -2(\eta(\eta + 1) \wp(-z) + B_k) 
= S(\mathbf{f}_k(z)),\] 
hence $\mathbf{f}_k(-z)$ is also a developing map for $u_k$. 
Since $\mathbf{f}_k(z)$ is normalized
in the sense of \ref{subsubsec:impose-cond}, 
$\mathbf{f}_k(-z) =\mathbf{f}_k\circ (-z)$ is also normalized,
so our results  applies to
\(\mathbf{f}_k(-z)\) as well.
We conclude that $f_k(-z) \to A$ outside $\{P_1, \ldots, P_n\}$. 
This implies that this set coincides with $\{-P_1, \ldots, -P_n\}$. 
We have prove Theorem \ref{blow-up thm}. \qed

\begin{corollary} \label{branch-pt}
Let $u_k$ be a blow-up sequence of solutions to 
\textup{(\ref{Liouville-eq})} 
with $\rho_k = 8\pi \eta_k \to 8\pi n$, $\eta_k \ne n$ for all \(k\).
Let $S = \{P_1, \ldots, P_n\}$ be the blow-up set. 
Then $S$ is a finite branch point of the hyperelliptic curve $C^2 = \ell_n(B)$.
\end{corollary}

\begin{proof}
The blow-up set $S$ satisfies equations (\ref{nGz}), 
or equivalently (\ref{zeta-eq}), thus $S \in Y_n$. 
Now Theorem \ref{blow-up thm} implies that $\{P_i\}_{i = 1}^n =
\{-P_i\}_{i = 1}^n$, 
hence $S \in Y_n \!\smallsetminus\! X_n$, 
i.e.~it is a branch point of $\bar\pi:\bar X_n\to \PP^1(\CC)$.
\end{proof}

\subsection{Further remarks}

\subsubsection{}
Consider the regular part of the Green function
$$
\tilde G(z, q) := G(z - q) + \frac{1}{2\pi} \log |z - q|.
$$

Let $S = \{P_1, \ldots, P_n\}\subset E$ be a set of \(n\) distinct
points on \(E=\CC/\Lambda\); pick representatives
\(p_1,\ldots,p_n\in\CC\) of \(P_1,\ldots,P_n\). 
For $i = 1, \ldots, n$ we define
\begin{equation*}
\begin{split}
f_{p_i}(z) &:= 8\pi (\tilde G(z, p_i) - \tilde G(p_i, p_i)) \\
&\quad + \sum_{j \ne i} (G(z - p_j) - G(p_i - p_j)) 
- 8\pi n(G(z) - G(p_i)), \\
\mu_i &:= \exp (8\pi (\tilde G(p_i, p_i) 
+ \sum_{j \ne i} G(p_j - p_i)) - 8\pi nG(p_i)).
\end{split}
\end{equation*}
Then we have the important global quantity associated to the set $S$:
$$
D(S) := \lim_{r \to 0} \sum_{i = 1}^n \mu_i 
\left( \int_{\Omega_i \backslash B_r(p_i)} 
\frac{e^{f_{p_i}(z)} - 1}{|z - p_i|^4} 
- \int_{\mathbb R^2 \backslash \Omega_i} \frac{1}{|z - p_i|^4} \right),
$$
where $\Omega_i$ is any open neighborhood of $p_i$ in $E$ such that 
$\Omega_i \cap \Omega_j = \emptyset$ for $i \ne j$, 
and $\bigcup_{i = 1}^n \overline{\Omega}_i = E$.

Under the hypothesis of Corollary \ref{branch-pt}, it was shown 
in \cite{ChenLW} for $n = 1$ and in \cite{LinYan} for general $n \in
\mathbb N$, but in a slightly different context, that
$$
\rho_k - 8\pi n = (D(S) + o(1)) \exp (-\max_T u_k).
$$

In general it is difficult to compute $D(S)$ even for $n = 1$. 
In the case $n = 1$, the hyperelliptic curve is the torus $E$ 
and the branch points consist of the three half-periods. 
In the very special case that $T$ is a rectangular torus, 
the sign of $D(\tfrac{1}{2}\omega_i)$ has been calculated: 
$D(\tfrac{1}{2} \omega_3) < 0$ and $D(\tfrac{1}{2}\omega_i) > 0$ for
$i = 1, 2$. 
Furthermore $D(\tfrac{1}{2} \omega_i) < 0$ if and only if
$\tfrac{1}{2} \omega_i$ 
is a minimal point \cite{ChenLW}.   

It is clear that when $D(S) \ne 0$ its sign provides important
information when we study bubbling solutions (blow-up sequence) 
with $\rho \ne 8\pi n$ (e.g.~if $D(S) > 0$ then the bubbling 
may only occur from the right hand side). 
Also in the case $\rho_k = 8\pi n$ for all $k$, 
if the blow-up family $u_\lambda$ exists then $D(S) = 0$ 
trivially for $S$ begin the blow-up set.

In particular, we pose the following

\begin{conjecturess}\label{last-conj}
For rectangular tori, there are $n$ branch points, 
on the associated hyperelliptic curve, with $D(S) < 0$ 
and $n + 1$ branch points with $D(S) > 0$.
\end{conjecturess}

\noindent
Conjecture \ref{last-conj} is known only when $n = 1$ as mentioned above.

\subsubsection{}
Theorem \ref{blow-up thm} provides some connection between 
mean field equations and the branch points of the associated 
hyperelliptic curve. We expect that it should hold true universally. 
For example we might ask the following question on Chern--Simons-Higgs equation:

Suppose that $u_\epsilon$ is a sequence of bubbling solutions of the Chern--Simons-Higgs equation
$$
\triangle u_\epsilon + \frac{1}{\epsilon^2} e^{u_\epsilon} (1 -
e^{u_\epsilon}) 
= 8\pi n \delta_0 \quad \mbox{in $E$}.
$$ 
Is the bubbling set $\{p_1, \ldots, p_n\}$, as $\epsilon \to 0$,
a branch point of the hyper elliptic curve $C^2 = \ell_n(B)$?

This has recently been answered affirmatively for $n = 1$ and 
for $E$ a rectangular domain \cite{LinYan, LinYan2}.

\end{document}